\documentclass[a4paper,openany,11pt]{book}

\usepackage[utf8]{inputenc}
\usepackage[english]{babel}
\usepackage[T1]{fontenc}
\usepackage[english,cleanlook]{isodate}
\usepackage{lmodern}
\usepackage[babel=true]{microtype}
\usepackage[bottom=120pt,hmarginratio=1:1]{geometry}

\usepackage{graphicx}
\usepackage{caption} 
\usepackage{amsmath}
\usepackage{amsthm}
\usepackage{amssymb}
\usepackage{esint} 
\usepackage{mathrsfs} 
\usepackage{dsfont} 
\usepackage[table]{xcolor}
\usepackage{array}
\usepackage{hhline}
\usepackage{enumitem} 
\usepackage{comment} 
\usepackage[toc,page]{appendix} 
\usepackage{xparse} 
\usepackage{mathtools}
\usepackage{ifthen} 
\usepackage{forloop} 
\usepackage{xstring}
\usepackage{emptypage} 
\usepackage{setspace}
\usepackage[initials,alphabetic]{amsrefs} 
\usepackage{caption} 
\usepackage{multirow}

\usepackage{tikz}
\usepackage{tikz-3dplot}
\usetikzlibrary{arrows,shapes,patterns,calc,fadings,decorations.pathreplacing,decorations.markings,decorations.pathmorphing,backgrounds}

\usepackage{fancyhdr} 
\usepackage{titlesec}
\usepackage{lmodern}
\titleformat{\chapter}[hang]
{\normalfont\huge\scshape}{\thechapter}{1em}{}

\titleformat{\section}[hang]
{\normalfont\bfseries\large}{\thesection}{0.8em}{}
\titleformat{\subsection}[hang]
{\normalfont\bfseries}{\thesubsection}{0.8em}{}

\titleclass{\part}{top}
\titleformat{\part}[display]
{\Huge\scshape\centering}{\partname~\thepart}{20pt}{}
\titlespacing*{\part}{0pt}{40pt}{40pt}

\usepackage{hyperref} 
\hypersetup{
	colorlinks = true,
	linkcolor = {blue},
	urlcolor = {blue},
	citecolor = {blue}
}

\usepackage[nameinlink,capitalise]{cleveref} 

\newcounter{results}[section] 

\theoremstyle{plain}
\newtheorem{theorem}[results]{Theorem}
\newtheorem{lemma}[results]{Lemma}
\newtheorem{proposition}[results]{Proposition}
\newtheorem{corollary}[results]{Corollary}

\newcounter{alpharesults}

\newtheorem{alphatheorem}[alpharesults]{Theorem}
\newtheorem{alphaproposition}[alpharesults]{Proposition}

\newtheorem*{theorem*}{Theorem}
\newtheorem*{lemma*}{Lemma}
\newtheorem*{proposition*}{Proposition}
\newtheorem*{corollary*}{Corollary}
\newtheorem*{exercise*}{Exercise}
\newtheorem*{fact*}{Fact}
\newtheorem*{open*}{Open problem}

\theoremstyle{remark}
\newtheorem{remark}[results]{Remark}

\newtheorem*{remark*}{Remark}
\newtheorem*{question*}{Question}

\theoremstyle{definition}
\newtheorem{definition}[results]{Definition}
\newtheorem{example}[results]{Example}

\newtheorem*{definition*}{Definition}
\newtheorem*{example*}{Example}

\numberwithin{equation}{section}

\crefname{figure}{Figure}{Figures}

\ifdefined\comma 
        \renewcommand{\comma}{\ensuremath{,}}
\else 
        \newcommand{\comma}{\ensuremath{,}}
\fi
\newcommand{\semicolon}{\ensuremath{;}}
\newcommand{\point}{\ensuremath{.}}
\newcommand{\pheq}{\ensuremath{\hphantom{{}={}}}}

\newcommand{\N}{\ensuremath{\mathbb N}}
\newcommand{\Z}{\ensuremath{\mathbb Z}}
\newcommand{\R}{\ensuremath{\mathbb R}}

\DeclarePairedDelimiter\ClOp{[}{)}
\DeclarePairedDelimiter\OpOp{(}{)}
\DeclarePairedDelimiter\ClCl{[}{]}

\newcommand{\co}[2]{\ensuremath{\ClOp{#1, #2}}}
\newcommand{\oo}[2]{\ensuremath{\OpOp{#1, #2}}}
\newcommand{\cc}[2]{\ensuremath{\ClCl{#1, #2}}}

\newcommand{\bigO}{\ensuremath{\mathcal{O}}}

\DeclarePairedDelimiter\abs{\lvert}{\rvert} 
\DeclarePairedDelimiter\norm{\lVert}{\rVert} 
\newcommand{\scal}[2]{\ensuremath{\langle #1 , #2 \rangle}} 

\newcommand \eps{\ensuremath{\varepsilon}}
\newcommand{\st}{\ensuremath{\ :\ }} 
\newcommand{\eqdef}{\ensuremath{\coloneqq}} 
\newcommand{\id}{\ensuremath{\mathrm{id}}}

\DeclareMathOperator{\tr}{tr}

\DeclareMathOperator{\supp}{supp}

\renewcommand{\d}{\ensuremath{d}} 
\newcommand{\de}{\ensuremath{\, d}} 

\let\div\undefined
\newcommand{\div}{\ensuremath{\mathrm{div}}} 
\newcommand{\grad}{\ensuremath{\nabla}} 
\newcommand{\lapl}{\ensuremath{\Delta}} 

\newcommand{\DerParz}[2]{\ensuremath{\frac{\partial #1}{\partial #2}}} 
\newcommand{\vf}{\ensuremath{\mathfrak X}} 
\newcommand{\cov}{\ensuremath{\nabla}} 
\DeclareMathOperator{\II}{I\!I} 
\newcommand{\Scal}{\ensuremath{R}} 
\DeclareMathOperator{\Ric}{Ric} 
\DeclareMathOperator{\vol}{vol} 
\DeclareMathOperator{\dist}{dist} 

\newcommand{\Haus}{\ensuremath{\mathscr H}} 
\newcommand{\ms}{\ensuremath{\Sigma}} 
\newcommand{\amb}{\ensuremath{M}} 
\DeclareMathOperator{\ind}{ind} 
\newcommand{\area}{\ensuremath{\Haus^2}} 
\newcommand{\length}{\ensuremath{\Haus^1}} 
\DeclareMathOperator{\genus}{genus} 
\DeclareMathOperator{\bdry}{boundaries} 
\newcommand{\Diff}{\ensuremath{D}} 
\newcommand{\jac}{\ensuremath{J}} 
\newcommand{\A}{\ensuremath{A}} 

\newcommand{\hypP}{\hyperref[HypP]{$(\mathfrak{P})$}} 
\newcommand{\hypC}{\hyperref[HypC]{$(\mathfrak{C})$}} 
\newcommand{\hypN}{\hyperref[HypN]{$(\mathfrak{N})$}} 

\newcommand{\lam}{\ensuremath{\mathcal L}} 
\newcommand{\set}{\ensuremath{\mathcal S}} 

\newcommand{\bu}[1]{\ensuremath{\bar{#1}}}

\newcommand{\notg}{\ensuremath{a}}
\newcommand{\notb}{\ensuremath{b}}

\newcommand{\euno}{\ensuremath{\scriptscriptstyle{(1)}}}
\newcommand{\edue}{\ensuremath{\scriptscriptstyle{(2)}}}
\newcommand{\ebi}{\ensuremath{\scriptscriptstyle{(\notb)}}}

\newcommand{\dih}{\ensuremath{\mathbb D}}
\DeclareMathOperator{\Prob}{Prob}
\newcommand{\Is}{\ensuremath{\mathfrak{Is}}}
\newcommand{\vard}{\ensuremath{\boldsymbol {\mathrm{F}}}}
\newcommand{\so}{\ensuremath{\boldsymbol{\Sigma}}}

\newcommand{\reduc}[1][\gamma]{\ensuremath{\underset{#1}{\ll}}}
\newcommand{\sreduc}[1][\gamma]{\ensuremath{\underset{#1}{<}}}

\newcommand{\smetric}{\ensuremath{\gamma}}

\newcommand{\selem}{\ensuremath{h}}

\newcommand{\limitv}{\ensuremath{\Xi}}
\newcommand{\avoids}{\ensuremath{\Theta}}

\newcommand{\euball}{\ensuremath{\mathscr{B}}}

\DeclarePairedDelimiter\interval{(}{)}
\DeclarePairedDelimiter\Interval{[}{)}
\DeclarePairedDelimiter\intervaL{(}{]}
\newcommand{\Eq}{\ensuremath{D}} 
\DeclareMathOperator{\asinh}{asinh}

\colorlet{myGray}{gray}
\colorlet{myBlue}{blue}
\colorlet{myBlack}{black}
\colorlet{myBackground}{gray!10}

\title{Contributions to the theory of free boundary minimal surfaces}
\author{Giada Franz}

\makeindex

\begin{document}

\frontmatter

\makeatletter
\begin{titlepage}
    \centering
    
    \large 
    
    \vspace{2ex}

    DISS. ETH NO. 28421 
    
    \vspace{10ex}
    
    {\Huge \scshape Contributions to the theory of\\[0.5ex] free boundary minimal surfaces}
    
    \vspace{10ex}
    
    A thesis submitted to attain the degree of\\
    Doctor of Sciences of ETH Zurich\\
    (Dr. sc. ETH Zurich)
    
    \vspace{8ex}
    
    presented by\\[0.2ex]
    {\Large \scshape Giada Franz}
    
    \vspace{5ex}
    
    MSc in Mathematics at University of Pisa\\
    born on 09.01.1994
    
    \vspace{18ex}
    
    accepted on the recommendation of\\[2ex]
    Prof. Dr. Alessandro Carlotto, examiner\\
    Prof. Dr. Ailana Fraser, co-examiner
    
    \vfill
    
    2022
    
\end{titlepage}
\makeatother

\setcounter{page}{2}

\chapter*{Abstract}

\addcontentsline{toc}{chapter}{Abstract/Sommario} 

In this thesis, we present various contributions to the study of free boundary minimal surfaces, which are critical points of the area functional with respect to variations that constrain the boundary of such surfaces to the boundary of the ambient manifold.
After introducing some basic tools and discussing some delicate aspects related to the definition of Morse index when allowing for a contact set, we divide the thesis in two parts.

In the first part of this dissertation, we study free boundary minimal surfaces with bounded Morse index in a three-dimensional ambient manifold. More specifically, we present a degeneration analysis of a sequence of such surfaces, proving that (up to subsequence) they converge smoothly away from finitely many points and that, around such `bad' points, we can at least `uniformly' control the topology and the area of the surfaces in question.
As a corollary, we obtain a complete picture of the way different `complexity criteria' (in particular: topology, area and Morse index) compare for free boundary minimal surfaces in ambient manifolds with positive scalar curvature and mean convex boundary.

In the second part, we focus on an equivariant min-max scheme to prove the existence of free boundary minimal surfaces with a prescribed topological type.
The principle is to choose a suitable group of isometries of the ambient manifold in order to obtain exactly the topology we are looking for. 
We recall a proof of the equivariant min-max theorem, and we also prove a bound on the Morse index of the resulting surfaces. Finally, we apply this procedure to show the existence of a new family of free boundary minimal surfaces with connected boundary and arbitrary genus in the three-dimensional unit ball.

\chapter*{Sommario}

In questa tesi, presentiamo vari contributi allo studio di superfici minime con bordo libero, definite come punti critici del funzionale area rispetto a variazioni che vincolano il bordo di queste superfici al bordo della varietà ambiente. Dopo aver introdotto degli strumenti di base e aver discusso degli aspetti delicati legati alla definizione di indice di Morse quando ammettiamo un insieme di contatto, dividiamo la tesi in due parti.

Nella prima parte di questa dissertazione, studiamo superfici minime con bordo liscio con indice di Morse limitato, in una varietà ambiente di dimensione tre. In particolare, presentiamo un'analisi del comportamento di una successione di tali superfici, dimostrando che una sottosuccessione converge in modo liscio a meno che in un numero finito di punti e che, intorno a questi punti `cattivi', possiamo quantomeno controllare `uniformemente' la topologia e l'area delle superfici in questione. Come corollario, otteniamo un quadro completo di come si confrontano diversi `criteri di complessità' (nello specifico: topologia, area e indice di Morse) di superfici minime con bordo libero in varietà ambiente con curvatura scalare positiva e bordo con curvatura media non negativa.

Nella seconda parte, ci concentriamo su uno schema di min-max equivariante per dimostrare l'esistenza di superfici minime con bordo libero di un tipo topologico prescritto. Il principio è scegliere accuratamente il gruppo di isometrie della varietà ambiente in modo tale da ottenere esattamente la topologia che stiamo cercando.
Richiamiamo una dimostrazione del teorema di min-max equivariante, dimostrando anche un controllo sull'indice di Morse della superficie risultante. Infine, applichiamo questa procedura per dimostrare l'esistenza di una nuova famiglia di superfici minime con bordo libero con bordo connesso e genere arbitrario nella palla unitaria di dimensione tre.


\chapter*{Acknowledgements}

\addcontentsline{toc}{chapter}{Acknowledgements} 

I would like to thank my advisor Alessandro Carlotto for his support and guidance, for having introduced me to so many fascinating problems and for countless fruitful conversations during these years. I also thank Ailana Fraser, for kindly accepting to be my co-examiner and for being an inspiration. Moreover, a special thank goes also to all the incredible colleagues and professors that I met in Pisa and in Zürich and have been fundamental for my mathematical development. In particular, I would like to mention some of my collaborators and friends that directly contributed in some way to this thesis: Mario Schulz, for many exciting discussions about minimal surfaces; Alessandro Pigati, for the patience in answering my questions; and Federico Trinca, for his precious feedback and encouragement.

Thanks to my friends and my family, my mum, my dad, and my brother, who supported me during these years and with whom I have spent so many beautiful moments. I tried to write down all the names, but it seemed impossible to give fair credit to everyone, so I decided to keep these acknowledgements short. There are many people that made my everyday life more enjoyable and people that I talked to once but left a great impression on me, from a mathematical or a human point of view. I thank all of them. I have learned so many things in these years and grown as a person, and I am grateful to the many people that I met in this process and who have been by my side.

\vfill

My doctoral studies have received funding from the European Research Council (ERC) under the
European Union’s Horizon 2020 research and innovation programme (grant agreement No. 947923).

\tableofcontents
\clearpage


\chapter*{Basic notation}
\addcontentsline{toc}{chapter}{Basic notation}  
\markboth{Basic notation}{Basic notation}
\label{chpt:notation}

Let $(\amb^3,\smetric)$ be a Riemannian manifold with boundary and let $\ms^2\subset\amb$ be an embedded surface with $\partial\ms\subset\partial\amb$. Then we denote by:
\begin{itemize}
\item $\vf(M)$ the set of vector fields on $M$ such that $X(x)\in T_x\partial M$ for all $x\in \partial M$. 
\item $X^\perp \in \Gamma(N\Sigma)$ the normal component to $\Sigma$ of a vector $X\in\vf(M)$, where $\Gamma(N\Sigma)$ denotes the sections of the normal bundle of $\Sigma$.
\item $\Diff$ the connection on $\amb$, $\cov$ the induced connection on $\ms$ and $\grad^\perp$ the induced connection on the normal bundle of $\Sigma$.
\item $\Ric_M$ and $\Scal_M$ the Ricci curvature and the scalar curvature of $M$, respectively. 
\item $\nu$ a choice of a global unit normal vector field on $\ms$, when $\ms$ is two-sided.
\item $\eta$ the outward unit co-normal vector field to $\partial \ms$.
\item $\hat\eta$ the outward unit co-normal vector field to $\partial\amb$ (which coincides with $\eta$ along $\partial\ms$ when $\ms$ satisfies the free boundary property). 
\item $\II^{\partial\amb}(X,Y) = g(D_X Y,\hat{\eta})$ the second fundamental form of $\partial\amb\subset\amb$. Observe that $\II^{\partial\amb} < 0$ if for example $\amb$ is the unit ball in $\R^3$ (and thus $\partial\amb$ is the unit sphere).
\item $\A(X,Y) = (D_X Y)^\perp$ the second fundamental form of $\ms\subset\amb$ and $\abs{A}^2$ its squared norm.
\item $H^{\partial\amb}$ the mean curvature of $\partial\amb$, that is $H^{\partial\amb} = - \II^{\partial\amb}(E_1,E_1) - \II^{\partial\amb}(E_2,E_2)$ for every choice of a local orthonormal frame $\{E_1,E_2\}$ of $\partial\amb$. 
We say that $\partial\amb$ is mean convex if $H^{\partial\amb} \ge 0$, and strictly mean convex if $H^{\partial\amb} > 0$ (that is, for example, the case of the unit ball in $\R^3$).
\item $\chi(\ms)$, $\genus(\ms)$, $\bdry(\ms)$ and $\area(\ms)$ respectively the Euler characteristic, the genus, the number of boundary components and the area of $\ms$.
\item $\Xi(a) \eqdef\{ x^1\ge a \}\subset \R^3$ and $\Pi(a)\eqdef\{ x^1= a \}\subset \R^3$, for $0 \ge a \ge -\infty$. When $a=-\infty$, we agree that $\Xi(a)$ coincides with $\R^3$ and $\Pi(a)$ is empty.
\end{itemize}
All these notions are easily extended to the case when the surface in question is only immersed, rather than embedded.

\begin{figure}[htpb]
\centering
\begin{tikzpicture}[line cap=round,line join=round,rotate=5]

\pgfmathsetmacro{\R}{3}
\pgfmathsetmacro{\T}{1.19968}
\pgfmathsetmacro{\a}{1/(\T*cosh(\T))}
\pgfmathsetmacro{\e}{0.403*\R}
\coordinate(Pa)at({\R*\a*cosh(\T)},{\R*\a*\T});
\coordinate(Pav)at({1.3*\R*\a*cosh(\T)},{1.3*\R*\a*\T});
\coordinate(Pao)at({\R*\a*cosh(\T)+0.3*\R*\a*\T},{\R*\a*\T-0.3*\R*\a*cosh(\T)});
\coordinate(Pb)at({\R*\a*cosh(-\T)},{-\R*\a*\T});
\coordinate(Qa)at({-\R*\a*cosh(\T)},{\R*\a*\T});
\coordinate(Qb)at({-\R*\a*cosh(-\T)},{-\R*\a*\T});

\pgfmathsetmacro{\wa}{215}
\pgfmathsetmacro{\wb}{140}
\pgfmathsetmacro{\ua}{0.4375*\R}
\pgfmathsetmacro{\ub}{0.35*\R}

\draw[myBlack] (0,0) circle (\R);
\node[above=2pt] at(0,\R) {$\amb$};

\begin{scope}[myBlue]
\node at(-0.65*\R,-0.2*\R) {$\ms$};
\draw[scale=1,domain=-\T:\T,smooth,variable=\s]  plot ({\R*\a*cosh(\s)},{\R*\a*\s});
\draw[scale=1,domain=-\T:\T,smooth,variable=\s]  plot ({-\R*\a*cosh(\s)},{\R*\a*\s});

\draw(Pa)..controls+(\wa:\ua)and +(180-\wa:\ua)..(Qa);
\draw(Pb)..controls+(\wa:\ua)and +(180-\wa:\ua)..(Qb);
\draw(Pa)..controls+(\wb:\ub)and +(180-\wb:\ub)..(Qa);
\draw[dashed](Pb)..controls+(\wb:\ub)and +(180-\wb:\ub)..(Qb);
\end{scope}

\draw[-latex] (Pa) -- (Pav) node[above]{$\eta$};
\draw[-latex] (Pa) -- (Pao) node[right]{$\nu$};
\end{tikzpicture}
\caption{An example of a free boundary minimal surface in the unit ball $B^3\subset \R^3$ with some notation included.} \label{fig:NotEx}
\end{figure}
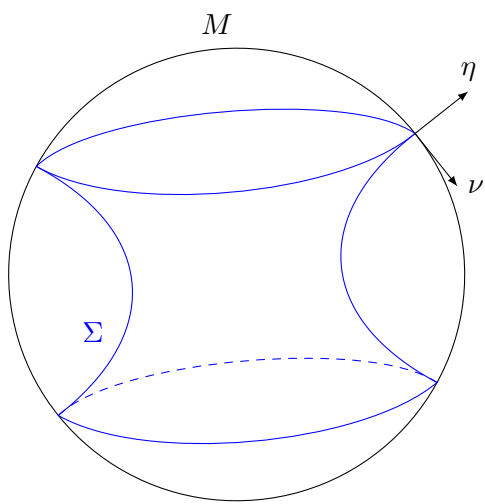


\mainmatter


\chapter*{Introduction} 
\addcontentsline{toc}{chapter}{Introduction} 
\markboth{Introduction}{Introduction}
\label{chpt:intro}

The focus of this thesis is the study of minimal surfaces, which are undoubtedly one of the main themes in Geometric Analysis.
Starting from the classical Plateau problem of finding an area minimizing disc whose boundary is a fixed curve in $\R^3$, the theory has spread in a variety of research directions.
Here, we shall be concerned with a class of surfaces whose study goes back at least to the work \cite{Courant1940} of Courant in the 1940s (see in particular Part II therein, and also \cite{Courant1977}*{Chapter VI}), when he posed the problem of finding an area minimizing surface whose boundary lies on a fixed submanifold in $\R^m$.
This was indeed the birth of the theory of \emph{free boundary minimal surfaces} (\emph{FBMS} in short), where the name comes from the fact that the boundary is `free' to move inside a fixed constraint.
We will now present the modern point of view on this problem, which has evolved to the study of \emph{stationary} points of the area functional rather than only minimizing ones.

Given a compact three-dimensional Riemannian manifold $(M^3,\gamma)$ with boundary and a smooth embedded surface $\Sigma^2\subset M$ with $\partial\Sigma = \Sigma\cap \partial M$, we say that $\Sigma$ is a \emph{free boundary minimal surface} if the first variation of the area functional at $\Sigma$ vanishes with respect to variations induced by proper diffeomorphisms of $M$ (namely, diffeomorphisms of $M$ sending the boundary into itself), or equivalently if $\Sigma$ has zero mean curvature and meets the ambient boundary orthogonally.

\begin{remark*}
More generally, one can define $k$-dimensional \emph{free boundary minimal submanifolds} in any compact Riemannian manifold $(M^{m},\gamma)$, with $k\le m-1$.
Here, we mostly focus on the case of surfaces in three-dimensional Riemannian manifolds since, as we will see, the theory is already extremely rich in this setting.
\end{remark*}

These surfaces are a special instance of a more general class of variational objects, i.e., the \emph{capillary surfaces}, which arise in Physics to model the interface between two incompressible and immiscible fluids at the equilibrium (see \cite{Finn1986} for a thorough monograph). Let us describe here an easy example (see \cref{fig:FluidsInterface}), for which the reader can also look at \cite{Nitsche1985}.

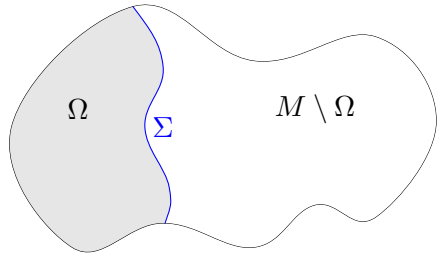
\begin{figure}[htpb]
\centering
\begin{tikzpicture}[scale =.8]

\clip plot [smooth cycle, tension=0.7] coordinates { (0,0) (2,2) (4,1.1) (6,1.5) (7,0.1) (6,-1.5) (5,-1.3) (4,-2) (2.5,-1.6) (1,-2)};

\fill[gray!20!white] plot [smooth, tension=0.7] coordinates {(0,3) (1,3) (2,2) (2.5,1) (2.2,0) (2.6, -1) (2.5,-1.7) (2,-3) (-0.1,-3)};

\draw[blue] plot [smooth, tension=0.7] coordinates {(0,3) (1,3) (2,2) (2.5,1) (2.2,0) (2.6, -1) (2.5,-1.7) (2,-3) (-0.1,-3)};

\draw plot [smooth cycle, tension=0.7] coordinates { (0,0) (2,2) (4,1.1) (6,1.5) (7,0.1) (6,-1.5) (5,-1.3) (4,-2) (2.5,-1.6) (1,-2)};

\node at (1.1,0.3) {$\Omega$};
\node[blue] at (2.5,0) {$\Sigma$};
\node at (5,0.3) {$M\setminus\Omega$};
\end{tikzpicture}
\caption{Capillary surface as interface between two fluids.}
\label{fig:FluidsInterface}
\end{figure}

Consider a hollow body $M^3$ filled with two incompressible and immiscible fluids, the first occupying the region $\Omega$ (grey in the figure) and the second occupying the region $M\setminus\Omega$.
With this notation, $\Sigma\eqdef\partial\Omega\setminus\partial M$ represents the interface between the two fluids.
Physically, at the equilibrium, the two fluids arrange in order to minimize the area of the interface, which is thus a solution of the relative isoperimetric problem of finding the surface that minimizes the area among those enclosing relative volume $\vol(\Omega)$ in $M$.
As one can easily show, the result is that $\Sigma$ has \emph{constant} mean curvature and meets the ambient boundary orthogonally.

Free boundary minimal surfaces are related to a variety of other problems in many different areas (for example partial differential equations, general relativity, complex analysis, conformal geometry, etc), which cannot be properly accounted for here.
However, following the series of articles \cite{FraserSchoen2011,FraserSchoen2013,FraserSchoen2016} by Fraser--Schoen, we wish to recall the connection between the \emph{Steklov eigenvalue problem} for compact surfaces with boundary and free boundary minimal surfaces in the unit ball $B^m$.
These results indicate, among other things, an analogy between free boundary minimal surfaces in $B^m$ and minimal surfaces in $S^m$, the latter having a relation with the \emph{Laplacian eigenvalue problem} on closed surfaces (cf. \cite{Nadirashvili1996}*{Theorem 2}, see also \cite{ElSoufiIlias2000}, \cite{FraserSchoen2013}*{Section 2}).
This often inspires the research of the two settings to move in parallel.

The general question that guides the study of (free boundary) minimal surfaces in our perspective is to classify, in some sense, all such surfaces in a given manifold. A good starting point is, for example, the case of ambient manifold $B^3\subset\R^3$. This is obviously a very ambitious goal, which we are not close to achieve, but it is useful to keep it in mind to contextualize results and open problems.

That being said, this thesis is divided in two parts. 
\begin{itemize}
\item In \cref{part:Complexity}, following \cite{CarlottoFranz2020}, we explore the relations among different measures of complexity of a free boundary minimal surface, such as its Morse index, its topology and its area.
\item In \cref{part:Existence}, based on \cite{CarlottoFranzSchulz2020} and \cite{Franz2021}, we investigate an equivariant min-max procedure, which we then we use to prove the existence of free boundary minimal surfaces with pre-assigned topology in the three-dimensional unit ball $B^3\subset\R^3$. We also explain how this method is useful to obtain information on the resulting surfaces (for example about their Morse index).
\end{itemize}

In the following sections, we present these two themes in more detail.
For the purposes of this introduction, we assume that $(M^3,\smetric)$ is a compact Riemannian manifold with boundary, which satisfies the following additional property.
\begin{description}
\item [$(\mathfrak{P})$]\label{HypP}\ If $\ms^2\subset \amb$ is a smooth, connected, complete (possibly noncompact), embedded surface with zero mean curvature which meets the boundary of the ambient manifold orthogonally along its own boundary, then $\partial\ms = \ms\cap\partial\amb$. 
\end{description}
Such a condition prevents the existence of minimal surfaces that touch $\partial\amb$ in their interior and it is implied by simple geometric assumptions (for example property \hypC{} in \cref{sec:properFBMS}, namely that the boundary of the ambient manifold is mean convex with no minimal component).
We contextualize this property, its relevance and the issues that arise when one drops it in \cref{sec:Properness}.
In particular, we point out how far from trivial the `properness issues' are. Indeed, if one allows for an interior contact set, then there are at least four natural definitions of Morse index one can adopt, and the value one computes depends not only on the adopted definition but also, even for a given definition, on the size of the contact set.

\section*{Part I: Relating topology, area and Morse index}
\addcontentsline{toc}{section}{Part I: Relating topology, area and Morse index}

As anticipated, \cref{part:Complexity} is based on \cite{CarlottoFranz2020} and concerns the study of different properties of a free boundary minimal surface, as (for example) the topology, the area and the Morse index.
Before entering in the details, let us give the definition of Morse index, which roughly is the number of negative directions of the second variation of the area functional. 

Let $\Sigma^2\subset M$ be a smooth embedded free boundary minimal surface. 
Let $(\psi_t)_{t\ge 0}$ be a smooth family of proper diffeomorphisms of $M$ such that $\psi_0=\id$ and let $X(x) = \frac{\d}{\d t}\big|_{t=0} \psi_t(x)$ be the associated variation vector field at time $0$. By definition of free boundary minimal surface, we have that $\frac{\d}{\d t}\big|_{t=0} \Haus^2(\psi_t(\Sigma)) = 0$, while one can check that the second variation of the area $\frac{\d^2}{\d t^2}\big|_{t=0} \Haus^2(\psi_t(\Sigma))$ depends only on the normal component $X^\perp$ of $X$ along $\Sigma$, in particular we can write it as a quadratic form $Q^\Sigma(X^\perp,X^\perp)$.
We define the \emph{Morse index} of $\Sigma$ as the maximal dimension of a linear subspace of the normal vector bundle $\Gamma(N\Sigma)$ where $Q^\Sigma$ is negative definite.
We refer to \cref{sec:SecondVariation} for further details.

Already from classical results about complete minimal surfaces in $\R^3$, we expect that the Morse index of a (free boundary) minimal surface should control its topology.
Indeed, having finite Morse index for a complete minimal surface in $\R^3$ is equivalent to having finite total curvature (see \cite{FischerColbrie1985}), which implies strong restrictions on the topology of the minimal surface by \cite{Osserman1964}.
In the last decades, these relation between index and topology of a minimal surface have been thoroughly investigated.
A recurrent method is using that the dimension of the space of harmonic one-forms is related to the topology via the Hodge theorem, and that harmonic forms give (in some sense) negative directions for the second derivative of the area. Exploiting this idea Chodosh--Maximo obtained quantitative bounds from above on the topology of a complete minimal surface in $\R^3$ in terms of the index in \cite{ChodoshMaximo2016,ChodoshMaximo2018}.

In the case of a boundaryless compact ambient manifold $(M^3,\smetric)$ with \emph{positive Ricci curvature}, it was conjectured by Marques, Neves, Schoen that there exists $C=C(M,\smetric)>0$ such that
\[
b_1(\Sigma)\le C\ind(\Sigma), 
\]
for every closed minimal surface $\Sigma^2\subset M$, where $b_1(\Sigma)$ is the first Betti number of $\Sigma$.
So far, this conjecture has been proven only under additional assumptions (e.g. for compact rank one symmetric spaces), see \cite{Ros2006,Savo2010,AmbrozioCarlottoSharp2018Index}, with methods similar to the one described above involving harmonic one-forms.

More recently, with completely different methods, Song proved in \cite{Song2020} that, for every three-dimensional boundaryless compact Riemannian manifold $(M^3,\smetric)$ \emph{without curvature assumptions}, there exists $C=C(M,\smetric)>0$ such that
\begin{equation} \label{eq:SongTopInd} 
b_1(\Sigma)\le C\area(\Sigma) (\ind(\Sigma)+1) \tag{$\ast$}
\end{equation}
for every minimal surface $\Sigma^2\subset M$.

For what concerns the free boundary setting, thanks to \cite{AmbrozioCarlottoSharp2018IndexFBMS} (see also \cite{Sargent2017}), we know that for every free boundary minimal surface $\Sigma$ in a compact strictly mean convex domain of $\R^3$ it holds
\[
\frac 13 (2 g + b - 1) \le\ind(\Sigma),
\]
where $g$ is the genus of $\Sigma$ and $b$ is the number of its boundary components.

As one can see, such advances leave several questions open. 
As an example, one may wonder if, unconditionally with respect to the ambient curvature, a bound on the index of a free boundary minimal surface is sufficient to have a bound on its topology. Observe that this is not implied by \eqref{eq:SongTopInd}, since that estimate involves the area as well. 

Based on several previous contributions, in \cref{part:Complexity} we manage to find all the missing answers to these problems in the setting of three-dimensional ambient manifolds with \emph{`weakly positive geometry'}, namely for ambient manifolds $(\amb^3,\smetric)$ such that
\begin{enumerate}[label={\normalfont(\roman*)}]
\item \emph{either} the scalar curvature of $\amb$ is positive and $\partial \amb$ is mean convex with no minimal components;
\item \emph{or} the scalar curvature of $\amb$ is nonnegative and $\partial\amb$ is strictly mean convex.
\end{enumerate}
\begin{remark*}
We use the expression \emph{`weakly positive geometry'}, because as ambient manifolds with \emph{`positive geometry'} we have in mind manifolds $(\amb^3,\smetric)$ such that
\begin{enumerate}[label={\normalfont(\roman*)}]
\item \emph{either} the Ricci curvature of $\amb$ is positive and $\partial \amb$ is convex with no minimal components;
\item \emph{or} the Ricci curvature of $\amb$ is nonnegative and $\partial\amb$ is strictly convex.
\end{enumerate}
We emphasize that these are not standard definitions, so we use them only here in the introduction to simplify the presentation.

One often expects a change of behaviour from minimal surfaces in manifolds without any curvature assumptions, or in `weakly positive geometry', or instead in `positive geometry'. To better understand what we mean, we refer to \cite{Song2020}*{Section 6.1} for a list of the conjectured quantitative estimates relating topology and Morse index in these different settings.
\end{remark*}

We outline the results of \cref{part:Complexity} in the following subsections, the first two being original work of this thesis, while the third being based on earlier work in the literature. 
We also summarize what qualitative relations hold and what do not hold, in the `weakly positive geometry' setting, in the diagram in \cref{fig:Complexity}.

\subsection*{Index bounds topology and area}

\begin{alphatheorem}[Index $\Rightarrow$ topology and area, {\cite{CarlottoFranz2020}*{Thm 1.4} \textrightarrow\ \cref{thm:AreaBound}}] \label{thm-intro:index}
Let $(M^3,\gamma)$ be a compact Riemannian manifold with `weakly positive geometry'. Given $I\in\N$, there exists $C=C(I)$ such that every compact connected free boundary minimal surface $\Sigma^2\subset M$ with nonempty boundary and index $I$ has area, total curvature, genus and number of boundary components bounded by $C$.
\end{alphatheorem}

This theorem constitutes the core of \cref{part:Complexity}.
The proof is based on a careful analysis of the behaviour of a sequence $\{\Sigma_j\}_{j\in\N}$ of free boundary minimal surfaces with bounded index (in a general ambient manifold, not only in `weakly positive geometry'), cf. \cref{sec:TopologicalDegeneration}. In fact, one can show that (up to subsequence) the surfaces $\Sigma_j$ satisfy curvature estimates away from a finite set $\mathcal{S}_\infty$ of points, hence they converge (again up to subsequence) to a free boundary minimal lamination away from $\mathcal{S}_\infty$ (see \cref{cor:ExistenceBlowUpSetAndCurvatureEstimate}). Observe that the `limit' is a minimal lamination (see \cref{sec:FBMLam} for definition and basic properties) rather than a minimal surface, because a priori we do not have any control on the area of the surfaces $\Sigma_j$. 
Moreover, near the points in $\mathcal{S}_\infty$ it is possible to develop an accurate blow-up analysis (see \cref{thm:GlobalDeg} and \cref{cor:Surgery}), thanks to the fact that the index bounds the topology for free boundary minimal surfaces in a half-space (see \cref{prop:GeometryOfHalfBubble}). 
The arguments retrace the results in the closed case treated in \cite{ChodoshKetoverMaximo2017}, with new nontrivial technical difficulties.

In order to prove that the limit lamination is a well-behaved surface in the case of `weakly positive geometry', the additional ingredient is a control on the size of \emph{stable} free boundary surfaces in a three-dimensional manifold with `weakly positive geometry'. More precisely, we prove the following result. 

\begin{alphaproposition} [{\cite{CarlottoFranz2020}*{Prop 1.8} \textrightarrow\ \cref{prop:StableImpliesCptness}}]
Let $(M^3,\gamma)$ be a complete Riemannian manifold. Denote by $\rho_0\eqdef \inf_M R_\smetric$ the infimum of the scalar curvature of $M$ and by $\sigma_0\eqdef \inf_{\partial M} H^{\partial M}$ the infimum of the mean curvature of $\partial M$. Assume that $M$ has `weakly positive geometry' and \emph{either} $\rho_0>0$ \emph{or} $\sigma_0>0$\footnote{Note that the assumption that either $\rho_0 > 0$ or $\sigma_0>0$ does not follow from having `weakly positive geometry' because we allow for $M$ to be noncompact.}.

Then every complete, connected, embedded, stable free boundary minimal surface $\ms^2\subset\amb$ that is two-sided and has nonempty boundary is compact, and its intrinsic diameter satisfies the bound
\[
\operatorname{diam}(\ms)\eqdef \sup_{x,y\in\ms} d_\ms(x,y) \le \min\left\{ \frac{2\sqrt 2 \pi}{\sqrt{3\varrho_0}}, \frac{\pi + 8/3}{\sigma_0}\right\}\point
\]
Moreover, one has that
\[
0< \frac{\varrho_0}2 \area(\ms) + \sigma_0 \length(\partial\ms) \le 2\pi\chi(\ms) \semicolon
\]
in particular, $\ms$ is diffeomorphic to a disc.
\end{alphaproposition}

The powerful principle behind this result is similar to the one underlying the estimate by Schoen--Yau \cite{SchoenYau1983} in the closed case, namely that stable minimal hypersurfaces inherit (in some sense) the positivity of the scalar curvature of the ambient manifold. 
However, the arguments for the closed case are not sufficient in the free boundary context, which turns out to be much more delicate (see \cref{sec:CptnessStable}).

\subsection*{Topology does not bound index and area}

\begin{alphatheorem} 
[Topology $\not\Rightarrow$ index and area, {\cite{CarlottoFranz2020}*{Thm 1.12} \textrightarrow\ \cref{thm:Counterexample}}]\label{thm-intro:topology}
For every $g\ge 0$ and $b>0$, there exists a compact Riemannian manifold $(M^3,\gamma)$ with positive scalar curvature and strictly mean convex boundary that contains a sequence of compact connected free boundary minimal surfaces of genus $g$ and $b$ boundary components, but whose area and Morse index attain arbitrary large values.
\end{alphatheorem}

In this case the proof consists in constructing the desired manifolds by gluing elementary blocks, each of them playing a specific role (see \cref{sec:Counterexample}), partly based on earlier work by Colding--De Lellis \cite{ColdingDeLellis2005} on the closed case.

\subsection*{Area does not bound topology and index}

\begin{alphatheorem}
[Area $\not\Rightarrow$ index and topology] \label{thm-intro:area}
There are sequences of free boundary minimal surfaces in the three-dimensional unit ball $B^3$ with uniformly bounded area, unbounded topology and unbounded Morse index.
\end{alphatheorem}

There are several examples proving this theorem. One is the family of free boundary minimal surfaces with connected boundary and arbitrary genus that we construct in \cref{thm:main-fbms-b1} in \cref{part:Existence}, but there are many others that are better described in the following section and in \cref{intro:existence}.
Observe that, if we have a family of free boundary minimal surfaces in the unit ball with bounded area and unbounded topology, then it has also unbounded Morse index thanks to \cite{AmbrozioBuzanoCarlottoSharp2018}*{Corollary 4}.

\begin{figure}[htpb]
\centering
\begin{tikzpicture}[scale=1.35]
\tikzset{bharrownegated/.style={
        decoration={markings,
            mark= at position 0.5 with {
                \node[transform shape, scale=1.7] (tempnode) {$\times$};
            },
            mark=at position 1 with {\arrow[scale=2]{>}}
        },
        postaction={decorate},
    }
}
\tikzset{bharrow/.style={
    decoration={markings,mark=at position 1 with {\arrow[scale=2]{>}}},
    postaction={decorate},
    }
}
\begin{scope}[every node/.style = {
    shape = rectangle,
    minimum width= 3.8cm,
    minimum height=1.5cm,
    align=center}]
\node (top) at (0,0.5) [draw=gray] {topological bounds\\ \footnotesize $\chi(\ms)$};
\node (ind) at (-4,-3.5) [draw=gray] {index bounds\\ \footnotesize $\ind(\ms)$};
\node (are) at (4,-3.5) [draw=gray] {area bounds\\ \footnotesize $\area(\ms)$};
\end{scope}
\begin{scope}[every node/.style={scale=0.85}]
\draw[bharrow] ([xshift=1mm,yshift=-1.5mm]ind.east) -- node[sloped,midway,below=2mm] {\cref{thm-intro:index}} ([xshift=-1mm,yshift=-1.5mm]are.west);
\draw[bharrownegated] ([xshift=-1mm,yshift=1.5mm]are.west) -- node[sloped,midway,above=2mm] {\cref{thm-intro:area}} ([xshift=1mm,yshift=1.5mm]ind.east);
\draw[bharrow] ([xshift=-2mm,yshift=1mm]ind.north) -- node[sloped,midway,above=2mm] {\cref{thm-intro:index} $\oplus$ \cite[Corollary 4]{AmbrozioBuzanoCarlottoSharp2018}} ([xshift=-2mm,yshift=-1mm]top.south west);
\draw[bharrownegated] ([xshift=3mm,yshift=-1mm]top.south west) -- node[sloped,midway,below=2mm] {\cref{thm-intro:topology}} ([xshift=3mm,yshift=1mm]ind.north);
\draw[bharrownegated] ([xshift=-3mm,yshift=-1mm]top.south east) -- node[sloped,midway,below=2mm] {\cref{thm-intro:topology}} ([xshift=-3mm,yshift=1mm]are.north);
\draw[bharrownegated]([xshift=2mm,yshift=1mm]are.north) -- node [sloped,midway,above=2mm] {\cref{thm-intro:area}} ([xshift=2mm,yshift=-1mm]top.south east);
\end{scope}
\end{tikzpicture}
\caption{A diagram comparing complexity criteria for compact Riemannian three-manifolds $(M,\smetric)$ satisfying \emph{either} $\Scal_\smetric>0$, $H^{\partial\amb}\ge 0$ \emph{or} $\Scal_\smetric\ge0$, $H^{\partial\amb}> 0$.}\label{fig:Complexity}
\end{figure}

\section*{Part II: Existence results} 
\addcontentsline{toc}{section}{Part II: Existence results}

In \cref{part:Existence}, based on \cite{CarlottoFranzSchulz2020} and \cite{Franz2021}, we focus on the \emph{realization problem}, roughly meaning the problem of constructing, in any given Riemannian manifold, a (free boundary) minimal hypersurface for any unobstructed topological type.

In this respect, it is due mentioning the outstanding advances by Marques--Neves in \cite{MarquesNeves2014} and \cite{MarquesNeves2021}, Liokumovic--Marques--Neves in \cite{LiokumovichMarquesNeves2018}, Irie--Marques--Neves in \cite{IrieMarquesNeves2018} and Song in \cite{Song2018} (among others), to push the Almgren--Pitts min-max theory up to proving the renowned \emph{Yau's conjecture}, namely that any closed Riemannian manifold of dimension between 3 and 7 contains infinitely many closed embedded minimal hypersurfaces. The analogous result in the free boundary case has been obtained combining works of Li--Zhou \cite{LiZhou2021}, Guang--Li--Wang--Zhou \cite{GuangLiWangZhou2021} and Wang \cite{Wang2019,Wang2020}.
However, observe that these results do not give a solution to the realization problem, since the general theory does not provide a control on the topology of the resulting surface. Indeed, so far, this has been (only very partially) obtained in the Almgren--Pitts theory by bounding the Morse index and then using one of the (ineffective) results cited above (as for example \cite{Song2020}) to get a bound on the topology.
This is one of the reasons why it is appropriate to first study the `simpler' case of $B^3$.

The equatorial disc and the critical catenoid are the easiest free boundary minimal surfaces in the three-dimensional unit ball, but we know examples with other topological types: genus $0$ and arbitrary number of boundary components (see \cite{FraserSchoen2016,Ketover2016FBMS,FolhaPacardZolotareva2017,KapouleasZou2021}), genus $1$ and sufficiently large number of boundary components (see \cite{FolhaPacardZolotareva2017}), connected boundary and sufficiently large genus (see \cite{KapouleasWiygul2017}), $3$ boundary components and sufficiently large genus (see \cite{Ketover2016FBMS,KapouleasLi2017,CarlottoSchulzWiygul2022}) and $4$ boundary components and sufficiently large genus (see \cite{KapouleasMcGrath2020}).

The methods employed in these constructions can be divided in three main categories: abstract existence results for the Steklov eigenvalue problem, gluing methods, and equivariant min-max schemes. These three approaches are discussed in more detail in \cref{intro:existence} (\ref{method:Steklov}, \ref{method:Perturbation} and \ref{method:EquivMinMax}, respectively).

Here we focus on the equivariant min-max approach, essentially based on the Almgren--Pitts min-max theory (see \cite{Almgren1965,Pitts1981}) but in the variant developed by Simon--Smith \cite{Smith1982} and Colding--De Lellis \cite{ColdingDeLellis2003}, which requires more regularity on the sweepouts but provides a better control on the topology of the resulting surface. 

Thanks to a suitable application of this procedure, we are able to contribute to the realization question in $B^3$ (see \cite{Nitsche1985}*{pp. 7}, \cite{Li2020}*{Open Question 1}), constructing a new family of free boundary minimal surfaces with connected boundary and arbitrary genus.
\begin{alphatheorem}[{\cite{CarlottoFranzSchulz2020}*{Thm 1.1} \textrightarrow\ \cref{thm:main-fbms-b1}}] \label{thm:NewFamilyFBMS}
For each $1\le g \in \N$ there exists an embedded free boundary minimal surface $M_g$ in $B^3$ with connected boundary and genus $g$.
\end{alphatheorem}

\begin{remark*}
Note that we manage to obtain surfaces for \emph{any} genus, which is in general difficult, since the low topology cases are more delicate to handle. In particular, we find a free boundary minimal surface in $B^3$ with genus $1$ and connected boundary, whose existence was previously regarded as a well-known open problem in this field (cf. Nitsche, Schoen, see also \cite{Li2020}*{Section 3}).
\end{remark*}

Since the topology of an orientable compact surface with boundary is determined by its genus and the number of its boundary components, we can summarize the known existence results of free boundary minimal surfaces in $B^3$ in the table of \cref{fig:ExistenceTable}.
One can observe that there are still many topological types for which existence of a free boundary minimal surface is still unknown. On the other hand, Lawson in \cite{Lawson1970} constructed closed minimal surfaces in $S^3$ for every genus, hence we are naturally led to conjecture that there are free boundary minimal surfaces in $B^3$ for any given topological type.

\begin{figure}[htb]
\definecolor{mycell}{gray}{.8}
\renewcommand{\arraystretch}{1.2}
\begin{center}
\begin{tabular}{   p{3ex}>{\centering\arraybackslash}m{2ex}|>{\centering\arraybackslash}m{2ex}>{\centering\arraybackslash}m{2ex}>{\centering\arraybackslash}m{2ex}>{\centering\arraybackslash}m{2ex}>{\centering\arraybackslash}m{2ex}>{\centering\arraybackslash}m{2ex}>{\centering\arraybackslash}m{2ex} } 
 && \multicolumn{6}{c}{genus}\\[2ex]
  && 0 & 1 & 2 & \multicolumn{2}{c}{$\ldots$} & \multicolumn{2}{c}{$g\gg 1$} \\ \hline
 &1 & \cellcolor{mycell}{\color{blue}$\checkmark$} & \cellcolor{mycell}{\color{blue}$\checkmark$} &  \cellcolor{mycell}{\color{blue}$\checkmark$} &\cellcolor{mycell}{\color{blue}$\checkmark$} & \cellcolor{mycell}{\color{blue}$\checkmark$} & \cellcolor{mycell}{\color{blue}$\checkmark$} & \cellcolor{mycell}{\color{blue}$\checkmark$}\\ 
 \parbox[t]{2mm}{\multirow{6}{*}{\rotatebox[origin=c]{90}{\# boundaries}}}&2 & \cellcolor{mycell}{\color{blue}$\checkmark$} & ? & ? & ? & ? & ? & ?\\ 
 &3 & \cellcolor{mycell}{\color{blue}$\checkmark$} & ? & ? & ? & ? & \cellcolor{mycell}{\color{blue}$\checkmark$} & \cellcolor{mycell}{\color{blue}$\checkmark$}\\ 
 &4 & \cellcolor{mycell}{\color{blue}$\checkmark$} & ? & ? & ? & ? & \cellcolor{mycell}{\color{blue}$\checkmark$} & \cellcolor{mycell}{\color{blue}$\checkmark$}\\ 
 &\parbox[t]{2mm}{\multirow{2}{*}{\rotatebox[origin=c]{90}{$\ldots$}}} & \cellcolor{mycell}{\color{blue}$\checkmark$} & ? & ?  & ? & ? & ? & ?\\ 
 & & \cellcolor{mycell}{\color{blue}$\checkmark$} & ? & ? & ? & ? & ? & ?\\ 
 &\parbox[t]{2mm}{\multirow{2}{*}{\rotatebox[origin=c]{90}{$b\gg 1$}}} & \cellcolor{mycell}{\color{blue}$\checkmark$} & \cellcolor{mycell}{\color{blue}$\checkmark$} & ? & ? & ? & ? & ?\\
 & & \cellcolor{mycell}{\color{blue}$\checkmark$} & \cellcolor{mycell}{\color{blue}$\checkmark$} & ? & ? & ? & ? & ?\\ 
\end{tabular}
\end{center}
\caption{Topological types for which we know existence of a FBMS in $B^3$.}
\label{fig:ExistenceTable}
\end{figure}

Let us now enter in more details about the equivariant min-max theory à la Simon--Smith and Colding--De Lellis.

\subsection*{Equivariant min-max theory}

As anticipated above, the theory that we present in this thesis, and we use to prove \cref{thm:NewFamilyFBMS}, is based on the variant of the Almgren--Pitts min-max theory à la Simon--Smith \cite{Smith1982} and Colding--De Lellis \cite{ColdingDeLellis2003}. This consists in a machinery to obtain (free boundary) minimal surfaces following the same mountain pass principle that one employs to obtain critical points of a functional (see \cite{AmbrosettiRabinowitz1973}, cf. \cite{AmbrosettiMalchiodi2007}*{Chapter 8}, \cite{Struwe2008}*{Chapter II}).

In the nonparametric setting, the foundations of the theory go back to Smith's thesis \cite{Smith1982}, where he proved the existence of an embedded minimal two-sphere in any Riemannian three-sphere.
The survey \cite{ColdingDeLellis2003} by Colding--De Lellis is now the main reference for a complete proof of the regularity and the embeddedness of minimal surfaces arising from this min-max procedure in three-dimensional ambient manifolds. The authors follow ideas from J. Pitts, L. Simon and F. Smith.

Then, the theory was extended:
\begin{itemize}
\item to prove a genus bound on the resulting surface in \cite{DeLellisPellandini2010}, improved to the optimal one in \cite{Ketover2019};
\item to the free boundary case in \cite{Li2015};
\item to the equivariant setting (i.e., imposing a symmetry to the entire procedure) in the closed case \cite{Ketover2016Equivariant} and in the case with boundary \cite{Ketover2016FBMS};
\item to the multi-parameter setting in \cite{ColdingGabaiKetover2018};
\item to closed minimal hypersurfaces in higher dimensions in \cite{DeLellisTasnady2013}, without proving any control on the topology.
\end{itemize}

In this thesis, we consider surfaces in a three-dimensional ambient manifold $(M^3,\smetric)$ with strictly mean convex boundary, in such a way that property \hypC{}, and thus also property \hypP{}, holds. Moreover, we consider a finite group $G$ of \emph{orientation-preserving} isometries of $M$.

Given a smooth $G$-sweepout $\{\Sigma_t\}_{t\in I^n}$, i.e., roughly, a curve of $G$-equivariant orientable surfaces varying smoothly with respect to $t\in I^n\eqdef[0,1]^n$ (see \cref{def:GSweepout} for a precise definition), we define its saturation as the set of $G$-sweepouts obtained from $\{\Sigma_t\}_{t\in I^n}$ by composing it with $G$-equivariant isotopies.
Assume also that the saturation satisfies a mountain pass condition.
Then, the equivariant min-max theorem roughly states that there exists a min-max sequence $\{\Sigma^j\}_{j\in\N}$, which consists of surfaces in $G$-sweepouts of the saturation, converging in the sense of varifolds to a $G$-equivariant free boundary minimal surface $\Sigma^\infty$ (up to multiplicity). Furthermore, the genus of $\Sigma^\infty$ is bounded by the genus of the surfaces $\Sigma_t$, $t\in I^n$.

The theorem is the result of the papers cited above (apart from \cite{DeLellisTasnady2013}, treating the higher dimensional case). In particular, a full statement for one-parameter sweepouts (i.e., $n=1$) can be found in \cite{Ketover2016FBMS}*{Theorem 3.2}. However, since the arguments are split in several different references and since we need to refine some of the results for this thesis, we present a proof of the regularity and genus bound in \cref{sec:ProofRegularityGenus}.

Furthermore, in \cite{Franz2021}, we are able to enrich the conclusion of the min-max theorem, proving a property of the index of the resulting surface.

\begin{alphatheorem} [\cite{Franz2021}*{Thm 1.10} \textrightarrow\ \cref{thm:EquivMinMax}]
The free boundary minimal surface $\Sigma^\infty$ obtained from the $G$-equivariant min-max theorem described above has $G$-equivariant index less or equal than $n$.
\end{alphatheorem}
Here, by \emph{$G$-equivariant index} we mean the maximal dimension of a linear subspace of the vector bundle of \emph{$G$-equivariant} normal vector fields (i.e., vector fields that are invariant with respect to the action of the symmetry group $G$) where $Q^{\Sigma^\infty}$ (introduced in the previous section) is negative definite (see \cref{def:GequivIndex}).
\begin{remark*}
The equivariant index is expected to be less or equal than $n$ since $\Sigma^\infty$ is constructed via an equivariant min-max procedure with $n$-parameters. However the proof is highly technically involved, as is the one proving the regularity and the genus bound of $\Sigma^\infty$, and consists in adapting the case without equivariance studied in \cite{MarquesNeves2016} in the Almgren--Pitts setting.
\end{remark*}

\subsection*{Free boundary minimal surfaces with connected boundary}

Given the equivariant min-max theorem, the first difficulty in proving existence of free boundary minimal surfaces with a given topology stands in choosing a suitable symmetry group $G$ and a suitable sweepout $\{\Sigma_t\}_{t\in I^n}$. This is far from being trivial, especially if one does not know \emph{a priori} how the resulting surface should look like.

\begin{figure}[htbp]
\centering
\includegraphics[scale=0.8]{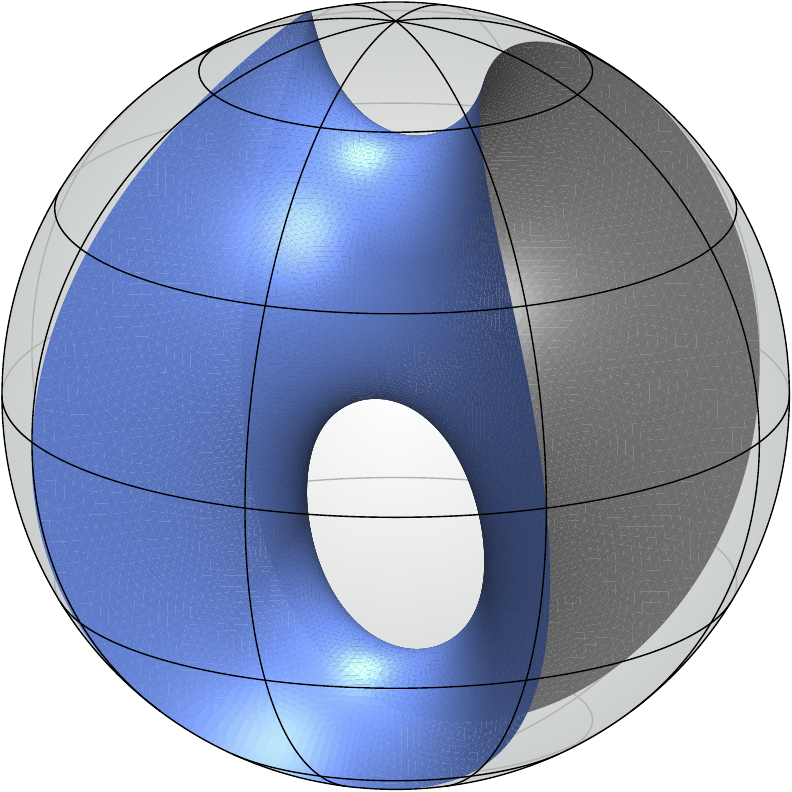}%
\hfill
\includegraphics[scale=0.8]{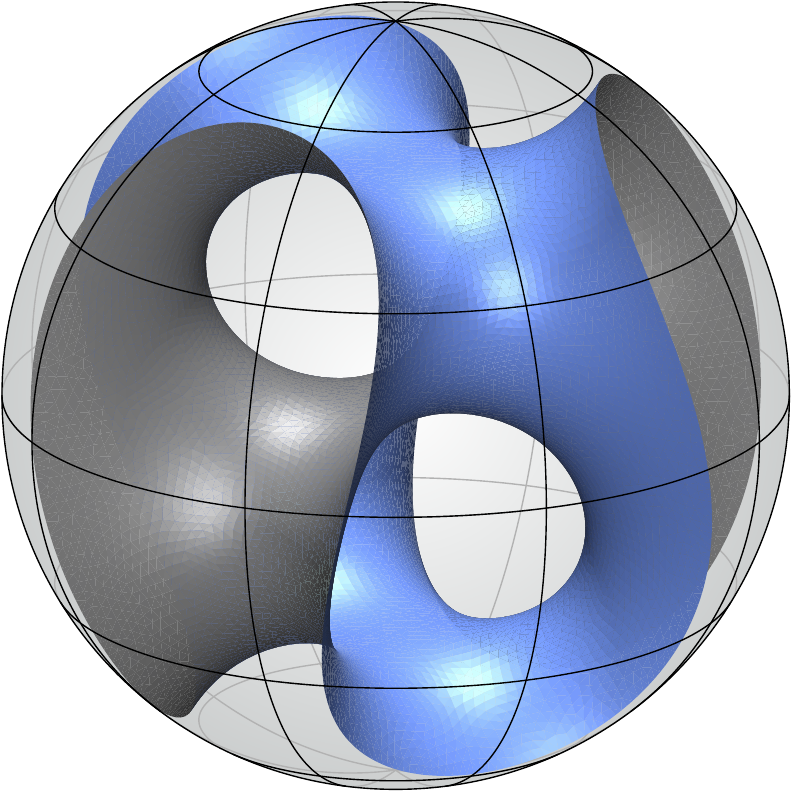}%
\caption{Numerical simulations from \cite{SchulzGallery} of free boundary minimal surfaces with connected boundary and genus $1$ and $2$, respectively.}%
\label{fig:fbms-b1}
\end{figure}

With the help of numerical simulations (see \cref{fig:fbms-b1}, and \cite{SchulzGallery}), we are able to successfully apply an equivariant min-max procedure and prove \cref{thm:NewFamilyFBMS}.
In the proof we use one-parameter sweepouts, and we impose a dihedral symmetry $\mathbb{D}_{g+1}$, where $g$ is going to be the genus of the resulting surface (see \cref{def:DihedralGroup} for a definition of the dihedral group). 
\begin{remark*}
From the numerical simulations we expect a further symmetry with respect to a plane passing through the vertical axis. However, the symmetry with respect to a plane is orientation-reversing, and the equivariant min-max theorem has been developed only for finite groups consisting of orientation-preserving isometries so far.
\end{remark*}

The $G$-sweepout $\{\Sigma_t\}_{t\in I}$ is explicitly described in the proof and, roughly, each surface $\Sigma_t$ for $0<t<1$ looks like three parallel discs joint by ribbons, with genus $g$ and connected boundary (see \cref{sec:sweepout}). A careful control on the area of the surfaces of the sweepout is important to show that the resulting free boundary minimal surface $M_g$ is obtained as a limit with multiplicity one (see \cref{cor:WidthBiggerPi} and \cref{lem:multiplicity}).

The control on the genus of $M_g$ follows from the semicontinuity of the genus in the min-max procedure and from a structure theorem for surfaces with dihedral symmetry (see \cref{lem:structural}). Unfortunately, no control is available on the number of boundary components of surfaces resulting from the min-max procedure. Hence, we prove that $M_g$ has connected boundary by an ad-hoc argument, first singling out a preferred component and then carefully applying Simon's Lifting Lemma (cf. \cite{DeLellisPellandini2010}*{Proposition 2.1}).

\begin{remark}
Similar methods have been employed by Buzano--Nguyen--Schulz in \cite{BuzanoNguyenSchulz2021} to construct noncompact self-shrinkers for the mean curvature flow with arbitrary genus. 
However, observe that, in their setting, the control on the number of ends of the resulting surface is still quite elusive and they are able to fully determine it only for large genus.
\end{remark}

Finally, we also obtain information on the index of $M_g$ and on the behavior as $g\to\infty$. 

\begin{alphaproposition} [cf. \cref{thm:main-fbms-b1}]
The surface $M_g$ obtained in \cref{thm:NewFamilyFBMS} has $\dih_{g+1}$-equivariant index equal to $1$ and converges in the sense of varifolds to the equatorial disc with multiplicity three, as $g\to\infty$.
\end{alphaproposition}
\begin{remark*}
The computation of the $\dih_{g+1}$-equivariant index is contained in \cite{Franz2021}*{Theorem 1.13} and is a consequence of the equivariant index bound discussed above. The behavior of $M_g$ as $g\to\infty$ is not presented in any other reference apart from this thesis, but it is the result of discussions after the publication of the article \cite{CarlottoFranzSchulz2020} among the authors.
\end{remark*}

So far, the only free boundary minimal surfaces in $B^3$ for which we know the index are the equatorial disc (with index $1$) and the critical catenoid (with index $4$, see \cite{Devyver2019,SmithZhou2019,Tran2020}). However, the situation in the closed case gives hope for improvement, since recently Kapouleas and Wiygul in \cite{KapouleasWiygul2020} computed the index of a family of Lawson surfaces in $S^3$. 
Hence, computing the equivariant index of the surfaces $M_g$ is meant as a first step towards the computation of the actual Morse index of such surfaces.

\clearpage

\section*{Schematic content of the thesis}

\addcontentsline{toc}{section}{Schematic content of the thesis} 

\vspace{-2ex}

\begin{center}
\usetikzlibrary{trees}
\tikzstyle{every node}=[anchor=west]
\tikzstyle{tit}=[shape=rectangle, rounded corners,
    draw, minimum width = 3cm]
\tikzstyle{sec}=[shape=rectangle, rounded corners]
\tikzstyle{optional}=[dashed,fill=gray!50]
\begin{tikzpicture}[scale=0.9]

\tikzstyle{mybox} = [draw, rectangle, rounded corners, inner sep=10pt, inner ysep=20pt]
\tikzstyle{fancytitle} =[fill=white, text=black, draw]

\begin{scope}[align=center, font=\small]
 \node [tit] (pre) {Preliminaries};
 \node [sec] [below of = pre, yshift=-0.5cm] (preFBMS) {Definitions and basic \\ results about FBMS \\ (\cref{chpt:preliminariesFBMS})};
 
 \node[sec] [right of =preFBMS, xshift=6.5cm,yshift=1.4cm] (intro) {\hyperref[chpt:intro]{Introduction}};
 \node[sec] [right of =preFBMS, xshift=6.5cm,yshift=-0.2cm] (mm) {About the definition of Morse index \\for FBMS that are not proper\\(\cref{sec:Properness})};

 \node[sec] [below of = preFBMS, yshift= -3cm] (inco) {Context and results \\ (\cref{intro:complexity})};
 
 \node [sec] [below of = inco, yshift=-1.2cm] (fbml) {Free boundary minimal \\ laminations\\ (\cref{sec:FBMLam})};
 \node [sec] [below of = fbml, yshift=-1.5cm] (morse) {Morse theoretic \\ arguments\\ (\cref{sec:MorseTheory})};
 \node [sec] [below of = morse, yshift=-4.1cm] (inex) {Context and results\\ (\cref{intro:existence})};
 \node [sec] [below of = inex, yshift=-1.5cm] (equiv) {Equivariant surfaces \\ and spectrum\\ (\cref{chpt:GroupIsom,chpt:EquivSpectrum,chpt:EquivFBMS})};
\end{scope}

\draw[gray,dashed,rounded corners] ($(pre)+(-2.3,0.5)$)  rectangle ($(pre)+(2.3,-21.9)$);

\begin{scope}[align=center, xshift=4.7cm, yshift=-4.1cm, font=\small]

\node [sec] (pi) {\textbf{\cref{part:Complexity}: Relating topology, area and Morse index}\\ based on \cite{CarlottoFranz2020}};

\node[sec] [below of = pi, yshift=-1cm,xshift=-2cm] (deg) {Analysis of sequence of FBMS \\ with bounded index\\
(\cref{sec:MacroDescr,sec:LocalDegeneration,sec:Surgery})};
\node[sec] [below of = pi, yshift=-1cm,xshift=3cm] (diam) {Diameter bound \\for stable FBMS\\ (\cref{sec:CptnessStable})};
\node[sec] [below of = pi, yshift=-3.5cm,xshift=0.5cm] (area) {Qualitative area and topology \\ bounds for FBMS with bounded index\\ (\cref{sec:AreaBound})};

\node[sec] [below of = area, yshift=-1.3cm] (count) {Topology does not bound \\ area and index of FBMS\\ (\cref{sec:Counterexample})};
\end{scope}

\tikzset{bigarr/.style={
decoration={markings,mark=at position 1 with {\arrow[scale=1.5]{>}}},
postaction={decorate}}}

\draw[bigarr,gray] ([xshift=1cm,yshift=0.1cm]deg.south) to ([xshift=1cm,yshift=-0.8cm]deg.south);
\draw[bigarr,gray] ([xshift=-0.6cm,yshift=0.1cm]diam.south) to ([xshift=-0.6cm,yshift=-0.8cm]diam.south);

\draw[blue,rounded corners] ($(pi)+(-10.8,0.8)$)  rectangle ($(pi)+(6,-8.7)$);

\begin{scope}[align=center, xshift=4.2cm, yshift=-14.2cm, font=\small]

\node [sec] (pii) {\textbf{\cref{part:Existence}: Existence of FBMS via equivariant min-max}\\ based on \cite{CarlottoFranzSchulz2020} and \cite{Franz2021}};

\node [sec] [below of = pii, yshift=-0.7cm, xshift=-2cm] (reg) {Regularity and genus bound \\(\cref{sec:ProofRegularityGenus})};
\node [sec] [below of = pii, yshift=-0.7cm, xshift=3cm] (ind) {Index bound \\ (\cref{chpt:IndexBound})};

\node [sec] [below of = pii, yshift=-2.7cm,xshift=0.5cm] (thm) {Equivariant min-max theorem\\ (\cref{thm:EquivMinMax})};

\node [sec] [below of = thm, yshift=-1.5cm] (fbb1) {Existence of FBMS in $B^3$ with\\ connected boundary and arbitrary genus\\ (\cref{chpt:FBMSb1})};
\end{scope}

\draw[bigarr,gray] ([xshift=1cm,yshift=0.1cm]reg.south) to ([xshift=1cm,yshift=-1cm]reg.south);
\draw[bigarr,gray] ([xshift=-0.6cm,yshift=0.1cm]ind.south) to ([xshift=-0.6cm,yshift=-1cm]ind.south);
\draw[bigarr,gray] ([yshift=0.1cm]thm.south) to ([yshift=-1cm]thm.south);

\draw[blue,rounded corners] ($(pii)+(-10.8,0.8)$)  rectangle ($(pii)+(6,-8)$);

\end{tikzpicture}
\end{center}


\chapter[Preliminaries on free boundary minimal surfaces]{Preliminaries on free boundary\\ minimal surfaces} \label{chpt:preliminariesFBMS}

In this chapter, we present several basic definitions and preliminary results about free boundary minimal surfaces in a three-dimensional ambient manifold.

\section{First variation of the area functional} \label{sec:FirstVar}

Given a Riemannian manifold $(\amb^3,g)$, we denote by $\vf(\amb)$ the linear space of smooth ambient vector fields $X$ such that 
\begin{enumerate}[label={\normalfont(\roman*)}]
\item $X(x) \in T_x\amb$ for all $x\in\amb$,
\item $X(x)\in T_x\partial\amb$ for all $x\in\partial\amb$.
\end{enumerate}

\begin{definition}
Given a two-dimensional varifold $V\in\mathcal{V}^2(\amb)$ in $M$ and a vector field $X\in\vf(M)$, we denote by $\delta V(X)$ the \emph{first variation of the area} of $V$ along $X$, namely
\[
\delta V(X) = \frac{\d}{\d t}\Big|_{t=0} \norm{(\psi_t)_*V}(M),
\]
where $\psi\colon[0,+\infty)\times M\to M$ is the flow generated by $X$.
\end{definition}

\begin{definition}
We say that a varifold $V\in\mathcal{V}^2(M)$ in $M$ is \emph{free boundary stationary} if $\delta V(X)=0$ for all $X\in\vf(\amb)$.
\end{definition}

In the case of a smooth, embedded surface $\ms^2\subset\amb$, the first variation of the area of $\Sigma$ along a vector field $X\in\vf(M)$ is given by
\begin{equation} \label{eq:FirstVar}
\begin{split}
\frac{\d}{\d t}\Big|_{t=0} \Haus^2(\psi_t(\ms))  &= \int_\ms \div_\ms (X) \de \Haus^2 \\
&= -\int_\ms \scal HX \de \Haus^2 + \int_{\partial\ms} \scal{X}\eta \de \Haus^1 ,
\end{split}
\end{equation}
where $\psi\colon[0,+\infty)\times M\to M$ is the flow generated by $X$.
The same formula holds true, modulo straightforward notational changes, in the case of immersed surfaces.

\section{Properly embedded surfaces} \label{sec:properFBMS}

When $M$ and $\ms$ are compact, it is customary to say that a surface $\ms^2\subset \amb$ is \emph{properly embedded} if $\partial\ms = \ms\cap \partial\amb$. If that is the case, the equation above implies that $\ms$ is a stationary point of the area functional if and only if it has zero mean curvature and meets the ambient boundary orthogonally. In this case $\ms$ is called \emph{free boundary minimal surface}.

\begin{remark}  \label{subsec:ImpRem}
It is possible to extend this definition also to nonproperly embedded surfaces, convening that $\ms^2\subset\amb^3$ is a free boundary minimal surface if it has zero mean curvature and meets the boundary of the ambient manifold orthogonally along its own boundary. 
\end{remark}

\begin{remark}\label{rem:edged}
In certain circumstances, we shall deal with free boundary minimal surfaces $\ms$ whose boundary is not entirely contained in $\partial\amb$, that is when $\partial\ms\setminus\partial\amb \not=\emptyset$ (cf. also \cite[Remark 13]{AmbrozioCarlottoSharp2018Compactness}).
For example, we look at the intersection of free boundary minimal surfaces with geodesic balls in the ambient manifold centered at boundary points: in that case we talk about \emph{edged} free boundary minimal surface and we mean that they satisfy the free boundary condition only with respect to the boundary of the ambient manifold that is contained in the ball, and not necessarily with respect to the relative boundary of the ball in question.
\end{remark}

In this thesis we focus our attention on compact Riemannian manifolds $(\amb^3,g)$ that satisfy property \hypP{}, that has been stated in the introduction. However, note that local limits of properly embedded free boundary minimal surfaces can be nonproperly embedded even assuming property \hypP{}.
In particular, this is the case of \cref{thm:LamCptness}.

\begin{remark} \label{rem:ConfusionProper}
We stress that, when $\ms$ or $\amb$ are noncompact, there is a possible confusion between \emph{proper} in the sense stated above, namely $\partial \ms = \ms\cap \partial\amb$, and \emph{proper} in the sense that the inclusion map $\ms\hookrightarrow\amb$ is proper as in General Topology.
Unfortunately, in \cref{part:Complexity}, we have to use both meanings, since we need to consider noncompact free boundary minimal surfaces in our ambient manifold and even free boundary minimal surfaces in noncompact ambient spaces.
Therefore, when we talk about \emph{properly embedded} surfaces with boundary we mean that \emph{both} properties are satisfied.
In all cases of possible ambiguity, we try to make the interpretation as explicit and clear as possible.
\end{remark}

Thanks to the maximum principle, property \hypP{} is implied by the following geometric condition on the boundary.
\begin{description}
\item [$(\mathfrak{C})$]\ \ \  The boundary $\partial\amb$ is mean convex and has no minimal component. \label{HypC}
\end{description}
\begin{remark}
Note that the assumptions in the main theorems stated in the \hyperref[chpt:intro]{Introduction} all imply \hypC{} and thus \hypP{}.
\end{remark}

\section{Second variation of the area functional} \label{sec:SecondVariation}

Given a properly embedded free boundary minimal surface $\ms^2\subset M^3$, one can consider the second variation of the area functional, which can be written as
\begin{equation} \label{eq:SecondVar}
\begin{split}
\frac{\d^2}{\d t^2}\Big|_{t=0} \Haus^n(\psi_t(\ms)) &= \int_\ms (\abs{\cov^{\perp}X^{\perp}}^2 - (\Ric_{\amb}(X^\perp,X^\perp) + \abs A^2 \abs{X^\perp}^2)) \de \Haus^2 {}\\
&\pheq + \int_{\partial\ms} \II^{\partial\amb}(X^\perp, X^\perp) \de \Haus^1\comma
\end{split}
\end{equation}
where $X^\perp$ is the normal component of $X$ and $\cov^\perp$ is the induced connection on the normal bundle $N\ms$ of $\ms\subset\amb$.
Thus, given a section $Y\in\Gamma(N\ms)$ of the normal bundle, the second variation along the flow generated by $Y$ is equal to the quadratic form
\[
 Q^\ms(Y,Y) \eqdef  \int_\ms (\abs{\cov^{\perp}Y}^2 - (\Ric_{\amb}(Y,Y) + \abs A^2 \abs{Y}^2)) \de \Haus^2 + \int_{\partial\ms} \II^{\partial\amb}(Y, Y) \de \Haus^1 \point
\]
The \emph{(Morse) index} of $\ms$ is defined as the maximal dimension of a linear subspace of $\Gamma(N\ms)$ where $Q^\ms$ is negative definite.
Under the above assumptions, this number equals the number of negative eigenvalues of the following elliptic problem on $\Gamma(N\ms)$:
\[
\begin{cases}
\lapl_\ms^\perp Y + \Ric^\perp_\amb(Y,\cdot) +\abs A^2 Y + \lambda Y = 0 & \text{on $\ms$}\comma\\
\cov^\perp_\eta Y = - (\II^{\partial\amb}(Y,\cdot))^\sharp & \text{on $\partial\ms$}\point
\end{cases}
\]

\begin{remark}
We postpone to \cref{sec:Properness} a thorough discussion about the definition of Morse index, in particular in the case when $\Sigma$ is not properly embedded. 
\end{remark}

Observe that, if $\ms^2\subset\amb^3$ is two-sided, then every vector field $Y\in\Gamma(N\ms)$ can be written as $Y = f\nu$ and $Q^\ms(Y,Y)$ coincides with the quadratic form $Q_\ms(f,f)$, defined as
\[
\begin{split}
Q_\ms(f,f) &\eqdef \int_\ms (\abs{\cov f}^2 - (\Ric_{\amb}(\nu,\nu) + \abs A^2 )f^2) \de \Haus^2 + \int_{\partial\ms} \II^{\partial\amb}(\nu, \nu)f^2 \de \Haus^1\\
&= -\int_\ms f\jac_\ms(f) \de\Haus^2 + \int_{\partial\ms} \left( f\DerParz f \eta + \II^{\partial\amb}(\nu,\nu) f^2 \right) \de\Haus^1 \comma
\end{split}
\]
where $\jac_\ms \eqdef \lapl_\ms + (\Ric_\amb(\nu,\nu) + \abs A^2)$ is the (scalar) Jacobi operator of $\ms$.
Hence, the index coincides with the number of negative eigenvalues of the elliptic problem
\[
\begin{cases}
\jac_\Sigma f = \lapl_\ms f + (\Ric_\amb(\nu,\nu) +\abs A^2)f  + \lambda f = 0 & \text{on $\ms$}\comma\\
\DerParz f\eta = - \II^{\partial\amb}(\nu,\nu) f & \text{on $\partial\ms$}
\end{cases}
\]
on $C^\infty(\Sigma)$ (see also \cref{chpt:EquivSpectrum} for further details in presence of a group of symmetries).

\section{Euler characteristic} \label{sec:EulChar}
Recall that the Euler characteristic of a compact surface $\ms$ is equal to
\[
\chi(\ms) = \begin{cases}
            2 - 2\genus(\ms) - \bdry(\ms) & \text{if $\ms$ is orientable} \comma\\
            1- \genus(\ms) - \bdry(\ms) & \text{if $\ms$ is not orientable} \point
            \end{cases}
\]
Now let $\ms_1,\ms_2$ be two compact oriented surfaces with boundary and consider $c_1$, $c_2$ be two boundary components of $\ms_1$ and $\ms_2$ respectively. Note that $c_1$ and $c_2$ are both homeomorphic to $S^1$.
As shown in \cref{fig:GlSurf}, we can glue $\ms_1$ and $\ms_2$ along $c_1$ and $c_2$ in two ways:
\begin{enumerate} [label={\normalfont(\roman*)}]
\item We can attach all $c_1$ to all $c_2$. \label{at:circle}
\item We can attach an arc of $c_1$ to an arc of $c_2$. \label{at:segment}
\end{enumerate}

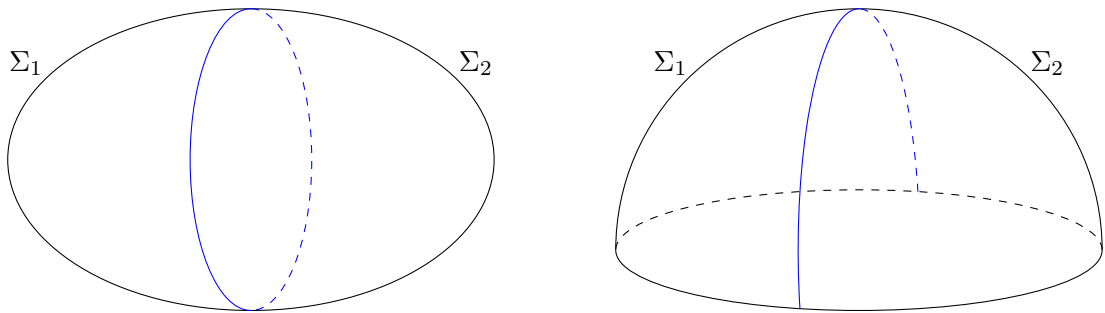
\begin{figure}[htpb]
\centering
\begin{tikzpicture}[scale=0.8]
\pgfmathsetmacro{\radius}{4}
\pgfmathsetmacro{\a}{4/sqrt(17)}
\draw (-\radius,0) arc (180:360:{\radius} and {\radius/4});
\draw[dashed] (\radius,0) arc (0:180:{\radius} and {\radius/4});
\draw (\radius,0) arc (0:180:{\radius});
\draw[dashed,blue] (\a,\a) arc ({atan(1/4)}:90:{\radius/4} and {\radius});
\draw[blue] (0,\radius) arc (90:{180+atan(1/4)}:{\radius/4} and {\radius});

\node at ({-\radius+0.9}, {\radius-0.9}) {$\ms_1$};
\node at ({\radius-0.9}, {\radius-0.9}) {$\ms_2$};

\pgfmathsetmacro{\xs}{-5/2*\radius}
\pgfmathsetmacro{\ys}{\radius/2-0.5}
\pgfmathsetmacro{\radiusa}{(\radius/2)+0.5}
\pgfmathsetmacro{\radiusb}{\radius}

\draw (\xs,\ys) ellipse ({\radiusb} and {\radiusa});
\draw[blue] (\xs, {\ys+\radiusa}) arc (90:270:{\radius/4} and {\radiusa});
\draw[dashed,blue] (\xs, {\ys+\radiusa}) arc (90:-90:{\radius/4} and {\radiusa});
\node at ({-\radius+0.3+\xs}, {\radius-0.9}) {$\ms_1$};
\node at ({\radius-0.3+\xs}, {\radius-0.9}) {$\ms_2$};
\end{tikzpicture}
\caption{Gluing of two discs via \ref{at:circle} (on the left) and \ref{at:segment} (on the right).} \label{fig:GlSurf}
\end{figure}

In general, we can construct an oriented surface $\ms$ by gluing $b$ boundary components of $\ms_1,\ms_2$ as in \ref{at:circle} and $b'$ boundary components as in \ref{at:segment}. Then the Euler characteristic of $\ms$ is given by
\[
\chi(\ms) = \chi(\ms_1)+\chi(\ms_2) - b' \point
\]
Therefore $\ms$ has genus equal to $\genus(\ms_1)+\genus(\ms_2) + b + b' - 1$ and number of boundary components equal to the sum of the boundary components of $\ms_1$ and $\ms_2$ minus $2b+b'$.

\section[Reflecting FBMS in a half-space]{Reflecting free boundary minimal surfaces in a half-space} \label{sec:ReflPrinc}

In this section, we recall a standard reflection lemma, which is useful to transfer information about minimal surfaces in $\R^3$ to free boundary minimal surfaces in a half-space. The proof consists in a well-known argument based on elliptic estimates.
\begin{lemma} \label{lem:ReflectionPrinciple}
If $\ms^2$ is an embedded free boundary minimal surface in $\Xi(a)\subset \R^3$ for some $0\ge a >-\infty$ (that is $\ms$ has zero mean curvature and meets $\Pi(a)$ orthogonally along $\partial\ms$), 
then the union of $\ms$ with its reflection with respect to $\Pi(a)$ is a smooth, embedded minimal surface in $\R^3$ without boundary.
\end{lemma}

We also recall the following regularity result for free boundary minimal immersions in a half-space of $\R^3$.

\begin{proposition} \label{cor:MinSurfInR3WithFiniteIndex}
Let $\ms\subset \Xi(a)\subset\R^3$ be a complete free boundary minimal injective immersion with finite index and with $\partial\ms = \Pi(a)\cap \ms$, for some $0\ge a >-\infty$. Then $\ms$ is two-sided, has finite total 
curvature and is properly embedded.
\end{proposition}
\begin{proof}
This is a consequence of Theorem B.1 in \cite{ChodoshKetoverMaximo2017}, after applying the reflection principle \cref{lem:ReflectionPrinciple} and the argument of \cite[Section 2.2]{AmbrozioBuzanoCarlottoSharp2018}, which implies that the reflected surface has finite index.
\end{proof}

We further apply \cite[Section 2.3]{AmbrozioBuzanoCarlottoSharp2018} together with \cite{ChodoshMaximo2016} to obtain topological information from index bounds for free boundary minimal surfaces in a half-space of $\R^3$.

\begin{definition}
Given a complete, connected, properly embedded, free boundary minimal surface $\ms^2$ in $\Xi(a)\subset \R^3$ for some $0\ge a >-\infty$, the number of ends of $\ms$ is the number of connected components of $\ms$ outside any sufficiently large compact set. We denote this number by $\mathrm{ends}(\ms)$.
\end{definition}
Observe that the number of ends of a properly embedded free boundary minimal surface in $\Xi(a)\subset \R^3$ as above is indeed well-defined (see \cite[Remark 26]{AmbrozioBuzanoCarlottoSharp2018}).

\begin{proposition} \label{prop:GeometryOfHalfBubble}
Given $I\ge 0$, there exists $\kappa(I)\ge 0$ such that every complete, connected, properly embedded, free boundary minimal surface  $\ms^2$ in $\Xi(a)\subset \R^3$ for some $0\ge a >-\infty$ of index at most $I$ 
has genus, number of ends and number of boundary components (in $\Pi(a)$) all bounded by $\kappa(I)$.
\end{proposition}
\begin{proof}
Let $\check\ms$ be the union of $\ms$ with its reflection with respect to $\Pi(a)$, as above.
Then, thanks to \cite[Section 2.3]{AmbrozioBuzanoCarlottoSharp2018} (in particular equation (2.5) therein), $\check\ms$ is a complete, connected, properly embedded, minimal surface of $\R^3$ with index less or equal than $2I$.
Hence, using the main estimate in \cite{ChodoshMaximo2016}, we obtain that
\[
\frac 23 (\genus(\check\ms) + \mathrm{ends}(\check\ms)) -1 \le 2I\comma
\]
where $\mathrm{ends}(\check\ms)$ denotes the number of ends of $\check\ms$.

Moreover, by Lemma 29 in \cite{AmbrozioBuzanoCarlottoSharp2018}, it holds that $\chi(\check\ms)-\mathrm{ends}(\check\ms) = 2(\chi(\ms)-\mathrm{ends}(\ms))$. Therefore, we obtain
\begin{align*}
\chi(\ms)-\mathrm{ends}(\ms) &= \frac 12(\chi(\check\ms)-\mathrm{ends}(\check\ms)) = 1 - \genus(\check\ms) - \mathrm{ends}(\check\ms)  \\
&\ge 1 -\frac 32(2I+1)= -3I-\frac 12 \point
\end{align*}
Observe that $\chi(\ms) = 2-2\genus(\ms) - \xi(\ms)$, where $\xi(\ms)$ is the number of boundary components of $\ms\cap B^{\R^3}_R(0)$, for any $R>0$ sufficiently large. Thus we get
\[
2\genus(\ms) + \xi(\ms) +\mathrm{ends}(\ms) \le 3I+\frac 52\comma
\]
from which it follows directly that $\genus(\ms)$ and $\mathrm{ends}(\ms)$ are both bounded by $3I+5/2$.

Finally, note that $\bdry(\ms)\le \xi(\ms)+\mathrm{ends}(\ms)$ and so
\[
3I+\frac 52 \ge 2\genus(\ms) + \xi(\ms) +\mathrm{ends}(\ms) \ge \xi(\ms) +\mathrm{ends}(\ms) \ge \bdry(\ms)\comma
\]
which concludes the proof once we choose $\kappa(I) = 3I +5/2$. 
\end{proof}

\section{Multiplicity one convergence} \label{sec:Mult1Conv}

In this section we prove that convergence to a surface with multiplicity one well-behaves in presence of isolated singularities, which essentially follows from Allard's regularity theory (see \cite{Allard1972}).

\begin{lemma} \label{lem:MultOneConvExtends}
Let $\ms_j^2\subset\amb^3$ be a sequence of connected free boundary minimal surfaces in a three-dimensional complete Riemannian manifold $\amb$. Moreover let $\ms^2\subset \amb$ be an embedded free boundary minimal surface in $\amb$. Assume that the sequence $\ms_j$ converges locally smoothly to $\ms$, with multiplicity one away from a finite set of points\footnote{Here we mean that for every $x\in\ms$ there exists a neighborhood $U$ of $x$ such that $\ms_j\cap U$ converges graphically smoothly to $\ms\cap U$ with multiplicity one.} $\set$.
Then $\ms_j$ converges locally smoothly to $\ms$ everywhere.
\end{lemma}
\begin{proof}
Let $p$ be a point in $\ms\cap\set$ and let ${r_0}>0$ be sufficiently small such that $B_{4{r_0}}(p)$ does not contain any other point of $\set$ apart from $p$. Fix $\eps>0$ and take ${r_0}$ possibly smaller in such a way that 
\[
\frac{\Haus^2(\ms\cap (B_{2{r_0}}(p)\setminus B_{r_0}(p)))}{\omega_2(4{r_0}^2-{r_0}^2)} < 
\begin{cases}
1+\eps/4 & \text{if $p\in \ms\setminus\partial\ms$}\\
\frac 12(1+\eps/4) & \text{if $p\in\partial\ms$} \point
\end{cases}
\]
Since the convergence of $\ms_j$ to $\ms$ is smooth and graphical in $B_{2{r_0}}(p)\setminus B_{r_0}(p)$, then we can assume that the same estimate holds for every $j$ sufficiently large substituting $\eps/4$ with $\eps/2$. By the extended monotonicity formula (cf. \cite{GuangLiZhou2020}) this implies that
\[
\frac{\Haus^2(\ms_j\cap B_{r_0}(p))}{\omega_2{r_0}^2} < 
\begin{cases}
1+\eps/2 & \text{if $p\in \ms\setminus\partial\ms$}\\
\frac 12(1+\eps/2) & \text{if $p\in\partial\ms$} \point
\end{cases}
\]
Thus, again by the same monotonicity formula, taking $r_0>0$ possibly smaller, we have
\[
\frac{\Haus^2(\ms_j\cap B_r(p))}{\omega_2r^2} < 
\begin{cases}
1+\eps & \text{if $p\in \ms\setminus\partial\ms$}\\
\frac 12(1+\eps) & \text{if $p\in\partial\ms$}\comma
\end{cases}
\]
for all $r<r_0$ and $j$ sufficiently large.
This concludes the proof by virtue of Theorem 17 in \cite{AmbrozioCarlottoSharp2018Compactness}.
\end{proof}


\chapter[Morse index for free boundary minimal surfaces]{Morse index for free boundary\\ minimal surfaces}\label{sec:Properness}

Let $(\amb^3,g)$ be a compact Riemannian manifold with nonempty boundary. Given an embedded surface $\ms^2$ in $\amb$ with $\partial\ms\subset \partial\amb$, we wish to compare different definitions of free boundary minimality and Morse index when one allows for an (arbitrary) contact set of $\ms$ along $\partial\amb$, as in \cref{fig:NonProp}. The aim is to analyze this matter, providing some flavour of the rather subtle nature of the question. To avoid ambiguities, throughout this chapter we always use the expression \emph{free boundary minimal surface} to refer to a surface with zero mean curvature that meets the boundary of the ambient manifold orthogonally along its own boundary.

\begin{figure}[htpb]
\centering
\begin{tikzpicture} [scale =0.6]
\pgfmathsetmacro{\R}{4}
\pgfmathsetmacro{\ta}{10}
\pgfmathsetmacro{\xa}{\R*sin(\ta)}
\pgfmathsetmacro{\ya}{\R*cos(\ta)}
\pgfmathsetmacro{\tb}{40}
\pgfmathsetmacro{\xb}{\R*sin(\tb)}
\pgfmathsetmacro{\yb}{\R*cos(\tb)}
\pgfmathsetmacro{\g}{0.02}

\draw (-\xa,-\ya) arc(-90-\ta:-360-90+\ta:\R);
\draw plot [smooth, tension = 0.45] coordinates {(-\xa,-\ya) ({-3/5*\xa},{-8/9*\ya}) (0,-\g) ({3/5*\xa},{-8/9*\ya}) (\xa,-\ya)};
\draw[blue,thick] (-\R,0) -- (\R,0);

\node at (0.85*\R,0.8*\R) {$\amb_0$};
\node[blue] at ({-3/4*\R}, 0.5) {$\ms$};

\begin{scope}[xshift=350]
\draw (-\xb,-\yb) arc(-90-\tb:-360-90+\tb:\R);
\draw plot [smooth, tension = 0.5] coordinates {(-\xb,-\yb) ({-5/6*\xb},{-\yb}) ({-3/4*\xb},{-1/8*\yb}) ({-2/3*\xb}, -\g)};
\draw plot [smooth, tension = 0.5] coordinates {({2/3*\xb}, -\g) ({3/4*\xb},{-1/8*\yb}) ({5/6*\xb},{-\yb}) (\xb,-\yb)};
\draw[blue,thick] (-\R,0) -- (\R,0);
\draw ({-2/3*\xb}, -\g) -- ({2/3*\xb}, -\g);

\node at (0.85*\R,0.8*\R) {$\amb_{r_0}$};
\node[blue] at ({-3/4*\R}, 0.5) {$\ms$};
\end{scope}
\end{tikzpicture}
\caption{Modified unit ball with nonproperly embedded free boundary minimal surface.} \label{fig:NonProp}
\end{figure}
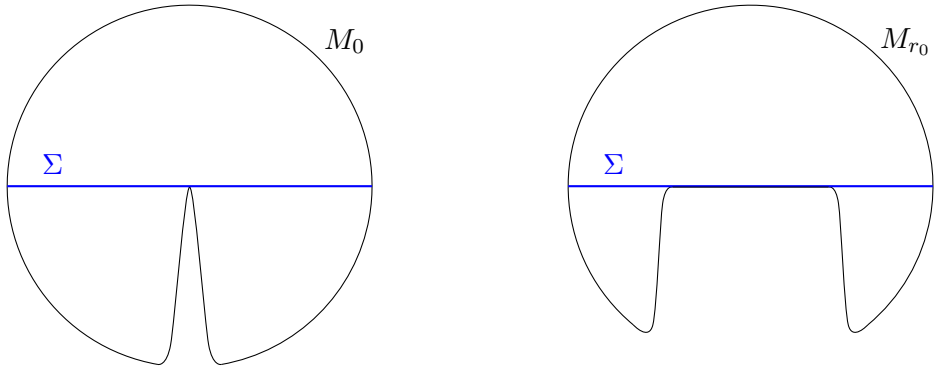

\begin{figure}
\centering
\begin{tikzpicture} [scale =0.6]
\pgfmathsetmacro{\R}{4}
\pgfmathsetmacro{\r}{1}
\pgfmathsetmacro{\t}{30}
\pgfmathsetmacro{\x}{\R*sin(\t)}
\pgfmathsetmacro{\y}{\R*cos(\t)}
\pgfmathsetmacro{\xa}{\R/4*8/3}
\pgfmathsetmacro{\ya}{\R/4*sqrt(80/9)}
\pgfmathsetmacro{\xb}{\xa*sqrt(5/64)}
\pgfmathsetmacro{\yb}{\ya*sqrt(5/64)}
\pgfmathsetmacro{\s}{asin(2/3)}
\pgfmathsetmacro{\g}{0.02}

\draw (-\x,-\y) arc(-90-\t:-360-90+\t:\R);

\draw (-\r,-1.5*\r) arc(180:0:\r);
\draw plot [smooth, tension = 0.7] coordinates {(-\x,-\y) ({-1.1*\r},{-\y}) ({-\r},{-1.5*\r})};
\draw plot [smooth, tension = 0.7] coordinates {(\x,-\y) ({1.1*\r},{-\y}) ({\r},{-1.5*\r})};

\draw[blue,thick] (-\xa,-\ya) -- (-\xb,-\yb);
\draw[blue,thick] (\xa,-\ya) -- (\xb,-\yb);
\draw[blue,thick] (-\xb,-\yb) arc(180-\s:\s:\r);

\node at (0.85*\R,0.8*\R) {$\amb$};
\node[blue] at ({-1/2*\R}, {-1/3*\R}) {$\ms$};
\end{tikzpicture}
\caption{A local minimizer of the area functional that is not a minimal surface.} 
\label{fig:MinNonMin}
\end{figure}
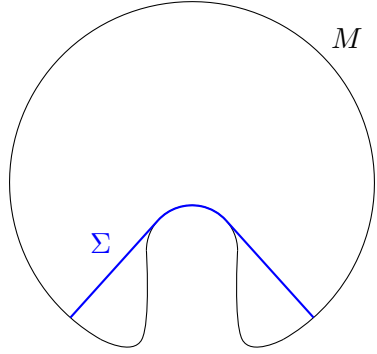

\begin{remark}
Observe that local minimizers of the area functional are not necessarily free boundary minimal surfaces with respect to this definition. See \cref{fig:MinNonMin} for an illustrative picture of a local minimizer which is not a free boundary minimal surface.
\end{remark}

\begin{example}
A useful example to keep in mind is the following one.
Let us denote by $D$ the horizontal equatorial disc in the three-dimensional unit ball $B^3\subset\mathbb{R}^3$.
Given any closed subset $C\subset D$ with $C\cap\partial D = \emptyset$, consider a smooth compact ambient manifold $\amb_C$ obtained from $B^3$ by removing a portion of the lower half-ball in such a way that $\partial\amb_{C}$ intersects the interior of $D$ exactly in $C$. This can be done for every such choice of $C$.
Observe that $D$ is a nonproperly embedded free boundary minimal surface in $\amb_C$. As a special instance, we consider $\amb_C$ for $C = \bar B_{r_0}(0)\subset \R^2$ where $0\le r_0<1$ and we denote it simply by $\amb_{r_0}$ (see \cref{fig:NonProp}).
\end{example}

\section{Possible sets of variations}

Given a surface $\ms\subset\amb$, there are several possible families of variations along which we can deform $\ms$. Some natural choices are (see also \cref{fig:PropScheme}):
\begin{align}
\vf_e(\amb,\ms) &\eqdef \{ X\in\vf(\amb) \st X(x) \in T_x\partial\amb\ \forall x\in\partial\ms  \}; \label{eq:vfe} \\
\vf_i(\amb, \ms) &\eqdef \{ X\in\vf(\amb) \st X(x) \in T_x\partial\amb\ \forall x\in\partial\ms\comma \label{eq:vfi}\\
&\phantom{\eqdef \{ X\in\vf(\amb) \st} g(X(x), \hat\eta(x))\leq 0 \ \forall x\in\partial\amb \};  \notag\\ 
\vf_\partial(\amb) &\eqdef \{ X\in\vf(\amb) \st X(x) \in T_x\partial\amb\ \forall x\in\partial\amb \}; \label{eq:vfd}\\
\vf_c(\amb,\ms) &\eqdef \{ X\in \vf(\amb) \st X(x) \in T_x\partial\amb\ \forall x\in\partial\ms\comma \label{eq:vfc}\\
&\phantom{\eqdef \{ X\in \vf(\amb) \st} \supp(X)\cap \ms \Subset \ms\setminus (\operatorname{int}(\ms)\cap \partial\amb) \}. \notag 
\end{align}
Recall that $\hat \eta$ is the outward unit co-normal to $\partial\amb$.
Moreover observe that 
\[
\vf_e(\amb,\ms)\supset \vf_i(\amb,\ms) \supset \vf_\partial(M) \supset \vf_c(\amb,\ms).
\]

\begin{remark}
Note that in the rest of the thesis (see \hyperref[chpt:notation]{Basic notation} and \cref{sec:FirstVar}) we write $\vf(M)$ instead of $\vf_\partial(M)$. In this chapter, we decided to use the notation $\vf_\partial(M)$ to emphasize that these vector fields are tangent to the boundary of $\partial M$.
\end{remark}

Hereafter, we write $\vf_*(\amb,\ms)$ to denote any of the previous subsets of $\vf(M)$.

\begin{figure}[htpb]
\centering
\begin{tikzpicture} [scale=0.58]
\pgfmathsetmacro{\R}{4}
\pgfmathsetmacro{\ta}{10}
\pgfmathsetmacro{\xa}{\R*sin(\ta)}
\pgfmathsetmacro{\ya}{\R*cos(\ta)}
\pgfmathsetmacro{\tb}{40}
\pgfmathsetmacro{\xb}{\R*sin(\tb)}
\pgfmathsetmacro{\yb}{\R*cos(\tb)}
\pgfmathsetmacro{\g}{0.02}
\pgfmathsetmacro{\h}{0.1*\R}

\begin{scope}[xshift=0]
\draw (-\xb,-\yb) arc(-90-\tb:-360-90+\tb:\R);
\draw plot [smooth, tension = 0.5] coordinates {(-\xb,-\yb) ({-5/6*\xb},{-\yb}) ({-3/4*\xb},{-1/8*\yb}) ({-2/3*\xb}, -\g)};
\draw plot [smooth, tension = 0.5] coordinates {({2/3*\xb}, -\g) ({3/4*\xb},{-1/8*\yb}) ({5/6*\xb},{-\yb}) (\xb,-\yb)};
\draw[blue,thick] (-\R,0) -- (\R,0);
\draw ({-2/3*\xb}, -\g) -- ({2/3*\xb}, -\g);

\def\up{{-0.9, -0.7, -0.5, -0.3,-0.1, 0.1, 0.3, 0.5, 0.7, 0.9}}
\foreach \i in {0,...,9}
  \draw[->] (\up[\i]*\R, 0) -- (\up[\i]*\R, \h);

\def\down{{-0.8, -0.6, -0.4, -0.2, 0, 0.2, 0.4, 0.6, 0.8}}
\foreach \i in {0,...,8}
  \draw[->] (\down[\i]*\R, 0) -- (\down[\i]*\R, -\h);
  
\node at (-0.9*\R, 0.9*\R) {\eqref{eq:vfe}};
\end{scope}

\begin{scope}[xshift=300] 
\draw (-\xb,-\yb) arc(-90-\tb:-360-90+\tb:\R);
\draw plot [smooth, tension = 0.5] coordinates {(-\xb,-\yb) ({-5/6*\xb},{-\yb}) ({-3/4*\xb},{-1/8*\yb}) ({-2/3*\xb}, -\g)};
\draw plot [smooth, tension = 0.5] coordinates {({2/3*\xb}, -\g) ({3/4*\xb},{-1/8*\yb}) ({5/6*\xb},{-\yb}) (\xb,-\yb)};
\draw[blue,thick] (-\R,0) -- (\R,0);
\draw ({-2/3*\xb}, -\g) -- ({2/3*\xb}, -\g);

\def\up{{-0.9, -0.7, -0.5, -0.3,-0.1, 0.1, 0.3, 0.5, 0.7, 0.9}}
\foreach \i in {0,...,9}
  \draw[->] (\up[\i]*\R, 0) -- (\up[\i]*\R, \h);

\def\down{{-0.8, -0.6, 0.6, 0.8}}
\foreach \i in {0,...,3}
  \draw[->] (\down[\i]*\R, 0) -- (\down[\i]*\R, -\h);
  
\node at (-0.9*\R, 0.9*\R) {\eqref{eq:vfi}};
\end{scope}

\begin{scope}[yshift=-270]
\draw (-\xb,-\yb) arc(-90-\tb:-360-90+\tb:\R);
\draw plot [smooth, tension = 0.5] coordinates {(-\xb,-\yb) ({-5/6*\xb},{-\yb}) ({-3/4*\xb},{-1/8*\yb}) ({-2/3*\xb}, -\g)};
\draw plot [smooth, tension = 0.5] coordinates {({2/3*\xb}, -\g) ({3/4*\xb},{-1/8*\yb}) ({5/6*\xb},{-\yb}) (\xb,-\yb)};
\draw[blue,thick] (-\R,0) -- (\R,0);
\draw ({-2/3*\xb}, -\g) -- ({2/3*\xb}, -\g);

\def\up{{-0.9, -0.7, -0.5, 0.5, 0.7, 0.9}}
\foreach \i in {0,...,5}
  \draw[->] (\up[\i]*\R, 0) -- (\up[\i]*\R, \h);

\def\down{{-0.8, -0.6, 0.6, 0.8}}
\foreach \i in {0,...,3}
  \draw[->] (\down[\i]*\R, 0) -- (\down[\i]*\R, -\h);

\node at (-0.9*\R, 0.9*\R) {\eqref{eq:vfd}};
\end{scope}

\begin{scope}[xshift=300, yshift=-270]
\draw (-\xb,-\yb) arc(-90-\tb:-360-90+\tb:\R);
\draw plot [smooth, tension = 0.5] coordinates {(-\xb,-\yb) ({-5/6*\xb},{-\yb}) ({-3/4*\xb},{-1/8*\yb}) ({-2/3*\xb}, -\g)};
\draw plot [smooth, tension = 0.5] coordinates {({2/3*\xb}, -\g) ({3/4*\xb},{-1/8*\yb}) ({5/6*\xb},{-\yb}) (\xb,-\yb)};
\draw[blue,thick] (-\R,0) -- (\R,0);
\draw ({-2/3*\xb}, -\g) -- ({2/3*\xb}, -\g);

\def\up{{-0.9, -0.7, 0.7, 0.9}}
\foreach \i in {0,...,3}
  \draw[->] (\up[\i]*\R, 0) -- (\up[\i]*\R, \h);

\def\down{{-0.8, -0.6, 0.6, 0.8}}
\foreach \i in {0,...,3}
  \draw[->] (\down[\i]*\R, 0) -- (\down[\i]*\R, -\h);
  
\node at (-0.9*\R, 0.9*\R) {\eqref{eq:vfc}};
\end{scope}
\end{tikzpicture}
\caption{Visualization of the different possible sets of variations.} \label{fig:PropScheme}
\end{figure}
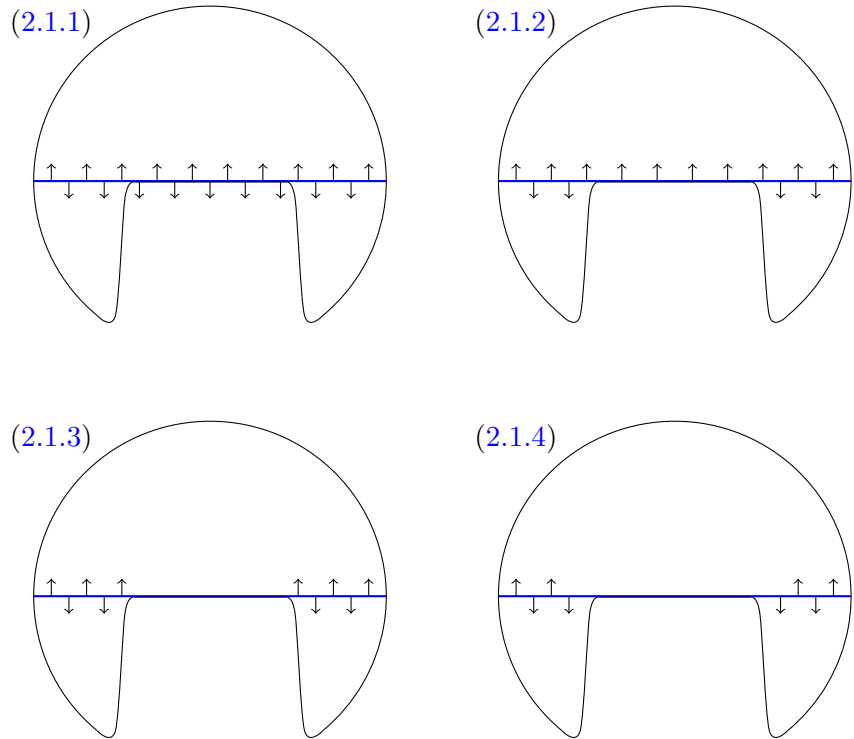

\begin{proposition} \label{prop:FirstSecMakeSense}
Let $\amb$, $\ms$ be as above and let $X\in \vf_e(\amb,\ms)$.
Consider $\check \amb$ to be a compact manifold in which $\amb$ embeds as a regular domain and such that $\ms$ is properly embedded in $\check \amb$, namely $\partial\ms = \ms\cap \partial\check\amb$. Moreover, let $\check X\in \vf(\check\amb)$ be a smooth extension of the vector field $X$ to all $\check \amb$. Then the first variation of the area of $\ms$ with respect to $\check X$ does not depend on the choice of the extensions $\check\amb$ and $\check X$, and is given by \eqref{eq:FirstVar}. 
Furthermore, if $\ms$ is a free boundary minimal surface, then we can compute the second variation of the area with respect to $\check X$ and the result again does not depend on the choice of the extensions and is given by \eqref{eq:SecondVar}.
\end{proposition}
\begin{proof}
The result follows by simply inspecting equations \eqref{eq:FirstVar} and \eqref{eq:SecondVar} for the first and second variation, respectively, applied to $\check X$.
Indeed, the only possible term in the formulae that is not obviously independent of the extensions is $\abs{\cov^\perp \check X^\perp}^2$. 
However, keeping in mind that the extensions $\check\amb$ and $\check X$ are smooth, the covariant derivatives of $\check X^{\perp}$ are uniquely determined by their values along $\ms$.
\end{proof}

\begin{remark}
Given \cref{prop:FirstSecMakeSense}, and thanks to \eqref{eq:FirstVar}, we can say that:
\begin{itemize}
\item If we consider as variations $\vf_e(\amb,\ms)$ or $\vf_i(\amb,\ms)$, the critical points are exactly all and only the free boundary minimal surfaces.
\item In the other two cases $\vf_\partial(\amb)$ and $\vf_c(\amb,\ms)$, free boundary minimal surfaces are critical points but the other implication is not true.
Indeed, being a critical point with respect to these variations does not impose \emph{any} condition on $\ms\cap (\partial\amb\setminus\partial\ms)$. In particular, surfaces as in \cref{fig:MinNonMin} are critical points.
\end{itemize}
\end{remark}

\section{Four alternative definitions}
Given a free boundary minimal surface $\ms$, we can now try to define the Morse index of $\ms$ with respect to variations in $\vf_*(\amb,\ms)$.

In the cases when $\vf_*(\amb,\ms)$ is a vector space, we can just mimic the classical definition of Morse index. 
Namely, denoting by $\vf_*^R(\amb,\ms)\subset \Gamma(T\ms)$ the set of restrictions of elements of $\vf_*(\amb,\ms)$ to $\Gamma(T\ms)$, we define the Morse index $\ind_*(\ms)$ with respect to the variations in $\vf_*(\amb,\ms)$ as the maximal dimension of a linear subspace of $\Gamma(N\ms)\cap \vf_*^R(\amb,\ms)$ where the second variation of the area (given by \eqref{eq:SecondVar}) is negative definite.
In this way we define $\ind_e(\ms)$, $\ind_\partial(\ms)$ and $\ind_c(\ms)$.
\begin{remark}
This definition of Morse index $\ind_c(\ms)$ coincides with the one given by Guang--Wang--Zhou in \cite{GuangWangZhou2021} and employed in \cite{GuangLiWangZhou2021}.
\end{remark}

Observe that $\vf_i^R(\amb,\ms)$ is only a convex cone in $\vf(M)$, thus the conceptual scheme above is not immediately applicable. So, let us recall some terminology and employ it to suggest a natural candidate for the definition of $\ind_i(\ms)$.
\begin{definition}
Given a surface $\ms\subset\amb$, we say that a set $\Theta\subset\Gamma(N\ms)\cap \vf_i^R(\amb,\ms)$ is a \emph{convex subcone} if it is closed under linear combinations with nonnegative coefficients.
The \emph{dimension} of a convex subcone $\Theta$ is the minimal dimension of a linear subspace of $\Gamma(N\ms)$ that contains $\Theta$.
\end{definition}

\begin{definition}
Given an embedded free boundary minimal surface $\ms\subset\amb$, we define $\ind_i(\ms)$ as the maximal dimension of a convex subcone of $\Gamma(N\ms)\cap \vf_i^R(\amb,\ms)$ where the second variation of the area functional is negative definite.
\end{definition}

Unfortunately, this definition is not always meaningful as shown in the following proposition.
\begin{proposition}
Let $\ms = D$ be the equatorial disc in the ambient manifold $\amb = \amb_{r_0}$ defined above, for some $0<r_0<1$.
Then $\ms$ is a free boundary minimal surface in $\amb$ and $\ind_i(\ms)=\infty$.
\end{proposition}
\begin{proof}
First of all, observe that $\Gamma(N\ms)\cap \vf_i^R(\amb,\ms)$ can be identified with the set of functions $\{f\in C^\infty(\ms) \st \text{$f\ge 0$ on $\bar B_{r_0}(0)$} \}$ and that the second variation of the area along any function $f$ in this set can be written as
\begin{equation} \label{eq:SecVarDisk}
Q_\ms(f,f) = \int_\ms\abs{\grad f}^2 \de \Haus^2 - \int_{\partial\ms} f^2\de\Haus^1 \point
\end{equation}

Now, given any $N\in\N$, let us consider functions $\rho_1,\ldots,\rho_N \in C^\infty(\ms)$, $0\le \rho_k\le1$ for $k=1,\ldots,N$, such that
\begin{enumerate} [label=(\roman*)]
\item $\supp(\rho_k)\subset B_{r_0}(0)\subset \ms$ for $k=1,\ldots,N$; \label{bf:supp}
\item $\supp(\rho_k)\cap \supp(\rho_h) = \emptyset$ for any $1\le k<h \le N$; \label{bf:disj}
\item $\int_\ms\abs{\grad\rho_k}^2\de\Haus^2<2\pi/N$ for $k=1,\ldots,N$; \label{bf:norm}
\item there exist $x_1,\ldots,x_N\in\ms$ such that $\rho_k(x_k) = 1$ for $k=1,\ldots,N$. \label{bf:value}
\end{enumerate}

Then define $\psi_k\eqdef 1-\sum_{h\not=k} \rho_h$ for $k=1,\ldots,N$ and denote by $\Theta\subset \Gamma(N\ms)\cap \vf_i^R(\amb,\ms)$ the convex subcone of dimension $N$ generated by $\{\psi_k\}_{k=1,\ldots,N}$ (observe that $\psi_k\ge 0$ in $\ms$). We want to prove that $Q_\ms$ is negative definite on $\Theta$, which would conclude the proof since $\ind_i(\ms)\ge \dim\Theta = N$.

Pick any $\psi\eqdef \sum_{k=1}^Na_k\psi_k \in \Theta$: using \ref{bf:disj} and \ref{bf:norm}, together with the fact that $\supp(\rho_k)\cap\partial\ms=\emptyset$ for $k=1,\ldots,N$ by \ref{bf:supp}, we have that
\[
\begin{split}
Q_\ms(\psi,\psi) &= \int_\ms \sum_{k=1}^N\abs{\grad\rho_k}^2\Bigg(\sum_{h\not=k}a_h\Bigg)^2 \de\Haus^2 - \int_{\partial\ms} \Bigg(\sum_{k=1}^N a_k \Bigg)^2\de\Haus^1\\
&\le \left(\sum_{k=1}^Na_k \right)^2\left( \int_\ms \sum_{k=1}^N\abs{\grad\rho_k}^2 \de\Haus^2 -2\pi\right) < 0 \comma
\end{split}
\]
which concludes the proof.
\end{proof}

\section{Basic comparison results} We now present the simplest inequality involving $\ind_e(\ms)$, $\ind_\partial(\ms)$, $\ind_c(\ms)$ and compute these three quantities in an explicit example.

\begin{proposition} \label{prop:IneqIndices}
Given an embedded free boundary minimal surface $\ms\subset\amb$, we have that $\ind_e(\ms)$, $\ind_\partial(\ms)$ and $\ind_c(\ms)$ are well-defined (finite) numbers and it holds
\[
\ind_e(\ms) \ge \ind_\partial(\ms) = \ind_c(\ms)\point
\]
\end{proposition}
\begin{proof}
First observe that, by \cref{prop:FirstSecMakeSense}, $\ind_e(\ms)$ coincides with the Morse index of $\ms$ seen as a properly embedded free boundary minimal surface in any extension $\check\amb$ of $\amb$ as in the statement of the proposition. Therefore $\ind_e(\ms)$ is a well-defined number.
Moreover, since $\vf_c(\amb,\ms) \subset \vf_\partial(\amb) \subset \vf_e(\amb,\ms)$, we have that 
$\ind_c(\ms)$ and $\ind_\partial(\ms)$ are well-defined as well, and $\ind_c(\ms)\le \ind_\partial(\ms)\le \ind_e(\ms)$.

We only need to prove that $\ind_c(\ms)$ is actually equal to $\ind_\partial(\ms)$. 
We first show that $\vf_c^R(\amb,\ms)$ is dense in $\vf_\partial^R(\amb)$ with respect to the $H^1$ norm, which is a consequence of the following lemma thanks to a standard partition of unity argument on $\Omega\eqdef \ms\setminus (\operatorname{int}(\ms)\cap \partial\amb)$ (in particular around $\partial\Omega\setminus\partial\ms$).

\begin{lemma}
Let $\Omega\subset \R^n$ be a bounded open domain and consider $u\in \operatorname{Lip}(\R^n)$ such that $u=0$ in $\Omega^{c}$. Then there exists a sequence of functions $u_k\in C^\infty_c(\Omega)$ that converge to $u$ in the $H^1(\R^n)$ norm, i.e.,
\[
\int_{\R^n}\abs{u-u_k}^2 + \abs{\grad u-\grad {u_k}}^2 \de x  \to 0 \point 
\]
\end{lemma}
\begin{proof}
Let $u_+,u_-\in \operatorname{Lip}(\R^n)$ be the positive and the negative part of $u$, respectively. Moreover, let us define $\rho \eqdef d(\cdot,\Omega^c)$, for which we have that $\rho\in \operatorname{Lip}(\R^n)$, $\rho = 0$ in $\Omega^c$ and $\rho>0$ in $\Omega$. The functions $u_1\eqdef \rho + u_+, u_2\eqdef \rho + u_- \in \operatorname{Lip}(\R^n)$ are zero in $\Omega^c$ and are strictly positive in $\Omega$. Note that it is sufficient to prove the result for these two functions.

Therefore, without loss of generality, let us assume that $u\ge d(\cdot,\Omega^c)>0$ in $\Omega$. Moreover, to simplify the notation, let us assume that the Lipschitz constant of $u$ is $1$. Then, let us consider $u_\eps\eqdef (u-\eps)_+ \in \operatorname{Lip}(\R^n)$. Observe that $u_\eps(x) = 0$ for every $x\in\Omega$ such that $d(x,\Omega^c) \le \eps$.
Furthermore, it holds that (see e.g. \cite[Theorem 3.38]{AmbrosioCarlottoMassaccesi2018})
\[
\int_{\overline\Omega} \abs{u-u_\eps}^2 \de x +  \int_{\overline\Omega} \abs{\grad u-\grad u_\eps}^2 \de x \le \eps^2\abs{\Omega} + \int_{\{u< \eps\}} \abs{\grad u}^2 \de x\comma
\]
which converges to $0$ as $\eps\to 0$ since $u\in H^1(\R^n)$ and $\bigcap_{\eps>0}\{u<\eps\} = \{u=0\}$.
Thus, let us choose a sequence $\eps_k\to 0$ and define $u_k\eqdef u_{\eps_k} \ast \varphi_{\eps_k/2}$, where $\varphi_{\eps}(x) \eqdef \eps^{-1}\varphi(x/\eps)$ and $\varphi\in C^\infty_c(B_1(0))$ is a cutoff function with $\varphi(0)=1$, $\varphi\le 1$.
The conclusion is then straightforward.
\end{proof}

Now, let $\Delta$ be an $(\ind_\partial(\ms))$-dimensional vector subspace of $\Gamma(N\ms)\cap \vf_\partial^R(\amb)$ where $Q_\ms$ is negative definite.
By density of $\vf_c^R(\amb,\ms)$ in $\vf_\partial^R(\amb)$, we can find an $(\ind_\partial(\ms))$-dimensional vector subspace $\tilde\Delta\subset \Gamma(N\ms)\cap \vf_c^R(\amb,\ms)$ such that $\tilde\Delta\cap \{ V\in \Gamma(N\ms)\st \norm V_{H^1} = 1\}$ is as close as we want to $\Delta\cap \{ V\in \Gamma(N\ms)\st \norm V_{H^1} = 1\}$ in the $H^1(\bar\ms)$ norm. In particular, we can choose $\tilde\Delta$ such that $Q_\ms$ is negative definite in $\tilde\Delta\cap \{ V\in \Gamma(N\ms)\st \norm V_{H^1} = 1\}$, and so in $\tilde\Delta$.
Therefore, we have that $\ind_\partial(\ms) = \dim\Delta=\dim\tilde\Delta \le \ind_c(\ms)$, which concludes the proof.
\end{proof}

We conclude this section by showing that the inequality above is actually strict in cases of interest.

\begin{proposition}
Let $\ms = D$ be the equatorial disc in the ambient manifold $M=M_{r_0}$ with $0\le r_0 <1$, defined above. Then $\ms$ is a free boundary minimal surface in $M_{r_0}$ with $\ind_e(\ms) = 1$ for all $0\le r_0 < 1$ and
\[
\ind_\partial(\ms) = \ind_c(\ms) = \begin{cases}
                                   1 & \text{if $r_0<e^{-1}$}\comma\\
                                   0 & \text{if $r_0\ge e^{-1}$} \point
                                   \end{cases}
\]
\end{proposition}
Informally, this means that the free boundary minimal surfaces in \cref{fig:NonProp} are respectively unstable (the one on the left) and stable (the one on the right) with respect to the variations $\vf_\partial(\amb)$ and $\vf_c(\amb,\ms)$.

\begin{proof}
First observe that $\ind_e(\ms)=1$, since it coincides with the index of the equatorial disc in $B^3$. Therefore, we only need to compute $\ind_\partial(\ms)= \ind_c(\ms)$, which are equal by \cref{prop:IneqIndices}.

Since $\ind_\partial(\ms) \le \ind_e(\ms)=1$, it is sufficient to determine whether $\ind_\partial(\ms)$ is $0$ or $1$, i.e., whether there exists $f\in C^\infty(\ms)$ such that $f =0$ on $\bar B_{r_0}(0)$ and $Q_\ms(f,f) < 0$.
For any such function $f\in C^\infty(\ms)$ it holds that
\begin{align*}
Q_\ms(f,f) &= \int_0^{2\pi} \int_{r_0}^1 (\abs{\partial_r f(r,\theta)}^2 + r^{-2} \abs{\partial_\theta f(r,\theta)}^2)r \de r \de \theta - \int_0^{2\pi} f(1,\theta)^2\de\theta\\
&\ge \int_0^{2\pi} \int_{r_0}^1 \abs{\partial_r f(r,\theta)}^2 r \de r \de \theta - \int_0^{2\pi} f(1,\theta)^2\de\theta\\
&= \int_0^{2\pi}\left( \int_{r_0}^1 \abs{\partial_r f(r,\theta)}^2 r \de r - f(1,\theta)^2\right)\de\theta \point
\end{align*}
However, observe that by the Cauchy-Schwartz inequality
\begin{align*}
\abs{f(1,\theta)} &= \abs{f(1,\theta) - f(r_0,\theta)} = \left \lvert \int_{r_0}^1 \partial_r f(r,\theta) \de r \right\rvert \le \left(\int_{r_0}^1 \frac 1r\de r \right)^{1/2} \left( \int_{r_0}^1 \abs{\partial_r f(r,\theta)}^2 r\de r\right)^{1/2}\comma
\end{align*}
thus
\[
f(1,\theta)^2 \le \ln(r_0^{-1}) \int_{r_0}^1 \abs{\partial_r f(r,\theta)}^2 r\de r \comma
\]
which implies that
\[
Q_\ms(f,f) \ge \left(\frac1{\ln(r_0^{-1})} - 1 \right) \int_0^{2\pi} f(1,\theta)^2\de\theta 
\]
with equality if and only if $f = f_{r_0}$, where $f_{r_0}(r,\theta)\eqdef c (\ln(r) - \ln (r_0))$ in $B_1(0)\setminus B_{r_0}(0)$ for some constant $c\in\R$ and $f_{r_0}(r,\theta) \eqdef 0$ in $\bar B_{r_0}(0)$.
Therefore, the index is $1$ if and only if $\ln(r_0^{-1}) >  1$ (observe that, rigorously, a smoothing of the function $f_{r_0}$ is needed), that means $r_0<e^{-1}$ as we wanted.
\end{proof}


\part[Inequivalent complexity criteria for FBMS]{Inequivalent complexity criteria for free boundary minimal surfaces}

We obtain a series of results in the global theory of free boundary minimal surfaces, which in particular provide a rather complete picture for the way different \emph{complexity criteria}, such as area, topology and Morse index compare, beyond the regime where effective estimates are at disposal. 

\begin{flushright}
{Based on \cite{CarlottoFranz2020}}    
\end{flushright}

\label{part:Complexity}


\chapter{Context and results} \label{intro:complexity}

The striking developments on the study of free boundary minimal surfaces in recent years pose a number of challenges. Among those, it is natural to ask how different pieces of information that one can associate to a free boundary minimal surface relate to each other. In this part, we primarily focus on three sources of data: the \emph{Euler characteristic} (as a topological descriptor, cf. \cref{sec:EulChar}), the \emph{area} (as a measure of geometric size) and the \emph{Morse index} (which plays the role of the most basic analytic invariant one can associate to the surface in question). 
 
The general scope of the present work is to investigate how these different \emph{complexity criteria} can be compared to each other, under mild \emph{positive curvature} assumptions, i.e., in a regime where effective/quantitative estimates are typically not available. More precisely, the context to have in mind is that of a compact three-manifold satisfying either of the following two pairs of curvature conditions:
\begin{enumerate}[label={\normalfont(\roman*)}]
\item the scalar curvature of $\amb$ is positive and $\partial \amb$ is mean convex with no minimal components; \label{ch:PosScal}
\item the scalar curvature of $\amb$ is nonnegative and $\partial\amb$ is strictly mean convex. \label{ch:PosMean}
\end{enumerate}
First of all, we consider questions of this sort: `Given a Riemannian three-manifold $(M,\smetric)$ as above, is it possible to construct a monotone function $f$ such that any free boundary minimal surface whose index is bounded by $C$ has area bounded by $f(C)$?'
We can combine two of the main results we present here with other recent advances in the field to fully determine whether each of the six natural implications that one can associate to the three pieces of data above hold true, thereby obtaining a rather complete description of the scenario in front of us.

In particular, we develop a detailed analysis of the topological degenerations that may occur, in the limit, to sequences of free boundary minimal surfaces solely subject to a uniform Morse index bound to ultimately prove that \emph{a bound on the index implies a bound on the area, the topology and the total curvature} (see \cref{thm:AreaBound} for a precise statement). In turn, this theorem implies novel unconditional compactness (\cref{cor:CptBoundIndPosRic}) and generic finiteness (\cref{cor:GenericFin}) results. We note that, prior to this work, no (effective or ineffective) counterpart of \cref{thm:AreaBound} was known for free boundary minimal surfaces, not even under the stronger curvature assumptions that the Ricci curvature of the ambient manifold be positive, and its boundary strictly convex (see also \cref{rmk:FraserLicomment} and \cref{rmk:Hersch}).

To show that a bound on the topology cannot possibly imply a bound on the area, nor on the index, we build a large class of pathological counterexamples: in \cref{thm:Counterexample} we construct, for any smooth three-manifold $M$ supporting Riemannian metrics of positive scalar curvature and mean convex boundary and any $g\geq 0, b>0$, one such Riemannian metric $\smetric=\smetric(g,b)$ in a way that $(M,\smetric)$ contains a sequence of connected, embedded, free boundary minimal surfaces of genus $g$ and exactly $b$ boundary components, and whose area and Morse index attain arbitrarily large values. In fact, we have some freedom on the geometric boundary conditions we impose, so that a few variants of the construction are actually possible. 

Concerning the last two (possible) implications above, we note how a bound on the area cannot possibly imply a bound on the index, nor on the topology (no matter how strong curvature conditions are imposed). 
Indeed, there are several existence results of family of free boundary minimal surfaces in the three-dimensional unit ball $B^3\subset\R^3$ with uniformly bounded area and unbounded topology. One of this examples is the family constructed in \cref{thm:main-fbms-b1} of this thesis. Other families are described in detail in \cref{intro:existence}, more precisely in \cref{method:Steklov,method:Perturbation,method:EquivMinMax}. It follows from \cref{thm:TopFromArea}, or even from \cite{AmbrozioBuzanoCarlottoSharp2018}*{Corollary 4}, that these families have also unbounded index.
Roughly speaking, a bound on the area only implies weaker forms of convergence, typically of measure-theoretic character (e.g. in the sense of varifolds, flat chains, currents, \ldots), but does not capture finer geometric properties.

The diagram in \cref{fig:Complexity} then implicitly defines a hierarchy of conditions, based on the implications that hold or do not hold true. Of course, it is then natural to ask whether \emph{pairs} of `weak conditions' imply a `strong' one, e.g. prototypically whether a bound on the area and the topology implies a bound on the index. This interesting question has recently been answered, in the affirmative, by Lima \cite[Theorem B]{Lima2017} who adapted to the free boundary setting some remarkable estimates by Ejiri--Micallef \cite{EjiriMicallef2008}: one can bound the Morse index from above by a linear function of the area and the Euler characteristic, with a multiplicative constant only depending on the ambient manifold. Thereby, the picture we obtain is quite exhaustive and final.

We now present the contents of this part of the dissertation in more detail, point out the technical challenges we faced and relate them to the pre-existing results in the literature, some of which played a fundamental role with respect to this project.

\section{Topological degeneration analysis} \label{sec:TopologicalDegeneration}

A good starting point for our discussion is the following result of Fraser--Li, which is the free boundary analogue of the classical compactness theorem \cite[Theorem 1]{ChoiSchoen1985} by Choi--Schoen.
\begin{theorem}[{\cite[Theorem 1.2]{FraserLi2014}}]
\label{thm:FraLiCpt}
Let $(\amb^3,\smetric)$ be a compact Riemannian manifold with nonempty boundary. Suppose that $\amb$ has nonnegative Ricci curvature and strictly convex boundary. Then the space of compact, properly embedded, free boundary minimal surfaces of fixed topological type in $\amb$ is compact in the $C^k$ topology for any $k\ge 2$.
\end{theorem}

Afterwards, a general investigation of the spaces of free boundary minimal hypersurfaces with bounded index and volume has been carried through in \cite{AmbrozioCarlottoSharp2018Compactness} and \cite{AmbrozioBuzanoCarlottoSharp2018}. In absence of any curvature assumptions, it is not possible to obtain a strong compactness result as in \cref{thm:FraLiCpt}, since curvature concentration can occur at certain (isolated) points. However, one can still prove a milder form of subsequential convergence (smooth, graphical convergence with multiplicity $m\geq 1$ away from finitely many points), and develop an accurate blow-up analysis near the points of bad convergence.
This analysis leads to several compactness and finiteness results, among which we want to recall the following theorem.
\begin{theorem}[{\cite[Corollary 4]{AmbrozioBuzanoCarlottoSharp2018}}] \label{thm:TopFromArea}
Let $(\amb^3,\smetric)$ be a compact Riemannian manifold with boundary and consider $I\in \N$, $\Lambda\ge 0$. Then there exist constants $g_0=g_0(\amb,\smetric,I,\Lambda)$ and $b_0 = b_0(\amb,\smetric,I,\Lambda)$ such that every compact, properly embedded, free boundary minimal surface with index bounded by $I$ and area bounded by $\Lambda$ has genus bounded by $g_0$ and number of boundary components bounded by $b_0$. Furthermore, there exists a constant $\tau_0=\tau_0(\amb,\smetric,I,\Lambda)$ so that the total curvature (i.e., the integral of the square length of the second fundamental form) of any such surface is bounded from above by $\tau_0$. 
\end{theorem}

\begin{remark}
The statement in \cite{AmbrozioBuzanoCarlottoSharp2018} is more general as it only requires a bound on some eigenvalue of the Jacobi operator instead of an index bound, and it applies to free boundary minimal hypersurfaces in ambient manifolds of dimension $3\leq n+1\leq 7$.
\end{remark}

Here we shall be concerned with the space of free boundary minimal surfaces (inside a three-dimensional Riemannian manifold) with bounded index but without any a priori bound on the area. The analogous task has been carried out, for the case of \emph{closed} minimal surfaces, in the remarkable article \cite{ChodoshKetoverMaximo2017} by Chodosh--Ketover--Maximo. Some of the methods developed there are essential for our analysis, and we often rely on the results presented in that article for interior points, although serious technical work (and specific tools) are needed to properly handle the possible degenerations occurring near the boundary of the ambient manifold. 

Let us now describe the key steps in this approach. Hence, let us consider a compact Riemannian manifold with nonempty boundary $(\amb^3,\smetric)$. For the sake of simplicity, let us assume that property \hypP{} holds.

As already explained in the \hyperref[chpt:intro]{Introduction} and in \cref{sec:properFBMS}, such a condition prevents the existence of minimal surfaces that touch $\partial\amb$ in their interior (see also \cref{sec:Properness} for the issues if we drop it) and it is implied by simple geometric assumptions, as for example \hypC{}.

Fixed $I\in\N$ and given a sequence of compact, properly embedded, free boundary minimal surfaces $\ms_j^2\subset\amb$ with $\ind(\ms_j)\le I$, we develop our analysis in two steps:
\begin{description}
\item [Macroscopic behavior.]  First we prove that, up to subsequence, the surfaces $\ms_j$ converge locally smoothly away from a finite set of points $\set_\infty$ to a smooth free boundary minimal lamination $\lam\subset\amb$, that is a suitable disjoint union of free boundary minimal surfaces (see \cref{def:fbmlam}). 
Note that we cannot expect anything better than a lamination without imposing uniform bounds on the area.
Moreover, it holds that the curvature of $\ms_j$ is locally uniformly bounded away from $\set_\infty$.
This tells us that the surfaces are well-controlled away from a finite set of points.
Let us remark that the points in $\set_\infty$ can belong to the boundary.

The main ingredient here is an extension of the curvature estimate for stable minimal surfaces (cf. \cite{Schoen1983Estimates,White1987Curvature} and \cite{SchoenSimon1981} for higher dimensions) to free boundary minimal surfaces with bounded index.

\item [Microscopic behavior.] The second step consists in carefully studying the local behavior of the surfaces $\ms_j$ near the `bad' points in $\set_\infty$. In particular, we prove that, for $\eps>0$ sufficiently small, $\ms_j\cap B_{\eps}(\set_\infty)$ contains only a finite number of components where the curvature is not bounded and these components have controlled topology and area.

The way to proceed in the proof is to blow-up the components of $\ms_j\cap B_{\eps}(\set_\infty)$ with unbounded curvature at the `scale of the curvature'. In this way we obtain a (free boundary) minimal surface in $\R^3$ or in a half-space of $\R^3$ (if initially we were near the boundary of $\amb$) with index less or equal than $I$. 
Thanks to \cite{ChodoshMaximo2016} (cf. also \cite{ChodoshMaximo2018}), possibly combined with the theory developed in Section 2 of \cite{AmbrozioBuzanoCarlottoSharp2018}, we are able to conclude the description of $\ms_j$ at small scale, paying attention to prevent data loss in the blow-up.
\end{description}

\begin{figure*}[htpb]
\centering
\begin{tikzpicture}[scale =0.55]
\coordinate (A1) at (0,0); 
\coordinate (B1) at (0.2,0.05); 
\coordinate (C1) at (0.4,0.12);
\coordinate (D1) at (0.6,0.21);
\coordinate (E1) at (0.8,0.35);

\coordinate (F1) at (4,1.97);
\coordinate (G1) at (4.1,2.05);
\coordinate (H1) at (4.3,2.17);
\coordinate (I1) at (4.5,2.24);
\coordinate (J1) at (4.7,2.28);
\coordinate (K1) at (4.85,2.27);

\coordinate (F2) at (5,-4.13);
\coordinate (G2) at (5.1,-4.13);
\coordinate (H2) at (5.3,-4.12);
\coordinate (I2) at (5.5,-4.11);
\coordinate (J2) at (5.7,-4.09);
\coordinate (K2) at (5.8,-4.07);

\coordinate (C2) at (0.7,-3);
\coordinate (D2) at (0.9,-3.05);
\coordinate (E2) at (1.1,-3.1);

\coordinate (X) at (-1.5,-0.2);
\coordinate (Y) at (-1.2, -3.2);
\coordinate (U) at (7.5, 1.8);
\coordinate (V) at (7.9, -2.6);

\draw plot [smooth cycle, tension=0.6] coordinates {(A1) (B1) (C1) (D1) (E1)
			      (1.5,0.9) (2,1.2) (2.5,1.3) (3,1.7) (3.7,1.71)
			      (F1) (G1) (H1) (I1) (J1) (K1)
			      (5.5,2.2) (6,1.8)
			      (U) (8.5,-0.5) (V)
			      (7,-2.9) (6.5, -3.3) (6.2, -3.93) 
			      (K2) (J2) (I2) (H2) (G2) (F2)
			      (4.5,-4) (3.5,-3.2) (2.7, -2.5) (2,-2.7) (1.5, -3.2) 
			      (E2) (D2) (C2)
			      (0.3, -3)
			      (Y) (-1.9,-2) (X) (-0.5,-0.05)};

\draw (F1) to [bend left=20] (F2);
\draw (G1) to [bend left=20] (G2);

\draw plot [smooth, tension = 0.5] coordinates {(H1) (5.2,0.2) (5.45,-1) (5.57, -1.2) (5.65, -1) (5.4,0.2) (I1)};
\draw plot [smooth, tension = 0.5] coordinates {(H2) (5.5,-2.5) (5.5,-1.5) (5.6, -1.3) (5.7, -1.5) (5.7,-2.5) (I2)};
\pgfmathsetmacro{\t}{atan(0.03/0.1)}
\begin{scope}[rotate around={\t:(5.57,-1.2)}]
\draw[gray] (5.57, -1.2) arc (90:270:0.03 and {sqrt(0.03*0.03+0.1*0.1)/2});
\draw (5.57, -1.2) arc (90:-90:0.03 and {sqrt(0.03*0.03+0.1*0.1)/2});
\end{scope}

\draw plot [smooth, tension = 0.5] coordinates {(J1) (5.25,1.2) (5.68,0) (5.9,-1.29) (5.97, -1.45) (6, -1.27) (5.8,0) (5.37,1.2) (K1)};
\draw plot [smooth, tension = 0.5] coordinates {(J2)  (5.9,-2.7) (5.93, -1.7) (5.97, -1.53) (6.03,-1.7) (6.02, -2.7) (K2)};
\draw[gray] (5.97, -1.45) arc (90:270:0.02 and {0.04});
\draw (5.97, -1.45) arc (90:-90:0.02 and {0.04});

\draw[gray] (0.45, -2.9) arc (90:182:0.02 and {0.08});
\draw (0.45, -2.9) arc (90:0:0.02 and {0.08});
\draw plot [smooth, tension = 0.55] coordinates {(A1) (0.35,-1.2) (0.35,-2.65) (0.45, -2.9) (0.55, -2.65) (0.55,-1.2) (B1)};

\pgfmathsetmacro{\t}{atan(0.01/0.1)}
\begin{scope}[rotate around={-\t:(0.87, -2.6)}]
\draw[gray] (0.87,-2.6) arc (90:270:0.03 and {sqrt(0.01*0.01+0.1*0.1)/2});
\draw (0.87,-2.6) arc (90:-90:0.03 and {sqrt(0.01*0.01+0.1*0.1)/2});
\end{scope}
\draw plot [smooth, tension = 0.55] coordinates {(C1) (0.75,-1.2) (0.77,-2.4) (0.87, -2.6) (0.97, -2.4) (0.95,-1.2) (D1)};
\draw plot [smooth, tension = 0.9] coordinates {(C2) (0.75, -2.8) (0.86, -2.7) (0.93, -2.8) (D2)};

\draw (E1) to [bend left = 15] (E2);

\pgfmathsetmacro\R{0.65}
\draw[gray] (C2) circle (\R);
\draw[gray] (5.685, -1.2) circle (\R);

\node(bp) at (-1,1) {\footnotesize `bad' points};
\draw[blue] (bp) to [bend right = 30] (0.2,-2.7);
\draw[blue] (bp) to [bend left=30] (5.3,-0.8);

\node at (1.7,-0.8) {$\ms_j$};

\draw[->] (9,-1) -- (12,-1);
\node at (10.5,-0.6) {\footnotesize $j\to\infty$};

\begin{scope}[xshift=420]
\coordinate (A1) at (0,0); 
\coordinate (B1) at (0.2,0.05); 
\coordinate (C1) at (0.4,0.12);
\coordinate (D1) at (0.6,0.21);
\coordinate (E1) at (0.8,0.35);

\coordinate (F1) at (4,1.97);
\coordinate (G1) at (4.1,2.05);
\coordinate (H1) at (4.3,2.17);
\coordinate (I1) at (4.5,2.24);
\coordinate (J1) at (4.7,2.28);
\coordinate (K1) at (4.85,2.27);

\coordinate (F2) at (5,-4.13);
\coordinate (G2) at (5.1,-4.13);
\coordinate (H2) at (5.3,-4.12);
\coordinate (I2) at (5.5,-4.11);
\coordinate (J2) at (5.7,-4.09);
\coordinate (K2) at (5.8,-4.07);

\coordinate (C2) at (0.7,-3);
\coordinate (D2) at (0.9,-3.05);
\coordinate (E2) at (1.1,-3.1);

\coordinate (X) at (-1.5,-0.2);
\coordinate (Y) at (-1.2, -3.2);
\coordinate (U) at (7.5, 1.8);
\coordinate (V) at (7.9, -2.6);

\draw plot [smooth cycle, tension=0.6] coordinates {(A1) (B1) (C1) (D1) (E1)
			      (1.5,0.9) (2,1.2) (2.5,1.3) (3,1.7) (3.7,1.71)
			      (F1) (G1) (H1) (I1) (J1) (K1)
			      (5.5,2.2) (6,1.8)
			      (U) (8.5,-0.5) (V)
			      (7,-2.9) (6.5, -3.3) (6.2, -3.93) 
			      (K2) (J2) (I2) (H2) (G2) (F2)
			      (4.5,-4) (3.5,-3.2) (2.7, -2.5) (2,-2.7) (1.5, -3.2) 
			      (E2) (D2) (C2)
			      (0.3, -3)
			      (Y) (-1.9,-2) (X) (-0.5,-0.05)};
			      
\draw[blue] (C1) to [bend left = 15] (C2);
\draw[blue] (F1) to [bend left = 20] (F2);
\draw[blue] (G1) to [bend left = 20] (G2);
\draw[blue] (I1) to [bend left = 20] (I2);

\fill[blue] (C2) circle [radius=0.2em] node[below] {$\set_\infty$};
\fill[blue] (5.685, -1.2) circle [radius=0.2em] node[right] {$\set_\infty$};
\node[blue] at (1.1,-0.8) {$\lam$};
\end{scope}

\draw[->] (2.5, -4.5) to [bend right = 30] (3.5,-6.5);
\node at (1.3,-5.7) {\footnotesize surgery};

\draw[->] (15.2, -6.9) to (17.5,-4.8);
\node[rotate=45] at (16,-5.5) {\footnotesize $j\to\infty$};

\begin{scope}[yshift=-200,xshift=160]

\coordinate (A1) at (0,0); 
\coordinate (B1) at (0.2,0.05); 
\coordinate (C1) at (0.4,0.12);
\coordinate (D1) at (0.6,0.21);
\coordinate (E1) at (0.8,0.35);

\coordinate (F1) at (4,1.97);
\coordinate (G1) at (4.1,2.05);
\coordinate (H1) at (4.3,2.17);
\coordinate (I1) at (4.5,2.24);
\coordinate (J1) at (4.7,2.28);
\coordinate (K1) at (4.85,2.27);

\coordinate (F2) at (5,-4.13);
\coordinate (G2) at (5.1,-4.13);
\coordinate (H2) at (5.3,-4.12);
\coordinate (I2) at (5.5,-4.11);
\coordinate (J2) at (5.7,-4.09);
\coordinate (K2) at (5.8,-4.07);

\coordinate (C2) at (0.7,-3);
\coordinate (D2) at (0.9,-3.05);
\coordinate (E2) at (1.1,-3.1);

\coordinate (X) at (-1.5,-0.2);
\coordinate (Y) at (-1.2, -3.2);
\coordinate (U) at (7.5, 1.8);
\coordinate (V) at (7.9, -2.6);

\draw plot [smooth cycle, tension=0.6] coordinates {(A1) (B1) (C1) (D1) (E1)
			      (1.5,0.9) (2,1.2) (2.5,1.3) (3,1.7) (3.7,1.71)
			      (F1) (G1) (H1) (I1) (J1) (K1)
			      (5.5,2.2) (6,1.8)
			      (U) (8.5,-0.5) (V)
			      (7,-2.9) (6.5, -3.3) (6.2, -3.93) 
			      (K2) (J2) (I2) (H2) (G2) (F2)
			      (4.5,-4) (3.5,-3.2) (2.7, -2.5) (2,-2.7) (1.5, -3.2) 
			      (E2) (D2) (C2)
			      (0.3, -3)
			      (Y) (-1.9,-2) (X) (-0.5,-0.05)};
			      
\draw (A1) to [bend left = 15] (0.3,-3);
\draw (B1) to [bend left = 15] (0.5,-2.99);
\draw (C1) to [bend left = 15] (C2);
\draw (D1) to [bend left = 15] (D2);
\draw (E1) to [bend left = 15] (E2);
\draw (F1) to [bend left = 20] (F2);
\draw (G1) to [bend left = 20] (G2);
\draw (H1) to [bend left = 20] (H2);
\draw (I1) to [bend left = 20] (I2);
\draw (J1) to [bend left = 20] (J2);
\draw (K1) to [bend left = 20] (K2);

\pgfmathsetmacro\R{0.65}
\draw[blue] (C2) circle (\R);
\draw[blue] (5.685, -1.2) circle (\R);

\node at (1.7,-0.8) {$\tilde\ms_j$};
\end{scope}
\end{tikzpicture}
\end{figure*}

At this point, due to this precise description of the degeneration, we are able to perform a `simplification surgery' on $\ms_j$.
Namely, fixing $\eps>0$ sufficiently small, we modify the surfaces $\ms_j$ inside $B_\eps(\set_\infty)$ to obtain new surfaces $\tilde\ms_j$ with the following properties:
\begin{itemize}
\item $\tilde\ms_j$ coincides with $\ms_j$ outside $B_\eps(\set_\infty)$ (the surgery is performed only near the `bad points');
\item the surfaces $\tilde\ms_j$ have uniformly bounded curvature;
\item the topology and the area of $\tilde\ms_j$ are comparable to those of $\ms_j$;
\item the surfaces $\tilde\ms_j$ converge locally smoothly to the lamination $\lam$ introduced above.
\end{itemize}

We refer the reader to \cref{cor:ExistenceBlowUpSetAndCurvatureEstimate}, \cref{thm:GlobalDeg} and \cref{cor:Surgery} for precise statements concerning the description of the topological degeneration and the surgery procedure, respectively. We note here that the \emph{refined} Morse-theoretic arguments we present in \cref{sec:MorseTheory} (where we need to separately count the number of curves $\Sigma_j$ traces, locally, along the two parts of the boundary of small geodesic balls near $\partial M$) are essential to make the whole machinery work. 

\section{Results in `weakly positive geometry'}

That being said, by means of this analysis we obtain the following result, which shows that it is possible to remove the assumption on the area bound from \cref{thm:TopFromArea} in the case of ambient manifolds satisfying the mild curvature assumptions mentioned above. 

\begin{theorem} \label{thm:AreaBound}
Let $(\amb^3,\smetric)$ be a compact Riemannian manifold with boundary. 
Moreover assume that
\begin{enumerate}[label={\normalfont(\roman*)}]
\item \emph{either} the scalar curvature of $\amb$ is positive and $\partial \amb$ is mean convex with no minimal components;
\item \emph{or} the scalar curvature of $\amb$ is nonnegative and $\partial\amb$ is strictly mean convex.
\end{enumerate}
Given $I\in\N$, there exist constants $\Lambda_0 = \Lambda_0(\amb, \smetric, I)$, $\tau_0=\tau_0(\amb, \smetric, I)$,  $g_0 = g_0(\amb, \smetric, I)$ and $b_0 = b_0(\amb, \smetric, I)$ such that for every compact, connected, embedded, free boundary minimal surface $\ms^2\subset\amb$ with nonempty boundary and with index at most $I$ we have that its area is bounded by $\Lambda_0$, its total curvature is bounded by $\tau_0$, its genus by $g_0$ and the number of its boundary components by $b_0$.
\end{theorem}

\begin{remark}
Note that the requirement that either of the curvature assumptions hold \emph{strictly} is actually necessary, for the manifold $S^1\times S^1\times I$, endowed with a flat metric, contains stable minimal annuli of arbitrarily large area.
\end{remark}

\begin{remark}\label{rmk:FraserLicomment}
An area bound for free boundary minimal surfaces is required also in the proof of \cref{thm:FraLiCpt} (see \cite[Proposition 3.4]{FraserLi2014}).
However, the assumption that is made there is a bound on the topology and, also, the method employed in that case is essentially analytic, relying on the aforementioned connection between free boundary minimal surfaces and the first Steklov eigenvalue.
\end{remark}

\begin{remark}\label{rmk:Hersch}
So far \cref{thm:AreaBound} was known only for surfaces with index $I=0$ or $I=1$ (see \cite[Appendix A]{AmbrozioBuzanoCarlottoSharp2018}). Indeed, for stable free boundary minimal surfaces the area bound follows from the stability inequality and a similar argument can be applied to surfaces with index $1$ based on the well-known \emph{Hersch trick}. 
\end{remark}

In order to prove the previous theorem, it turns out that one needs to gain some control on the size of \emph{stable subdomains} of free boundary minimal surfaces. Prior to this work, this result was known only in the closed case (see \cite{Carlotto2015}, based on ideas going back to the work by Schoen--Yau \cite{SchoenYau1983}).

\begin{proposition} \label{prop:StableImpliesCptness}
Let $(\amb^3,\smetric)$ be a three-dimensional Riemannian manifold with boundary. Denote by $\varrho_0\eqdef \inf_M \Scal_g$ the infimum of the scalar curvature of $\amb$ and by $\sigma_0\eqdef \inf_{\partial\amb}H^{\partial\amb}$ the infimum of the mean curvature of $\partial\amb$. Assume that 
\begin{enumerate}[label={\normalfont(\roman*)}]
\item \label{sic:posscal} \emph{either} the scalar curvature of $\amb$ is uniformly positive ($\varrho_0>0$) and $\partial \amb$ is mean convex ($\sigma_0\ge 0$) with no minimal components;
\item \label{sic:strictmc} \emph{or} the scalar curvature of $\amb$ is nonnegative ($\varrho_0\ge 0$) and $\partial\amb$ is uniformly strictly mean convex ($\sigma_0 > 0$).
\end{enumerate}
Then every complete, connected, embedded, stable free boundary minimal surface $\ms^2\subset\amb$ that is two-sided and has nonempty boundary is compact, and its intrinsic diameter satisfies the bound
\[
\operatorname{diam}(\ms)\eqdef \sup_{x,y\in\ms} d_\ms(x,y) \le \min\left\{ \frac{2\sqrt 2 \pi}{\sqrt{3\varrho_0}}, \frac{\pi + 8/3}{\sigma_0}\right\}\point
\]
Moreover, one has that
\[
0< \frac{\varrho_0}2 \area(\ms) + \sigma_0 \length(\partial\ms) \le 2\pi\chi(\ms) \semicolon
\]
in particular, $\ms$ is diffeomorphic to a disc.
\end{proposition}

It is interesting to note that the variational argument that allows to prove the corresponding estimate in the closed case is not sufficient in the free boundary context and, indeed, this result turns out to be much more delicate (see \cref{sec:CptnessStable}). Yet, the key idea remains similar: stable (free boundary) minimal hypersurfaces inherit the `positivity' curvature properties of the ambient manifold (cf. \cite{SchoenYau1979PMT,SchoenYau2017} for a striking application of this principle to the proof of the \emph{positive mass theorem}).

We shall now present two significant consequences of \cref{thm:AreaBound}. The first one descends by combining the area bound with the geometric compactness result of \cref{thm:FraLiCpt}.
\begin{corollary} \label{cor:CptBoundIndPosRic}
Let $(\amb^3,\smetric)$ be a compact Riemannian manifold with boundary. Suppose that $\amb$ has nonnegative Ricci curvature and that the boundary $\partial\amb$ is strictly convex.
Then any set of compact embedded free boundary minimal surfaces with uniformly bounded index is compact in the $C^k$ topology for every $k\ge 2$.
\end{corollary}

In addition, one can employ the Baire-type result given in \cite[Theorem 9]{AmbrozioCarlottoSharp2018Compactness} to obtain the following \emph{generic finiteness} result.

\begin{corollary}\label{cor:GenericFin}
Let $\amb^3$ be a compact manifold with boundary. For a generic choice of $\smetric$ in the class of Riemannian metrics such that $M$ has positive scalar curvature and $\partial M$ has strictly mean convex boundary, the space of compact, embedded, free boundary minimal surfaces with index bounded by $I$ is finite (for any $I\in\mathbb{N}$). Hence, the set of all such surfaces (regardless of their Morse index) is countable. Analogous conclusions hold true for $\smetric$ chosen in a dense subclass of metrics satisfying the curvature conditions \ref{ch:PosScal} or \ref{ch:PosMean} given above.
\end{corollary}

\begin{remark}\label{rmk: GenericYau}
Guang--Li--Wang--Zhou proved in \cite{GuangLiWangZhou2021}*{Proposition 5.3} that, for a generic choice of metric on a compact Riemannian manifold with boundary, the space of (compact embedded) free boundary minimal surfaces with index bounded by $I$ is \emph{countable}.
Observe that their result proves countability instead of finiteness, but it holds in any compact Riemannian manifold with boundary.
\end{remark}

\section[Pathological families of FBMS]{Pathological families of free boundary minimal surfaces}

Explicit examples, based on equivariant constructions due to Hsiang \cite{Hsiang1983} and Calabi \cite{Calabi1967} show that the conclusion of the geometric compactness theorem by Choi--Schoen cannot possibly hold in higher dimension or codimension, respectively, no matter how restrictive curvature assumptions one considers on the ambient manifold. A related question, explicitly posed by B. White in 1984, see \cite{Brothers1986}, is whether strong compactness still holds under weaker curvature conditions but in ambient dimension three, specifically in the class of compact three-manifolds of positive scalar curvature. For the closed case, this question was fully answered by Colding and De Lellis in \cite{ColdingDeLellis2005} (after earlier, significant contributions by Hass--Norbury--Rubinstein \cite{HassNorburyRubinstein2003}): given any compact three-manifold $M$ supporting metrics of positive scalar curvature, and a nonnegative integer $g$, one can construct a metric of positive scalar curvature $\smetric=\smetric(g)$ so that the ambient manifold $(M,\smetric)$ contains a sequence of pairwise distinct, closed minimal surfaces of genus $g$ which is not compact in the sense above. In fact, one can analyze the limit behavior of such a sequence in detail: one witnesses convergence to a singular minimal lamination (inside the given ambient manifold) with precisely two singular points, located at two antipodes on a stable minimal sphere. Furthermore, the value of the area and of the Morse index of these surfaces diverge. 
Here we discuss, and solve, the corresponding problem in the setting of \emph{free boundary} minimal surfaces inside smooth compact three-manifolds with boundary. 

\begin{theorem}\label{thm:Counterexample}
Let $M^3$ be a compact, orientable, three-manifold supporting Riemannian metrics of positive scalar curvature and mean convex boundary and let $g\geq 0$ and $b>0$ be integers. Then there exists a Riemannian metric $\smetric=\smetric(g,b)$ of positive scalar curvature and totally geodesic boundary such that $(M,\smetric)$ contains a sequence of connected, embedded, free boundary minimal surfaces of genus $g$ and exactly $b$ boundary components, and whose area and Morse index attain arbitrarily large values. Analogous conclusions hold true requiring the ambient manifold to have positive scalar curvature and strictly mean convex boundary.
\end{theorem}

This shows that such curvature conditions, i.e., $R_\smetric> 0$ in $M$ and $H^{\partial M}_\smetric\geq 0$ on its boundary $\partial M$ are in general too weak to ensure any form of geometric compactness. In fact, in each of the above examples we witness not only the lack of smooth single-sheeted subsequential convergence, but even of a milder form of subsequential convergence (namely: smooth, possibly with integer multiplicity $m\geq 1$, away from finitely many points) to a smooth free boundary minimal surface. The limit object is, as above, more pathological than one would hope for. On the other hand, the reader may want to compare these `negative' results with some of the geometric applications in \cite{AmbrozioBuzanoCarlottoSharp2018}, based on a bubbling analysis: it is shown that (in three-manifolds satisfying the aforementioned curvature conditions) sequences of free boundary minimal surfaces of fixed topological type \emph{and correspondingly low index} shall indeed subconverge in the strongest geometric sense.

\begin{remark}\label{rmk:ConvexBC}
In certain cases, for instance when $M$ is a ball and $b=1$, it is possible to deform the metric $\smetric$ near the boundary $\partial M$ so that $(M,\smetric)$ has positive scalar curvature, \emph{strictly convex} boundary and contains a sequence of free boundary minimal surfaces whose limit behavior is as above.
\end{remark}

\begin{remark}
A simpler variant of the very same construction we shall present in the proof of \cref{thm:Counterexample} allows to prove that, when dropping any curvature requirement, these noncompactness phenomena can be made to occur inside \emph{any} pre-assigned topological three-manifold.
\end{remark}

\begin{remark}
A full topological characterization of those compact three-manifolds that support metrics of positive scalar curvature and mean convex boundary has been obtained by the first author and Li in \cite[Theorem A]{CarlottoLi2019}. Roughly speaking, one can assert that those curvature conditions are only `mildly' restrictive from the topological perspective, as in particular the boundary $\partial M$ can be the disjoint union of closed surfaces whose genera correspond to any pre-assigned string of nonnegative integers.  
\end{remark}

The examples we construct partly rely on earlier work by Colding--De Lellis \cite{ColdingDeLellis2005}. Their construction is, in some sense, \emph{modular}: they build some simple blocks and develop tools to glue such blocks together by means of a suitable \emph{wire matching argument}. The novel ingredients one needs to prove our results is a new family of building blocks: for any $b>0$ we construct a (conveniently simple) Riemannian metric of positive scalar curvature on the three-ball so that the resulting three-manifold contains a sequence of free boundary surfaces of genus 0 and exactly $b$ boundary components. The construction we present is rather different depending on whether $b\geq 2$ or instead $b=1$, the latter relying on a smoothing lemma by P. Miao employed to the scope of desingularizing a preliminary edgy model. A detailed proof of \cref{thm:Counterexample} is provided in \cref{sec:Counterexample}.

\section{Structure of \texorpdfstring{\cref{part:Complexity}}{Part I}}

We provide an outline of the contents of this part by means of the following diagram.

\begin{center}
\usetikzlibrary{trees}
\tikzstyle{every node}=[anchor=west]
\tikzstyle{tit}=[shape=rectangle, rounded corners,
    draw, minimum width = 3cm]
\tikzstyle{sec}=[shape=rectangle, rounded corners]
\tikzstyle{optional}=[dashed,fill=gray!50]
\begin{tikzpicture}[scale=0.9]

\tikzstyle{mybox} = [draw, rectangle, rounded corners, inner sep=10pt, inner ysep=20pt]
\tikzstyle{fancytitle} =[fill=white, text=black, draw]

\begin{scope}[align=center, font=\small]
 \node [tit] (pre) {Preliminaries};
 \node [sec] [below of = pre, yshift=-0.5cm, xshift=-3cm] (fbml) {Free boundary minimal\\ laminations (\cref{sec:FBMLam})};
 \node [sec] [below of = pre, yshift=-0.5cm, xshift=3cm] (mor) {Morse-theoretic arguments\\ (\cref{sec:MorseTheory})};
\end{scope}

\draw[gray,dashed,rounded corners] ($(pre)+(-6,0.7)$)  rectangle ($(pre)+(6,-2.6)$);

\begin{scope}[align=center, yshift=-4cm]
\node [tit] (res) {Main results};
\node [sec] [below of = res, yshift=-0.5cm] (are) {Area bound\\ (\cref{sec:AreaBound})};
\node [sec] [below of = are, yshift=-1.2cm] (cou) {Counterexamples\\ (\cref{sec:Counterexample})};
\begin{scope}[font=\small]
\node [sec] [right of = are, xshift=4.1cm] (sur) {Simplification surgery\\ (\cref{sec:Surgery})};
\node [sec] [left of = are, xshift=-4.1cm] (cpt) {Diameter bounds \\ for stable FBMS \\ (\cref{sec:CptnessStable})};
\node [sec] [below of = sur, xshift=-1.3cm, yshift=-1.2cm] (mic) {Microscopic\\ behavior \\ (\cref{sec:LocalDegeneration})} ;
\node [sec] [below of = sur, xshift=1.3cm, yshift=-1.2cm] (mac) {Macroscopic\\ behavior \\ (\cref{sec:MacroDescr})} ;

\draw[blue,rounded corners] ($(cou)+(2,-1.2)$)  rectangle ($(cpt)+(-2,1)$);
\node[blue] at ($(cou)+(-7.6,-0.7)$) {\footnotesize Core of the part};

\end{scope}

\tikzset{bigarr/.style={
decoration={markings,mark=at position 1 with {\arrow[scale=1.5]{>}}},
postaction={decorate}}}
 
\begin{scope}[gray]
\draw[bigarr] ([xshift=0.4cm]mic.north) to ([xshift=0.4cm,yshift=0.8cm]mic.north);
\draw[bigarr] ([xshift=-0.4cm]mac.north) to ([xshift=-0.4cm,yshift=0.8cm]mac.north);
\draw[bigarr] (sur) to (are);
\draw[bigarr] (cpt) to (are);
\end{scope}
\end{scope}
\end{tikzpicture}
\end{center}



\chapter{Free boundary minimal laminations}\label{sec:FBMLam}

In this chapter we want to generalize the definitions and the results about minimal laminations (see for example \cite[Definition 2.1]{Carlotto2015} or \cite[Definition 2.2]{MeeksPerezRos2016}) to the case of manifolds with boundary (see also \cite[Section 5]{GuangLiZhou2020}). 
Indeed, laminations naturally arise as limits of (free boundary) minimal surfaces that are assumed to have uniformly bounded index, but not necessarily uniformly bounded area.

\section{Definitions and compactness results}

\begin{definition} \label{def:fbmlam}
A \emph{free boundary minimal lamination} $\lam$ in a three-dimensional Riemannian manifold $(\amb^3,\smetric)$ with boundary $\partial\amb$ is the union of a collection of pairwise disjoint, connected, embedded free boundary minimal surfaces of $\amb$. Moreover we require that $\cup_{L\in \lam} L$ is a closed subset of $\amb$ and that, for each $x\in\amb$, one of the following assertions holds:
\begin{enumerate}  [label={\normalfont(\roman*)}]
\item 
\label{fbml:interior}
$x\in \amb\setminus \partial\amb$ and there exists an open neighborhood $U$ of $x$ and a local coordinate chart $\varphi\colon  B_1^2(0)\times \oo01\subset \R^3 \to U$ such that $\varphi^{-1}((\cup_{L\in \lam} L) \cap U) = B_1^2(0) \times C$ for a closed subset $C\subset \oo 01$; 
\item \label{fbml:goodbdry}
$x\in \partial\amb$ and there exists a relatively open neighborhood $U$ of $x$ and a local coordinate chart $\varphi\colon (B_1^2(0) \cap \{x^1\ge 0\})\times \oo01 \subset \Xi(0)\to U$ such that $\varphi^{-1}((\cup_{L\in \lam} L) \cap U) = (B_1^2(0)\cap\{x^1\ge 0\})\times C$ for a closed subset $C\subset \oo 01$; 
\item \label{fbml:badbdry}
$x\in\partial\amb$ and there exists an extension $\check\amb$ without boundary and an open neighborhood $U$ of $x$ in $\check\amb$ such that property \ref{fbml:interior} is satisfied for the neighborhood $U\ni x$.
\end{enumerate}
\end{definition}
A schematic representation of the three different situations in \cref{def:fbmlam} can be found in \cref{fig:DefLam}.

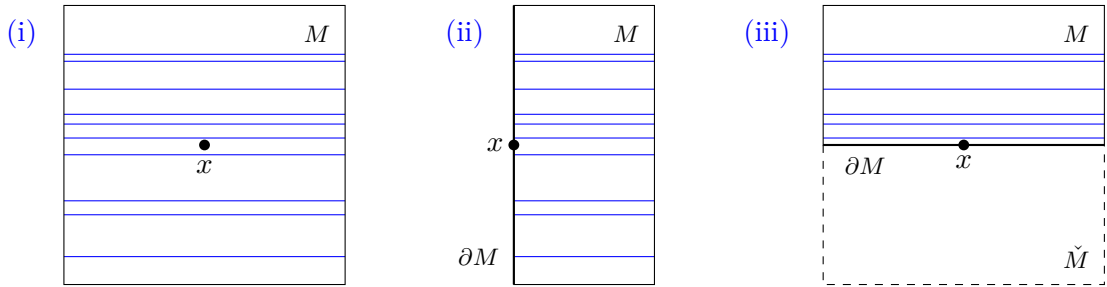
\begin{figure}[htpb]
\centering
\begin{tikzpicture}[scale=1.85]
\pgfmathsetmacro{\xa}{-3.2}
\pgfmathsetmacro{\xb}{-1}
\pgfmathsetmacro{\xc}{2.2}
\pgfmathsetmacro{\r}{1}
\pgfmathsetmacro{\s}{0.8}
\pgfmathsetmacro{\t}{0.3}

\pgfmathsetmacro{\hp}{5}
\def\heightspos{{0.05, 0.15, 0.22, 0.4, 0.6, 0.65}}
\pgfmathsetmacro{\hn}{3}
\def\heightsneg{{-0.07, -0.4, -0.5, -0.8}}

\foreach \i in {0,...,\hp}
  \draw[myBlue] ({\xa-\r},\heightspos[\i]) -- ({\xa+\r},\heightspos[\i]);
\foreach \i in {0,...,\hn}
  \draw[myBlue] ({\xa-\r},\heightsneg[\i]) -- ({\xa+\r},\heightsneg[\i]);
  
\fill (\xa, 0) circle [radius=0.1em] node[below=2pt] {$x$};
\draw  ({\xa+\r},\r) -- ({\xa-\r},\r) -- ({\xa-\r},-\r) -- ({\xa+\r},-\r) -- cycle;
\node at ({\xa+\s}, \s) {\footnotesize $\amb$};
\node at ({\xa-\r-\t}, \s) {\ref{fbml:interior}};

\foreach \i in {0,...,\hp}
  \draw[myBlue] (\xb,\heightspos[\i]) -- ({\xb+\r},\heightspos[\i]);
\foreach \i in {0,...,\hn}
  \draw[myBlue] (\xb,\heightsneg[\i]) -- ({\xb+\r},\heightsneg[\i]);
  
\fill (\xb, 0) circle [radius=0.1em] node[left] {$x$};
\draw  ({\xb+\r},\r) -- ({\xb},\r) -- ({\xb},-\r) -- ({\xb+\r},-\r) -- cycle;
\draw[thick] ({\xb},\r) -- ({\xb},-\r);
\node at ({\xb+\s}, \s) {\footnotesize $\amb$};
\node at ({\xb-0.25}, -\s) {\footnotesize $\partial \amb$};
\node at ({\xb-\t-0.05}, \s) {\ref{fbml:goodbdry}};

\foreach \i in {0,...,\hp}
  \draw[myBlue] ({\xc-\r},\heightspos[\i]) -- ({\xc+\r},\heightspos[\i]);

\fill (\xc, 0) circle [radius=0.1em] node[below] {$x$};
\draw ({\xc+\r},\r) -- ({\xc-\r},\r) -- ({\xc-\r},0) -- ({\xc+\r},0) -- cycle;
\draw[thick] ({\xc-\r},0) -- ({\xc+\r},0);
\node at ({\xc+\s}, \s) {\footnotesize $\amb$};
\node at ({\xc-0.7}, -0.15) {\footnotesize $\partial\amb$};
\draw[dashed]  ({\xc-\r},0) -- ({\xc-\r},-\r) -- ({\xc+\r},-\r) -- ({\xc+\r},0);
\node at ({\xc+\s}, -\s) {\footnotesize $\check{\amb}$};
\node at ({\xc-\r-\t-0.1}, \s) {\ref{fbml:badbdry}};
\end{tikzpicture}
\caption{Definition of lamination in chart.} \label{fig:DefLam}
\end{figure}
\begin{remark}
Note that, if we require property \hypP{} on $\amb$, then case \ref{fbml:badbdry} cannot occur. We have included it in the definition because we need to consider free boundary minimal laminations in half-spaces of $\R^3$ (which do not fulfill \hypP{}) and local limits of free boundary minimal surfaces (or laminations) which a priori can be nonproperly embedded (see \cref{subsec:ImpRem} and \cref{thm:LamCptness} below).
\end{remark}

\begin{definition}
We say that a point $p$ of a minimal lamination $\lam$ is a \emph{limit point} if there exists a coordinate chart $(U,\varphi)$ with $p\in U$ as in the previous definition such that $\varphi^{-1}(p) = (t,x)$ and $t$ is an accumulation point for $C$.
\end{definition}

\begin{remark} \label{rem:LimLeaf}
Thanks to the Harnack inequality, if $p$ is a limit point of a lamination $\lam$, then the entire leaf through $p$ consists of limit points of $\amb$. In this case, we shall call it a \emph{limit leaf}.
\end{remark}

As anticipated, we introduce the concept of lamination to gain compactness. Indeed the following theorem holds.
\begin{theorem}[{\cite[Proposition B.1]{ColdingMinicozzi2004IV} and \cite[Theorem 5.5]{GuangLiZhou2020}}] \label{thm:LamCptness}
Let $(\amb^3, \smetric)$ be a complete three-dimensional Riemannian manifold with boundary. Given $x\in\amb$ and $r>0$, let $\lam_j$ be a sequence of free boundary minimal laminations with uniformly bounded curvatures in $B_{2r}(x) \subset \amb$. Then there exists a subsequence which converges in $B_r(x)$ in the $C^{0,\alpha}$ topology for any $\alpha<1$ to a Lipschitz minimal lamination $\lam$ whose (possibly nonproperly embedded) leaves have free boundary with respect to $\partial\amb$. Moreover the leaves of $\lam$ are smooth free boundary minimal surfaces and the leafwise convergence is $C^\infty$. 
\end{theorem}

\begin{remark} \label{rem:ConvOfLam} In the scenario described in the previous statement, we will say that $\lam_j$ locally converges to $\lam$ \emph{in the sense of laminations}.
In that respect, we recall that the convergence of a sequence of laminations $\lam_j$ to a lamination $\lam$ is leafwise $C^\infty$ if, for every sequence of points $x_j\in \lam_j$ that converges to a point $x\in \amb$, there exists a leaf $L\subset\lam$ such that $x\in L$ and a neighborhood $U$ of $x$ such that the connected component of $\mathcal{L}_j\cap U$ that contains $x_j$ converges to the connected component of $L\cap U$ that contains $x$ smoothly with multiplicity one (in the sense of graphs).
\end{remark}

It is well-known that a two-sided minimal surface having a positive Jacobi field is stable (cf. e.g. \cite[Lemma 1.36]{ColdingMinicozzi2011}). In turn, the existence of a positive Jacobi field can be deduced whenever multi-sheeted convergence occurs. We shall state here a helpful variation on this theme, whose proof is, by now, rather standard (cf. e.g. Theorem 22 in \cite{White2016}, and Proposition 2.1 in \cite{CarlottoChodoshEichmair2016}).
\begin{lemma} \label{lem:StableUnivOrMultOne}
Let $(\amb^3,\smetric)$ be a complete three-dimensional Riemannian manifold with boundary and let $\ms_j^2\subset\amb$ be a sequence of (connected) properly embedded free boundary minimal surfaces that locally converge to a free boundary minimal lamination $\lam\subset \amb\setminus\set_\infty$ away from a finite set of points $\set_\infty$. Consider a leaf $L\in\lam$, then one of the following assertions holds:
\begin{enumerate} [label={\normalfont(\arabic*)}]
\item $L$ has stable universal cover; \label{suomo:stable}
\item the convergence of $\ms_j$ to $L$ is locally smooth with multiplicity one (in the sense of graphs); namely, for every $x\in L$ there exists a neighborhood $U$ of $x$ in $\amb$ such that $U\cap (\bigcup_{L'\in\lam} L') = U\cap L$ and $U\cap \ms_j$ converges smoothly to $U\cap L$ with multiplicity one (as graphs). \label{suomo:multone}
\end{enumerate}
\end{lemma}

\section{Removable singularities}

Thanks to \cref{thm:LamCptness}, in \cref{sec:MacroDescr} we prove that free boundary minimal surfaces with uniformly bounded index converge, possibly after extracting a subsequence, to a free boundary minimal lamination which is smooth away from a finite number of points (so that, in particular, we are in the scenario described in the statement of \cref{lem:StableUnivOrMultOne}). The aim of the following propositions is to provide tools to show that these singularities are actually removable.

\begin{corollary} \label{cor:StableLam}
Let $\varphi\colon \ms\to \Xi(0)\setminus\{0\}\subset\R^3$ be a stable, two-sided minimal immersion that has free boundary with respect to $\Pi(0)$ and is complete away from $\{0\}$. Then the closure of $\varphi(\ms)$ is a plane or a half-plane.
\end{corollary}
\begin{proof} If $\partial\ms=\emptyset$ then the result is a direct consequence of the Bernstein-type theorem by Gulliver--Lawson \cite{GulliverLawson1986}. Otherwise,
consider the double $\check{\ms}$ of $\ms$ (that is to say: the boundaryless surface that is obtained by reflecting $\ms$ across $\partial\ms$) and let $\tau\colon \check{\ms}\to \check{\ms}$ be the associated involution. 
Denoting with $\varrho\colon \R^3\to\R^3$ the reflection with respect to $\Pi(0)$, we define the map $\check\varphi\colon \check{\ms} \to \R^3\setminus\{0\}$ as follows:
\[
\check\varphi(x) \eqdef \begin{cases}
                   \varphi(x) & \text{if $x\in \ms\subset \check{\ms}$}\comma\\
                   \varrho(\varphi(\tau(x))) & \text{if $x\in \check{\ms}\setminus\ms$} \point
                  \end{cases}
\]
Observe that $\check\varphi$ is a two-sided minimal immersion that is complete away from $\{0\}$. Moreover, it follows from the discussion presented in \cite[Section 2]{AmbrozioBuzanoCarlottoSharp2018} that $\check\varphi$ is stable, so we can conclude as above.
\end{proof}

The following result mirrors the one obtained for interior points in Proposition D.3 of \cite{ChodoshKetoverMaximo2017}.

\begin{proposition} \label{prop:RemovableSingStab}
Let $(\amb^3,\smetric)$ be a complete Riemannian manifold with boundary. Fix $p\in\partial\amb$ and $\eps_0>0$ and consider an embedded minimal surface $\hat\ms\subset B_{\eps_0}(p)\setminus\{p\}$ having free boundary with respect to $\partial\amb$, and stable universal cover.
Then $\hat\ms$ smoothly extends across $p$, i.e., there exists a free boundary minimal surface $\ms\subset B_{\eps_0}(p)$ such that $\hat\ms = \ms\setminus\{p\}$.
\end{proposition}
\begin{remark}
Note that we are not requiring that $\hat\ms$ is properly embedded in $B_{\eps_0}(p)\setminus \{p\}$ (in particular it could be nonproperly embedded in the sense of maps).
\end{remark}

\begin{proof}
First observe that we can assume that $p$ belongs to the topological closure of $\hat \ms$, otherwise the result would be obvious.
Taking $\eps_0>0$ possibly smaller and using \cite[Theorem 1.2]{GuangLiZhou2020}, we can assume that in $B_{\eps_0}(p)$ it holds $\abs{\A_{\hat\ms}}(x)d_\smetric(p,x) \le C$ for all $x\in{\hat\ms}$, for some $C>0$. Therefore, for any $r_j\to 0$, the surfaces $r_j^{-1}({\hat\ms} - p)$ locally converge (in the sense of laminations), up to subsequence, to a free boundary minimal lamination ${\lam}_\infty$ of $\Xi(0)\setminus\left\{0\right\}$ thanks to \cref{thm:LamCptness}.

Observe that each leaf of $\lam_\infty$ is complete away from $\{0\}$. We now argue that each leaf also has stable universal cover. Consider a leaf $L\in \lam_\infty$; then, by \cref{lem:StableUnivOrMultOne}, $L$ has stable universal cover or the convergence to $L$ is locally smooth with multiplicity one. However, if the second case occurs, the stability of the universal cover of the surfaces $r_j^{-1}({\hat\ms} - p)$ is inherited by the universal cover of $L$.

Hence, we can apply \cref{cor:StableLam} to obtain that ${\lam}_\infty$ consists of parallel planes or half-planes and thus, possibly further restricting the ball we are considering, we can improve the curvature estimate to (say)
\begin{equation}\label{eq:14estLam}
\abs{\A_{\hat\ms}}(x) d_\smetric(x,p) \le \frac 14
\end{equation}
for all $x\in {\hat\ms} \subset B_{\eps_0}(p)\setminus\{p\}$.

We now want to prove that $\lam_\infty$ is either a single half-plane $\Delta$ passing through the origin and orthogonal to $\Pi(0)$ or $\Pi(0)$ itself.
If not the case, then there would exist another plane or half-plane not passing through the origin which appears in the (lamination) limit of the rescalings $r_j^{-1}(\hat\ms-p)$; hence one could define $\delta\in \oo0{\eps_0}$ sufficiently small such that $\hat\ms\cap (B_\delta(p)\setminus \{p\})$ contains a properly embedded component diffeomorphic to a disc or a half-disc, which does not contain $p$. As a result, we can choose $\delta$ (and possibly taking $\eps_0$ smaller) in such a way that $\hat\ms$ satisfies the assumptions of \cref{cor:14PuncturedBall} between radii $\delta$ and $\eps_0$ in a suitable Fermi chart.
Applying the corollary, we conclude that $\hat\ms$ itself contains a disc or half-disc (obtained, roughly speaking, by gluing the previous disc or half-disc with its corresponding `trivial' component of $\hat\ms\setminus B_\delta(p)$), but this is a contradiction since $\hat\ms$ is connected and its closure contains $p$.
Therefore $\lam_\infty$ consists of only one leaf\footnote{Note that $\lam_\infty$ cannot contain two intersecting (half-)planes, because this would contradict the embeddedness of ${\hat\ms}$.}, which is either $\Delta$ or $\Pi(0)$.

 Let us consider $\delta\in \oo0{\eps_0/3}$ such that $\hat\ms\cap (B_{3\delta}(p)\setminus B_\delta (p))$ intersects $\partial B_{2\delta}(p)$ transversely, is sufficiently (smoothly) close to the limit half-plane or plane and is a multigraph over it.
In particular $\hat\ms\cap \partial B_{2\delta}(p)$ is the union of injectively immersed curves.

If $\hat\ms\cap \partial B_{2\delta}(p)$ contains a compact curve, which can be either an $S^1$ or an arc, then $\hat\ms\cap B_{3\delta}(p)$ is a properly embedded topological punctured disc or half-disc in $B_{3\delta}(p)\setminus \{p\}$ thanks to estimate \eqref{eq:14estLam}, since we can apply \cref{cor:14PuncturedBall} (in Fermi chart) in the variant described in \cref{rem:14PunctBallVariant}.
Thus, using Proposition D.1 in \cite{ChodoshKetoverMaximo2017} or its free boundary analogue (based on \cite[Theorem 4.1]{FraserLi2014}), respectively, we obtain that $\hat\ms$ extends smoothly across $\{p\}$.

The only other situation that can happen is that $\hat\ms\cap \partial B_{2\delta}(p)$ consists of one or more spiraling curves. Observe that this case can only occur if $\lam_\infty$ has $\Pi(0)$ as its only leaf since, otherwise,  $\Delta\cap (B_{3\delta}(p)\setminus B_\delta (p))$ is simply connected and thus it is not possible to see a spiraling behavior.
However, the method to deal with the spiraling situation in the case when $\lam_\infty=\left\{\Pi(0)\right\}$ is completely analogous to the one in the spiraling case in the proof for interior points in \cite[Proposition D.3]{ChodoshKetoverMaximo2017}, therefore we can employ the same argument to conclude.
\end{proof}

We are now able to discuss the general case, when one needs to deal with isolated singularities arising when taking limits of free boundary minimal surfaces with bounded index.

\begin{theorem} \label{thm:RemSingLimLam}
Let $(\amb^3,\smetric)$ be a compact three-dimensional Riemannian manifold with boundary.
Let $\ms_j^2\subset\amb$ be a sequence of properly embedded free boundary minimal surfaces with uniformly bounded index, which locally converge to a free boundary minimal lamination ${\hat\lam}$ in $\amb\setminus\set_\infty$ away from a finite set of points $\set_\infty$.
Then ${\hat\lam}$ extends smoothly through $\set_\infty$ to a smooth lamination $\lam$ of $\amb$.
Moreover, given any leaf $L\in\lam$, one of the following assertions holds:
\begin{enumerate} [label={\normalfont(\arabic*)}]
\item $L$ has stable universal cover; \label{rsll:stable}
\item the convergence of $\ms_j$ to $L$ is locally smooth with multiplicity one (in the sense of graphs); namely, for every $x\in L$ there exists a neighborhood $U$ of $x$ in $\amb$ such that $U\cap (\bigcup_{L'\in\lam} L') = U\cap L$ and $U\cap \ms_j$ converges smoothly to $U\cap L$ with multiplicity one (as graphs). \label{rsll:multone}
\end{enumerate}
\end{theorem}
\begin{proof}
We split the proof in two steps: in the first step we show that $\hat\lam$ extends through $\set_\infty$, while in the second step we prove properties \ref{rsll:stable} and \ref{rsll:multone} of the leaves of $\lam$.

\vspace{2mm}\textbf{Step 1.} Since the result is local, we can assume to work around a single point $p\in \set_\infty$.
First, let us prove that there exists $\eps_0>0$ sufficiently small such that each leaf of ${\hat\lam}$ in $B_{\eps_0}(p)$ has stable universal cover.

Thanks to \cref{lem:StableUnivOrMultOne}, either $\hat L\in \hat \lam$ has stable universal cover or the convergence of $\ms_j$ to $\hat L$ is locally smooth and graphical with multiplicity one. Observe that there can be only a finite number of unstable leaves of the second type, otherwise the uniform bound on the index of $\ms_j$ would be violated.
If we focus on \emph{one} such leaf, we can argue as follows.
Taking $\eps_0$ possibly smaller, we can assume that $B_{\eps_0}(p)$ is simply connected and thus $\Sigma_j\cap B_{\eps_0}(p)$ is two-sided\footnote{This is a general topological fact, whose proof can be found for example in \cite[Lemma C.1]{ChodoshKetoverMaximo2017} (where the result is stated in a particular case but the proof is completely analogous).}. Hence, it is easily argued that (by virtue of the multiplicity one convergence) ${\hat L}\cap B_{\eps_0}(p)$ has to be two-sided as well. Now, a straightforward variation of the same argument as in \cite[Proposition 1]{FischerColbrie1985} proves that we can pick $\eps_0$ even smaller, in such a way that $\hat L\cap (B_{\eps_0}(p)\setminus \{p\})$ is stable and, in addition, the Jacobi operator has a positive solution in the same domain. Hence, using such a function, we can derive that $\hat L\cap (B_{\eps_0}(p)\setminus \{p\})$ has, in fact, stable universal cover.

We have thus proved that there exists $\eps_0>0$ such that every leaf of $\hat\lam$ in $B_{\eps_0}(p)\setminus \{p\}$ has stable universal cover. Then we can apply \cref{prop:RemovableSingStab} to obtain that each leaf extends smoothly across $p$. Note that the extended leaves cannot meet at $p$ (since $\hat\lam$ is assumed to be a lamination in $M\setminus \mathcal{S}_{\infty}$).

It is only left to prove that $\lam$ obtained as union of the extended leaves has the structure of lamination around the point $p$. This fact follows from the proof of \cref{prop:RemovableSingStab}, where it was shown that that for any sequence of positive numbers $r_j\to 0$ we have that $r_j^{-1}(\lam-p)$ converge to a lamination consisting of parallel planes or half-planes. 

\vspace{2mm}\textbf{Step 2.} If $L$ is a limit leaf of $\lam$ (see \cref{rem:LimLeaf}), then $L$ has stable universal cover by standard arguments (cf. \cref{lem:StableUnivOrMultOne}).
On the other hand, if the convergence of $\ms_j$ to $L$ is locally smooth with multiplicity one (in the sense of graphs) away from $\set_\infty$, then the convergence extends across $\set_\infty$ thanks to \cref{lem:MultOneConvExtends} and we end up in case \ref{rsll:multone}.

Therefore, let us assume that $L$ is not a limit leaf and that $\ms_j$ does not converge to $L$ with multiplicity one.
Possibly passing to the double cover, we can assume that $L$ is two-sided (cf. \cite[Section 6]{AmbrozioCarlottoSharp2018Compactness}). We then show that $L$ admits a positive Jacobi function, which proves that $L$ is stable as well as its universal cover.

Let us consider a regular domain $\Omega\Subset L\setminus \set_\infty$. Consider a vector field $X\in \vf(M)$ that has unit length and is normal to $L$ along $L$. Denote by $\Phi(x,t)$ the flow associated to $X$ and, for every $\eps>0$, define
\[
\Omega_\eps\eqdef \{ \Phi(x,t) \st x\in\Omega\comma \abs t< \eps \} \point
\]
Observe that we can fix $\eps>0$ such that the convergence of $\ms_j$ to $\lam$ is smooth in the sense of laminations in $\Omega_\eps$ and such that the only component of $\lam$ in $\Omega_\eps$ is $L\cap \Omega_\eps$.

By definition of convergence in the sense of laminations, $\Omega_\eps\setminus \Omega_{\eps/2}$ does not intersect $\ms_j$ for $j$ sufficiently large, since it does not intersect any leaf of $\lam$.
Then, for every $j$, define
\begin{align*}
u_j^+(x) &\eqdef \sup  \{t\in \oo{-\eps}\eps \st \Phi(x,t)\in \ms_j\}\comma \\
u_j^-(x) &\eqdef \inf \{t\in \oo{-\eps}\eps \st \Phi(x,t)\in \ms_j\} \point
\end{align*}
Note that $-\eps/2< u_j^-(x) < u_j^+(x) < \eps/2$, since $\ms_j\cap(\Omega_\eps\setminus\Omega_{\eps/2}) = \emptyset$ and $\ms_j$ does not converge to $L$ with multiplicity one.
Furthermore observe that, by \cref{rem:ConvOfLam} together with the compactness of $\Omega$, for every $j$ sufficiently large (depending on $\Omega$) the surfaces $\ms_j^\pm \eqdef \{ \Phi(x, u_j^\pm(x)) \st x\in\Omega \}\subset \ms_j$ are well-defined smooth surfaces (with boundary) that converge uniformly smoothly to $\Omega$.

At this point, we can go through the very same argument as in the proof of Theorem 5 in \cite[Section 6]{AmbrozioCarlottoSharp2018Compactness}, so we just sketch it briefly.

Fixing $x_0\in\Omega\setminus\partial\amb$, define $\tilde h_j\eqdef u_j^+ - u_j^->0$ and $h_j(x)\eqdef \tilde h_j(x_0)^{-1} \tilde h_j(x)$. Then, following exactly the same proof of Claim 1 in \cite[Section 6]{AmbrozioCarlottoSharp2018Compactness}, we have that $h_j$ is bounded in $C^l(\Omega)$ for all $l\in\N$ and converges smoothly to $h\in C^\infty(\Omega)$, solution of
\begin{equation} \label{eq:JacEqSub}
\begin{cases}
\jac_L (h) = \lapl_L h +(\frac 12 \Scal_\smetric+\frac 12 \abs A^2 - K)h = 0 & \text{in $\Omega$}\comma\\
\frac{\partial h}{\partial \eta} = - \II^{\partial M}(\nu,\nu)h & \text{on $\partial\amb\cap\Omega$}\point
\end{cases}
\end{equation}
Then, by taking an exhaustion of $L\setminus\set_\infty$ by domains $\Omega\Subset L\setminus\set_\infty$ containing $x_0$, we obtain a function $h\in C^\infty(L\setminus\set_\infty)$ solving \eqref{eq:JacEqSub} in $L\setminus\set_\infty$. Moreover $h(x_0)=1$ and $h\ge 0$.
Finally, thanks to the same proof of Claim 2 in \cite[Section 6]{AmbrozioCarlottoSharp2018Compactness}, it holds that $h$ is uniformly bounded and thus extends to a smooth Jacobi function on all $L$, which is positive everywhere thanks to the maximum principle and the Hopf boundary point lemma. 
\end{proof}

In the setting of Euclidean half-spaces, \cref{thm:RemSingLimLam} reads as follows.
\begin{proposition}\label{prop:LimLamInR3}
Let $\smetric_j$ be a sequence of Riemannian metrics in $\R^3$ that converges, locally smoothly, to the Euclidean metric.
Let $\ms_j^2\subset \Xi(a_j) \subset \R^3$ for some $0\ge a_j\ge -\infty$ be a sequence of properly embedded, \emph{edged}, free boundary minimal surfaces in $(\Xi(a_j),\smetric_j)$, such that for every $j\in\mathbb{N}$ we have $\Sigma_j\subset\Delta_j$ for a sequence of compact domains $\Delta_j$ exhausting $\Xi(a)$ for $a=\lim_{j\to\infty} a_j$, each being the intersection of a smooth domain of $\R^3$ with $\Xi(a_j)$.
Assume that the surfaces $\ms_j$ have index bounded by $I\in\N$ and locally converge (in the sense of laminations) to a free boundary minimal lamination $\hat\lam\subset \Xi(a)\setminus\set_\infty$ away from a finite set of points $\set_\infty$.
Then $\hat\lam$ extends smoothly through $\set_\infty$ to a free boundary minimal lamination $\lam\subset\Xi(a)$ and the following dichotomy holds:
\begin{enumerate} [label={\normalfont(\arabic*)}]
\item $\lam$ consists of parallel planes or half-planes; \label{llir:stable}
\item $\lam$ is a complete, nonflat, connected, properly embedded, free boundary minimal surface in $\Xi(a)$ of (positive) index at most $I$ and (a subsequence of) $\ms_j$ converges to $\lam$ locally smoothly (in the sense of graphs) with multiplicity one. \label{llir:multone}
\end{enumerate}
\end{proposition}
\begin{proof}
The first part of the statement follows directly from \cref{thm:RemSingLimLam}, so let us prove properties \ref{llir:stable} and \ref{llir:multone}.
Let us consider a leaf $L\in\lam$. By \cref{thm:RemSingLimLam}, $L$ has stable universal cover or the convergence of $\ms_j$ to $L$ is locally smooth with multiplicity one (in the sense of graphs).
In the first case, $L$ must be a plane by Corollary 22 in \cite{AmbrozioBuzanoCarlottoSharp2018}. In the second case, $L$ has index bounded by $I$, otherwise the bound on the index of $\ms_j$ would be violated. In particular, all the leaves of $\lam$ have bounded index, thus the result is just the free boundary analogue of Corollary B.2 in \cite{ChodoshKetoverMaximo2017}, from which our result follows using \cref{lem:ReflectionPrinciple} and \cref{cor:MinSurfInR3WithFiniteIndex}, which is needed (in particular) to ensure that the reflected minimal lamination still has finite index.
\end{proof}

\chapter{Some Morse-theoretic arguments} \label{sec:MorseTheory}

In this chapter we collect a few lemmas that will be useful to obtain topological information at intermediate scales. The basic idea is that if a surface fulfills suitable curvature estimates then it is `locally simple'.

\section[FBMS with bounded curvature]{Free boundary minimal surfaces with bounded curvature}

\begin{lemma} \label{lem:GradEst}
Let $M^3 = \{\abs x\le 2\}\cap \Xi(a)\subset \R^3$ be endowed with the Euclidean metric, for some $0\ge a\ge -\infty$. Fix $p\in B_{1/4}(0)\cap M$ and $r>0$.
Let $\ms^2\subset \amb\setminus B_r(p)$ be a connected embedded surface, having free boundary with respect to $\Pi(a)$, with $\partial\ms = \ms\cap \partial (\amb\setminus B_r(p))$. Assume that for every $x\in\ms$ it holds $\abs{\A_\ms}(x) \abs{x-p} \le \eps$, for some constant $0<\eps\le 1/3$. Moreover suppose that $\ms$ intersects $\partial B_r(p)$.
Then, denoting by $f\colon \ms\to\R$ the function $f(x)\eqdef \abs{x-p}^2$, we have that
\[
\abs{\grad_\ms f(x)} \ge 2(1-\eps) (\abs{x-p} - r) \semicolon  
\]
furthermore, if $\ms\cap\partial B_r(p)$ contains a compact component $\Gamma$, then in fact
\[
\abs{\grad_\ms f(x)} \ge 2(1-\eps) (\abs{x-p} - r) + \min_{y\in \Gamma}\ \abs{\grad_\ms f(y)}\point
\]
\end{lemma}
\begin{proof}
Consider a point $x\in\ms\setminus\partial\amb$ and take a unit-speed geodesic $\alpha\colon\cc 0l\to \ms$ such that $\alpha(0)\in \ms\cap \partial B_r(p)$ and $\alpha(l)=x$.
Note that $\alpha$ exists since $\ms$ is connected and a geodesic starting inside $\ms$ cannot touch $\ms\cap\partial\amb$ tangentially. Indeed $\ms\cap\Pi(a)$ is a union of geodesics in $\ms$ thanks to the free boundary condition and a simple computation shows that $\ms\cap(\partial\amb\setminus\Pi(a))$ is strictly convex in $\ms$ by the curvature estimate.

Then observe that $\grad_\ms f(x)$ is the projection of the vector $2(x-p)$ on $\ms$ for all $x\in\ms$.
In addition, for every $v\in T_x\ms$, it holds
\begin{align*}
(\mathrm{Hess}_\ms f)_x(v,v)  &= 2(\abs v ^2 - \scal{A_\ms(x)(v,v)}{x-p}) \ge 2\abs v ^2 (1 - \abs{A_\ms}(x)\abs{x-p}) \\
&\ge  2\abs v^2 (1-\eps) \point
\end{align*}
Thus we have that
\[
\frac{\d}{\d t} (f\circ \alpha) = \scal{\grad_\ms f(\alpha)}{\alpha'} \quad\text{and} \quad \frac{\d^2}{\d t^2}(f\circ\alpha) = (\mathrm{Hess}_\ms f)_\alpha (\alpha',\alpha') \ge 2(1-\eps) \point
\]
Therefore, since clearly $\scal{\grad_\ms f(\alpha(0))}{ \alpha'(0)} \ge 0$, we obtain that
\[
\begin{split}
\abs{\grad_\ms f(x)} &\ge \scal{\grad_\ms f(x)}{\alpha'(l)} = \frac{\d}{\d t}\Big|_{t=l} (f\circ \alpha) \ge 2(1-\eps)l \ge  2(1-\eps)(\abs{x-p}-r)\point
\end{split}
\]

Now assume that $\ms\cap\partial B_r(p)$ contains a compact component $\Gamma$. Take a point $x\in\ms$ and consider a \emph{length minimizing} unit-speed \emph{curve} $\alpha\colon \cc 0l \to\ms$ with $\alpha(0) \in \Gamma$ and $\alpha(l) =x$. 
Note that $\alpha$ cannot touch $\ms\cap \partial B_r(p)$ at time larger than zero. Indeed, if that happened, there would exist an intermediate time $t_0>0$ such that $f\circ \alpha$ has a (local) maximum at $t_0$ and $\alpha$ is a geodesic between $0$ and $t_0$. However this fact leads to a contradiction with the previous computation since we would have at the same time $(f\circ\alpha)'(t_0) =0$ and $(f\circ\alpha)'(t_0) \ge 2(1-\eps)t_0>0$. Therefore $\alpha$ is a geodesic, since it cannot touch $\partial\ms$ in its interior.

Performing the same computation as above, this time using that $\scal{\grad_\ms f(\alpha(0))}{\alpha'(0)} = \abs{\grad_\ms f(\alpha(0))}$ (by minimality of $\alpha$), we thus obtain
\[
\abs{\grad_\ms f(x)} \ge 2(1-\eps)(\abs{x-p}-r) + \abs{\grad_\ms f(\alpha(0))} \ge 2(1-\eps)(\abs{x-p}-r) + \min_{y\in\Gamma}\ \abs{\grad_\ms f(y)}\comma
\]
which concludes the proof.
\end{proof}

\begin{corollary} \label{cor:14PuncturedBall}
Let $\smetric$ be a metric on $\{\abs x\le 2\}\cap \Xi(0)\subset\R^3$ sufficiently close to the Euclidean one, and denote by $M^3\eqdef B_1(0)\subset \Xi(0)$ the unit ball with respect to this metric\footnote{\label{fn:MetricBalls}That is to say: here $B_1(0)$ is the metric ball with respect to the metric $\smetric$, and the same comment applies to the balls $B_1(0), B_r(p)$ in \cref{cor:14TopInfo}.}.
Assume that being orthogonal to $\Pi(0)$ with respect to the Euclidean metric is equivalent\footnote{\label{fn:FermiChart}In the applications we obtain this orthogonality condition working in a Fermi chart, both here and in \cref{cor:14TopInfo} below.} to being orthogonal to $\Pi(0)$ with respect to $\smetric$.
Given $0<r<1/4$, consider a connected, embedded surface $\ms^2\subset \amb\setminus B_r^{\R^3}(0)$ with $\partial\ms=\ms\cap\partial(\amb\setminus B_r^{\R^3}(0))$ such that
\begin{enumerate} [label={\normalfont(\roman*)}]
\item \label{pb:FB} $\ms$ is free boundary with respect to $\Pi(0)$;
\item \label{pb:cptcomp} $\ms\cap \partial B_r^{\R^3}(0)$ contains a compact component $\Gamma$ and $\ms$ intersects $\partial B_r^{\R^3}(0)$ transversely along $\Gamma$; 
\item \label{pb:curv} for every $x\in\ms$ it holds $\abs{\A_\ms}(x) d_\smetric(x,0) \le 1/4$.
\end{enumerate}
Then $\ms$ is properly embedded in $B_1(0)\setminus B_r^{\R^3}(0)$ and is either a topological disc or a topological annulus.
\end{corollary}
\begin{proof}
Taking the metric $\smetric$ sufficiently close to the Euclidean metric we can assume that $\abs{A_\ms^{\R^3}}(x)\abs{x}\le 1/3$.
Let $f\colon \ms\to \R$ be given by $f(x) \eqdef \abs{x}^2$. Then, thanks to \cref{lem:GradEst} and \ref{pb:cptcomp}, $f$ has no critical points on $\ms$; in particular it holds
\begin{equation}\label{eq:GradAwayFromZero}
\abs{\grad_\ms^{\R^3}f(x)} \ge \min_{y\in\Gamma}\ \abs{\grad_\ms^{\R^3}f(y)} > 0 \quad \text{for all $x\in\ms$}\point
\end{equation}
Moreover we obtain that $\abs{\grad_\ms^{\R^3}f(x)} \ge 1/2$ for all $x\in\ms\cap (\partial B_1(0)\setminus \Pi(0))$;
hence, taking $\smetric$ possibly closer to the Euclidean metric (independently of $\ms$), we can assume that $\grad_\ms^{\R^3}f(x)$ points strictly out of $\ms$ along $\partial B_1(0)\setminus \Pi(0)$.
Observe also that $\grad_\ms^{\R^3}f$ is parallel to $\Pi(0)$ along $\ms\cap\Pi(0)$ by \ref{pb:FB}.

Thus, defining $\Phi(t,x)$ as the flow of the vector field $\grad_\ms^{\R^3}f$ on $\ms$,
we have that the exit-time map $\ms\to \co 0\infty$ given by $x\mapsto t(x)$ (i.e., $t(x)$ is the supremum of the times $t$ for which $\Phi(t,x)$ is well-defined in $\ms$) is continuous. Furthermore $0 < t(y) <\infty$ for every $y\in\Gamma$ by \eqref{eq:GradAwayFromZero}.
Hence the map $\mathcal F\colon \Gamma\times\cc 01 \to \ms$ given by $\mathcal F(y,t) \eqdef \Phi(t(y)t,y)$ is a homeomorphism with its image, which must coincide with all of $\ms$ by connectedness.
This concludes the proof, showing that $\ms$ is a properly embedded topological disc if $\Gamma$ is an arc and a properly embedded topological annulus if $\Gamma$ is a circle.
\end{proof}
\begin{remark} \label{rem:14PunctBallVariant}
A variation of the previous proof works in the case when $\ms^2\subset \amb \setminus\{0\}$ is a connected embedded surface satisfying assumptions \ref{pb:FB}, \ref{pb:curv} and such that $\ms\cap (\partial B_1(0)\setminus\Pi(0))$ contains a compact component $\Gamma$. What we obtain is that $\ms$ is properly embedded in $B_1(0)\setminus\{0\}$ and is either a punctured disc (the puncture being in the interior) or a punctured half-disc (the puncture being on the boundary).
Observe that, in this case, a combination of the first estimate of \cref{lem:GradEst} and a flow from $\Gamma$ along $-\grad_\ms^{\R^3}f$ is needed.
\end{remark}

\section{Curvature bounds imply topological control}

We shall further need the following refinement of the previous statement in order to transfer topological information from a small to a bigger scale. Observe that in the free boundary case a standard Morse-theoretic argument as in \cite[Lemma 3.1]{ChodoshKetoverMaximo2017} is not sufficient. Indeed, given (for example) a surface $\ms\subset B_1^{\R^3}(0)\cap\Xi(0)\subset\R^3$, the number of connected components of $\ms\cap\partial B_1^{\R^3}(0)$ can be arbitrarily large even if $\ms\cap (\partial B_1^{\R^3}(0)\cup \Pi(0))$ consists of a single connected component.

\begin{definition} \label{def:StrongEq}
Let $\smetric$ be a metric on $\amb^3\eqdef \{\abs x\le 2\}\cap \Xi(a)\subset\R^3$ sufficiently close to the Euclidean one, for some $0\ge a\ge-\infty$, and let $B_r(p)\subset \amb$ be the metric ball with respect to $\smetric$ with center $p\in\amb$ and radius $r>0$. Given a surface $\ms^2\subset \amb$, we say that $\ms\cap (\partial B_r(p)\setminus\Pi(a))$ is \emph{$\eps$-strongly equatorial} for some $\eps>0$ if the following properties hold:
\begin{enumerate} [label={\normalfont(\roman*)}]
    \item each connected component of $\ms$ intersects $\partial B_r(p)\setminus\Pi(a)$;
    \item there exists a plane $\Delta\subset \R^3$ passing through $p$ which is either orthogonal or parallel to $\Pi(a)$ such that each component of $\ms\cap (B_{2r}(p)\setminus B_{r/2}(p))$ is a graph over $\Delta\cap (B_{2r}(p)\setminus B_{r/2}(p))$ with $C^0$ and $C^1$ norm less than $2\eps r$ and $2\eps$ respectively.
\end{enumerate}
\end{definition}

\begin{proposition} \label{cor:14TopInfo}
There exists a universal constant $\mu_0>0$ with the following property.

Let $\smetric$ be a metric on $\{\abs x \le 2\}\cap \Xi(a) \subset \R^3$ sufficiently close to the Euclidean metric, for some $0\ge a \ge -\infty$, and denote by $M^3\eqdef B_1(0)\subset \Xi(a)$ the unit ball with respect to this metric.
Suppose that being orthogonal to $\Pi(a)$ with respect to the Euclidean metric is equivalent to being orthogonal to $\Pi(a)$ with respect to $\smetric$.
Let $\ms^2\subset \amb$ be a compact, connected, embedded surface having free boundary with respect to $\Pi(a)$, with $\partial\ms=\ms\cap \partial\amb$.
Assume that there exist $p\in B_{\mu_0}(0)$  and $0< r\le \mu_0$ so that:
\begin{enumerate} [label={\normalfont(\roman*)}]
    \item $\abs{A_\ms}(x) d_\smetric(x,p) \le \mu_0$ for all $x\in\ms\setminus B_r(p)$; \label{tp:curv}
    \item \label{tp:equat1} $\ms\cap (\partial B_r(p)\setminus\Pi(a))$ is $\mu_0$-strongly equatorial.
\end{enumerate}
Then, if the genus of $\ms\cap B_r(p)$ and the number of connected components of $\ms\cap (\partial B_r(p)\setminus\Pi(a))$ are bounded by $\kappa$, then also $\ms$ has genus and number of connected components of $\ms\cap(\partial B_1(0)\setminus\Pi(a))$ bounded by $\kappa$.
Moreover, if we have also that the number of connected components of $\ms\cap (\Pi(a)\cap B_r(p))$ is bounded by $\kappa$, then the number of connected components of $\ms\cap \Pi(a)$ is bounded by $2\kappa$.
\end{proposition}
\begin{proof}
First observe that, taking $\smetric$ sufficiently close to the Euclidean metric, by \ref{tp:curv} and \ref{tp:equat1} we can assume that $\abs{A_\ms^{\R^3}}(x) \abs{x-p} \le 2\mu_0$ for all $x\in\ms\setminus B_r^{\R^3}(p)$ and that $\ms\cap (\partial B_r^{\R^3}(p)\setminus\Pi(a))$ is $\mu_0$-strongly equatorial.
That being said, and since (as will be clear) the argument we are about to present runs exactly the same in $B_1(0)$ and $B_1^{\R^3}(0)$, for notational convenience we will work with respect to the Euclidean metric (in particular $B_r(p)$, $B_1(0)$ will be balls in $\Xi(a)\subset \R^3$, the Euclidean space). 

By \cref{lem:GradEst}, denoting by $f\colon \ms\to\R$ the function $f(x)\eqdef\abs{x-p}^2$, for all $x\in\ms\setminus B_r(p)$ we have
\begin{equation*}
\abs{\grad_\ms f(x)} \ge 2(1-2\mu_0)(\abs{x-p}-r) + \min_{y\in\ms\cap \partial B_r(p)}\  \abs{\grad_\ms f(y)} \ge 2(1-2\mu_0)\abs{x-p} > 0\comma
\end{equation*}
where we have used that $\grad_\ms f(x)$ is the orthogonal projection of the vector $2(x-p)$ on $T_x\ms$ and thus we can choose $\mu_0$ small enough that $\abs{\grad_\ms f(y)} \ge 2(1-2\mu_0)r$ for all $y\in\ms\cap (\partial B_r(p)\setminus\Pi(a))$, by definition of $\mu_0$-strongly equatorial.
Hence, we have that
\begin{equation} \label{eq:AlmostTangent}
\begin{split}
\abs{\grad_\ms f(x)-2(x-p)} &= \sqrt{4\abs{x-p}^2 - \abs{\grad_\ms f(x)}^2} \\
&\le \sqrt{4\abs{x-p}^2 - 4(1-2\mu_0)^2\abs{x-p}^2} \\
&= 2\abs{x-p}\sqrt{4\mu_0-4\mu_0^2} \le 4 \abs{x-p}\sqrt{\mu_0} \le 8 \sqrt{\mu_0} 
\end{split}
\end{equation}
for all $x\in \ms\setminus B_r(p)$.

Now let us denote by $\Phi(t,x)$ the flow of the vector field $\grad_\ms f$ on $\ms$.
Observe that $\grad_\ms f$ points (strictly) towards $\ms\setminus B_r (p)$ along $\partial B_r (p)\setminus\Pi(a)$ and out of $\ms$ along $\partial B_1(0)\setminus\Pi(a)$ (as long as $\mu_0$ is sufficiently small). Moreover, $\grad_\ms f$ points (strictly) out of $\ms$ along $\Pi(a)\setminus \partial B_r (p)$ if $p\not\in \Pi(a)$ and is parallel to $\Pi(a)$ if $p\in\Pi(a)$ thanks to the free boundary property.
Therefore we can argue similarly to \cref{cor:14PuncturedBall} (keeping in mind that, this time, the exit-time is zero in particular for points in $\ms\cap \partial (\Pi(a)\cap B_r (p))$ to obtain that $\ms$ has genus and number of boundary components in $\partial B_1(0)\setminus B_r (p)$ bounded by $\kappa$.
We only need to show that the number of connected components of $\ms\cap(\partial B_1(0)\setminus \Pi(a))$ is also bounded by $\kappa$.

Since we can work separately on each connected component of $\ms\setminus B_r (p)$ and different connected components correspond to different connected components of $\ms\cap (\partial B_r (p)\setminus\Pi(a))$, for the sake of simplicity we can assume that $\Gamma\eqdef \ms\cap (\partial B_r (p)\setminus\Pi(a))$ consists of only one compact curve (i.e., an $S^1$ or an arc).
Then we want to prove that $\ms\cap (\partial B_1(0)\setminus \Pi(a))$ also consists of a single connected component.

If $\ms\cap (\partial B_1(0)\setminus \Pi(a))$ contains a closed curve (i.e., an $S^1$), then this must be the only connected component of $\ms\cap(\partial B_1(0)\setminus B_r (p))$ and thus we obtain what we want.
Otherwise all components of $\ms\cap (\partial B_1(0)\setminus \Pi(a))$ are arcs with endpoints in $\partial(\Pi(a)\cap B_1(0))$.
Consider one of these arcs and parametrize it with a unit-speed curve $\alpha\colon \cc 0l\to\partial B_1(0)\setminus \Pi(a)$ in such a way that $\alpha(0),\alpha(l)\in \partial(\Pi(a)\cap B_1(0))$.

Denote by $\beta\colon \cc 0l\to S^2$ a choice of the unit normal vector field along $\alpha$ orthogonal to $\alpha'$ in $\partial B_1(0)$. Then, thanks to \eqref{eq:AlmostTangent} and for $\mu_0$ sufficiently small, we can assume that $\scal{\beta(t)}{\nu(\alpha(t))}\ge 1/2$ for all $t\in\cc 0l$, where $\nu$ is a choice unit normal to $\ms$.
Therefore, we have that
\[
\begin{split}
\abs{\cov_{\alpha'}^{\partial B_1(0)}{\alpha'}} &= \abs{\scal{\cov_{\alpha'}^{\partial B_1(0)} \alpha'}{\beta}} = \scal{\beta}{\nu(\alpha)}^{-1} \abs{\scal{\cov_{\alpha'}^{\partial B_1(0)} \alpha'}{\nu(\alpha)}} \\
&\le 2(\abs{\scal{\Diff_{\alpha'}\alpha'}{\nu(\alpha)}} + \abs{\scal{\cov_{\alpha'}^{\partial B_1(0)}\alpha' - \Diff_{\alpha'}\alpha'}{\nu(\alpha)}})\\
&\le 2(\abs {A_{\ms}} + \abs{\scal{\alpha}{\nu(\alpha)}}) \point
\end{split}
\]
Hence note that we can choose $\mu_0$ small enough that $\abs{\cov_{\alpha'}^{\partial B_1(0)}{\alpha'}}\le \eps_1$, for some constant $\eps_1>0$ to be chosen later.

Therefore, since $\alpha'(0)$ is almost (depending on $\mu_0$) orthogonal to $\partial(\Pi(a)\cap B_1(0))$ by the free boundary condition, $\alpha$ remains close to the geodesic in $\partial B_1(0)\setminus\Pi(a)$ connecting $\alpha(0)$ and the north pole of $\partial B_1(0)$, which is the intersection between the normal to $\Pi(a)$ passing through the origin and $\partial B_1(0)\setminus\Pi(a)$. In particular, given $\eps_2>0$, we can choose $\eps_1$ (and $\mu_0$) in such a way that the $\eps_2$-neighborhood of the north pole of $\partial B_1(0)$ contains $\alpha(t_0)$ for some $t_0\in\cc 0l$.

Now assume by contradiction that $\ms\cap (\partial B_1(0)\setminus\Pi(a))$ consists of two or more connected components. Performing the above argument for two of them, we find two points $y,z\in\ms\cap (\partial B_1(0)\setminus\Pi(a))$ in two different connected components of $\ms\cap (\partial B_1(0)\setminus\Pi(a))$ that are both contained in an $\eps_3$-neighborhood of the north pole of $\partial B_1(0)$.

Consider the points $y',z'\in \ms\cap(\partial B_r (p)\setminus\Pi(a))$ such that $\Phi(t(y'),y') = y$ and $\Phi(t(z'),z') = z$, where $t(x)$ is the exit-time of $x\in\ms\setminus B_r(p)$ as in \cref{cor:14PuncturedBall}. By \eqref{eq:AlmostTangent}, given any constant $\eps_3>0$, we can choose $\mu_0$ sufficiently small such that $\abs{\Phi(t,x) - \hat\Phi(t,x)}\le \eps_3$ for all $x\in\ms$ and $t\le t(x)$, where $\hat\Phi$ is the flow of the vector field $2(x-p)$ in $\R^3$.
Then we have that
\[
\abs{\hat \Phi(t(y'),y') - \hat\Phi(t(z'),z')} \le 2(\eps_2+\eps_3).
\]
Now, choose $\eps_2,\eps_3$ so that $\abs{y'-z'} \le r/20$.
Hence, since $\ms\cap(\partial B_r (p)\setminus\Pi(a))$ consists of a single connected component and it is $\mu_0$-strongly equatorial, there exists an arc $\Gamma'\subset \ms\cap(\partial B_r (p)\setminus\Pi(a))$ of length less than $r/10$ that connects $y'$ and $z'$.
Therefore $\Gamma'' \eqdef \{ \Phi(t(x),x) \st x\in\Gamma'\}$ is a connected subset of $\ms\cap (\partial B_1(0)\setminus\Pi(a))$ containing $y$ and $z$, and contained in a $(1/4)$-neighborhood of the north pole of $\partial B_1(0)$. However, this contradicts the fact that $y$ and $z$ are contained in two different connected components of $\ms\cap (\partial B_1(0)\setminus \Pi(a))$ and thus proves what we wanted. Finally, the last assertion of the statement is easily verified given all the previous information.
\end{proof}


\chapter{Degeneration analysis}


This chapter is devoted to the analysis of a sequence of free boundary minimal surfaces with bounded index in a three-dimensional ambient manifold, as described in \cref{sec:TopologicalDegeneration}.

\section{Macroscopic behavior} \label{sec:MacroDescr}

In this section we study the behavior at macroscopic scale of a sequence of free boundary minimal surfaces with bounded index. 

\subsection{Smooth blow-up sets}

Similarly to \cite[Section 2.1.1]{ChodoshKetoverMaximo2017}, we first need to define a finite set of `bad points', away from which a sequence of free boundary minimal surfaces with bounded index is well-controlled.

\begin{definition} \label{def:BlowUpSets}
Let $(\amb^3,\smetric)$ be a compact three-dimensional Riemannian 
manifold and suppose that $\ms_j^2\subset\amb$ is a sequence 
of properly embedded \emph{edged} free boundary minimal surfaces.
A sequence of finite sets of points $\set_j \subset \ms_j$ is said to be a 
\emph{sequence of smooth blow-up sets} if
\begin{enumerate} [label={\normalfont(\roman*)}]
\item The second fundamental form blows up at points in $\set_j$, i.e., \[\liminf_{j\to\infty}\, \min_{p\in \set_j}\ \abs{\A_{\ms_j}}(p) = \infty\point\]
\item Chosen a sequence of points $p_j \in \set_j$, the 
rescaled surfaces $\bu \ms_j \eqdef \abs{\A_{\ms_j}}(p_j) (\ms_j - p_j)$ 
converge (up to subsequence) locally smoothly with multiplicity one (in the sense of graphs) to a complete, nonflat, properly embedded, free 
boundary minimal surface $\bu \ms_\infty \subset \Xi(a)$, for some $0\ge a \ge -\infty$, satisfying 
\[\abs{\A_{\bu \ms_\infty}}(x) \le \abs{\A_{\bu \ms_\infty}} (0)\] for all $x\in \bu 
\ms_\infty$. \label{bus:BlowUpLimit}
\item Points in $\set_j$ do not appear in the blow-up limit of other points in $\set_j$, i.e., \[\liminf_{j\to\infty}\, \min_{p\neq q\in \set_j}\ \abs{\A_{\ms_j}}(p) d_{\smetric_j}(p,q) = 
\infty\point\] \label{bus:PointsAreFar}
\end{enumerate}
\end{definition}

\subsection{Curvature estimates}

The starting point for our study is a curvature estimate for stable minimal surfaces, here stated in the free boundary setting.
\begin{theorem}[{\cite[Theorem 1.2]{GuangLiZhou2020}}]\label{thm:CurvEstStable}
Let $(\amb^3, \smetric)$ be a compact Riemannian manifold with boundary. Then there 
exists a constant $C =C(\amb,\smetric)$ such that, if $\ms^2\subset \amb$ is a 
compact, properly embedded, stable, \emph{edged}, free boundary minimal surface, then
\[
\sup_{x\in \ms}\ \abs{\A_\ms}(x) \min\{ 1, d_\smetric(x, \partial \ms \setminus \partial 
\amb) \} \le C \point
\]
\end{theorem}
\begin{remark}
Actually, in \cite{GuangLiZhou2020} the theorem is stated only for \emph{nonedged} free 
boundary minimal surfaces, but it is observed in Remark 1.3 therein
that the conclusion still holds in such more general setting.
\end{remark}

We will need the following extension of the previous result.

\begin{lemma} \label{lem:CurvEstBoundedInd}
Let $(\amb^3,\smetric)$ be a compact Riemannian manifold with boundary and fix $I\in 
\N$. Suppose that $\ms_j^2\subset \amb$ is a sequence of compact, properly 
embedded, \emph{edged}, free boundary minimal surfaces with $\ind(\ms_j) \le I$. Then, up 
to subsequence, there exist a constant $C=C(\amb,\smetric)$ and a sequence of smooth 
blow-up sets $\set_j\subset \ms_j$ with $\abs {\set_j} \le I$ and
\[
\sup_{x\in \ms_j} \ \abs{\A_{\ms_j}}(x) \min\{ 1, d_\smetric(x, \set_j \cup (\partial \ms_j 
\setminus \partial 
\amb)) \} \le C \point
\]
\end{lemma}
\begin{remark}
Note that the reason to state the lemma with \emph{edged} free boundary minimal surfaces is to perform the inductive procedure in the proof.
\end{remark}
\begin{proof}
We proceed by induction on $I\in \N$. If $I=0$ the statement is exactly 
\cref{thm:CurvEstStable}, thus assume $I>0$.
Note that we can suppose that \[\rho_j \eqdef \max_{x\in\ms_j}\ \abs{\A_{\ms_j}}(x) \min\{ 1, d_\smetric(x, \partial 
\ms_j \setminus \partial \amb) \} \to \infty\comma \] for otherwise one can take $\set_j=\emptyset$ and there is nothing left to do.
Take $q_j \in \ms_j$, a point where $\rho_j$ is attained, and define $R_j 
\eqdef d_\smetric(q_j, \partial 
\ms_j \setminus \partial \amb)/2$.

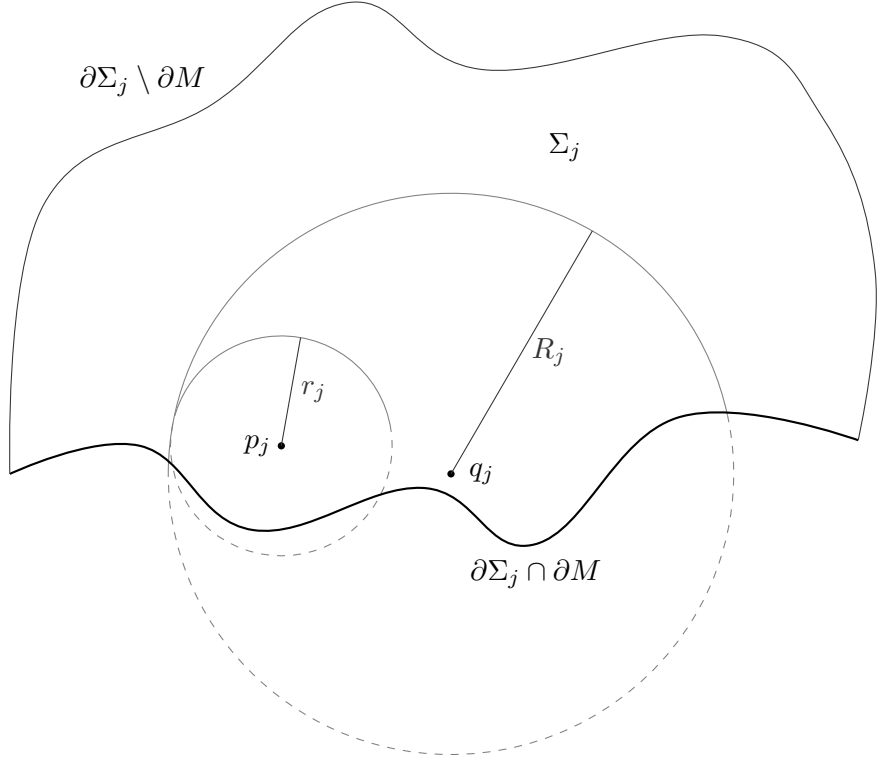
\begin{figure}[htpb]
\centering
\begin{tikzpicture} [scale=1.24]
\pgfmathsetmacro{\R}{3}
\pgfmathsetmacro{\r}{\R*(1-sqrt(37/100))}
\pgfmathsetmacro{\k}{1.2}
\coordinate(Q)at(0,0);
\coordinate(P)at(-0.6*\R,0.1*\R);

\coordinate(C1)at(-1.3*\R*\k,0*\k);
\coordinate(C2)at(-1.1*\R,0.1*\R);
\coordinate(C3)at(-0.7*\R,-0.2*\R);
\coordinate(C4)at(-0.1*\R,-0.05*\R);
\coordinate(C5)at(0.3*\R,-0.25*\R);
\coordinate(C6)at(0.8*\R,0.2*\R);
\coordinate(C7)at(1.2*\R*\k, 0.1*\R*\k);

\coordinate(D1)at(1.25*\R*\k,0.6*\R*\k);
\coordinate(D2)at(1.1*\R*\k,1.05*\R*\k);
\coordinate(D3)at(0.8*\R*\k,1.3*\R*\k);
\coordinate(D4)at(0.1*\R*\k,1.2*\R*\k);
\coordinate(D5)at(-0.3*\R*\k,1.4*\R*\k);
\coordinate(D6)at(-0.7*\R*\k,1.1*\R*\k);
\coordinate(D7)at(-1.2*\R*\k, 0.8*\R*\k);

\pgfmathsetmacro{\ta}{12}
\pgfmathsetmacro{\tb}{179}
\draw[myGray] ({\R*cos(\ta)}, {\R*sin(\ta)}) arc (\ta:\tb:\R);
\draw[dashed, myGray] ({\R*cos(\tb)}, {\R*sin(\tb)}) arc (\tb:360+\ta:\R);

\pgfmathsetmacro{\sa}{8}
\pgfmathsetmacro{\sb}{160}
\draw[myGray] ({-0.6*\R+\r*cos(\sa)}, {0.1*\R+\r*sin(\sa)}) arc (\sa:\sb:\r);
\draw[dashed, myGray] ({-0.6*\R+\r*cos(\sb)}, {0.1*\R+\r*sin(\sb)}) arc (\sb:360+\sa:\r);

\fill (Q) circle[radius=0.1em];
\node[right=3pt] at (Q) {$q_j$};
\fill (P) circle[radius=0.1em];
\node[left] at (P) {$p_j$};

\draw[black!80] (Q)--({cos(60)*\R},{sin(60)*\R}) node [midway,right] {$R_j$};
\draw[black!80] (P)--({-0.6*\R+cos(80)*\r},{0.1*\R+sin(80)*\r}) node [midway,right] {$r_j$};

\draw[thick] plot [smooth, tension=0.7] coordinates {(C1) (C2) (C3) (C4) (C5) (C6) (C7)};
\draw[black!80] plot [smooth, tension=0.7] coordinates {(C7) (D1) (D2) (D3) (D4) (D5) (D6) (D7) (C1)};

\node[below=2pt] at (C5) {$\partial\ms_j\cap\partial \amb$};
\node[above left=-1pt] at (D6) {$\partial\ms_j\setminus\partial \amb$};

\node[below right=20pt] at (D4) {$\ms_j$};
\end{tikzpicture}
\caption{Point picking argument.} \label{fig:PickPt}
\end{figure}

Now consider $p_j\in \ms_j \cap B_{R_j}(q_j)$ which 
realizes
\[
\max_{x\in \ms_j \cap B_{R_j}(q_j)}\, \abs{\A_{\ms_j}}(x) d_{\smetric}(x, \partial B_{R_j}(q_j)\setminus \partial \amb)
\]
and define $r_j \eqdef d_{\smetric}(p_j,  \partial B_{R_j}(q_j) \setminus\partial \amb)$ and 
$\lambda_j \eqdef \abs{\A_{\ms_j}}(p_j)$.
Note that 
\[
\begin{split}
\lambda_jr_j &= \abs{\A_{\ms_j}}(p_j) d_{\smetric}(p_j,  \partial B_{R_j}(q_j)\setminus\partial\amb) \ge  \abs{\A_{\ms_j}}(q_j) d_{\smetric}(q_j,  \partial B_{R_j}(q_j)\setminus \partial\amb) \\ &= \abs{\A_{\ms_j}}(q_j) R_j = \frac 12 \abs{\A_{\ms_j}}(q_j)d_\smetric(q_j, 
\partial 
\ms_j \setminus \partial \amb) \ge \frac{\rho_j}2 \to\infty \point
\end{split}
\]
Thus in particular $\abs{\A_{\ms_j}}(p_j) \min\{ 1, d_\smetric(p_j, \partial 
\ms_j \setminus \partial \amb) \}\to \infty$, where we are using that $d_\smetric(p_j, 
\partial 
\ms_j \setminus \partial \amb)\ge r_j$, since there are no points of $\partial 
\ms_j\setminus \partial \amb$ in $B_{R_j}(q_j)$.

Now, we have that $d_{\smetric}(x, \partial B_{r_j}(p_j) \setminus\partial\amb) \le 
d_{\smetric}(x, \partial B_{R_j}(q_j)\setminus\partial\amb)$ for every $x\in B_{r_j}(p_j)$ with equality in $p_j$. Therefore $p_j$ also realizes 
\[
\max_{x\in \ms_j \cap B_{r_j}(p_j)}\, \abs{\A_{\ms_j}}(x) d_{
g}(x, \partial B_{r_j}(p_j)\setminus\partial\amb) \point
\]

\clearpage

Hence, we can now perform a blow-up argument around the points 
$p_j$. In particular, let us define 
$\bu \ms_j \eqdef \lambda_j(\ms_j - p_j)$, a surface in the manifold $\amb_j \eqdef 
\lambda_j (\amb - p_j)$ endowed with the rescaled\footnote{One can think of $M$ as isometrically embedded in some $\R^K$ and consider the blow-ups in this Euclidean space. Thus, in particular, we have that $\smetric_j(x)=\lambda_j g(x+p_j)$.} metric $\smetric_j$. Then we have that
\[
\abs{\A_{\bu \ms_j}}(x)\, d_{
\smetric_j}(x, \partial B_{\lambda_jr_j}(0)\setminus \partial M_j) \le \lambda_jr_j \to\infty
\] 
for all $x \in \bu \ms_j \cap B_{\lambda_jr_j}(0)$. 
Note that here $B_{\lambda_jr_j}(0)$ is the ball in $\amb_j$ with respect to the metric $\smetric_j$. We do not write explicitly the dependence on $j$ since it is always clear from the context.

Hence, fixing $R>0$, for all $x\in \bu \ms_j \cap B_{R}(0)$ it 
holds that
\begin{equation} \label{eq:BehaviourSecFundBlowUp}
\abs{\A_{\bu \ms_j}}(x) \le \frac{\lambda_jr_j}{\lambda_jr_j - R} \nearrow 1
\end{equation}
as $j \to \infty$. Moreover, observe that $\abs{\A_{\bu \ms_j}}(0) = 1$ and the domains $B_{\lambda_jr_j}(0) \subset \amb_j$ do not contain points of $\partial\bu\ms_j\setminus \partial\amb_j$, since $\ms_j$ has no points of 
$\partial \ms_j\setminus \partial \amb$ in $B_{r_j}(p_j)$.
By \cref{thm:LamCptness} together with \cref{prop:LimLamInR3} (we are in case \ref{llir:multone} since the limit lamination cannot be flat), this implies that the surfaces $\bu\ms_j$ converge locally smoothly (in the sense of graphs) with multiplicity one to a properly embedded free boundary minimal surface $\bu\ms_\infty\subset \Xi(a)$ (for some $-\infty\leq a\leq 0$). Furthermore, the index of $\bu\ms_\infty$ is strictly positive and less or equal than $I$.

Hence, with a standard argument as in \cite[Proposition 1]{FischerColbrie1985}, there exists $R_0>0$ such that $\bu \ms_\infty \cap B_{R_0}(0)$ has 
index greater than $0$ and $\bu \ms_\infty\setminus B_{R_0}(0)$ is stable. 
Moreover we can assume, without loss of generality, that $\bu \ms_\infty$ intersects $\partial B_{R_0}(0)$ 
transversely and, thanks to Proposition 28 in \cite{AmbrozioBuzanoCarlottoSharp2018}, also that 
\begin{equation}\label{eq:SmallSecFundSigmaInfty}
\abs{\A_{\bu \ms_\infty}}(x) \le 1/4
\end{equation}
on $\bu \ms_\infty\setminus B_{R_0}(0)$.

Now consider $\ms_j' \eqdef \ms_j\setminus B_{R_0/\lambda_j}(p_j)$. Choosing 
$j$ sufficiently large, we can assume that $\partial B_{R_0/\lambda_j}(p_j)$ intersects $\ms_j$ transversely and that the ball $B_{R_0/\lambda_j}(p_j)$ does not intersect the portion of the boundary of $\ms_j$ that is not contained in 
$\partial \amb$. Indeed we know that $\lambda_j d_{\smetric}(p_j, \partial 
\ms_j\setminus \partial\amb)\to\infty$.
Therefore $\ms_j'$ is a sequence of manifolds which still fulfills the 
assumptions of the lemma, but with $\ind(\ms_j')\le I-1$ for $j$ 
sufficiently large.
Hence, by the inductive hypothesis, up to subsequence there exist a constant 
$C'>0$ and a sequence of smooth blow-up sets $\set_j'\subset \ms_j'$ 
with $\abs{\set_j'} \le I-1$ and
\begin{equation} \label{eq:EstimateInductive}
 \abs{\A_{\ms_j'}} (x) \min\{ 1, d_{ \smetric}(x,  \set_j'\cup (\partial 
\ms_j'\setminus\partial\amb)) \} \le C' \point
\end{equation}

We now want to show that $\set_j\eqdef  \set_j'\cup \{ p_j \}$ is the desired 
sequence of blow-up sets. The only nonobvious point to check is 
point \ref{bus:PointsAreFar} of \cref{def:BlowUpSets}, for which it suffices to verify that
\[
\lim_{j\to\infty}\, \min_{q\in \set_j'}\ \abs{\A_{\ms_j}}(p_j) d_{\smetric}(p_j,q) = 
\lim_{j\to\infty}\, \min_{q\in \set_j'}\ \abs{\A_{\ms_j}}(q) d_{\smetric}(p_j,q) = 
\infty \point
\]
We first observe that indeed
\[
\liminf_{j\to\infty}\, \min_{q\in \set_j'}\ \abs{\A_{\ms_j}}(q) d_{\smetric}(p_j,q)\geq  \liminf_{j\to\infty}\, \min_{q\in \set_j'}\ \abs{\A_{\ms_j}}(q) d_{\smetric}(\partial \ms'_j\setminus\partial M,q)=
\infty 
\]
based on \ref{bus:BlowUpLimit} for $q\in\set'_j$ (the associated limit surface is \emph{not edged}). However, thanks to \eqref{eq:BehaviourSecFundBlowUp} and \eqref{eq:SmallSecFundSigmaInfty}, we derive
 $\abs{\A_{\ms_j}}(q) \le 
\abs{\A_{\ms_j}}(p_j)/2$ for all $q\in \set_j'$ provided one takes $j$ sufficiently large; hence $\min_{q\in \set_j'}\ \abs{\A_{\ms_j}}(p_j) d_{\smetric}(p_j,q)$ could not be uniformly bounded either.

Therefore, it remains to check that there exists $C>0$ such that
\[
\abs{\A_{\ms_j}} (x) \min\{ 1, d_{\smetric}(x, \set_j\cup (\partial \ms_j\setminus \partial\amb)) \} \le 
C
\]
for all $x\in \ms_j$.
This inequality easily holds on $B_{R_0/\lambda_j}(p_j)$, thus it is sufficient 
to check it for points $x\in \ms_j'$.
Assume by contradiction that there exists a sequence of points $z_j\in \ms_j'$ such that
\begin{equation} \label{eq:ContradictionAssumption}
\limsup_{j\to\infty}\ \abs{\A_{\ms_j}}(z_j) \min\{ 1, d_{\smetric}(z_j, \set_j\cup 
(\partial \ms_j\setminus \partial \amb)) \} = \infty\point
\end{equation}
First observe that $\liminf_{j\to\infty}\lambda_jd_\smetric(z_j,p_j) = \infty$, because otherwise $\lambda_j(z_j-p_j)$ would converge to some point $\bar z\in \bu\ms_\infty$ and we would obtain
\[
\begin{split}
\limsup_{j\to\infty}\ &\abs{\A_{\ms_j}}(z_j) \min\{ 1, d_{\smetric}(z_j, \set_j\cup 
(\partial \ms_j\setminus \partial \amb)) \} \le \limsup_{j\to\infty}\ \abs{\A_{\ms_j}}(z_j) d_{\smetric}(z_j,p_j) \\
& = \limsup_{j\to\infty}\ \abs{\A_{\bu\ms_j}}(\lambda_j(z_j-p_j))  d_{\smetric_j}(\lambda_j(z_j-p_j),0) = \abs{\A_{\bu\ms_\infty}}(\bar z) d_{{\R^3}}(\bar z,0) < \infty \point
\end{split}
\]
Moreover, since both \eqref{eq:EstimateInductive} and \eqref{eq:ContradictionAssumption} hold and $[\set_j'\cup(\partial\ms_j'\setminus\partial\amb)]\setminus [\set_j\cup(\partial\ms_j\setminus\partial\amb)] = \partial\ms_j'\setminus\partial\ms_j$, we have that
\[
d_\smetric(z_j,\set_j'\cup(\partial\ms_j'\setminus\partial\amb)) = d_\smetric(z_j,\partial\ms_j'\setminus\partial\ms_j) = d_\smetric(z_j,p_j) - \frac{R_0}{\lambda_j} \point
\]
Thus, we can conclude that
\begingroup
\allowdisplaybreaks
\[
\begin{split}
\limsup_{j\to\infty}\ &\abs{\A_{\ms_j}}(z_j) \min\{ 1, d_{\smetric}(z_j, \set_j\cup 
(\partial \ms_j\setminus \partial \amb)) \} \le\limsup_{j\to\infty}\ \abs{\A_{\ms_j}}(z_j) d_{\smetric}(z_j,p_j) \\
&\le\limsup_{j\to\infty}\ \frac{C'}{d_\smetric(z_j,\set_j'\cup(\partial\ms_j'\setminus\partial\amb))}d_{\smetric}(z_j,p_j) \\
&= \limsup_{j\to\infty}\ \frac{C'}{d_{\smetric}(z_j,p_j) - R_0/\lambda_j}d_{\smetric}(z_j,p_j) \\
& = \limsup_{j\to\infty}\ \frac{C'}{1 - R_0/(\lambda_jd_{\smetric}(z_j,p_j) )} = C'\comma
\end{split}
\]
\endgroup
which is a contradiction and completes the proof.
\end{proof}

Given the previous lemma and the tools to handle free boundary minimal laminations presented in \cref{sec:FBMLam}, we can conclude the description of the limit picture at macroscopic scale.
\begin{corollary} \label{cor:ExistenceBlowUpSetAndCurvatureEstimate}
Let $(\amb^3,\smetric)$ be a compact Riemannian manifold with boundary and fix $I\in 
\N$. Suppose that $\ms_j^2\subset \amb$ is a sequence of compact, properly 
embedded, free boundary minimal surfaces with $\ind(\ms_j) \le I$. Then, up 
to subsequence, there exist a constant $C=C(\amb,\smetric)$ and a sequence of smooth 
blow-up sets $\set_j\subset \ms_j$ with $\abs {\set_j} \le I$ and
\begin{equation}\label{eq:CurvEstSj}
\sup_{x\in \ms_j} \ \abs{\A_{\ms_j}}(x) \min\{ 1, d_\smetric(x, \set_j ) \} \le C \point
\end{equation}
Moreover, the sets $\set_j$ converge to a set of points $\set_\infty$ and the surfaces $\ms_j$ converge locally smoothly away from $\set_\infty$ to some smooth free boundary minimal lamination $ \lam$ of $\amb$. Furthermore, if the ambient manifold satisfies \hypP{} then $\partial L = L \cap \partial\amb$ for all $L\in\lam$. 
\end{corollary}

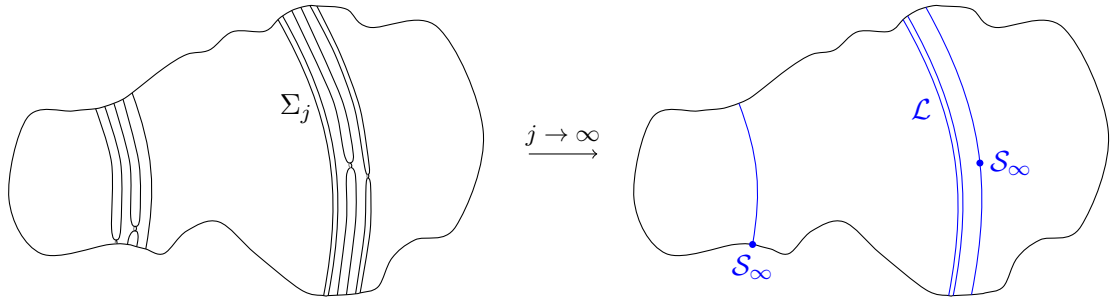
\begin{figure}[htpb]
\centering
\begin{tikzpicture}[scale=0.60]
\coordinate (A1) at (0,0); 
\coordinate (B1) at (0.2,0.05); 
\coordinate (C1) at (0.4,0.12);
\coordinate (D1) at (0.6,0.21);
\coordinate (E1) at (0.8,0.35);

\coordinate (F1) at (4,1.97);
\coordinate (G1) at (4.1,2.05);
\coordinate (H1) at (4.3,2.17);
\coordinate (I1) at (4.5,2.24);
\coordinate (J1) at (4.7,2.28);
\coordinate (K1) at (4.85,2.27);

\coordinate (F2) at (5,-4.13);
\coordinate (G2) at (5.1,-4.13);
\coordinate (H2) at (5.3,-4.12);
\coordinate (I2) at (5.5,-4.11);
\coordinate (J2) at (5.7,-4.09);
\coordinate (K2) at (5.8,-4.07);

\coordinate (C2) at (0.7,-3);
\coordinate (D2) at (0.9,-3.05);
\coordinate (E2) at (1.1,-3.1);

\coordinate (X) at (-1.5,-0.2);
\coordinate (Y) at (-1.2, -3.2);
\coordinate (U) at (7.5, 1.8);
\coordinate (V) at (7.9, -2.6);

\draw plot [smooth cycle, tension=0.6] coordinates {(A1) (B1) (C1) (D1) (E1)
			      (1.5,0.9) (2,1.2) (2.5,1.3) (3,1.7) (3.7,1.71)
			      (F1) (G1) (H1) (I1) (J1) (K1)
			      (5.5,2.2) (6,1.8)
			      (U) (8.5,-0.5) (V)
			      (7,-2.9) (6.5, -3.3) (6.2, -3.93) 
			      (K2) (J2) (I2) (H2) (G2) (F2)
			      (4.5,-4) (3.5,-3.2) (2.7, -2.5) (2,-2.7) (1.5, -3.2) 
			      (E2) (D2) (C2)
			      (0.3, -3)
			      (Y) (-1.9,-2) (X) (-0.5,-0.05)};
			      
\draw (F1) to [bend left=20] (F2);
\draw (G1) to [bend left=20] (G2);

\draw plot [smooth, tension = 0.5] coordinates {(H1) (5.2,0.2) (5.45,-1) (5.57, -1.2) (5.65, -1) (5.4,0.2) (I1)};
\draw plot [smooth, tension = 0.5] coordinates {(H2) (5.5,-2.5) (5.5,-1.5) (5.6, -1.3) (5.7, -1.5) (5.7,-2.5) (I2)};
\pgfmathsetmacro{\t}{atan(0.03/0.1)}
\begin{scope}[rotate around={\t:(5.57,-1.2)}]
\draw[gray] (5.57, -1.2) arc (90:270:0.03 and {sqrt(0.03*0.03+0.1*0.1)/2});
\draw (5.57, -1.2) arc (90:-90:0.03 and {sqrt(0.03*0.03+0.1*0.1)/2});
\end{scope}

\draw plot [smooth, tension = 0.5] coordinates {(J1) (5.25,1.2) (5.68,0) (5.9,-1.29) (5.97, -1.45) (6, -1.27) (5.8,0) (5.37,1.2) (K1)};
\draw plot [smooth, tension = 0.5] coordinates {(J2)  (5.9,-2.7) (5.93, -1.7) (5.97, -1.53) (6.03,-1.7) (6.02, -2.7) (K2)};
\draw[gray] (5.97, -1.45) arc (90:270:0.02 and {0.04});
\draw (5.97, -1.45) arc (90:-90:0.02 and {0.04});

\draw[gray] (0.45, -2.9) arc (90:182:0.02 and {0.08});
\draw (0.45, -2.9) arc (90:0:0.02 and {0.08});
\draw plot [smooth, tension = 0.55] coordinates {(A1) (0.35,-1.2) (0.35,-2.65) (0.45, -2.9) (0.55, -2.65) (0.55,-1.2) (B1)};

\pgfmathsetmacro{\t}{atan(0.01/0.1)}
\begin{scope}[rotate around={-\t:(0.87, -2.6)}]
\draw[gray] (0.87,-2.6) arc (90:270:0.03 and {sqrt(0.01*0.01+0.1*0.1)/2});
\draw (0.87,-2.6) arc (90:-90:0.03 and {sqrt(0.01*0.01+0.1*0.1)/2});
\end{scope}
\draw plot [smooth, tension = 0.55] coordinates {(C1) (0.75,-1.2) (0.77,-2.4) (0.87, -2.6) (0.97, -2.4) (0.95,-1.2) (D1)};
\draw plot [smooth, tension = 0.9] coordinates {(C2) (0.75, -2.8) (0.86, -2.7) (0.93, -2.8) (D2)};

\draw (E1) to [bend left = 15] (E2);

\node at (4.4,0) {$\ms_j$};

\draw[->] (9.5,-1) -- (11,-1);
\node at (10.25,-0.6) {\footnotesize $j\to\infty$};

\begin{scope}[xshift=390]
\coordinate (A1) at (0,0); 
\coordinate (B1) at (0.2,0.05); 
\coordinate (C1) at (0.4,0.12);
\coordinate (D1) at (0.6,0.21);
\coordinate (E1) at (0.8,0.35);

\coordinate (F1) at (4,1.97);
\coordinate (G1) at (4.1,2.05);
\coordinate (H1) at (4.3,2.17);
\coordinate (I1) at (4.5,2.24);
\coordinate (J1) at (4.7,2.28);
\coordinate (K1) at (4.85,2.27);

\coordinate (F2) at (5,-4.13);
\coordinate (G2) at (5.1,-4.13);
\coordinate (H2) at (5.3,-4.12);
\coordinate (I2) at (5.5,-4.11);
\coordinate (J2) at (5.7,-4.09);
\coordinate (K2) at (5.8,-4.07);

\coordinate (C2) at (0.7,-3);
\coordinate (D2) at (0.9,-3.05);
\coordinate (E2) at (1.1,-3.1);

\coordinate (X) at (-1.5,-0.2);
\coordinate (Y) at (-1.2, -3.2);
\coordinate (U) at (7.5, 1.8);
\coordinate (V) at (7.9, -2.6);

\draw plot [smooth cycle, tension=0.6] coordinates {(A1) (B1) (C1) (D1) (E1)
			      (1.5,0.9) (2,1.2) (2.5,1.3) (3,1.7) (3.7,1.71)
			      (F1) (G1) (H1) (I1) (J1) (K1)
			      (5.5,2.2) (6,1.8)
			      (U) (8.5,-0.5) (V)
			      (7,-2.9) (6.5, -3.3) (6.2, -3.93) 
			      (K2) (J2) (I2) (H2) (G2) (F2)
			      (4.5,-4) (3.5,-3.2) (2.7, -2.5) (2,-2.7) (1.5, -3.2) 
			      (E2) (D2) (C2)
			      (0.3, -3)
			      (Y) (-1.9,-2) (X) (-0.5,-0.05)};
			      
\draw[blue] (C1) to [bend left = 15] (C2);
\draw[blue] (F1) to [bend left = 20] (F2);
\draw[blue] (G1) to [bend left = 20] (G2);
\draw[blue] (I1) to [bend left = 20] (I2);

\fill[blue] (C2) circle [radius=0.2em] node[below] {$\set_\infty$};
\fill[blue] (5.685, -1.2) circle [radius=0.2em] node[right] {$\set_\infty$};
\node[blue] at (4.4,0) {$\lam$};
\end{scope}
\end{tikzpicture}
\caption{Macroscopic description of degeneration.} \label{fig:MacroDescr}
\end{figure}

\begin{proof}
The first part of the statement is a special case of \cref{lem:CurvEstBoundedInd}. Then, possibly extracting a further subsequence, we can assume that the sets $\set_j$ converge to a set $\set_\infty$ (of cardinality at most $I$) and, thanks to \cref{thm:LamCptness}, the surfaces $\ms_j$ converge to a free boundary minimal lamination $\hat\lam$ in $\amb\setminus \set_\infty$ smoothly away from $\set_\infty$. 
However, \cref{thm:RemSingLimLam} ensures that the lamination $\hat\lam$ extends smoothly through $\set_\infty$; namely, there exists a smooth free boundary minimal lamination $\lam$ in $\amb$ extending $\hat\lam$. The last claim is straightforward.
\end{proof}


\section{Microscopic behavior} \label{sec:LocalDegeneration}

In this section we study the behavior of our minimal surfaces at small scales around the points where concentration of curvature occurs, that is to say around the points in $\set_\infty$ in \cref{cor:ExistenceBlowUpSetAndCurvatureEstimate}.

\subsection{Setting description} \label{HypN}

We denote by $(\mathfrak{N})$ the following set of assumptions:
\begin{enumerate} [label={\normalfont(\roman*)}]
\item $\smetric_j$ is a sequence of metrics on $\amb_j^3 \eqdef \Xi(a_j)\cap \{\abs x< R_j\}\subset \R^3$, with $0\ge a_j\ge -\infty$ and $R_j\to\infty$, locally smoothly converging to the Euclidean metric as 
$j\to\infty$.
\item $\ms_j^2\subset \amb_j$ is a sequence of properly embedded \emph{edged} minimal surfaces (with $\partial\ms_j\subset\partial\amb_j$) that have free boundary with respect to $\Pi(a_j)$. 
\item \label{N:BlowUpSet} $\ind(\ms_j)\le I$ for some natural constant $I>0$ independent of $j$, and
$\set_j\subset \ms_j\cap B_{{\mu_0}}(0)$, {for $\mu_0$ given by \cref{cor:14TopInfo}}, is a sequence of nonempty smooth blow-up sets\footnote{Hereafter we denote by $B_r(p)$ the ball of center $p$ and radius $r$ in the metric $\smetric_j$, without specifying $j$ when this is clear from the context. Also, note that the setting in \cref{def:BlowUpSets} is slightly different from the setting here, but the definition can be easily adapted to this context.} 
with $\abs {\set_j}\le I$ and \[\abs{\A_{\ms_j}}(x) d_{\smetric_j}(x, \set_j\cup 
(\partial\ms_j\setminus \Pi(a_j) )) \le C\] for all $x\in \ms_j$, for some constant 
$C>0$ independent of $j$.
\end{enumerate}

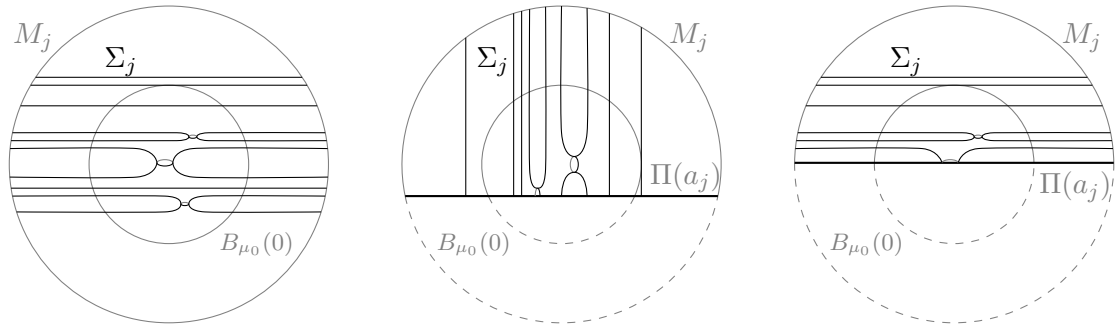
\begin{figure}[htpb]
\centering
\begin{tikzpicture}[scale=1.05]
\pgfmathsetmacro{\R}{2}
\pgfmathsetmacro{\r}{1}


\coordinate (OA) at (0,0);

\draw[gray] (OA) circle (\R);
\node[gray] at (-0.85*\R,0.8*\R) {$\amb_j$};
\draw[gray] (OA) circle (\r);
\node[gray] at (1.1*\r,-1*\r) {\footnotesize $B_{\mu_0}(0)$};

\node at (-0.3*\R, 0.65*\R) {$\ms_j$};

\def\h{{-0.15, 0.37, 0.5, 0.55}}

\foreach \i in {0,...,3}
  \draw ({-sqrt(\R*\R*(1-\h[\i]*\h[\i]))},{\R*\h[\i]}) -- ({sqrt(\R*\R*(1-\h[\i]*\h[\i]))},{\R*\h[\i]});

\pgfmathsetmacro{\ha}{-0.3*\R}
\pgfmathsetmacro{\xa}{sqrt(\R*\R-\ha*\ha)}
\pgfmathsetmacro{\hb}{-0.20*\R}
\pgfmathsetmacro{\xb}{sqrt(\R*\R-\hb*\hb)}

\draw[gray] (0.15,{(\ha+\hb)/2}) arc (180:0:0.05 and 0.02);
\draw (0.15,{(\ha+\hb)/2}) arc (180:360:0.05 and 0.02);
\draw plot [smooth, tension=0.35] coordinates {(-\xa,\ha) (-0.05,\ha) (0.15,{(\ha+\hb)/2}) (-0.05,\hb) (-\xb,\hb)};
\draw plot [smooth, tension=0.4] coordinates {(\xb,\hb) (0.45,\hb) (0.25,{(\hb+\ha)/2}) (0.45,\ha) (\xa,\ha)};

\pgfmathsetmacro{\hd}{-0.08*\R}
\pgfmathsetmacro{\xd}{sqrt(\R*\R-\hd*\hd)}
\pgfmathsetmacro{\he}{0.10*\R}
\pgfmathsetmacro{\xe}{sqrt(\R*\R-\he*\he)}

\draw[gray] (-0.15,{(\he+\hd)/2}) arc (180:0:0.1 and 0.04);
\draw (-0.15,{(\he+\hd)/2}) arc (180:360:0.1 and 0.04);
\draw plot [smooth, tension=0.4] coordinates {(-\xe,\he) (-0.4,\he) (-0.15,{(\he+\hd)/2}) (-0.4,\hd) (-\xd,\hd)};
\draw plot [smooth, tension=0.4] coordinates {(\xe,\he) (0.3,\he) (0.05, {(\he+\hd)/2}) (0.3,\hd) (\xd,\hd)};

\pgfmathsetmacro{\hf}{0.15*\R}
\pgfmathsetmacro{\xf}{sqrt(\R*\R-\hf*\hf)}
\pgfmathsetmacro{\hg}{0.20*\R}
\pgfmathsetmacro{\xg}{sqrt(\R*\R-\hg*\hg)}

\draw[gray] (0.25,{(\hf+\hg)/2}) arc (180:0:0.05 and 0.02);
\draw (0.25,{(\hf+\hg)/2}) arc (180:360:0.05 and 0.02);
\draw plot [smooth, tension=0.25] coordinates {(-\xg,\hg) (0.1,\hg) (0.25,{(\hg+\hf)/2}) (0.1,\hf) (-\xf,\hf)};
\draw plot [smooth, tension=0.32] coordinates {(\xg,\hg) (0.5,\hg) (0.35,{(\hg+\hf)/2}) (0.5,\hf) (\xf,\hf)};

\begin{scope}[xshift=70*\R]

\pgfmathsetmacro{\hl}{-0.20*\R}
\pgfmathsetmacro{\xl}{sqrt(\R*\R-\hl*\hl)}
\coordinate (OA) at (0,0);

\pgfmathsetmacro{\t}{atan(-\hl/(\xl))}
\draw[gray] (\xl,\hl) arc(-\t:180+\t:\R);
\draw[gray, dashed] (-\xl,\hl) arc(180+\t:360-\t:\R);
\node[gray] at (0.8*\R,0.8*\R) {$\amb_j$};

\pgfmathsetmacro{\s}{atan(-\hl/(sqrt(\r*\r-\hl*\hl)))}
\draw[gray] ({(sqrt(\r*\r-\hl*\hl))},\hl) arc(-\s:180+\s:\r);
\draw[gray, dashed] (-{(sqrt(\r*\r-\hl*\hl))},\hl) arc(180+\s:360-\s:\r);
\node[gray] at (-1.1*\r,-1*\r) {\footnotesize $B_{\mu_0}(0)$};

\draw[thick] (-\xl,\hl) -- (\xl,\hl);
\node[gray] at (0.78*\R,\hl+0.12*\R) {$\Pi(a_j)$};

\node at (-0.43*\R, 0.65*\R) {$\ms_j$};

\def\c{{-0.3,-0.25, 0.3, -0.6,0.5}}
 
\foreach \i in {0,...,4}
 \draw ({\R*\c[\i]}, {sqrt(\R*\R*(1-\c[\i]*\c[\i]))}) -- ({\R*\c[\i]}, \hl);
 
\pgfmathsetmacro{\ca}{-0.2*\R}
\pgfmathsetmacro{\ya}{sqrt(\R*\R-\ca*\ca)}
\pgfmathsetmacro{\cb}{-0.10*\R}
\pgfmathsetmacro{\yb}{sqrt(\R*\R-\cb*\cb)}

\draw[gray] ({(\ca+\cb)/2}, {\hl+0.05*\R}) arc (90:180:0.03 and {0.05*\R});
\draw ({(\ca+\cb)/2}, {\hl+0.05*\R}) arc (90:0:0.03 and {0.05*\R});
\draw plot [smooth, tension=0.4] coordinates {(\ca,\ya) (\ca, {\hl+0.17*\R}) ({(\ca+\cb)/2}, {\hl+0.05*\R}) (\cb,{\hl+0.17*\R}) (\cb,\yb)};

\pgfmathsetmacro{\cc}{0*\R}
\pgfmathsetmacro{\yc}{sqrt(\R*\R-\cc*\cc)}
\pgfmathsetmacro{\cd}{0.16*\R}
\pgfmathsetmacro{\yd}{sqrt(\R*\R-\cd*\cd)}

\draw[gray] ({(\cc+\cd)/2}, {\hl+0.25*\R}) arc (90:270:0.05 and {0.05*\R});
\draw ({(\cc+\cd)/2}, {\hl+0.25*\R}) arc (90:-90:0.05 and {0.05*\R});
\draw plot [smooth, tension=0.5] coordinates {(\cc,\yc) (\cc, {\hl+0.4*\R}) ({(\cc+\cd)/2}, {\hl+0.25*\R}) (\cd,{\hl+0.4*\R}) (\cd,\yd)};
\draw plot [smooth, tension=1.8] coordinates {(\cc,\hl) ({(\cc+\cd)/2}, {\hl+0.15*\R}) (\cd,\hl)};
\end{scope}

\begin{scope}[xshift=140*\R]

\pgfmathsetmacro{\hl}{0.01*\R}
\pgfmathsetmacro{\xl}{sqrt(\R*\R-\hl*\hl)}

\coordinate (OA) at (0,0);

\pgfmathsetmacro{\t}{atan(-\hl/(\xl))}
\draw[gray] (\xl,\hl) arc(-\t:180+\t:\R);
\draw[gray, dashed] (-\xl,\hl) arc(180+\t:360-\t:\R);
\node[gray] at (0.8*\R,0.8*\R) {$\amb_j$};

\pgfmathsetmacro{\s}{atan(-\hl/(sqrt(\r*\r-\hl*\hl)))}
\draw[gray] ({(sqrt(\r*\r-\hl*\hl))},\hl) arc(-\s:180+\s:\r);
\draw[gray, dashed] (-{(sqrt(\r*\r-\hl*\hl))},\hl) arc(180+\s:360-\s:\r);
\node[gray] at (-1.1*\r,-1*\r) {\footnotesize $B_{\mu_0}(0)$};

\draw[thick] (-\xl,\hl) -- (\xl,\hl);
\node[gray] at (0.76*\R,\hl-0.15*\R) {$\Pi(a_j)$};

\node at (-0.3*\R, 0.65*\R) {$\ms_j$};

\def\h{{0.37, 0.5, 0.55}}

\foreach \i in {0,...,2}
  \draw ({-sqrt(\R*\R*(1-\h[\i]*\h[\i]))},{\R*\h[\i]}) -- ({sqrt(\R*\R*(1-\h[\i]*\h[\i]))},{\R*\h[\i]});

\pgfmathsetmacro{\hd}{-0.08*\R}
\pgfmathsetmacro{\xd}{sqrt(\R*\R-\hd*\hd)}
\pgfmathsetmacro{\he}{0.10*\R}
\pgfmathsetmacro{\xe}{sqrt(\R*\R-\he*\he)}

\draw[gray] (-0.15,{(\he+\hd)/2}) arc (180:0:0.1 and 0.04);
\draw plot [smooth, tension=0.4] coordinates {(-\xe,\he) (-0.4,\he) (-0.15,{(\he+\hd)/2})};
\draw plot [smooth, tension=0.4] coordinates {(\xe,\he) (0.3,\he) (0.05, {(\he+\hd)/2})};

\pgfmathsetmacro{\hf}{0.15*\R}
\pgfmathsetmacro{\xf}{sqrt(\R*\R-\hf*\hf)}
\pgfmathsetmacro{\hg}{0.20*\R}
\pgfmathsetmacro{\xg}{sqrt(\R*\R-\hg*\hg)}

\draw[gray] (0.25,{(\hf+\hg)/2}) arc (180:0:0.05 and 0.02);
\draw (0.25,{(\hf+\hg)/2}) arc (180:360:0.05 and 0.02);
\draw plot [smooth, tension=0.25] coordinates {(-\xg,\hg) (0.1,\hg) (0.25,{(\hg+\hf)/2}) (0.1,\hf) (-\xf,\hf)};
\draw plot [smooth, tension=0.32] coordinates {(\xg,\hg) (0.5,\hg) (0.35,{(\hg+\hf)/2}) (0.5,\hf) (\xf,\hf)};
\end{scope}
\end{tikzpicture}
\caption{Different possible situations in setting \hypN{}.} \label{fig:MicroSett}
\end{figure}

\begin{proposition} \label{prop:StructureLimLam}
Let us consider the set of assumptions \hypN{}. Then, up to subsequence, the smooth blow-up sets $\set_j$ converge to a finite set of points $\set_\infty$, of cardinality at most $I$, and there 
is a free boundary minimal lamination $\lam$ in $\Pi(a)$ (where we assume the existence of $a = \lim_{j\to\infty}a_j\in \cc{-\infty} 0$) consisting of parallel planes or half-planes and such that $\ms_j$ locally converges (in the sense of laminations) to $\lam$ away from $\set_\infty$.
\end{proposition}
\begin{proof}
First observe that, up to subsequence, we can assume that $\set_j$ converge to a finite set of points $\set_\infty$ of cardinality at most $I$ and contained in the open unit ball centered at the origin.
Then, because of the curvature assumptions we are making, thanks to \cref{prop:LimLamInR3} we gain smooth (subsequential) convergence to a free boundary minimal lamination $\hat\lam$ in $\Pi(a)\setminus\set_\infty$;
in fact $\hat\lam$ extends to a smooth lamination $\lam$ of $\Pi(a)$. 
We now need to rule out alternative \ref{llir:multone} of the proposition, namely the possibility that $\lam$ consists of a (single) nonflat, two-sided, properly embedded, free boundary minimal surface.
However, in this case the convergence must be locally smooth and graphical with multiplicity one at (all points of) such unique leaf: hence this would imply locally uniform curvature estimates for the sequence $\Sigma_j$, which is in contradiction with the presence of a smooth blow-up set in \hypN{}, though.
Thus the only possibility is that $\lam$ is a lamination in $\Xi(a)$ of parallel planes or half-planes, which concludes the proof.
\end{proof}

Later on, we will separately study the components where `bad things' happen (but which are in finite number) and the others. Therefore it will be useful to introduce the following definition. 
\begin{definition}
In the setting \hypN{}, let us denote by $\ms_j^{\euno}$ the union of the connected components of 
$\ms_j\cap B_1(0)$ that contain at least one point in $\set_j$ (informally: the ones with the necks in \cref{fig:MicroSett}) and with
$\ms_j^{\edue}$ the union of the connected components of $\ms_j\cap B_1(0)$ that do not.
\end{definition}
\begin{remark}
It is sufficient to work in $B_1(0)$, since the information about $\ms_j$ in $B_1(0)^c$ are then obtained in the applications thanks to suitable Morse-theoretic arguments (see \cref{sec:MorseTheory}).
\end{remark}

\begin{remark} \label{rem:ConvSigma1}
Observe that the surfaces $\ms_j^{\euno}$ locally converge (in the sense of laminations) to the union of the leaves of the lamination $\lam$ (given by \cref{prop:StructureLimLam}) passing through points of $\set_\infty$, the convergence happening away from $\set_\infty$. Indeed, the convergence to any other component of $\lam$ in $B_1(0)$ is uniformly smooth (in the sense of laminations), but each component of $\ms_j^{\euno}$ contains a point where the curvature diverges.

In particular, if the number of components of $\ms_j^{\euno}\cap(\partial B_1(0)\setminus\Pi(a_j))$ is uniformly bounded, then $\ms_j^{\euno}\cap(\partial B_1(0)\setminus\Pi(a_j))$ is $\mu_0$-strongly equatorial (as per \cref{def:StrongEq}) for any $j$ sufficiently large.
\end{remark}

\subsection{Neck components}

In this section, we deal with the behavior of $\ms_j^{\euno}$, which is the part of $\ms_j\cap B_1(0)$ that `carries the topology' of the surface $\ms_j$. 
The components $\ms_j^{\edue}$ are instead well-controlled, in the sense that they have uniformly bounded curvature and they are topological discs. We postpone the investigation of these properties of $\ms_j^{\edue}$ to the proof \cref{thm:GlobalDeg}.

The following proposition is essentially the base case of the induction to prove \cref{prop:LocDeg}, which is the full description of what happens around the origin along the sequence $\ms_j^{\euno}$. 

\begin{lemma} \label{lem:LocDegOnePt}
Let us assume to be in the setting \hypN{} with $\abs {\set_j} = 1$ for all $j$.
Then there exists $\kappa(I)\ge0$ (depending on $I$) such that, for $j$ sufficiently large, the following assertions hold true:
\begin{enumerate} [label={\normalfont(\arabic*)}]
\item The surfaces $\ms_j^{\euno}$ have genus at most $\kappa(I)$.
\item {The surfaces $\ms_j^{\euno}$ intersect both $\partial B_{1}(0)\setminus \Pi(a_j)$ and $\Pi(a_j)$ transversely in at most $\kappa(I)$ components.}
\end{enumerate}

\end{lemma}
\begin{proof}
Let $\set_j = \{p_j\}$, $\set_\infty = \{p_\infty\}$ and $\lambda_j \eqdef \abs{\A_{\ms_j}}(p_j)$. By definition of smooth blow-up set, the surfaces $\bu\ms_j\eqdef \lambda_j(\ms_j-p_j)$ converge up to subsequence to a complete, nonflat, properly embedded, free boundary minimal surface $\bu\ms_\infty$ in $\Xi(a)$ for some $0\ge a\ge -\infty$. Furthermore $\bu\ms_\infty$ has index at most $I$.
Therefore, by \cref{prop:GeometryOfHalfBubble}, the genus, number of ends and number of boundary components of $\bu\ms_\infty$ are all bounded by $\kappa(I)$.

Consider $\mu_0$ given by \cref{cor:14TopInfo} and take $R_0>0$ such that 
\begin{equation} \label{eq:CurvEstSigmaBarInfty}
\abs{\A_{\bu\ms_\infty}} (x) d_{\R^3}(0,x) < {\mu_0}
\end{equation}
for $x\in\bu\ms_\infty \setminus B_{R_0}(0)$.
{Assume $R_0$ large enough that the genus and the number of connected components of both $\bu\ms_\infty\cap { (\partial B_{R_0}(0)\setminus \Pi(a))}$ and $\bu\ms_\infty\cap \Pi(a)$ are bounded by $\kappa(I)$.
{Moreover, thanks to \cite[Proposition 1]{Schoen1983Uniqueness}, we can also suppose that $\bu\ms_\infty\cap (\partial B_{R_0}(0)\setminus\Pi(a))$ is $(\mu_0/2)$-strongly equatorial (as per \cref{def:StrongEq}).}
Hence observe that, for $j$ sufficiently large, $\ms_j\cap B_{R_0/\lambda_j}(p_j)$ has genus and number of boundary components on $\partial B_{R_0/\lambda_j}(p_j)\setminus \Pi(a_j)$ both bounded by $\kappa(I)$, and $\ms_j\cap(\partial B_{R_0/\lambda_j}(p_j)\setminus\Pi(a_j))$ is $\mu_0$-strongly equatorial.}
In order to transfer this information to all $B_1(0)$ we want to prove that the estimate
\begin{equation}\label{eq:14EstMsJ}
\abs{\A_{\ms_j}}(x) d_{\smetric_j} (p_j,x) < {\mu_0}
\end{equation}
holds for every $x\in \ms_j \cap ( B_{1}(0)\setminus B_{R_0/\lambda_j}(p_j) )$, for $j$ sufficiently large.

To this purpose, it is enough to prove that there exists $\delta>0$ such that the estimate holds in $\ms_j \cap ( B_\delta(p_j)\setminus B_{R_0/\lambda_j}(p_j) )$. Then we deduce the desired estimate using that $\ms_j$ converges (in the sense of laminations) in $B_1(0)\setminus B_\delta (p_j)$ to a lamination consisting of planes (indeed $p_\infty\in B_\delta(p_j)$ for $j$ sufficiently big).

Assume by contradiction that such $\delta>0$ does not exist. Then one would find a sequence $z_j\in \ms_j\setminus B_{R_0/\lambda_j}(p_j)$ such that $\delta_j\eqdef d_{\smetric_j}(p_j,z_j) \to 0$ and $\abs{\A_{\ms_j}}(z_j)\delta_j \ge {\mu_0}$.
Then consider $\check\ms_j \eqdef \delta_j^{-1}(\ms_j-p_j)$, for which we have \begin{equation}\label{eq:BoundBelowCurvature}
\abs{\A_{\check\ms_j}}(\delta_j^{-1}(z_j-p_j)) \ge {\mu_0}\point
\end{equation}
Note that there cannot possibly be a uniform curvature bound for the sequence $\check\ms_j$ around $0$, otherwise the scales $\lambda_j$ and $\delta^{-1}_j$ would be comparable, hence the surfaces $\check\ms_j$ would converge to a homothety of $\bu\ms_\infty$, but this is not possible for the choice of $z_j$ together with \eqref{eq:CurvEstSigmaBarInfty}.
As a result, possibly extracting a subsequence (which we do not rename) $\check\ms_j$ still fulfills the assumptions of the setting \hypN{} and therefore it converges to a lamination consisting of planes away from $0$. However, observe that this implies $\abs{\A_{\check\ms_j}}(\delta_j^{-1}(z_j-p_j))\to 0$, which contradicts \eqref{eq:BoundBelowCurvature} and thus proves \eqref{eq:14EstMsJ}.

Thus, all the assumptions of \cref{cor:14TopInfo} are satisfied and therefore the genus and the number of connected components of both $\ms_j^{\euno}\cap(\partial B_1(0)\setminus \Pi(a_j))$ and $\ms_j^{\euno}\cap\Pi(a_j)$ are bounded by $\kappa(I)$ (possibly renaming $\kappa(I)$ as the double of the constant introduced above).
\end{proof}

We can now proceed and prove the corresponding result for any set $\set_j$ of uniformly bounded cardinality.

\begin{proposition} \label{prop:LocDeg}
Assume to be in the setting \hypN{}. 
Then there exists $\kappa(I)\ge0$ such that, for $j$ sufficiently large, the following assertions hold true:
\begin{enumerate} [label={\normalfont(\arabic*)}]
\item The surfaces $\ms_j^{\euno}$ have genus at most $\kappa(I)$.
\item {The surfaces $\ms_j^{\euno}$ intersect both $\partial B_{1}(0)\setminus \Pi(a_j)$ and $\Pi(a_j)$ transversely in at most $\kappa(I)$ components.}
\end{enumerate}
\end{proposition}
\begin{proof}
Let us proceed by induction on $I> 0$. 
The case $I=1$ has been treated in \cref{lem:LocDegOnePt}, thus let us assume $I>1$.
We distinguish two cases, and we first consider the case when $\abs{\set_\infty} \ge2$. 
Choose $\delta>0$ small enough that $\min_{p,q\in \set_\infty}d_\smetric(p,q)\ge 4\delta$, moreover fix one point $p_\infty\in \set_\infty$. Since $\ind(\ms_j\cap B)\ge 1$ for every connected component $B$ of $B_\delta(\set_\infty)$, then $\ind(\ms_j\cap B_\delta(p_\infty)) \le I-1$.

Now choose positive numbers $r_j\to 0$ such that $\set_j\subset B_{\mu_0r_j}(\set_\infty)$ and 
\[\liminf_{j\to\infty} \left(r_j \min_{p\in \set_j} \abs{\A_{\ms_j}} (p)\right) = \infty\point\]
Note that the surfaces $r_j^{-1}(\ms_j-p_\infty)$ still fulfill the assumptions \hypN{} (with blow-up sets that are rescalings of the blow-up sets $\set_j$) thanks to the choice of $r_j$ and thus we can apply the inductive hypothesis to these surfaces. In particular we obtain that the components $\ms_j^{\euno}\cap B_{r_j}(p_\infty)$ have genus at most $\kappa(I-1)$ and {intersect both $\partial B_{r_j}(p_\infty){\setminus \Pi(a_j)}$ and $\Pi(a_j)$ transversely in at most $\kappa(I-1)$ components.}

We now prove that, choosing $\delta>0$ possibly smaller, we have that
\[
\abs{A_{\ms_j}}(x) d_{\smetric_j}(x,p_\infty) < {\mu_0}
\]
for all $x \in B_\delta(p_\infty)\setminus B_{r_j}(p_\infty)$.
If this is not the case, then there exists a sequence of points $z_j\in \ms_j$ with $\delta_j\eqdef d_{\smetric_j}(z_j,p_\infty) >r_j$, $\delta_j\to0$ and $\abs{A_{\ms_j}}(z_j)\delta_j\ge {\mu_0}$. The rescaled surfaces $\bu \ms_j \eqdef \delta_j^{-1}(\ms_j-p_\infty)$ still satisfy \hypN{} with blow-up set $\delta_j^{-1}(\set_j-p_\infty)$ and therefore they converge away from $0$ to a lamination consisting of parallel planes by \cref{prop:StructureLimLam}. In particular $\abs{A_{\bu\ms_j}}(\delta_j^{-1}(z_j-p_\infty)) \to 0$, which contradicts the choice of $z_j$.

{Note that, choosing $j$ sufficiently large, we can assume that $\ms_j^{\euno}\cap(\partial B_{r_j}(p_\infty)\setminus \Pi(a_j))$ is $\mu_0$-strongly equatorial by \cref{rem:ConvSigma1}.
As a result, we can invoke \cref{cor:14TopInfo} and conclude that the components $\ms_j^{\euno}\cap B_{\delta}(p_\infty)$ also have genus at most $\kappa(I-1)$ and intersect both $\partial B_{\delta}(p_\infty)\setminus \Pi(a_j)$ and $\Pi(a_j)$ transversely in at most $2\kappa(I-1)$ components.}

Now we follow the very same argument on each ball of radius $\delta$ (as above) and centered at a point of $\set_{\infty}$. We obtain analogous bounds, hence (keeping in mind that we have uniform curvature estimates away from such balls) we can exploit the topological bounds we have gained
to get that $\ms_j^{\euno}$ converges graphically smoothly with finite multiplicity in $B_1(0)\setminus B_\delta(\set_\infty)$ to the leaves passing through $\set_\infty$ of the limit lamination $\lam$ in \cref{prop:StructureLimLam} and conclude with the basic topological tools presented in \cref{sec:EulChar}.

Therefore, now we have to deal only with the case $\abs{\set_\infty} = 1$.
We can assume $\abs{\set_j}\ge 2$, since otherwise we could just apply \cref{lem:LocDegOnePt}. Take $p_j,q_j\in\set_j$ that realize the maximum distance between points in $\set_j$ and define $r_j\eqdef  2 d_{\smetric_j}(p_j,q_j) \mu_0^{-1}\to 0$.
The sequence of surfaces $\check\ms_j\eqdef r_j^{-1}(\ms_j-p_j)$ still satisfy the assumptions \hypN{} with index $I$. Moreover $\abs{\check \set_\infty}\ge 2$, therefore we can apply the first part of the proposition to $\check\ms_j$, obtaining all the desired information in $B_{r_j}(p_j)$. However, now we can argue as above to obtain the ${\mu_0}$-curvature estimate in $\ms_j\cap (B_{1}(0)\setminus B_{r_j}(p_j))$ and transfer the information to $B_1(0)$.
\end{proof}

\subsection{Fine description} \label{sec:GlobalScheme}

We are now ready to put together all the information obtained above and present a fine description of degeneration for a sequence of surfaces with bounded index.

\begin{theorem} \label{thm:GlobalDeg}
Given $I\in \N$ there exists $\kappa(I)\ge 0$ such that the following assertions hold true.

Let $(\amb^3,\smetric)$ be a three-dimensional Riemannian manifold with boundary that satisfies \hypP{}. Let $\ms_j^2\subset \amb$ be a sequence of compact, properly embedded, free boundary minimal surfaces with $\ind(\ms_j)\le I$, for a fixed constant $I\in\N$.
 In the setting of \cref{cor:ExistenceBlowUpSetAndCurvatureEstimate}, one can find a constant $\eps_0>0$ (depending on the sequence in question) so that $\min_{p,q\in\set_\infty}d_\smetric(p,q)\ge 4\eps_0$ and such that, taken $\eps\leq \eps_0$ and defined $\ms_j^{\euno}$ the union of the components of $\ms_j\cap B_{\eps}(\set_\infty)$ which contain at least one point of $\set_j$ and $\ms_j^{\edue}$ the union of the other components of $\ms_j\cap B_{\eps}(\set_\infty)$, for $j$ sufficiently large: 
\begin{enumerate} [label={\normalfont{\arabic*.}}]
\item 
\begin{enumerate} [label={\normalfont{(\alph*)}},ref=(\arabic{enumi}\alph*)]
\item No component of $\ms_j^{\euno}$ is a disc or a half-disc.
\item The genus of $\ms_j^{\euno}$ is bounded by $\kappa(I)$.
\item \label{gd:NumCompSigmaI}
{$\ms_j^{\euno}$ intersects both $\partial B_\eps(\set_\infty)\setminus\partial M$ and $\partial\amb\cap B_\eps(\set_\infty)$ transversely in at most $\kappa(I)$ components.}
\item \label{gd:AreaBoundSigmaI}
$\ms_j^{\euno}$ has uniformly bounded area, namely
\[ 
\limsup_{j\to\infty}\ \area(\ms_j^{\euno}) \le 4\pi \kappa(I) \eps^2\point
\]
\end{enumerate}
\item 
\begin{enumerate} [label={\normalfont{(\alph*)}},ref=(\arabic{enumi}\alph*)]
\item Each component of $\ms_j^{\edue}$ is a disc or a half-disc.
\item $\ms_j^{\edue}$ has uniformly bounded curvature, that is 
\[
\limsup_{j\to\infty}\,\sup_{x\in\ms_j^{\edue}}\ \abs{\A_{\ms_j}}(x) < \infty \point
\]
\item Each component of $\ms_j^{\edue}$ has area uniformly bounded, namely
\[ 
\limsup_{j\to\infty}\, \sup_{\substack{C\subset \ms_j^{\edue}\\ \text{connected}}} \area(C) \le 2\pi \eps^2\point
\]
 \label{gd:AreaBoundSigmaII}
\end{enumerate}
\end{enumerate}
\end{theorem}

\begin{proof}

Let $\eps_0>0$ be such that $\min_{p,q\in\set_\infty}d_\smetric(p,q)\ge 4\eps_0$ and each component of $B_{\eps_0}(\set_\infty)$ admits a chart to the unit ball in $\R^3$ (around the points of $\set_\infty$ in $\amb\setminus\partial\amb$) or in $\Xi(0)$ (around the points of $\set_\infty$ in $\partial\amb$) where the limit lamination (cf. statement of \cref{cor:ExistenceBlowUpSetAndCurvatureEstimate}) $\lam$ is sufficiently close (in the sense of laminations) to a union of parallel planes or half-planes. 
Moreover, choose $r_j\to 0$ such that $\set_j\subset B_{ \mu_0r_j}(\set_\infty)$ and
\[
\liminf_{j\to\infty}\ \left( r_j \min_{p\in\set_j}\ \abs{\A_{\ms_j}}(p) \right) = \infty \point
\]
Observe that, fixing a point $p_\infty\in \set_\infty$, the rescaled surfaces $r_j^{-1}(\ms_j - p_\infty)$ in the ambient manifolds $r_j^{-1}(B_{\eps_0}(p_\infty)-p_\infty)$ (seeing everything in chart) fit in the setting \hypN{}.
Therefore, by \cref{prop:LocDeg}, we obtain that the surfaces $\ms_j^{\euno} \cap B_{r_j}(p_\infty)$ have genus at most $\kappa(I)$ and intersect $\partial\amb$ and $ \partial B_{r_j}(p_\infty)\setminus\partial M$ transversely in at most $\kappa(I)$ components.

\begin{figure}[htpb]
\centering

\begin{tikzpicture}[scale=0.9]
\coordinate (A1) at (0,0); 
\coordinate (B1) at (0.2,0.05); 
\coordinate (C1) at (0.4,0.12);
\coordinate (D1) at (0.6,0.21);
\coordinate (E1) at (0.8,0.35);

\coordinate (F1) at (4,1.97);
\coordinate (G1) at (4.1,2.05);
\coordinate (H1) at (4.3,2.17);
\coordinate (I1) at (4.5,2.24);
\coordinate (J1) at (4.7,2.28);
\coordinate (K1) at (4.85,2.27);

\coordinate (F2) at (5,-4.13);
\coordinate (G2) at (5.1,-4.13);
\coordinate (H2) at (5.3,-4.12);
\coordinate (I2) at (5.5,-4.11);
\coordinate (J2) at (5.7,-4.09);
\coordinate (K2) at (5.8,-4.07);

\coordinate (C2) at (0.7,-3);
\coordinate (D2) at (0.9,-3.05);
\coordinate (E2) at (1.1,-3.1);

\coordinate (X) at (-1.5,-0.2);
\coordinate (Y) at (-1.2, -3.2);
\coordinate (U) at (7.5, 1.8);
\coordinate (V) at (7.9, -2.6);

\draw plot [smooth cycle, tension=0.6] coordinates {(A1) (B1) (C1) (D1) (E1)
			      (1.5,0.9) (2,1.2) (2.5,1.3) (3,1.7) (3.7,1.71)
			      (F1) (G1) (H1) (I1) (J1) (K1)
			      (5.5,2.2) (6,1.8)
			      (U) (8.5,-0.5) (V)
			      (7,-2.9) (6.5, -3.3) (6.2, -3.93) 
			      (K2) (J2) (I2) (H2) (G2) (F2)
			      (4.5,-4) (3.5,-3.2) (2.7, -2.5) (2,-2.7) (1.5, -3.2) 
			      (E2) (D2) (C2)
			      (0.3, -3)
			      (Y) (-1.9,-2) (X) (-0.5,-0.05)};
			      
\draw (F1) to [bend left=20] (F2);
\draw (G1) to [bend left=20] (G2);

\draw plot [smooth, tension = 0.5] coordinates {(H1) (5.2,0.2) (5.45,-1) (5.57, -1.2) (5.65, -1) (5.4,0.2) (I1)};
\draw plot [smooth, tension = 0.5] coordinates {(H2) (5.5,-2.5) (5.5,-1.5) (5.6, -1.3) (5.7, -1.5) (5.7,-2.5) (I2)};
\pgfmathsetmacro{\t}{atan(0.03/0.1)}
\begin{scope}[rotate around={\t:(5.57,-1.2)}]
\draw[gray] (5.57, -1.2) arc (90:270:0.03 and {sqrt(0.03*0.03+0.1*0.1)/2});
\draw (5.57, -1.2) arc (90:-90:0.03 and {sqrt(0.03*0.03+0.1*0.1)/2});
\end{scope}

\draw plot [smooth, tension = 0.5] coordinates {(J1) (5.25,1.2) (5.68,0) (5.9,-1.29) (5.97, -1.45) (6, -1.27) (5.8,0) (5.37,1.2) (K1)};
\draw plot [smooth, tension = 0.5] coordinates {(J2)  (5.9,-2.7) (5.93, -1.7) (5.97, -1.53) (6.03,-1.7) (6.02, -2.7) (K2)};
\draw[gray] (5.97, -1.45) arc (90:270:0.02 and {0.04});
\draw (5.97, -1.45) arc (90:-90:0.02 and {0.04});

\draw[gray] (0.45, -2.9) arc (90:182:0.02 and {0.08});
\draw (0.45, -2.9) arc (90:0:0.02 and {0.08});
\draw plot [smooth, tension = 0.55] coordinates {(A1) (0.35,-1.2) (0.35,-2.65) (0.45, -2.9) (0.55, -2.65) (0.55,-1.2) (B1)};

\pgfmathsetmacro{\t}{atan(0.01/0.1)}
\begin{scope}[rotate around={-\t:(0.87, -2.6)}]
\draw[gray] (0.87,-2.6) arc (90:270:0.03 and {sqrt(0.01*0.01+0.1*0.1)/2});
\draw (0.87,-2.6) arc (90:-90:0.03 and {sqrt(0.01*0.01+0.1*0.1)/2});
\end{scope}
\draw plot [smooth, tension = 0.55] coordinates {(C1) (0.75,-1.2) (0.77,-2.4) (0.87, -2.6) (0.97, -2.4) (0.95,-1.2) (D1)};
\draw plot [smooth, tension = 0.9] coordinates {(C2) (0.75, -2.8) (0.86, -2.7) (0.93, -2.8) (D2)};
\draw (E1) to [bend left = 15] (E2);
\node at (-0.2,-0.6) {$\ms_j$};

\pgfmathsetmacro\R{0.7}
\pgfmathsetmacro\r{0.5}
\fill[gray] (C2) circle [radius=0.1em];
\draw[gray] (C2) circle (\R);

\node[gray] at (1.75, -2.4) {\scriptsize $B_\varepsilon(p_\infty)$};

\fill[gray] (5.685, -1.2) circle [radius=0.1em];
\draw[gray] (5.685, -1.2) circle (\R);
\node[gray] at (4.4, -1) {\scriptsize $B_\varepsilon(p_\infty)$};

\draw[blue] (0.17,-2.5) to [bend left=20] (-2,-3.85);
\draw[blue] (1.1,-3.6) to [bend right=20] (-0.39,-5.57);

\draw[blue] (6,-0.55) to [bend right=20] (10,-1.2);
\draw[blue] (5.7,-1.92) to [bend left=20] (9.5,-3.42);

\begin{scope}[scale=0.8,xshift=-2.5cm,yshift=-6.8cm]

\pgfmathsetmacro{\R}{2}
\pgfmathsetmacro{\r}{1}

\pgfmathsetmacro{\hl}{-0.20*\R}
\pgfmathsetmacro{\xl}{sqrt(\R*\R-\hl*\hl)}
\coordinate (OA) at (0,0);

\pgfmathsetmacro{\t}{atan(-\hl/(\xl))}
\draw[gray] (\xl,\hl) arc(-\t:180+\t:\R);
\draw[gray, dashed] (-\xl,\hl) arc(180+\t:360-\t:\R);

\pgfmathsetmacro{\s}{atan(-\hl/(sqrt(\r*\r-\hl*\hl)))}
\draw[gray] ({(sqrt(\r*\r-\hl*\hl))},\hl) arc(-\s:180+\s:\r);
\draw[gray, dashed] (-{(sqrt(\r*\r-\hl*\hl))},\hl) arc(180+\s:360-\s:\r);

\draw[thick] (-\xl,\hl) -- (\xl,\hl);

\def\c{{-0.3,-0.25, 0.3, -0.6,0.5}}

\foreach \i in {0,...,4}
  \draw ({\R*\c[\i]}, {sqrt(\R*\R*(1-\c[\i]*\c[\i]))}) -- ({\R*\c[\i]}, \hl);

\pgfmathsetmacro{\ca}{-0.2*\R}
\pgfmathsetmacro{\ya}{sqrt(\R*\R-\ca*\ca)}
\pgfmathsetmacro{\cb}{-0.10*\R}
\pgfmathsetmacro{\yb}{sqrt(\R*\R-\cb*\cb)}

\draw[gray] ({(\ca+\cb)/2}, {\hl+0.05*\R}) arc (90:180:0.03 and {0.05*\R});
\draw ({(\ca+\cb)/2}, {\hl+0.05*\R}) arc (90:0:0.03 and {0.05*\R});
\draw plot [smooth, tension=0.4] coordinates {(\ca,\ya) (\ca, {\hl+0.17*\R}) ({(\ca+\cb)/2}, {\hl+0.05*\R}) (\cb,{\hl+0.17*\R}) (\cb,\yb)};

\pgfmathsetmacro{\cc}{0*\R}
\pgfmathsetmacro{\yc}{sqrt(\R*\R-\cc*\cc)}
\pgfmathsetmacro{\cd}{0.16*\R}
\pgfmathsetmacro{\yd}{sqrt(\R*\R-\cd*\cd)}

\draw[gray] ({(\cc+\cd)/2}, {\hl+0.25*\R}) arc (90:270:0.05 and {0.05*\R});
\draw ({(\cc+\cd)/2}, {\hl+0.25*\R}) arc (90:-90:0.05 and {0.05*\R});
\draw plot [smooth, tension=0.5] coordinates {(\cc,\yc) (\cc, {\hl+0.4*\R}) ({(\cc+\cd)/2}, {\hl+0.25*\R}) (\cd,{\hl+0.4*\R}) (\cd,\yd)};
\draw plot [smooth, tension=1.8] coordinates {(\cc,\hl) ({(\cc+\cd)/2}, {\hl+0.15*\R}) (\cd,\hl)};
\end{scope}

\begin{scope}[scale=0.8,yshift=-3.2cm, xshift=13.6cm]
\pgfmathsetmacro{\R}{2}
\pgfmathsetmacro{\r}{1}

\coordinate (OA) at (0,0);
\draw[gray] (OA) circle (\R);
\draw[gray] (OA) circle (\r);

\def\h{{-0.15, 0.37, 0.5, 0.55}}

\foreach \i in {0,...,3}
  \draw ({-sqrt(\R*\R*(1-\h[\i]*\h[\i]))},{\R*\h[\i]}) -- ({sqrt(\R*\R*(1-\h[\i]*\h[\i]))},{\R*\h[\i]});  

\pgfmathsetmacro{\ha}{-0.3*\R}
\pgfmathsetmacro{\xa}{sqrt(\R*\R-\ha*\ha)}
\pgfmathsetmacro{\hb}{-0.20*\R}
\pgfmathsetmacro{\xb}{sqrt(\R*\R-\hb*\hb)}

\draw[gray] (0.15,{(\ha+\hb)/2}) arc (180:0:0.05 and 0.02);
\draw (0.15,{(\ha+\hb)/2}) arc (180:360:0.05 and 0.02);
\draw plot [smooth, tension=0.35] coordinates {(-\xa,\ha) (-0.05,\ha) (0.15,{(\ha+\hb)/2}) (-0.05,\hb) (-\xb,\hb)};
\draw plot [smooth, tension=0.4] coordinates {(\xb,\hb) (0.45,\hb) (0.25,{(\hb+\ha)/2}) (0.45,\ha) (\xa,\ha)};

\pgfmathsetmacro{\hd}{-0.08*\R}
\pgfmathsetmacro{\xd}{sqrt(\R*\R-\hd*\hd)}
\pgfmathsetmacro{\he}{0.10*\R}
\pgfmathsetmacro{\xe}{sqrt(\R*\R-\he*\he)}

\draw[gray] (-0.15,{(\he+\hd)/2}) arc (180:0:0.1 and 0.04);
\draw (-0.15,{(\he+\hd)/2}) arc (180:360:0.1 and 0.04);
\draw plot [smooth, tension=0.4] coordinates {(-\xe,\he) (-0.4,\he) (-0.15,{(\he+\hd)/2}) (-0.4,\hd) (-\xd,\hd)};
\draw plot [smooth, tension=0.4] coordinates {(\xe,\he) (0.3,\he) (0.05, {(\he+\hd)/2}) (0.3,\hd) (\xd,\hd)};

\pgfmathsetmacro{\hf}{0.15*\R}
\pgfmathsetmacro{\xf}{sqrt(\R*\R-\hf*\hf)}
\pgfmathsetmacro{\hg}{0.20*\R}
\pgfmathsetmacro{\xg}{sqrt(\R*\R-\hg*\hg)}

\draw[gray] (0.25,{(\hf+\hg)/2}) arc (180:0:0.05 and 0.02);
\draw (0.25,{(\hf+\hg)/2}) arc (180:360:0.05 and 0.02);
\draw plot [smooth, tension=0.25] coordinates {(-\xg,\hg) (0.1,\hg) (0.25,{(\hg+\hf)/2}) (0.1,\hf) (-\xf,\hf)};
\draw plot [smooth, tension=0.32] coordinates {(\xg,\hg) (0.5,\hg) (0.35,{(\hg+\hf)/2}) (0.5,\hf) (\xf,\hf)};
\end{scope}
\end{tikzpicture}
\caption{Blow-up around the points of degeneration.} \label{fig:GlobalDescr}
\end{figure}
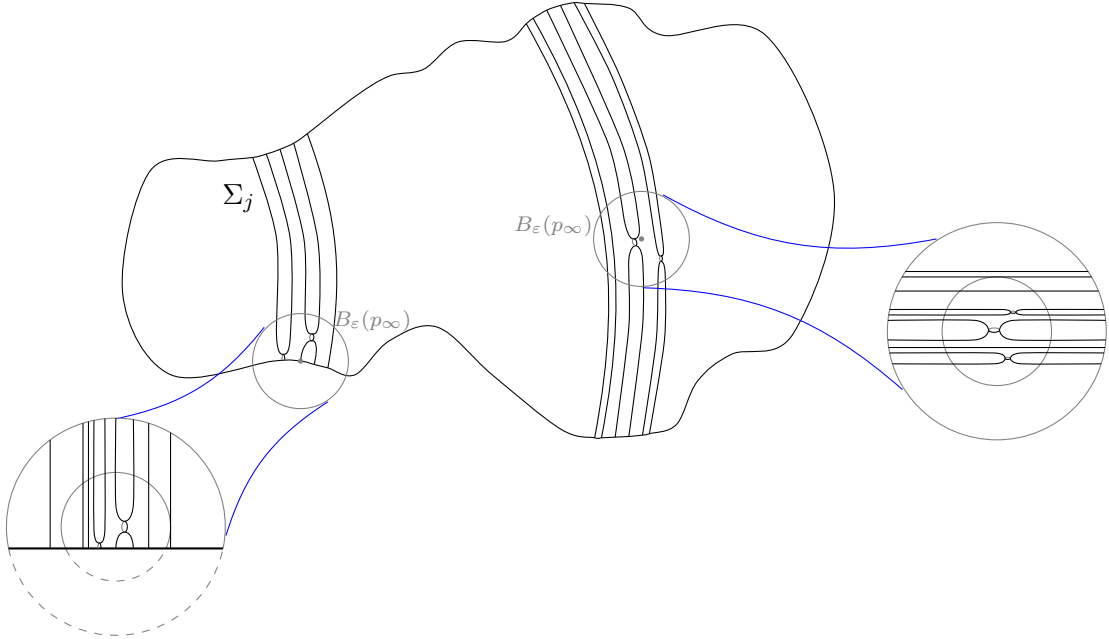
Then note that, exactly as in the proof of \cref{prop:LocDeg}, one can choose $\eps_0>0$, possibly smaller than before, in such a way that 
\[
\abs{\A_{\ms_j}}(x) d_\smetric(x,\set_\infty) < {\mu_0}
\]
for $x\in \ms_j \cap (B_{\eps_0}(\set_\infty) \setminus B_{r_j}(\set_\infty))$, for $j$ sufficiently large.
{Therefore, thanks to \cref{rem:ConvSigma1}, we can apply \cref{cor:14TopInfo} and thus transfer the information on the genus and the boundary components of $\ms_j^{\euno}\cap B_{r_j}(p_\infty)$ to $\ms_j^{\euno}\cap B_{\eps_0}(p_\infty)$ (possibly taking $\kappa(I)$ as double the constant given by \cref{prop:LocDeg}).}

Let us now prove that the components of $\ms_j^{\edue}$ have uniformly bounded curvature, i.e.,
\[
\limsup_{j\to\infty}\sup_{x\in \ms_j^{\edue}} \abs{A_{\ms_j}}(x) < \infty \point
\]
Assume by contradiction that (possibly after passing to a subsequence) there exists a sequence $z_j\in \ms_j^{\edue}$ satisfying
\[
\lambda_j^{\edue}\eqdef \abs{\A_{\ms_j}}(z_j) = \sup_{x\in \ms_j^{\edue}} \abs{\A_{\ms_j}}(x) \to \infty \point
\]
Observe that, by \eqref{eq:CurvEstSj}, this implies that (up to subsequence) the sequence $z_j$ converges to some point $p_\infty\in \set_\infty$. In particular the distance between $z_j$ and $\partial\ms_j^{\edue}\setminus\partial\amb$ is bounded from below by a positive constant, hence $\lambda_j^{\edue}(B_{\eps_0}(z_j)-z_j)$ is an exhausting sequence of domains of $\Pi(a)$ for some $0\ge a \ge -\infty$.

Now consider the rescaled surfaces $\bu\ms_j\eqdef \lambda_j^{\edue}(\ms_j -z_j)$: they have bounded curvature away from a finite set of points and that $\abs{A_{\bu\ms_j}}(0) = 1$. Therefore, by \cref{thm:LamCptness} and \cref{prop:LimLamInR3}, $\bu\ms_j$ converges locally smoothly with multiplicity one (in the sense of graphs) to a complete, nonflat, connected, properly embedded, two-sided, free boundary minimal surface $\bu \ms_\infty$ in $\Xi(a)$, for some $0\ge a \ge-\infty$.

The limit of the surfaces $\lambda_j^{\edue}(\ms_j^{\euno}-z_j)$ must be nonempty since \[\limsup_{j\to\infty} \min_{p\in \set_j} \ \lambda_j^{\edue} d_{\smetric}(z_j,p) < \infty \] by \eqref{eq:CurvEstSj}. 
However, this contradicts the fact that $\lambda_j^{\edue}(\ms_j -z_j)$ converges to $\bu\ms_\infty$ with multiplicity one since $\lambda_j^{\edue}(\ms_j^{\euno} -z_j)$ and $\lambda_j^{\edue}(\ms_j^{\edue} -z_j)$ are two different connected components of $\lambda_j^{\edue}(\ms_j^{\euno}-z_j)$ and the limit of both of them is nonempty.

Given the bound on the curvature of the components of $\ms_j^{\edue}$, the information on their topology follows easily (possibly taking $\eps_0$ smaller), using that the convergence of $\ms_j^{\edue}$ must be smooth (in the sense of lamination) everywhere in $B_{\eps_0}(\set_\infty)$.
Therefore each component of $\ms_j^{\edue}$ converges smoothly with multiplicity one to a leaf of $\lam$ (by simply connectedness of the components of $\lam$ in $B_\eps(p_\infty)$).

There only remains to prove the bounds \ref{gd:AreaBoundSigmaI} and \ref{gd:AreaBoundSigmaII} on the area.
Fixing $\eps\le\eps_0$, the surfaces $\ms_j^{\euno}$ intersects $\partial B_{\eps}(p_\infty)\setminus \partial\amb$ in at most $\kappa(I)$ components that are simple closed curves or arcs. This implies that $\ms_j^{\euno} \cap (B_\eps(p_\infty)\setminus B_{\eps/2}(p_\infty))$ converges graphically smoothly with finite multiplicity bounded by $\kappa(I)$ to the leaf of $\lam$ in $B_\eps(p_\infty)$ passing through $p_\infty$ (similarly to \cref{rem:ConvSigma1}). As a result, choosing $\eps_0$ sufficiently small such that all the leaves of $\lam$ in $B_{\eps}(p_\infty)$ are sufficiently close to discs or half-discs, we have that 
\[
\limsup_{j\to\infty}\ \area(\ms_j^{\euno} \cap (B_{\eps}(\set_\infty)\setminus B_{\eps/2}(\set_\infty))) \le 2\kappa(I)\pi\eps^2\point
\]
Finally the estimate on $\area(\ms_j^{\euno}\cap B_\eps(\set_\infty))$ follows from the monotonicity formula since, for $\eps>0$ sufficiently small, it holds
\[
\area(\ms_j^{\euno}\cap B_\eps(\set_\infty)) \le 2 \area(\ms_j^{\euno}\cap (B_{\eps}(\set_\infty)\setminus B_{\eps/2}(\set_\infty))) \point
\]
The area estimate for the components of $\ms_j^{\edue}$ is even easier since $\ms_j^{\edue}$ converges (in the sense of laminations) to $\lam$ everywhere in $B_\eps(\set_\infty)$ (in particular, each component of $\ms_j^{\edue}$ converges smoothly with multiplicity one to a leaf of $\lam$ in $B_\eps(\set_\infty)$, as observed above), thus we omit the details.
\end{proof}

\section{Surgery procedure}\label{sec:Surgery}

As a corollary of the degeneration description in \cref{thm:GlobalDeg}, we are able to perform surgeries on the surfaces $\ms_j$ to obtain new surfaces `similar' to $\ms_j$ but with bounded curvature.

\begin{corollary} \label{cor:Surgery}

Given $I\in \N$ there exists $\tilde\kappa(I)\ge 0$ such that the following assertions hold true.

In the setting of \cref{thm:GlobalDeg}, and for the same value of $\eps_0$,
for all $0< \eps\le \eps_0$ and $j$ sufficiently large, there exist properly embedded surfaces $\tilde \ms_j\subset \amb$ satisfying the following properties:
\begin{enumerate} [label={\normalfont(\arabic*)}]
\item $\tilde\ms_j$ coincides with $\ms_j$ outside $B_\eps(\set_\infty)$.
\item The curvature of $\tilde\ms_j$ is uniformly bounded, i.e.,
\[
\limsup_{j\to\infty}\sup_{x\in \tilde\ms_j} \abs{\A_{\tilde\ms_j}}(x) < \infty \point
\]
\item The genus, the number of boundary components, the area and the number of connected components of $\tilde\ms_j$ are controlled by the ones of $\ms_j$, namely \label{s:SimChar}
\begin{align*}
&\genus(\ms_j) - \tilde\kappa(I) \le \genus(\tilde\ms_j) \le \genus(\ms_j)\comma \\
&\bdry( \ms_j) - \tilde\kappa(I) \le \bdry(\tilde\ms_j) \le \bdry( \ms_j) \comma \\
&\area( \ms_j) - \tilde\kappa(I) \le \area(\tilde\ms_j) \le \area(\ms_j) + \tilde\kappa(I) \comma \\
&\abs{\pi_0(\ms_j)} \le \abs{\pi_0(\tilde\ms_j)} \le \abs{\pi_0(\ms_j)} + \tilde\kappa(I) \point
\end{align*}
\item The surfaces $\tilde\ms_j$ locally converge (in the sense of laminations) to the lamination $\lam$ in \cref{cor:ExistenceBlowUpSetAndCurvatureEstimate}.

\end{enumerate}

\end{corollary}
\begin{proof}
Consider $p_\infty\in \set_\infty$; then, if $p_\infty\in\amb\setminus\partial\amb$, we can perform the surgery as in \cite[Corollary 1.19]{ChodoshKetoverMaximo2017} (possibly restricting $\eps_0$). Therefore, let us assume that $p_\infty\in \partial\amb$. Pick the leaf $L$ of $ \lam \cap B_\eps(p_\infty)$ passing through $p_\infty$, which satisfies $\partial L = L\cap\partial\amb$ thanks to property \hypP{}.
Then fix a diffeomorphism $\psi\colon B_\eps(p_\infty)\to B_3(0)\cap \Xi(0)\subset \R^3$ such that $\psi$ maps $B_{\eps/3}(p_\infty)$ diffeomorphically onto $B_1(0)\cap \Xi(0)$ and $L$ onto the flat half-disc $\{ x^3 = 0\} \cap (B_3(0)\cap \Xi(0))$.

Consider all the connected components of $\ms_j\cap B_\eps(p_\infty)$ which are converging smoothly to $L$ in $B_\eps(p_\infty)\setminus \bar B_{\eps/3}(p_\infty)$. These include all the neck components, called $\ms_j^{\euno}$ in \cref{thm:GlobalDeg}, as observed in the proof of the theorem.
Note that the convergence to the leaf $L$ in $B_\eps(p_\infty)\setminus \bar B_{\eps/3}(p_\infty)$ might possibly occur with infinite multiplicity; however, the convergence of the components relative to $\ms_j^{\euno}$ occurs with uniformly bounded multiplicity (for example thanks to the area bound \ref{gd:AreaBoundSigmaI} in \cref{thm:GlobalDeg}, or from the proof of \ref{gd:AreaBoundSigmaI} itself).

Let $\Gamma_j$ be the image through $\psi$ of the union of the necks components of $\ms_j\cap B_\eps(p_\infty)$ together with the disc components directly above or below a neck component.
We are going to perform our surgery on $\Gamma_j$ leaving invariant its disc components, in this way we are sure that also all the other disc components remain untouched.
For the sake of convenience, let us identify $L$ with its image in $B_3(0)\cap \Xi(0)$ and let us denote by $D(2)$ the disc with radius $2$ and center $0$ in $L$, that is $D(2)\eqdef L\cap \{\abs x < 2\}$, and $A(2,1)$ the annulus between radii $1$ and $2$ in $L$, namely $A(2,1)\eqdef L\cap \{1<\abs x < 2\}$. 
Then choose $\chi\colon \R\to \cc 01$ a smooth, nondecreasing cutoff function with $\chi(t) = 0$ for $t\le 5/4$ and $\chi(t) = 1$ for $t\ge 7/4$.

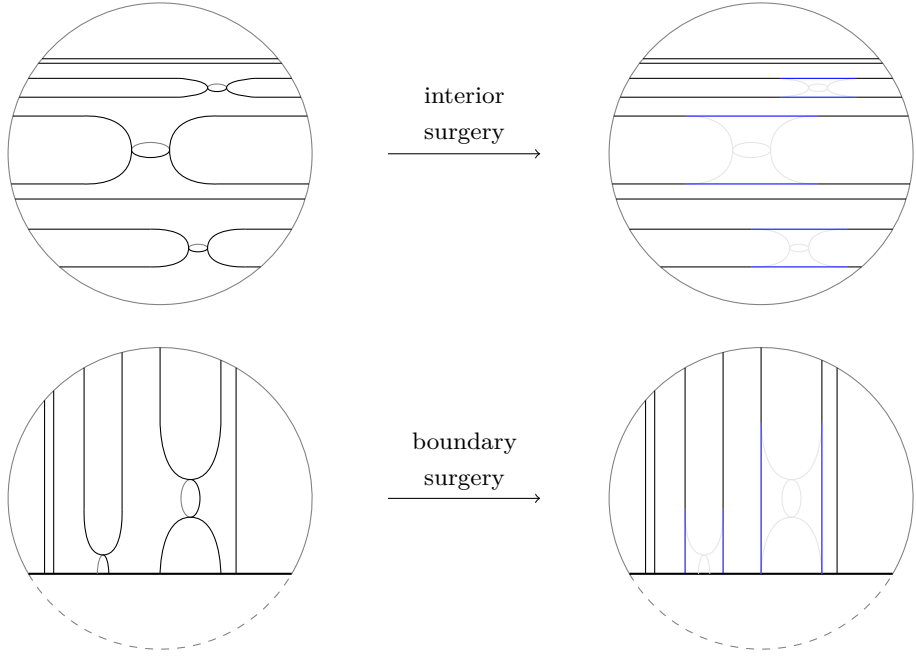
\begin{figure}[htpb]
\centering
\begin{tikzpicture}[scale=2.5]
\pgfmathsetmacro{\R}{2}
\pgfmathsetmacro{\r}{0.8}


\coordinate (OA) at (0,0);

\draw[gray] (OA) circle (\r);
\def\h{{-0.3, 0.6, 0.63}}

\foreach \i in {0,...,2}
  \draw ({-sqrt(\r*\r*(1-\h[\i]*\h[\i]))},{\r*\h[\i]}) -- ({sqrt(\r*\r*(1-\h[\i]*\h[\i]))},{\r*\h[\i]});

\pgfmathsetmacro{\ha}{-0.3*\R}
\pgfmathsetmacro{\xa}{sqrt(\r*\r-\ha*\ha)}
\pgfmathsetmacro{\hb}{-0.20*\R}
\pgfmathsetmacro{\xb}{sqrt(\r*\r-\hb*\hb)}

\begin{scope}
\draw[gray]  (0.15,{(\ha+\hb)/2}) arc (180:0:0.05 and 0.02);
\draw  (0.15,{(\ha+\hb)/2}) arc (180:360:0.05 and 0.02);
\draw plot [smooth, tension=2] coordinates {(-0.05,\ha) (0.15,{(\ha+\hb)/2}) (-0.05,\hb)};
\draw plot [smooth, tension=2] coordinates {(0.45,\hb) (0.25,{(\hb+\ha)/2}) (0.45,\ha)};

\draw (-\xa,\ha) -- (-0.05,\ha);
\draw (-0.05,\hb) -- (-\xb,\hb);
\draw (\xa,\ha) -- (0.45,\ha);
\draw (0.45,\hb) -- (\xb,\hb);

\pgfmathsetmacro{\hd}{-0.08*\R}
\pgfmathsetmacro{\xd}{sqrt(\r*\r-\hd*\hd)}
\pgfmathsetmacro{\he}{0.10*\R}
\pgfmathsetmacro{\xe}{sqrt(\r*\r-\he*\he)}
\draw[gray] (-0.15,{(\he+\hd)/2}) arc (180:0:0.1 and 0.04);
\draw  (-0.15,{(\he+\hd)/2}) arc (180:360:0.1 and 0.04);
\draw  plot [smooth, tension=2] coordinates {(-0.4,\he) (-0.15,{(\he+\hd)/2}) (-0.4,\hd)};
\draw  plot [smooth, tension=2] coordinates {(0.3,\he) (0.05, {(\he+\hd)/2}) (0.3,\hd)};
\draw (-\xe,\he) -- (-0.4,\he);
\draw (-0.4,\hd) -- (-\xd,\hd);
\draw (\xe,\he) -- (0.3,\he);
\draw (0.3,\hd) -- (\xd,\hd);

\pgfmathsetmacro{\hf}{0.15*\R}
\pgfmathsetmacro{\xf}{sqrt(\r*\r-\hf*\hf)}
\pgfmathsetmacro{\hg}{0.20*\R}
\pgfmathsetmacro{\xg}{sqrt(\r*\r-\hg*\hg)}
\draw[gray]  (0.25,{(\hf+\hg)/2}) arc (180:0:0.05 and 0.02);
\draw  (0.25,{(\hf+\hg)/2}) arc (180:360:0.05 and 0.02);
\draw plot [smooth, tension=1.5] coordinates {(0.1,\hg) (0.25,{(\hg+\hf)/2}) (0.1,\hf)};
\draw plot [smooth, tension=1.5] coordinates {(0.5,\hg) (0.35,{(\hg+\hf)/2}) (0.5,\hf)};
\draw (-\xg,\hg) -- (0.1,\hg);
\draw (\xg,\hg) -- (0.5,\hg);
\draw (0.1,\hf) -- (-\xf,\hf);
\draw (0.5,\hf) -- (\xf,\hf);
\end{scope}

\draw[->] (1.2,0) -- (2,0);
\node at (1.6,0.2) [text width=1.5cm,align=center]{\footnotesize interior surgery};

\begin{scope}[xshift=45*\R]

\coordinate (OA) at (0,0);
\draw[gray] (OA) circle (\r);
\def\h{{-0.3, 0.6, 0.63}}

\foreach \i in {0,...,2}
  \draw ({-sqrt(\r*\r*(1-\h[\i]*\h[\i]))},{\r*\h[\i]}) -- ({sqrt(\r*\r*(1-\h[\i]*\h[\i]))},{\r*\h[\i]});

\pgfmathsetmacro{\ha}{-0.3*\R}
\pgfmathsetmacro{\xa}{sqrt(\r*\r-\ha*\ha)}
\pgfmathsetmacro{\hb}{-0.20*\R}
\pgfmathsetmacro{\xb}{sqrt(\r*\r-\hb*\hb)}

\begin{scope}
\draw[gray!20]  (0.15,{(\ha+\hb)/2}) arc (180:0:0.05 and 0.02);
\draw[gray!20]  (0.15,{(\ha+\hb)/2}) arc (180:360:0.05 and 0.02);
\draw[gray!20] plot [smooth, tension=2] coordinates {(-0.05,\ha) (0.15,{(\ha+\hb)/2}) (-0.05,\hb)};
\draw[gray!20] plot [smooth, tension=2] coordinates {(0.45,\hb) (0.25,{(\hb+\ha)/2}) (0.45,\ha)};

\draw (-\xa,\ha) -- (-0.05,\ha);
\draw (-0.05,\hb) -- (-\xb,\hb);
\draw (\xa,\ha) -- (0.45,\ha);
\draw (0.45,\hb) -- (\xb,\hb);

\pgfmathsetmacro{\hd}{-0.08*\R}
\pgfmathsetmacro{\xd}{sqrt(\r*\r-\hd*\hd)}
\pgfmathsetmacro{\he}{0.10*\R}
\pgfmathsetmacro{\xe}{sqrt(\r*\r-\he*\he)}

\draw[gray!20] (-0.15,{(\he+\hd)/2}) arc (180:0:0.1 and 0.04);
\draw[gray!20]  (-0.15,{(\he+\hd)/2}) arc (180:360:0.1 and 0.04);
\draw[gray!20]  plot [smooth, tension=2] coordinates {(-0.4,\he) (-0.15,{(\he+\hd)/2}) (-0.4,\hd)};
\draw[gray!20]  plot [smooth, tension=2] coordinates {(0.3,\he) (0.05, {(\he+\hd)/2}) (0.3,\hd)};

\draw (-\xe,\he) -- (-0.4,\he);
\draw (-0.4,\hd) -- (-\xd,\hd);
\draw (\xe,\he) -- (0.3,\he);
\draw (0.3,\hd) -- (\xd,\hd);

\pgfmathsetmacro{\hf}{0.15*\R}
\pgfmathsetmacro{\xf}{sqrt(\r*\r-\hf*\hf)}
\pgfmathsetmacro{\hg}{0.20*\R}
\pgfmathsetmacro{\xg}{sqrt(\r*\r-\hg*\hg)}

\draw[gray!20]  (0.25,{(\hf+\hg)/2}) arc (180:0:0.05 and 0.02);
\draw[gray!20]  (0.25,{(\hf+\hg)/2}) arc (180:360:0.05 and 0.02);
\draw[gray!20] plot [smooth, tension=1.5] coordinates {(0.1,\hg) (0.25,{(\hg+\hf)/2}) (0.1,\hf)};
\draw[gray!20] plot [smooth, tension=1.5] coordinates {(0.5,\hg) (0.35,{(\hg+\hf)/2}) (0.5,\hf)};

\draw (-\xg,\hg) -- (0.1,\hg);
\draw (\xg,\hg) -- (0.5,\hg);
\draw (0.1,\hf) -- (-\xf,\hf);
\draw (0.5,\hf) -- (\xf,\hf);

\begin{scope}[solid,blue]
\draw[blue] (-0.05,\ha) -- (0.45,\ha);
\draw[blue] (-0.05,\hb) -- (0.45,\hb);
\draw[blue] (-0.4,\he) -- (0.3,\he);
\draw[blue] (-0.4,\hd) -- (0.3,\hd);
\draw[blue] (0.1,\hg) -- (0.5,\hg);
\draw[blue] (0.1,\hf) -- (0.5,\hf);
\end{scope}
\end{scope}
\end{scope}

\begin{scope}[yshift=-26*\R]
\pgfmathsetmacro{\hl}{-0.20*\R}
\pgfmathsetmacro{\xl}{sqrt(\r*\r-\hl*\hl)}
\coordinate (OA) at (0,0);
\pgfmathsetmacro{\s}{atan(-\hl/(sqrt(\r*\r-\hl*\hl)))}
\draw[gray] ({(sqrt(\r*\r-\hl*\hl))},\hl) arc(-\s:180+\s:\r);
\draw[gray, dashed] (-{(sqrt(\r*\r-\hl*\hl))},\hl) arc(180+\s:360-\s:\r);
\draw[thick] (-\xl,\hl) -- (\xl,\hl);

\def\c{{-0.76,-0.7, 0.5}}
\foreach \i in {0,...,2}
\draw ({\r*\c[\i]}, {sqrt(\r*\r*(1-\c[\i]*\c[\i]))}) -- ({\r*\c[\i]}, \hl);

\pgfmathsetmacro{\ca}{-0.2*\R}
\pgfmathsetmacro{\ya}{sqrt(\r*\r-\ca*\ca)}
\pgfmathsetmacro{\cb}{-0.10*\R}
\pgfmathsetmacro{\yb}{sqrt(\r*\r-\cb*\cb)}

\draw[gray] ({(\ca+\cb)/2}, {\hl+0.05*\R}) arc (90:180:0.03 and {0.05*\R});
\draw ({(\ca+\cb)/2}, {\hl+0.05*\R}) arc (90:0:0.03 and {0.05*\R});
\draw plot [smooth, tension=2] coordinates {(\ca, {\hl+0.17*\R}) ({(\ca+\cb)/2}, {\hl+0.05*\R}) (\cb,{\hl+0.17*\R})};

\draw (\ca,\ya) -- (\ca, {\hl+0.17*\R});
\draw (\cb,{\hl+0.17*\R}) -- (\cb,\yb);

\pgfmathsetmacro{\cc}{0*\R}
\pgfmathsetmacro{\yc}{sqrt(\r*\r-\cc*\cc)}
\pgfmathsetmacro{\cd}{0.16*\R}
\pgfmathsetmacro{\yd}{sqrt(\r*\r-\cd*\cd)}

\draw[gray] ({(\cc+\cd)/2}, {\hl+0.25*\R}) arc (90:270:0.05 and {0.05*\R});
\draw ({(\cc+\cd)/2}, {\hl+0.25*\R}) arc (90:-90:0.05 and {0.05*\R});
\draw plot [smooth, tension=1.8] coordinates {(\cc, {\hl+0.4*\R}) ({(\cc+\cd)/2}, {\hl+0.25*\R}) (\cd,{\hl+0.4*\R})};
\draw plot [smooth, tension=1.8] coordinates {(\cc,\hl) ({(\cc+\cd)/2}, {\hl+0.15*\R}) (\cd,\hl)};

\draw (\cc,\yc) -- (\cc, {\hl+0.4*\R});
\draw (\cd,{\hl+0.4*\R}) -- (\cd,\yd);

\draw[->] (1.2,0) -- (2,0);
\node at (1.6,0.2) [text width=1.5cm,align=center]{\footnotesize boundary surgery};

\begin{scope} [xshift=45*\R]

\pgfmathsetmacro{\hl}{-0.20*\R}
\pgfmathsetmacro{\xl}{sqrt(\r*\r-\hl*\hl)}
\coordinate (OA) at (0,0);

\pgfmathsetmacro{\s}{atan(-\hl/(sqrt(\r*\r-\hl*\hl)))}
\draw[gray] ({(sqrt(\r*\r-\hl*\hl))},\hl) arc(-\s:180+\s:\r);
\draw[gray, dashed] (-{(sqrt(\r*\r-\hl*\hl))},\hl) arc(180+\s:360-\s:\r);
\draw[thick] (-\xl,\hl) -- (\xl,\hl);

\def\c{{-0.76,-0.7, 0.5}}

\foreach \i in {0,...,2}
\draw ({\r*\c[\i]}, {sqrt(\r*\r*(1-\c[\i]*\c[\i]))}) -- ({\r*\c[\i]}, \hl);

\pgfmathsetmacro{\ca}{-0.2*\R}
\pgfmathsetmacro{\ya}{sqrt(\r*\r-\ca*\ca)}
\pgfmathsetmacro{\cb}{-0.10*\R}
\pgfmathsetmacro{\yb}{sqrt(\r*\r-\cb*\cb)}

\draw[gray!20] ({(\ca+\cb)/2}, {\hl+0.05*\R}) arc (90:180:0.03 and {0.05*\R});
\draw[gray!20] ({(\ca+\cb)/2}, {\hl+0.05*\R}) arc (90:0:0.03 and {0.05*\R});
\draw[gray!20] plot [smooth, tension=2] coordinates {(\ca, {\hl+0.17*\R}) ({(\ca+\cb)/2}, {\hl+0.05*\R}) (\cb,{\hl+0.17*\R})};

\draw[blue] (\ca, {\hl+0.17*\R}) -- (\ca,\hl);
\draw[blue] (\cb, {\hl+0.17*\R}) -- (\cb,\hl);
\draw (\ca,\ya) -- (\ca, {\hl+0.17*\R});
\draw (\cb,{\hl+0.17*\R}) -- (\cb,\yb);

\pgfmathsetmacro{\cc}{0*\R}
\pgfmathsetmacro{\yc}{sqrt(\r*\r-\cc*\cc)}
\pgfmathsetmacro{\cd}{0.16*\R}
\pgfmathsetmacro{\yd}{sqrt(\r*\r-\cd*\cd)}

\draw[gray!20] ({(\cc+\cd)/2}, {\hl+0.25*\R}) arc (90:270:0.05 and {0.05*\R});
\draw[gray!20] ({(\cc+\cd)/2}, {\hl+0.25*\R}) arc (90:-90:0.05 and {0.05*\R});
\draw[gray!20] plot [smooth, tension=1.8] coordinates {(\cc, {\hl+0.4*\R}) ({(\cc+\cd)/2}, {\hl+0.25*\R}) (\cd,{\hl+0.4*\R})};
\draw[gray!20] plot [smooth, tension=1.8] coordinates {(\cc,\hl) ({(\cc+\cd)/2}, {\hl+0.15*\R}) (\cd,\hl)};

\draw[blue] (\cc, {\hl+0.4*\R}) -- (\cc,{\hl});
\draw[blue] (\cd, {\hl+0.4*\R}) -- (\cd,{\hl});

\draw (\cc,\yc) -- (\cc, {\hl+0.4*\R});
\draw (\cd,{\hl+0.4*\R}) -- (\cd,\yd);
\end{scope}
\end{scope}
\end{tikzpicture}
\caption{Surgery at small scales.} \label{fig:SurgeryMicro}
\end{figure}

We can assume to have only two disc components in $\Gamma_j$, one at the top and one at the bottom. If this is not the case, we work separately on subsets of the components of $\Gamma_j$ in this form, eventually adding the extremal disc components if missing.
Pick a smooth function $w_j\colon D(2)\to \R$ such that 
\begin{itemize}
 \item the graph of $w_j$ is contained in $B_3(0)\cap \Xi(0)$ and it lies strictly between the two disc components of $\Gamma_j$;
 \item the graph of $w_j$ intersects transversely $\Pi(0)$;
 \item $w_j$ smoothly converges to $0$ as $j\to\infty$.
\end{itemize}
Moreover choose real numbers $\delta_j\to 0$ such that $w_j+\delta$ has the same properties for every $0\le \delta \le\delta_j$.

Now, since the convergence of $\Gamma_j$ to $L$ is smooth graphical with finite multiplicity on $A(2,1)$, we can find functions $u_{j,1},\ldots,u_{j,n(j)}\colon A(2,1)\to \R$ so that the non-disc components of $\Gamma_j$ on the cylinder over $A(2,1)$ are exactly the graphs of $u_{j,k}$ for $k=1,\ldots,n(j)$.
Moreover observe that
\begin{itemize}
\item $n(j)$ is uniformly bounded for $j\to\infty$ for what we said before;
\item for all $h\in\N$ we have $\sup_{k=1,\ldots,n(j)} \norm{u_{j,k}}_{C^h(A(2,1))}\to 0$ as $j\to\infty$;
\item by embeddedness of $\Gamma_j$ we can assume $u_{j,1}(x) < u_{j,2}(x) < \ldots < u_{j,n(j)}(x)$ for all $x\in A(2,1)$.
\end{itemize}
We then interpolate, i.e., we define
\[
\tilde u_{j,k}(x) \eqdef \chi(\abs{x}) u_{j,k}(x) + (1-\chi(\abs x)) \left(w_j(x) + \frac{k}{n(j)} \delta_j\right)\comma 
\]
for all $x\in A(2,1)$, and the surface $\tilde \Gamma_j$ as the union of the graphs of $\tilde u_{j,k}$ for $k=1,\ldots,n(j)$ inside the cylinder over $D(2)$ and coinciding with $\Gamma_j$ outside that cylinder.
Now just define $\tilde\ms_j$ as $\ms_j$ outside $B_\eps(\set_\infty)$ and as the preimages of $\tilde\Gamma_j$ through $\psi$ inside the balls $B_\eps(\set_\infty)$.
All the properties required to $\tilde\ms_j$ are easily fulfilled by construction thanks to the bounds on the genus and the number of boundary components of $\ms^{\euno}_j$ inside $B_\eps(\set_\infty)$, proven in \cref{thm:GlobalDeg}.
\end{proof}


\chapter{`Weakly positive geometry' setting}


In this chapter, we focus on three-dimensional ambient manifolds with `weakly positive geometry', namely manifolds with nonnegative scalar curvature and mean convex boundary.

\section[Diameter bounds for stable FBMS]{Diameter bounds for stable free boundary minimal surfaces} \label{sec:CptnessStable}

In this section we prove \cref{prop:StableImpliesCptness}, which is key to derive the area bounds claimed in the statement of \cref{thm:AreaBound}. The main idea of the proof is that, under the assumptions of the proposition, the stable minimal surface $\ms$  admits a (complete) conformal metric with nonnegative curvature and convex boundary. Moreover, at least one of the two inequalities is strict: either the curvature is strictly positive or the boundary is strictly convex, hence we expect the size of $\Sigma$ to be bounded as a consequence of this property.
In order to implement this heuristic idea, we employ several different techniques, mainly from \cite{FischerColbrie1985} and \cite{White1986} and, of course, from the same result in the closed case \cite{SchoenYau1983}, in the form proposed in \cite[Proposition 2.12]{Carlotto2015}.

\begin{proof} [Proof of \cref{prop:StableImpliesCptness}]

Suppose by contradiction that $\ms$ is noncompact\footnote{It is straightforward to note that the same proof goes through, to provide the bound for $\text{diam}(\ms)$ in the case when $\Sigma$ is a compact, properly embedded free boundary minimal surface: indeed in that case we know that $\ms$ is a (two-sided) disc, and it suffices to take $\omega$ the first eigenfunction of the Jacobi operator $\jac_{\ms}$ (subject to the usual oblique boundary condition).}. Then by standard arguments as in \cite{FischerColbrieSchoen1980} (cf. also \cite[Section 2.2.2]{Volkmann2015}), there exists a positive function $\omega$ on $\ms$ such that
\begin{equation} \label{eq:EllProb}
\begin{cases} 
\jac_\ms (\omega) = \lapl \omega + (\frac 12 \Scal_\smetric + \frac 12 \abs \A^2 - K)\omega = 0 & \text{on $\ms$}\\
\frac{\partial \omega}{\partial \eta} = -\II^{\partial\amb}(\nu,\nu) \omega &\text{on $\partial\ms$} \point
\end{cases}
\end{equation}
Recall that we can choose $\omega$ strictly positive in all $\ms$ (including on $\partial\ms$) by the strong maximum principle and the Hopf boundary point lemma.

Consider on $\ms$ the conformal change of metric $\tilde \smetric = \omega^2\smetric$ where $\smetric$ (which we also denote by $\scal\cdot\cdot$) is the metric on $\ms$ that is induced by the ambient metric on $\amb$.
Then we know (see e.g. \cite[pp. 126-127]{FischerColbrieSchoen1980}) that the curvature of $(\ms,\tilde \smetric)$ is given by
\begin{equation}\label{eq:ChangeSectional}
\begin{split}
\tilde K &= \omega^{-2}(K-\lapl \log \omega) = \omega^{-2}\left( K - \frac{\omega\lapl\omega - \abs{\grad \omega}^2 }{\omega^2} \right) \\
&=\omega^{-2}\left(\frac 12 \Scal_\smetric + \frac 12 \abs A^2\right) + \frac{\abs{\grad \omega}^2}{\omega^4} \ge \omega^{-2}\frac{\Scal_\smetric}2 \ge\omega^{-2}\frac{\varrho_0}2\ge 0\point
\end{split}
\end{equation}
Now let $\tau$ be a local choice of unit vector tangent to $\partial\ms$ in $(\ms, \smetric)$ and, as usual, let $\eta$ be the outward unit normal to $\partial\ms$.
After the conformal change of coordinates, $\tau/\omega$ and $\eta/\omega$ are an orthonormal basis. Hence we can compute the geodesic curvature of $\partial\ms$ in $(\ms,\tilde \smetric)$ as follows
\begin{equation}\label{eq:ChangeGeodCurv}
\begin{split}
\tilde k &= -\tilde \smetric({\tilde\cov_{\tau/\omega}\tau/\omega},{\eta/\omega}) = -\omega^2 \scal{\tilde\cov_{\tau/\omega}\tau/\omega}{\eta/\omega}
= -\omega^{-1}\scal{\tilde \cov_\tau \tau}\eta \\
&=-\omega^{-1}\left(\scal{\cov_\tau\tau}{\eta} - \omega^{-1} \DerParz{\omega}\eta\right) = -\omega^{-1}(\II^{\partial\amb}(\tau,\tau) + \II^{\partial\amb}(\nu,\nu)) \\ &= \omega^{-1}H^{\partial\amb}\ge \omega^{-1}\sigma_0\ge 0\point
\end{split}
\end{equation}
Therefore $(\ms,\tilde \smetric)$ is a surface with nonnegative Gaussian curvature and convex boundary and, as claimed above, either of the two functions $\tilde{K}$ or $\tilde k$ is strictly positive.

\begin{lemma}
The surface $(\ms,\tilde \smetric)$ with boundary is complete (as a metric space).
\end{lemma}
\begin{proof}
The proof is similar to that of Theorem 1 in \cite{FischerColbrie1985} with some precautions to be taken when dealing with the boundary.
Consider a point $x_0\in \ms\setminus \partial\ms$ and let $B_R(x_0)\subset \ms$ be the intrinsic ball of center $x_0$ and radius $R>0$ with respect to the complete metric $\smetric$. Then $\{B_R(x_0)\}_{R>0}$ is an exhaustion of $\ms$.

Now consider the shortest geodesic $\beta_R$ in the metric $\tilde \smetric$ connecting $x_0$ to the closure of $\partial B_R(x_0)\setminus\partial\ms$, which exists by the following argument. Let us consider $\omega_R = \omega + \epsilon_R$ where $0\le\epsilon_R\le 1$ is a smooth function with $\epsilon_R = 0$ in $B_R(x_0)$ and $\epsilon_R = 1$ in $B_{R+1}(x_0)^c$. Then $\omega_R$ is bounded from below and thus the metric $\tilde \smetric_R \eqdef \omega_R^2 \smetric$ is complete on $\ms$. 
Based on the fact, and on the convexity of the boundary (that we just saw above), there exists a length-minimizing geodesic connecting $x_0$ to the closure of $\partial B_R(x_0)\setminus \partial\ms$ with respect to the metric $\tilde \smetric_R$. Since this geodesic is obviously contained in $B_R(x_0)$ it is also length-minimizing with respect to $\tilde \smetric$, since $\tilde \smetric$ coincides with $\tilde \smetric_R$ in $B_R(x_0)$.

Let us assume $\beta_R$ to be parametrized by arc length with respect to $\smetric$. By standard compactness arguments, we can find a sequence $R_i\to\infty$ such that $\beta_{R_i}$ locally smoothly converge to a curve $\beta\colon  \co 0{\infty} \to \ms$ such that $\beta(0)=x_0$, minimizing length with respect to $\tilde \smetric$ between any two of its points, and that is also parametrized by arc length with respect to $\smetric$.
Observe that $\beta$ cannot touch the boundary $\partial\ms$ since it starts at an interior point and $\partial\ms$ is convex.

To prove the completeness of $(\ms,\tilde \smetric)$ it is now sufficient to prove that $\beta$ has infinite length with respect to $\tilde \smetric$, that is to say
\[
\int_0^{\infty}\omega(\beta(t))\de t = \infty \semicolon
\]
indeed, by construction, any divergent ray starting from $x_0$ has length equal or bigger than $\beta$.
However, to prove this, we can apply exactly the same argument as in the first part of the proof of \cite[Theorem 1]{FischerColbrie1985}, since $\beta$ is actually contained in $\ms\setminus\partial\ms$. \end{proof}

Given the completeness of $(\ms,\tilde \smetric)$, we can apply the Gauss-Bonnet theorem on its metric balls as follows.
Fix $x_0\in \ms$ and let consider $\Omega_r \eqdef \tilde B_r(x_0)$, the metric ball of radius $r$ and center $x_0$ in metric $\tilde \smetric$.

\begin{remark}
We are now going to apply some results on geodesic balls of a Riemannian manifold taken from \cite[Section 4.4]{ShiohamaShioyaTanaka2003}, which are stated for complete manifolds without boundary.
The following observation clarifies why the same results hold in our case.

Given $r_0>0$, we can regard our manifold as a smooth subdomain of a complete manifold $(\check\ms,\check \smetric)$ without boundary (the extension depending on $r_0>0$), such that, thanks to the convexity of $\partial \Sigma$, for $r$ close to $r_0$ the set
$\tilde B_r(x_0)\subset\ms$ is the intersection of the metric ball $\check B_r(x_0)\subset \check \ms$ with respect to the metric $\check \smetric$ with $\ms$.
\end{remark}

Thanks to this remark, we can invoke a classical theorem by Hartman \cite{Hartman1964} (cf. also Theorem 4.4.1 in \cite{ShiohamaShioyaTanaka2003}) and obtain that the boundary of $\Omega_r$ is a piecewise smooth embedded closed curve for almost every $r>0$. 
Moreover the length $l(r)$ of $\partial\Omega_r$ is differentiable almost everywhere with derivative given by
\[
l'(r) = \int_{\partial\Omega_r\setminus\partial\ms} (\text{geodesic curvature of $\partial\Omega_r\setminus\partial\ms$}) + \sum (\text{exterior angles of $\Omega_r$});
\]
note that \[\limsup_{r\to\infty}\ l'(r)\ge 0\comma\] since $l$ only attains positive values.

Let us now pick a radius $r$ for which $\partial\Omega_r$ is a piecewise smooth embedded closed curve. Then, by the Gauss-Bonnet theorem on $\Omega_r$, it holds that
\begin{multline*}
l'(r) = \int_{\partial\Omega_r\setminus \partial\ms} (\text{geodesic curvature of $\partial\Omega_r\setminus \partial\ms$})+ \sum (\text{exterior angles of $\Omega_r$}) =\\ = 2\pi \chi (\Omega_r) - \int_{\Omega_r} \tilde K \de\tilde \Haus^2 - \int_{\partial\Omega_r\cap \partial\ms} \tilde k \de\tilde\Haus^1\comma
\end{multline*}
where $\tilde\Haus^1$ and $\tilde \Haus^2$ are the Hausdorff measures with respect to $\tilde \smetric$ in $\ms$.
Hence, taking the upper limit on both sides and using that $\tilde K, \tilde k\ge 0$, we can conclude that
\[
0\le 2\pi\limsup_{r\to\infty}\chi(\Omega_r) - \int_\ms \tilde K\de\tilde\Haus^2 - \int_{\partial\ms} \tilde k\de\tilde\Haus^1 \point
\]
Recalling that $\chi(\Omega_r)\le 1$, this implies that both the integrals $\int_\ms \tilde K\de\tilde\Haus^2$ and $\int_{\partial\ms}\tilde k\de\tilde\Haus^1$ are finite.

Now observe that $\d\tilde\Haus^1 = \omega\d\Haus^1$ and $\d\tilde\Haus^2 = \omega^2\d\Haus^2$ by definition of $\tilde \smetric$, where $\d\Haus^1$ and $\d\Haus^2$ are the one-dimensional and two-dimensional Hausdorff measures on $(\ms, \smetric)$.
Hence, applying \eqref{eq:ChangeSectional} and \eqref{eq:ChangeGeodCurv}, we obtain that
\[
\int_\ms\tilde K \de\tilde\Haus^2 \ge \int_\ms \frac{\varrho_0}{2\omega^2}\de\tilde\Haus^2 = \int_\ms \frac{\varrho_0}{2\omega^2}\omega^2\de\Haus^2 = \frac{\varrho_0}2 \Haus^2(\ms)
\]
and that
\[
\int_{\partial\ms} \tilde k\de\tilde\Haus^1  = \int_{\partial\ms} \frac{\sigma_0}\omega\de\tilde\Haus^1 = \int_{\partial\ms} \frac{\sigma_0}\omega \omega\de\Haus^1 = \sigma_0 \Haus^1(\partial\ms) \point
\]
In particular we deduce\footnote{Observe that this inequality in the compact case follows directly from the stability inequality applied to a constant function.} that
\begin{equation} \label{eq:GBFinal}
0< \frac{\varrho_0}2 \Haus^2(\ms) + \sigma_0 \Haus^1(\partial\ms) \le 2\pi\limsup_{r\to\infty}\chi(\Omega_r) \le 2\pi\point
\end{equation}
Note that this also proves that $\ms$ has well-defined Euler characteristic equal to $\chi(\ms) = \limsup_{r\to\infty}\chi(\Omega_r)=1$.

In particular we have obtained that $\partial\ms$ is compact if $\sigma_0>0$. 
However, since we might only have $\sigma_0=0$ and, in any case, a priori we are not in the position to invoke the isoperimetric inequality in \cite{White2009}, we actually need the following lemma.

\begin{lemma} \label{lem:DistBdryEst}
Let $\ms^2 \subset \amb^3$ be as in \cref{prop:StableImpliesCptness} and let $x_0\in \ms$. Then $x_0$ has distance from the boundary of $\ms$ bounded by a constant depending only on $\varrho_0 = \inf_\amb R_\smetric$ and $\sigma_0 = \inf_{\partial\amb} H^{\partial\amb}$. Namely it holds
\[
d_\ms(x_0,\partial\ms) \le \min\left\{ \frac{2\sqrt 2 \pi}{\sqrt{3\varrho_0}}, \frac{4}{3\sigma_0} \right\} \point
\]
\end{lemma}
\begin{proof}
Let us consider the functional
\[
\tilde l(\beta) \eqdef \int_\beta\omega
\]
on the class of $W^{1,2}$-curves $\beta$ lying in $\ms$ and connecting $x_0$ to a point in $\partial\ms$. Note that here with $\int_\beta \omega$ we mean the integral with respect to the arc length (in metric $\smetric$) of $\beta$ and thus $\tilde l(\beta)$ is instead the length of the curve $\beta$ with respect to $\tilde \smetric$.
In particular let $\beta$ be a curve minimizing this functional. Observe that this curve is smooth and has finite length since $\omega$ is strictly positive on the closure of $\ms$.

We now compute the first and second variation of the functional $\tilde l$ along $\beta$. Without loss of generality, we can assume that $\beta$ is parametrized by arc length, that is $\abs{\beta'} = 1$, and it has length $l$.
Thus let us choose a variation $\alpha\colon  \oo {-\eps}{\eps}\times \cc 0l  \to \ms$ with $\alpha(0,\cdot) = \beta$, $\alpha(s,0) =x_0$ and $\alpha(s,1)\in\partial\ms$.
Computing the first variation we obtain\footnote{We always omit the dependence from $s$ and $t$ if clear from the context.}
\begin{align*}
\frac{\d}{\d s} \tilde l(\alpha(s, \cdot)) &= \frac{\d}{\d s} \int_0^l \omega(\alpha(s,t))\left|{\frac{\partial \alpha}{\partial t}(s,t)}\right| \de t\\
&= \int_0^l \d\omega(\alpha)\left[ \DerParz{\alpha}{s} \right]\left|{\DerParz{\alpha}{t}}\right| + \omega(\alpha) \frac{\scal{\cov_{\DerParz \alpha s} \DerParz\alpha t}{\DerParz\alpha t}}{\left|{\DerParz{\alpha}{t}}\right|} \de t\\
&= \int_0^l \Bigg( \d\omega(\alpha)\left[ \DerParz{\alpha}{s} \right]\left|{\DerParz{\alpha}{t}}\right| 
-\d\omega(\alpha)\left[ \DerParz\alpha t \right] \frac{\scal{ \DerParz\alpha s}{\DerParz\alpha t}}{\left|{\DerParz{\alpha}{t}}\right|}  +{}\\
&\pheq{}-\omega(\alpha) \frac{\scal{\DerParz\alpha s}{\cov_{\DerParz\alpha t}\DerParz\alpha t}}{\left|{\DerParz{\alpha}{t}}\right|} +\omega(\alpha) \frac{\scal{\DerParz\alpha s}{\DerParz\alpha t}  \scal{\cov_{\DerParz\alpha t}\DerParz\alpha t}{\DerParz\alpha t}}{\left|{\DerParz{\alpha}{t}}\right|^3} \Bigg) \de t  +{}\\
&\pheq {}+  \omega(\alpha(s,l)) \frac{\scal{ \DerParz\alpha s(s,l)}{\DerParz\alpha t(s,l)}}{\left|{\DerParz{\alpha}{t}}(s,l)\right|} \point
\end{align*}
In particular, evaluating at $s=0$ and setting $X(t)=\DerParz\alpha s(t,0)$, we have that
\[
\begin{split}
0 = \frac{\d}{\d s}\Big|_{s=0} \tilde l(\alpha(s,\cdot)) &= \int_0^l \Big(\d\omega(\beta)[X] - \d\omega(\beta)[\beta'] \scal{\beta'}{X} -\omega(\beta)\Big\langle{\cov_{\beta'}\beta'},{X}\Big\rangle +{}\\
&\pheq +\omega(\beta) \Big\langle{\cov_{\beta'}\beta'},{\beta'}\Big\rangle\scal{\beta'}{X} \Big)\de t + \omega(\beta(l)) \scal{\beta'(l)}{X(l)}\\
&= \int_0^l \Big(\d\omega(\beta)[X] - \d\omega(\beta)[\beta'] \scal{\beta'}{X} -\omega(\beta)\Big\langle{\cov_{\beta'}\beta'},{X}\Big\rangle\Big)\de t  +{}\\
&\pheq{} +\omega(\beta(l)) \scal{\beta'(l)}{X(l)}\comma
\end{split}
\]
which holds for all variations $\alpha$ as above. Note that we have used that $0 = \frac{\d}{\d t}\abs{\beta'}^2 = 2\scal{\cov_{\beta'}\beta'}{\beta'}$.

As a result, since $X(l)$ is tangent to $\partial\ms$, we have that
\[
\begin{cases}
\grad\omega(\beta) - \d\omega(\beta)[\beta'] \beta' -\omega(\beta)\cov_{\beta'}\beta' = 0 \\
\beta'(l)\perp \partial\ms \point
\end{cases}
\]

We can then compute the second variation for $s=0$, obtaining
\[
\begin{split}
\frac{\d^2}{\d s^2}\Big|_{s=0} &\tilde l(\alpha(s,\cdot))=\\ 
&=\int_0^l \Big\langle \cov_{\DerParz\alpha s}\Big|_{s=0} \bigg( 
\grad\omega(\alpha) \left|{\DerParz{\alpha}{t}}\right|^4 - \d\omega(\alpha)\left[{\DerParz{\alpha}{t}}\right] \left|{\DerParz{\alpha}{t}}\right|^2 {\DerParz{\alpha}{t}} +{}\\
&\pheq - \omega(\alpha) \left|{\DerParz{\alpha}{t}}\right|^2 \cov_{\DerParz\alpha t}\DerParz\alpha t + \omega(\alpha) \Big\langle{\cov_{\DerParz\alpha t}\DerParz\alpha t },{\DerParz\alpha t}\Big\rangle \DerParz\alpha t
\bigg),
X\Big\rangle \de t+{}\\
&\pheq + \omega(\beta(l)) \left( \Big\langle{\cov_{\DerParz\alpha s}\Big|_{s=0}\DerParz\alpha s (s,l)},{\beta'(l)}\Big\rangle + \Big\langle{X(l)},{\cov_{\DerParz\alpha s}\Big|_{s=0} \DerParz\alpha t(s,l)} \Big\rangle\right) \point
\end{split}
\]
Now assume that $X(t)$ is of the form $\psi(t)\tau$ where $\tau$ is a unit vector field orthogonal to $\beta'$. Then we have that
\[
\begin{split}
\frac{\d^2}{\d s^2}\Big|_{s=0} \tilde l(\alpha(s,\cdot)) &= \int_0^l (\psi^2\lapl \omega - \psi^2\omega'' - \psi\psi'\omega' - \psi\psi''\omega - \psi^2K\omega)\de t +{}\\
&\pheq{}+\omega(l)(\psi\psi' + \psi^2\II^{\partial\amb}(\tau,\tau))\\
&= \int_0^l \left(\psi^2\jac_\ms\omega - \frac 12\psi^2(R_\smetric + \abs \A^2) \omega - \psi^2\omega'' - \psi\psi'\omega' - \psi\psi''\omega\right)\de t  +{}\\
&\pheq{}+\omega(l)(\psi\psi'(l) + \psi^2(l)\II^{\partial\amb}(\tau,\tau))
\point
\end{split}
\]

Let $h\colon \cc 0l\to \R$ be the first (positive) eigenfunction for the eigenvalue problem
\[
\begin{cases}
\psi'' + \omega^{-1}\omega'\psi' + (\frac 12 R_\smetric + \frac 12 \abs \A^2- \omega^{-1}\jac_\ms\omega + \omega^{-1}\omega'') \psi + \lambda\psi = 0\\
\psi(0) = 0\\
\psi'(l) = -\II^{\partial\amb}(\tau,\tau)\psi(l) \point
\end{cases}
\]
Then, since $\frac{\d^2}{\d s^2}|_{s=0} \tilde{l}(\alpha(s,\cdot))\ge 0$ for all variations $\alpha$, we have in particular that
\begin{multline*}
 h^{-1}h'' + \omega^{-1}\omega'h^{-1}h' + \frac 12 \varrho_0 +\omega^{-1}\omega''\le \\
 \le h^{-1}h'' + \omega^{-1}\omega'h^{-1}h' + \frac 12 R_\smetric + \frac 12 \abs \A^2- \omega^{-1}\jac_\ms\omega + \omega^{-1}\omega''=-\lambda \le 0.
\end{multline*}
Therefore, multiplying this inequality by a test function $\xi\in C^\infty(\cc 0l)$ with $\xi(0)=0$ and integrating by parts, we obtain
\begingroup
\allowdisplaybreaks
\begin{align*}
0&\ge \int_0^l \left(h^{-1}h'' + \omega^{-1}\omega'h^{-1}h' + \frac 12 \varrho_0 + \omega^{-1}\omega'' \right)\xi^2 \de t\\
& = \int_0^l  \left( h^{-2}(h')^2+ \omega^{-2}(\omega')^2+ \omega^{-1}\omega'h^{-1}h'  \right)\xi^2 + \frac 12\varrho_0\xi^2 - 2 (h^{-1}h' + \omega^{-1}\omega')\xi\xi' \de t+ {}\\
&\pheq + (h^{-1}(l)h'(l)+\omega^{-1}(l)\omega'(l)) \xi^2(l) \\ 
& = \int_0^l  \frac 12\left( h^{-2}(h')^2+ \omega^{-2}(\omega')^2 \right)\xi^2+ \frac 12\varrho_0 \xi^2 + \frac 12\left(\frac{\d}{\d t}(\log(\omega h))\right)^2\xi^2 \de t+{}\\
&\pheq- 2\int_0^l \frac{\d}{\d t}(\log(\omega h))\xi\xi' \de t- (\II^{\partial\amb}(\tau,\tau)+\II^{\partial\amb}(\nu,\nu)) \xi^2(l) \\
& \ge \int_0^l  \frac 12\left( h^{-2}(h')^2+ \omega^{-2}(\omega')^2 \right)\xi^2+ \frac 12\varrho_0 \xi^2 + \frac 12\left(\frac{\d}{\d t}(\log(\omega h))\right)^2\xi^2 \de t+{} \\
&\pheq- 2\int_0^l \frac{\d}{\d t}(\log(\omega h))\xi\xi' \de t+ \sigma_0 \xi^2(l)\comma
\end{align*}
\endgroup
where we have used the boundary assumptions on $h$ and $\omega$, noting that $\omega'(l) = \scal{\grad\omega}{\beta'(l)} = \DerParz\omega\eta$, since $\beta'(l) = \eta(\beta(l))$ ($\beta'(l)$ is orthogonal to $\partial\ms$ and outward-pointing).
Furthermore, recall that $-\II^{\partial\amb}(\tau,\tau)-\II^{\partial\amb}(\nu,\nu) = H^{\partial\amb} \ge\sigma_0$.

We can then rely on the inequality
\[
2\left| \frac{\d}{\d t}(\log(\omega h))\xi\xi' \right| \le \frac 12\left( h^{-2}(h')^2+ \omega^{-2}(\omega')^2 \right)\xi^2+ \frac 12\left(\frac{\d}{\d t}(\log(\omega h))\right)^2\xi^2 + \frac 43(\xi')^2
\]
to conclude that
\[
\frac 12 \varrho_0\int_0^l \xi^2 \de t +  \sigma_0 \xi^2(l) \le \frac 43 \int_0^l (\xi')^2 \de t
\]
for all $\xi$ as above.
This proves that 
\[
l\le \delta_0 \eqdef \min \left\{\frac{2\sqrt 2\pi}{\sqrt{3\varrho_0}}, \frac 4{3\sigma_0}\right\} < \infty \comma
\]
as in \cite[Proposition 2.12]{Carlotto2015} for the first term and for example by taking $\xi(t) = t$ for the second term.
As a result we have proven that $\ms$ is contained in the $\delta_0$-neighborhood of $\partial\ms$ with respect to the intrinsic distance.
\end{proof}

 If $\varrho_0>0$ pick two points $x_0,y_0\in\ms\setminus\partial\ms$ and consider the curve $\beta$ minimizing $\tilde l(\beta)$ and connecting $x_0$ to $y_0$. Since $\beta$ is a minimizing curve in the metric $\tilde \smetric$ of $\ms$, then it cannot touch the convex boundary $\partial\ms$. Therefore we can follow the very same argument as in \cite[Proposition 2.12]{Carlotto2015} to prove that the length of $\beta$ in the metric $\smetric$ is bounded by $\frac{2\sqrt 2 \pi}{\sqrt{3\varrho_0}}$.

If $\sigma_0>0$ we conclude by observing that \cref{lem:DistBdryEst}, together with the compactness of $\partial\ms$, proves the compactness of $\ms$. The diameter estimate follows by simply combining equation \eqref{eq:GBFinal}, with \cref{lem:DistBdryEst}.
\end{proof}


\section{Ambient manifolds with `weakly positive geometry'} \label{sec:AreaBound}
We can now capitalize our efforts and present the proof of \cref{thm:AreaBound}, which crucially exploits \cref{prop:StableImpliesCptness}.

\begin{proof}[Proof of \cref{thm:AreaBound}]
	We only need to prove the area bound, since all other conclusions then follow from \cref{thm:TopFromArea}.
	Assume by contradiction that there exists a sequence of connected, compact, properly embedded, free boundary minimal surfaces $\Sigma_j^2\subset\amb$ with nonempty boundary and with $\ind(\ms_j)\le I$ and $\area(\ms_j)\to \infty$. Then, there exists a point $x_0\in\amb$ such that $\Haus^2(\ms_j\cap B_r(x_0))\to\infty$ for all $r>0$.
	
	Denote by $\lam$ the limit lamination given by \cref{cor:ExistenceBlowUpSetAndCurvatureEstimate} and consider the leaf $L\in\lam$ passing through $x_0$. Thanks to \cref{thm:RemSingLimLam}, $L$ must have stable universal cover (the other case is excluded since the area is diverging around $x_0$), which we can represent by a stable free boundary minimal immersion $\varphi\colon  \ms\to M$.
	By means of a variation of the pull-back construction presented e.g. in \cite[Section 6]{AmbrozioCarlottoSharp2018Compactness}, we can then reduce to applying \cref{prop:StableImpliesCptness} so to conclude that $\ms$ must be a disc, hence the map $\varphi$ must be an embedding (therefore $L$ is a stable, free boundary minimal surface in $(M,\smetric)$).
	
	Let then $\tilde\ms_j$ be the sequence obtained from $\ms_j$ by means of the surgery procedure, as per
	\cref{cor:Surgery}. Observe that the new sequence satisfies uniform curvature bounds, and still has diverging area. Also, since \cref{cor:Surgery} provides a uniform bound $\tilde{\kappa}(I)+1$ on the number of connected components of $\tilde\ms_j$, we can select and rename $\tilde\ms_j$ so that it is connected for every $j\in\mathbb{N}$, and the area concentrates near the point $x_0$.
	
	Now, fix $r>0$ sufficiently small, and assume to consider a connected component of the intersection $\tilde\ms_j\cap B_r(x_0)$ which smoothly converges with multiplicity one to $L\cap B_r(x_0)$. Since $L$ is a disc, a standard monodromy argument allows to conclude that, a posteriori, (the whole component) $\tilde\ms_j$ converges to $L$ smoothly with multiplicity one. From this fact, we derive a uniform bound for the areas of $\tilde\ms_j$, which is a contradiction.
\end{proof}


\chapter[Noncompact families of FBMS of fixed topology] {Noncompact families of free \\boundary minimal surfaces of fixed topology}\label{sec:Counterexample}


\cref{thm:Counterexample} is proven via a suitable, rather explicit, gluing construction aimed at attaching some elementary blocks. In all cases we shall now list, the word \emph{block} refers to an ambient manifold together with a sequence of minimal surfaces (closed or having free boundary) satisfying additional requirements.

\section{The building blocks}

\subsection{Spiraling spheres}\label{subs:BlocksSphere}

\begin{lemma}[cf. {\cite[Proposition 8]{ColdingDeLellis2005}}] \label{lem:intblock}
On $S^3$ there exists a Riemannian metric $\smetric_0$ of positive scalar curvature such that $(S^3, \smetric_0)$ contains a sequence of minimal spheres $\Sigma_j$ with arbitrarily large area and index, and converging to a singular lamination $\lam$ whose singular set consists of exactly two points lying on a leaf which is a strictly stable (two-dimensional) sphere. Furthermore, the metric in question coincides with the unit round metric in a neighborhood of two distinct points $x_0$ and $y_0$, and on both those neighborhoods such minimal spheres can be completed to a local foliation by great spheres parallel at $x_0$ and $y_0$ respectively.
\end{lemma} 

\begin{remark}
An important aspect is to clarify what we mean when writing that a local foliation is \emph{parallel at a point} (cf. \cite[Definition 7]{ColdingDeLellis2005}). Given $z\in\Omega$, an open subset of the round three-sphere, and $\mathscr{F}$, a local foliation of $\Omega$ by great spheres, we say that the foliation is parallel at $z\in F$ (for some $ F\in\mathscr{F}$, to be called the \emph{central leaf}) if 
\[
\sup_{w\in F} d(w, F')=d(z, F')
\]
for any $ F'\in\mathscr{F}$. 

The geometric picture this definition captures is easily described. Take $S^3\subset \R^4$ isometrically embedded as unit sphere and consider the foliation of $S^3\setminus \left\{ (0,0,0,\pm 1)\right\}$ consisting of the two-spheres obtained by slicing via vertical hyperplanes. Given any point $x=(x^1,x^2,x^3,0)$ and any open set $\Omega\ni x$ then the restriction of the above foliation to $\Omega$ is parallel at $x$, and the unique great sphere passing through $x$ is the central leaf.
\end{remark}

\begin{remark}\label{rmk:half}
For later use (cf. \cref{rmk:bmodel} and \cref{subs:BlocksAnnulus}), it is helpful to introduce some related terminology. We consider the set
$\Omega^+\eqdef \left\{x \in \Omega \st  x^4\geq 0 \right\}
$
and the corresponding foliation $\mathscr{F}^+$ obtained by considering the intersection of each leaf of $\mathscr{F}$ with the domain $\Omega^+$.
We will say that $\mathscr{F}^+$ is a foliation of $\Omega^+$ by half great spheres, parallel at $x$.
\end{remark}

\begin{remark}\label{rmk:index}
It follows from the construction (see, specifically, the second paragraph of \cite[pp. 30]{ColdingDeLellis2005}) that for any open set $\Omega$ containing the limit sphere, one has that $\ind(\Sigma_j\cap\Omega)\to\infty$ as $j\to\infty$, where it is understood that one only considers variations that are compactly supported in $\Omega$. Of course, it is also true that $\area({\Sigma_j\cap\Omega})\to\infty$ as $j\to\infty$.
\end{remark}

\subsection{Minimal tori}\label{subs:BlocksTorus}

We first provide the relevant statement and then mention the key points in the construction, to the extent this is needed in \cref{subs:BlocksAnnulus} below to produce, for any given $\notb>1$, a Riemannian metric of positive scalar curvature on the three-ball so that the resulting three-manifold contains a family of free boundary surfaces of genus $0$ and exactly $\notb$ boundary components.

\begin{lemma} [cf. {\cite[Lemma 12]{ColdingDeLellis2005}}]
\label{lem:tori} 
On $S^3$ there exists a Riemannian metric $\smetric_{1}$ of positive scalar curvature such that:
\begin{enumerate} [label={\normalfont(\arabic*)}]
\item {$(S^3, \smetric_{1})$ contains a family of minimal tori $\Pi_\theta$ parametrized by $\theta\in (-\theta_0,\theta_0)$, for some $\theta_0>0$;}
\item {the metric in question coincides with the unit round metric in a neighborhood of given points $x_1$ and $y_1$ and on both those neighborhoods such minimal tori provide a local foliation by great spheres parallel at $x_1$ and $y_1$, respectively.}	
\end{enumerate}	
\end{lemma}

The construction can be schematically described as follows:
\begin{itemize}
\item{On the (topological) product manifold $\cc{-\pi/2}{\pi/2}\times S^1\times S^1$, one can consider the equivalence relation $\sim$ given by
\[
(-\pi/2,p,q) \sim (-\pi/2,p,q') \quad \forall \ p \in S^1, \ \forall \ q,q'\in S^1\comma
\]
and
\[
(\pi/2,p,q)\sim (\pi/2, p', q) \quad \forall \ p, p' \in S^1,\  \forall \ q\in S^1\comma
\]
that corresponds to `collapsing vertical (resp. horizontal) fibers on $\left\{-\pi/2\right\}\times S^1\times S^1$ (resp. on $\left\{\pi/2\right\}\times S^1\times S^1$)'. Set $\tilde{M}=M/\sim $, this manifold can be endowed with a smooth Riemannian metric $\tilde{\smetric}$ so that $(\tilde{M},\tilde{\smetric})$ is a three-sphere of positive scalar curvature, and in a neighborhood of $\left\{0\right\}\times S^1\times S^1$ the metric is isometric to the Riemannian product of $S^2\times S^1$. Hence $(\tilde{M},\tilde{\smetric})$ contains a one-parameter family of \emph{totally geodesic} tori all having two circles in common, say $\left\{0\right\}\times \left\{0\right\}\times S^1$ and $\left\{0\right\}\times \left\{\pi\right\}\times S^1$ (where we are conveniently identifying the round unit $S^1$ with the interval $\cc{0}{2\pi}$ with endpoints attached).
}
\item{At this stage one can perform a \emph{local} modification of the metric $\tilde{\smetric}$ near the points $(0,\pi/2,0)$ and $(0,3\pi/2,0)$ so to make it round; with the family of tori being locally isometric to a family of standard great spheres in such neighborhoods. The construction is performed by explicitly interpolating between the metric of round $S^3$ and the product metric of $S^2\times S^1$.}	
\end{itemize}	

\begin{remark}\label{rmk:divide}
We observe that:
\begin{itemize}
\item {In $(S^3,\smetric_1)$ the surface 
\[
\tilde{\Sigma}\eqdef (\cc{-\pi/2}{\pi/2}\times\left\{\pi/2\right\}\times S^1 \cup \cc{-\pi/2}{\pi/2}\times\left\{3\pi/2\right\}\times S^1) /\sim
\]  
is a totally geodesic two-sphere, which divides the closed manifold into two (pairwise isometric) three-dimensional balls.}
\item {Denoted by $\Omega$ one of such balls, for any $\theta\in(-\theta_0,\theta_0)$ the intersection $\Xi_\theta\eqdef \Pi_\theta \cap \overline{\Omega}$ is checked to be a free boundary minimal annulus.}
\end{itemize}

Hence, the Riemannian manifold $(\overline{\Omega}, \smetric_1)$ contains a family of minimal annuli $\Xi_{\theta}$ parametrized by $\theta\in (-\theta_0,\theta_0)$, for some $\theta_0>0$; the metric $\smetric_1$ coincides with the unit round metric in a neighborhood of boundary points $x_1$ and $y_1$ (these points belonging to the two boundary circles of $\Xi_{0}$) and on both those neighborhoods $\partial\Omega$ is isometric to an equatorial two-sphere and the minimal annuli provide \emph{half of} a local foliation by great spheres parallel at $x_1$ and $y_1$, respectively.
\end{remark}

\subsection{Free boundary minimal discs}\label{subs:BlocksDisk}

Let us now, instead, move to the free boundary models. We prove this ancillary result.

\begin{lemma}\label{lem:fbpiece}For any $\eps\in (0,\pi/4)$ there exists a smooth Riemannian metric $\smetric_2=\smetric_2(\eps)$ on the closed ball $\overline{B^3}= (\cc{0}{1+\pi/2}\times S^2)/\sim$ with coordinates $(r,\omega)$ (where $\sim$ is the equivalence relation collapsing $\left\{1+\pi/2\right\}\times S^2$ to a point) having positive scalar curvature, and such that the following properties are satisfied:
\begin{enumerate} [label={\normalfont(\arabic*)}]
\item {$\smetric_2$ coincides with the unit round metric of $S^3$ on the domain $(\cc{1+\eps}{1+\pi/2}\times S^2)/\sim$, and coincides with the cylindrical metric on the domain $([0,1/3]\times S^2)/\sim$;}	
\item {the resulting manifold $(\overline{B^3}, \smetric_2)$ contains a one-parameter family $\Delta^{\euno}_\theta$ of embedded free boundary minimal discs (here $\theta\in (-\theta_0,\theta_0)$ for some $\theta_0>0$);}
\item {there exist two points $x^{\euno}_2, y^{\euno}_2$ with $r(x^{\euno}_2)= r(y^{\euno}_2)=1/2$ and $\omega(x^{\euno}_2)=-\omega(y^{\euno}_2)$, and open neighborhoods $\Omega(x^{\euno}_2)$, $\Omega(y^{\euno}_2)$ respectively, where $\smetric_2$ is round (isometric to domains of the unit round metric of $S^3$) and the minimal discs above restrict to give two local foliations by great spheres that are parallel at $x^{\euno}_2, y^{\euno}_2$ respectively.}
\end{enumerate}
\end{lemma}

\begin{proof}
	
Consider the cylinder $I\times S^2$ (where $I=[0,1]$ and the sphere $S^2$ is endowed with the standard unit round metric), fix two antipodal points $p, q\in S^2$ and consider the family of circles passing through $(0,p)$ and $(0,q)$. Said $\Gamma$ any of those circles, then $I\times \Gamma$ is a smooth, free boundary minimal surface $\Delta$ in $I\times S^2$ (with two boundary components). Now, cap off $I\times S^2$ by identifying its upper boundary with the boundary of a hemisphere in $S^3$ (with its unit round metric). The resulting metric $\hat{\smetric}$ is smooth away from the interface (and has, near it, the form of a warped product with the warping factor being a $C^{1,1}$ function of the distance coordinate from the interface). Furthermore, one can attach a two-dimensional half-hemisphere to $\Delta$ so to get a free boundary minimal surface, which we shall not rename, that is not yet smooth along the connecting circle but has a mild singularity there. 
Now, since (in the resulting three-manifold) the mean curvature of the interface on both sides match (for the interface is totally geodesic on both sides), we can apply the smoothing theorem by P. Miao \cite{Miao2002} in the simplest possible case (see in particular Theorem 1 therein) to get a smooth metric $\check{\smetric}$ on the closed three-dimensional ball, that coincides with $\hat{\smetric}$ away from a small neighborhood of the gluing interface, and whose scalar curvature is positive. In fact, by the very way the construction is defined, namely by fiberwise convolution, it follows that the regularized metric $\check{\smetric}$ takes the form of a warped product, i.e., we have (in the coordinates $(r,\omega)$ introduced in the statement)
\[
\check{\smetric}=\d r^2+f^2(r)\smetric_{S^2}
\]
where 
\[
f(r)=
\begin{cases}
1 & \text{if $r\in \cc{0}{1-\eps}$} \\
\sin(r-1+\pi/2) & \text{if $r\in \cc{1+\eps}{1+\pi/2}$}
\end{cases}
\]
for small $\eps>0$. The \emph{set} $\Delta$ is then a totally geodesic surface in this smooth Riemannian manifold, and since the metric has not been modified near the boundary component which has not been capped off one has that the free boundary condition is still fulfilled.

At this stage, let us consider the construction above starting from a one parameter family of circles $\Gamma_{\theta}$ (for $\theta$ varying in a subset of $S^1$, say $\theta\in (-\theta_0,\theta_0)$ for some $\theta_0>0$) passing through the points $(0,p), (0,q) \in \left\{0\right\} \times S^2$ and let $p', q' \in S^2$ be two antipodal points (on the great circle equidistant from $p$ and $q$) chosen so that $(0,p'), (0,q')$ have small neighborhoods that are foliated by such a family. Let $\Delta^{\euno}_{\theta}$ be the corresponding free boundary minimal surfaces.

Since the cylindrical metric has, so far, not been modified away from a small neighborhood of the gluing interface we can consider the points $x^{\euno}_2=(1/2,p')$, $y^{\euno}_2=(1/2,q')$, and open neighborhoods thereof (named $\Omega(x^{\euno}_2)$, $\Omega(y^{\euno}_2)$ respectively) that are foliated by the surfaces $\Delta^{\euno}_{\theta}$ as $\theta$ varies. Hence, we perform a local deformation of the metric in each of these neighborhoods to make it round (for which it is enough to follow, without any modification, the argument given in the second part of Appendix B of \cite{ColdingDeLellis2005}). Possibly renaming those open neighborhoods (to be taken slightly smaller than we had originally defined), the metric $\smetric_2=\smetric_2(\eps)$, resulting from such local modifications, satisfies all desired properties.
\end{proof}

\begin{remark}\label{rmk:bmodel}
Said $(M^{\euno}_2, \smetric^{\euno}_2)$ the manifold constructed in \cref{lem:fbpiece}, we will also need the following variant: the local modifications of the metric are performed at \emph{one interior point} (as above) and at \emph{one boundary point}. Thereby, one obtains a three-manifold $(M^{\euno}_{2,\text{bdry}}, \smetric^{\euno}_{2,\text{bdry}})$ of positive scalar curvature which contains a family of free boundary minimal discs, still denoted by ${\Delta^{\euno}_{\theta}}$, and having two points:
\begin{itemize}
\item $x^{\euno}_{2,\text{bdry}}$ (in the interior) having a full neighborhood where such minimal discs provide a local foliation by great spheres (parallel at $x^{\euno}_{2,\text{bdry}}$) and 
\item $y^{\euno}_{2,\text{bdry}}$ (on the boundary) having a half neighborhood where such minimal discs provide  a local foliation by half great spheres (parallel at $y^{\euno}_{2,\text{bdry}}$).
\end{itemize}
With respect to the notation employed in the proof above, one can take (for instance)
\[
x^{\euno}_{2,\text{bdry}}=(1/2,p')\comma \quad  y^{\euno}_{2,\text{bdry}}=(0,q') \point
\]
\end{remark}

\subsection{Free boundary minimal \texorpdfstring{$k$}{k}-annuli}\label{subs:BlocksAnnulus}

In order to prove \cref{thm:Counterexample}, we also need to construct metrics on the closed three-ball having positive scalar curvature and containing families of free boundary minimal surfaces of genus zero and any pre-assigned number $\notb$ of boundary components. So far, this has only been accomplished for $\notb=1$ (in \cref{subs:BlocksDisk}) and for $\notb=2$, as a result of \cref{rmk:divide}. To proceed we need the following \emph{free boundary analogue} of Lemma 11 in \cite{ColdingDeLellis2005}.

\begin{lemma}\label{lem:join}
Let $\delta>0$ and let $\Omega^+_1\eqdef  B_{\delta}(u_1)$, $\Omega^+_2\eqdef  B_{\delta}(u_2)$ be two (relatively) open half-balls of round unit hemispheres $N^+_1, N^+_2$, centered at boundary points $u_1, u_2$ respectively. Suppose that $\mathscr{F}^+_1$ and $\mathscr{F}^+_2$ are (locally defined) foliations of $\Omega_1$ and, respectively, $\Omega_2$ by half great spheres parallel at $u_1$ and, respectively, $u_2$. Then we can join those hemispheres to obtain a smooth Riemannian manifold with boundary (having the topology of $\overline{B^3}$) of positive scalar curvature, and (possibly by considering smaller neighborhoods) the leaves of $\mathscr{F}^+_1$ and $\mathscr{F}^+_2$ can pairwise be matched to obtain free boundary minimal surfaces in the ambient manifold.
\end{lemma}

\begin{proof}
Let us double each of the two given manifolds with boundary to round spheres $N_1$ and $N_2$ and consider the families $\mathscr{F}_1, \mathscr{F}_2$ obtained by extending $\mathscr{F}^+_1$ and $\mathscr{F}^+_2$ in the obvious fashion (each leaf of $\mathscr{F}^+_i$ is extended to an equatorial two-sphere in $N_i$, for $i=1,2$). Since our construction is purely local, this does not affect the generality of the argument.

We know (by Lemma 11 in \cite{ColdingDeLellis2005}) that one can construct a connected sum $N_1\# N_2$ and pairwise match the leaves of the foliations $\mathscr{F}_1$ and $\mathscr{F}_2$. The connecting neck $N$, diffeomorphic to $\cc{-\tau}{\tau}\times S^2$, can be described by means of spherical coordinates
\[ 
(r,\phi,\theta) \in \cc{-\tau}{\tau}\times \cc 0\pi \times S^1
\]
and each minimal surface that we produce is, when restricted to the neck, the lift of a \emph{graphical} curve $\sigma\colon \cc{-\tau}{\tau}\to \cc{0}{\pi}$ solving a suitable ODE (that is nothing but a geodesic equation in a degenerate metric). In other words, each such minimal surface takes (in the neck) the form
\[
\Sigma\eqdef \left\{(r,\sigma(r),\theta) \st r\in \cc{-\tau}\tau \comma \theta\in S^1 \right\} \point
\] 
That being said, these coordinates can be chosen so that the condition $\theta\in S^1_+$ (for $S^1_+\subset S^1$ a half-circle that is fixed now and for all) identifies the half-sphere $\partial\Omega^+_1$, and the same conclusion holds true for $\partial\Omega^+_2$ as well. Now, defining 
$
N^+\eqdef \left\{(r,\phi,\theta) \in \cc{-\tau}{\tau}\times \cc 0\pi \times S^1_{+} \right\}
$,
the totally geodesic two-sphere $(\partial N^+_1 \setminus\Omega^+_1 ) \cup \partial_{\ell} N^+ \cup (\partial N^+_2\setminus \Omega^+_2), \text{where} \ \partial_{\ell} N^{+}=[{-\tau},{\tau}]\times \cc 0\pi \times \partial S^1_{+}$, divides the connected sum into two, mutually isometric balls. If we consider the one, among those, containing $N^+_1 \setminus\Omega^+_1$ (which we might call the \emph{upper copy}) it is straightforward to check that
\[
\Sigma^+\eqdef \left\{(r,\sigma(r),\theta) \st r\in\cc{-\tau}{\tau}\comma \theta\in S^1_+ \right\}
\] 
is indeed a free boundary minimal surface. This correspondence holds true for any closed minimal surface that is obtained by means of the \emph{wire-matching} argument by Colding--De Lellis; in particular, the matching is certainly possible for the central leaves and a family of nearby leaves, so the proof is complete.
\end{proof}

We shall now present the main, straightforward application of this gluing lemma.

\begin{lemma}
\label{lem:2model} Given any $\notb\geq 2$ there exists, 
on $\overline{B^3}$, a Riemannian metric $\smetric^{\ebi}_{2}$ of positive scalar curvature such that:
\begin{enumerate} [label={\normalfont(\arabic*)}]
\item {$(\overline{ B^3}, \smetric^{\ebi}_{2})$ contains a family of embedded, free boundary minimal surfaces $\Delta^{\ebi}_\theta$, having genus zero and $b$ boundary components, parametrized by $\theta\in (-\theta_0,\theta_0)$ for some $\theta_0>0$;}
\item {the metric in question coincides with the unit round metric in a neighborhood of given points $x^{\ebi}_2$ (in the interior) and $y^{\ebi}_2$ (on the boundary) and on both those neighborhoods such minimal annuli provide a local foliation by (half-)great spheres parallel at $x^{\ebi}_2$ and $y^{\ebi}_2$, respectively.}	
\end{enumerate}	
\end{lemma}

\begin{proof}
When $\notb=2$ this follows by applying \cref{lem:join} to the blocks $(M^{\euno}_{2,\text{bdry}}, \smetric^{\euno}_{2,\text{bdry}})$ near $y^{\euno}_{2,\text{bdry}}$ (see \cref{rmk:bmodel}) and $(\overline{\Omega},\smetric_1)$ near $x_1$ (see \cref{rmk:divide}).
The case $\notb>2$ follows by simply repeating the operation, namely joining the resulting manifold with further copies of $(\overline{\Omega},\smetric_1)$. 
\end{proof}

\section{Construction of the counterexamples}

\begin{proof}[Proof of \cref{thm:Counterexample}] It is convenient to divide the argument in two steps.

{\vspace{2mm}\textbf{Step 1.}}
Given integers $\notg\geq 0$ and $\notb>0$ as in the statement, let us consider:
\begin{itemize}
\item {one copy of the Riemannian three-sphere of positive scalar curvature as per \cref{lem:intblock}, which we shall refer to as $(M_0,\smetric_0)$ and let $(x_0, y_0)$ be the pair of points mentioned in that statement;}
\item{$\notg$ copies of the Riemannian three-sphere of positive scalar curvature produced by \cref{lem:tori}, denoted by $(M^1_1, \smetric^1_1),\ldots ,(M^{\notg}_1, \smetric^{\notg}_1)$, and let $(x^1_1, y^1_1)$,\ldots,$(x_1^{\notg},y_1^{\notg})$ be the pairs of points mentioned in that statement, respectively; each $(M_1^{i}, \smetric_1^i)$ contains a family of minimal tori that provide a foliation by great two-spheres of suitably small neighborhoods of $x_1^i$ and $y_1^i$;}
\item {one copy of the Riemannian (closed) three-ball of positive scalar curvature and totally geodesic boundary produced via \cref{lem:2model}, which we shall refer to as $(M^{\ebi}_2, \smetric^{\ebi}_2)$ and let $(x^{\ebi}_2, y^{\ebi}_2)$ be the pair of points mentioned in that statement.}
\end{itemize}	

\begin{figure}[htpb]
\centering
\begin{tikzpicture}[scale=0.29]
\newcommand\connect{%
    \draw (-0.4,0.2) to[bend right=30] (0.4,0.2);
    \draw (-0.4,-0.2) to[bend left=30] (0.4,-0.2);
}
\pgfmathsetmacro\R{3}
\pgfmathsetmacro\a{0.6}
\pgfmathsetmacro\b{1.3}
\pgfmathsetmacro\c{1.2}

\begin{scope} [xshift=0]
\coordinate (C1) at (0,0);

\draw (C1) circle (\R);
\begin{scope}[blue]
\draw (-0.8,0.05) to[bend left=20] (0.8,0.05);
  \draw (-0.9,.1) to[bend right=25] (0.9,.1);
\draw (C1) ellipse ({\R-0.2} and {\R/3});
\end{scope}
\node at ($(C1)+(\a*\R,\b*\R)$) {$(\amb_1^1,\smetric_1^1)$};
\node at ($(C1)+(-\c*\R,0)$) {$x_1^1$};
\node at ($(C1)+(\c*\R,0)$) {$y_1^1$};
\end{scope}

\begin{scope} [xshift=140] \connect \end{scope}

\begin{scope} [xshift=280]
\coordinate (C1) at (0,0);
\draw (C1) circle (\R);
\begin{scope}[blue]
\draw (-0.8,0.05) to[bend left=20] (0.8,0.05);
\draw (-0.9,.1) to[bend right=25] (0.9,.1);
\draw (C1) ellipse ({\R-0.2} and {\R/3});
\end{scope}
\node at ($(C1)+(\a*\R,\b*\R)$) {$(\amb_1^2,\smetric_1^2)$};
\node at ($(C1)+(-\c*\R,0)$) {$x_1^2$};
\node at ($(C1)+(\c*\R,0)$) {$y_1^2$};
\end{scope}

\begin{scope} [xshift=420] \connect \end{scope}

\begin{scope} [xshift=560]
\coordinate (C1) at (0,0);
\draw (C1) circle (\R);
\begin{scope}[blue]
\draw plot [smooth cycle, tension = 0.55] coordinates {(-0.9*\R, 0) (-0.65*\R, 0.4*\R) (-0.5*\R, 0*\R) (0,0.5*\R) (0.3*\R, -0.1*\R) (0.6*\R, 0.3*\R) (0.95*\R, 0) (0.6*\R, 0*\R) (0.3*\R, -0.5*\R) (-0.05*\R,0.05*\R) (-0.5*\R, -0.4*\R) (-0.7*\R, 0.1*\R)};
\end{scope}
\node at ($(C1)+(\a*\R,\b*\R)$) {$(\amb_0,\smetric_0)$};
\node at ($(C1)+(-\c*\R,0)$) {$x_0$};
\node at ($(C1)+(\c*\R,0)$) {$y_0$};
\end{scope}

\begin{scope} [xshift=700] \connect \end{scope}

\begin{scope} [xshift=865]
\coordinate (C1) at (0,-0.5*\R);
\pgfmathsetmacro\r{1.1*\R}

\draw ($(C1)+(-\r,{0.4*\r})$) arc (180:0:{\r});
\draw ($(C1)+(-\r,{0.4*\r})$) -- ($(C1)+(-\r,0)$);
\draw ($(C1)+(\r,{0.4*\r})$) -- ($(C1)+(\r,0)$);

\draw ($(C1)+(-\r,0)$) arc (180:360:{\r} and {0.3*\r});
\draw[dashed] ($(C1)+(-\r,0)$) arc (180:0:{\r} and {0.3*\r});

\begin{scope}[blue]

\coordinate (C2) at (-0.7*\r,-0.5*\R);
\coordinate (C3) at (-0.25*\r,-0.5*\R);
\coordinate (C4) at (0.25*\r,-0.5*\R);
\coordinate (C5) at (0.7*\r,-0.5*\R);

\pgfmathsetmacro\l{0.15*\R}

\draw ($(C2)+(-\l,0)$) arc (180:360:{\l} and {0.3*\l});
\draw[dashed] ($(C2)+(-\l,0)$) arc (180:0:{\l} and {0.3*\l});

\draw ($(C3)+(-\l,0)$) arc (180:360:{\l} and {0.3*\l});
\draw[dashed] ($(C3)+(-\l,0)$) arc (180:0:{\l} and {0.3*\l});

\draw ($(C4)+(-\l,0)$) arc (180:360:{\l} and {0.3*\l});
\draw[dashed] ($(C4)+(-\l,0)$) arc (180:0:{\l} and {0.3*\l});

\draw ($(C5)+(-\l,0)$) arc (180:360:{\l} and {0.3*\l});
\draw[dashed] ($(C5)+(-\l,0)$) arc (180:0:{\l} and {0.3*\l});

\draw plot [smooth, tension = 1] coordinates {($(C2)+(\l,0)$)  (-0.48*\r,-0.35*\R) ($(C3)+(-\l,0)$) };

\draw plot [smooth, tension = 1] coordinates {($(C3)+(\l,0)$)  (0*\r,-0.35*\R) ($(C4)+(-\l,0)$) };

\draw plot [smooth, tension = 1] coordinates {($(C4)+(\l,0)$)  (0.48*\r,-0.35*\R) ($(C5)+(-\l,0)$) };

\draw plot [smooth, tension = 1] coordinates {($(C2)+(-\l,0)$)  (-0.5*\r,0.3*\R) (0.5*\r,0.3*\R) ($(C5)+(\l,0)$) };

\end{scope}
\node at ($(0,0.1)+(\a*\R,\b*\R)$) {$(\amb_2^{\scriptscriptstyle(4)},\smetric_2^{\scriptscriptstyle(4)})$};
\node at ($(-0.5,0)+(-\c*\R,0)$) {$x_2^{\scriptscriptstyle(4)}$};
\end{scope}

\begin{scope} [xshift=990,yshift=-40] \connect \end{scope}

\begin{scope} [xshift=1080,scale=1.2]
\coordinate (C1) at (0,-0.5*\R);
\pgfmathsetmacro\r{0.5*\R}

\draw ($(C1)+(-\r,0)$) arc (180:360:{\r} and {0.3*\r});
\draw[dashed] ($(C1)+(-\r,0)$) arc (180:0:{\r} and {0.3*\r});

\draw plot [smooth, tension = 0.75] coordinates {($(C1)+(-\r,0)$)  ($(C1)+(-0.9*\r,1*\R)$) ($(C1)+(0.2*\r,1.5*\R)$) ($(C1)+(1.8*\r,1.15*\R)$) ($(C1)+(3*\r,1.5*\R)$) ($(C1)+(3.9*\r,1.2*\R)$) ($(C1)+(3.5*\r,0.5*\R)$) ($(C1)+(2.5*\r,0.25*\R)$) ($(C1)+(1.3*\r,0.3*\R)$) ($(C1)+(\r,0)$)};

\begin{scope}[scale=0.8,rotate=10,yshift=20,xshift=10]
\draw (-0.8,0.05) to[bend left=20] (0.8,0.05);
  \draw (-0.9,.1) to[bend right=25] (0.9,.1);
\end{scope}

\begin{scope}[rotate=20,yshift=-10,xshift=110]
\draw (-0.8,0.05) to[bend left=20] (0.8,0.05);
  \draw (-0.9,.1) to[bend right=25] (0.9,.1);
\end{scope}

\node at (1.4*\R,-0.6*\R) {$(\amb,\smetric_T)$};

\end{scope}

\draw [decorate,decoration={brace,mirror,amplitude=10pt}] (-3,-4) -- (13,-4) node [black,midway,yshift=-20] {$(\amb',\smetric')$};

\draw [decorate,decoration={brace,mirror,amplitude=10pt}] (-3,-6.5) -- (23,-6.5) node [black,midway,yshift=-20] {$(\amb'',\smetric'')$};

\draw [decorate,decoration={brace,mirror,amplitude=10pt}] (-3,-9) -- (33,-9) node [black,midway,yshift=-20] {$(\amb''',\smetric''')$};
\end{tikzpicture}
\caption{Scheme of the construction in the proof of \cref{thm:Counterexample} for $\notg=2$ and $\notb=4$.} \label{fig:CountEx}
\end{figure}

Invoking Lemma 11 in \cite{ColdingDeLellis2005}, we proceed as follows (see \cref{fig:CountEx}).
We first attach $(M^i_1,\smetric^i_1)$, near $y_1^i$, to $(M^{i+1}_1, \smetric^{i+1}_1)$, near $x_1^{i+1}$, as $i$ varies from $1$ to $\notg-1$; let $(M',\smetric')$ be the resulting manifold (of positive scalar curvature and empty boundary); let $y'\in M'$ be the point corresponding to $y^{\notg}_1\in M^{\notg}_1$, thus with a neighborhood that is foliated by great spheres.
Similarly we attach $(M',\smetric')$, near $y'$, to $(M_0, \smetric_0)$ near $x_0$; let $(M'', \smetric'')$ be the resulting manifold (of positive scalar curvature and empty boundary) and let $y''\in M''$ be the point corresponding to $y_0\in M_0$, thus with a neighborhood that is foliated by great spheres.
Lastly, we attach $(M'', \smetric'')$, near $y''$, to $(M^{\ebi}_2, \smetric^{\ebi}_2)$, near $x^{\ebi}_2$; let $(M''', \smetric''')$ be the resulting manifold (of positive scalar curvature and totally geodesic boundary).
The manifold $(M''',\smetric''')$ is connected, has the topology of a ball, and it contains a sequence of free boundary minimal surfaces of genus $\notg$, exactly $\notb$ boundary components, that have unbounded area and Morse index (cf. \cref{rmk:index} above).

\vspace{2mm}\textbf{Step 2.}
Let $M$ be as in the statement: possibly applying Lemma C.1 in \cite{CarlottoLi2019} we can, and we shall, assume that this manifold comes endowed with a Riemannian metric of positive scalar curvature, and such that $\partial M$ is strictly mean convex. At that stage we know, by virtue of Theorem 5.7 in \cite{GromovLawson1980Spin}, that there exists a new metric $\smetric_{T}$ on $M$ still having positive scalar curvature but totally geodesic boundary (in fact this manifold can be \emph{doubled} to a smooth Riemannian manifold $(M_D, \smetric_D)$ without boundary).
Hence, we just observe that one can perform the Gromov--Lawson connected sum of $(M''',\smetric''')$ and $(M,\smetric_T)$ so to obtain a compact three-manifold with positive scalar curvature and totally geodesic boundary. 

The combination of the two steps above allows to complete the proof of the first assertion of \cref{thm:Counterexample}. Instead, to obtain strictly positive mean curvature, it suffices to have the previous construction followed by the perturbation argument given in Lemma C.1 in \cite{CarlottoLi2019}. 
\end{proof}

Concerning the claim in \cref{rmk:ConvexBC}, it suffices to observe that when $\notb=1$ one modifies the block $(M^{\euno}_2, \smetric^{\euno}_2)$ constructed in \cref{subs:BlocksDisk} as follows: by the statement of \cref{lem:fbpiece}, the metric we obtained equals that of the cylinder $I\times S^2$ near the boundary sphere, thus we can just consider a smooth warping factor $f_W\colon \cc 0\eps \times S^2\to \mathbb{R}$, only depending on the first coordinate, that is monotone decreasing and equals $f$ on $\cc{\eps/2}{\eps}$. For \emph{any} such choice the boundary is convex (umbilic, with constant mean curvature); furthermore if the derivative of $f_W$ is small enough then the scalar curvature of the ambient manifold shall still be positive.


\part[Existence of FBMS via equivariant min-max]{Existence of free boundary minimal surfaces via equivariant min-max}

Given a three-dimensional Riemannian manifold with boundary and a finite group of orientation-preserving isometries of this manifold, we present a proof of the regularity of a free boundary minimal surface obtained via an equivariant min-max procedure à la Simon--Smith with $n$-parameters and we show that its equivariant index is bounded above by $n$.
Moreover, we employ these min-max techniques to show that the unit ball in $\R^3$ contains embedded free boundary minimal surfaces with connected boundary and arbitrary genus.

\begin{flushright}
{Based on \cite{CarlottoFranzSchulz2020} and \cite{Franz2021}}    
\end{flushright}

\label{part:Existence}


\chapter{Context and results} \label{intro:existence}

Over the last decade, the theory of free boundary minimal surfaces has been developed in various interesting directions, yet many fundamental questions remain open. One of the most basic ones regards the existence of free boundary minimal surfaces of any given topological type.
We refer the reader to the \hyperref[chpt:intro]{Introduction} and to Section 3 of \cite{Li2020} for a broad overview of existence results, including those in higher-dimensional Euclidean balls or in the general setting of compact Riemannian manifolds with boundary, while we will focus here on the special case of $B^3$. 

The existence question in this setting can be phrased as follows (see also \emph{Open Question 1} in the survey \cite{Li2020}): does the unit ball of $\R^3$ contain free boundary minimal surfaces of any given genus $g\geq 0$ and any number of boundary components $b\geq 1$? 
In spite of significant advances, which we survey below, the answer to such a question has proven to be quite elusive.

There is a well-known analogy between the free boundary theory for the unit ball $B^3\subset\R^3$ and the theory concerning closed minimal surfaces in the round three-dimensional sphere $S^3$: since for the latter Lawson proved in \cite{Lawson1970} that there indeed exist in the sphere embedded minimal surfaces of arbitrary genus, there might be some reason to lean towards an affirmative answer.
Unfortunately, an approach similar to the one adopted by Lawson does not work in the free boundary setting. 
For this reason, many different methods to prove existence of free boundary minimal surfaces in $B^3$ are being investigated.

\subsection{Free boundary minimal surfaces and Steklov problem} \label{method:Steklov}
The first nontrivial examples of (embedded) free boundary minimal surfaces, besides the flat disc and the critical catenoid, were obtained by Fraser and Schoen in \cite{FraserSchoen2016} (see also \cite{GirouardLagace2021}*{Appendix A}): these have genus zero and $b\geq 2$ boundary components and converge in the sense of varifolds to the sphere $S^2$ as $b\to\infty$. 
The proof is based on the relation between free boundary minimal surfaces and the extremes for Steklov eigenvalues (see \cite{FraserSchoen2011, FraserSchoen2013, FraserSchoen2016}). 
These methods have been pushed till proving by Matthiesen--Petrides in \cite{MatthiesenPetrides2020} that, for any compact surface $\Sigma^2$ with nonempty boundary, there exists a branched free boundary minimal immersion $\Sigma\to B^{m+1}$, for some $m\ge 2$. However, so far, one cannot control the dimension $m+1$ of the ambient ball, nor it is known that the map is an embedding, apart from the aformentioned case where $\Sigma$ as genus zero (see \cite{KarpukhinKokarevPolterovich2014}*{Corollary 1.3} and \cite{FraserSchoen2016}*{Proposition 8.1}).

\subsection{Gluing method} \label{method:Perturbation}
Another very fruitful procedure to construct free boundary minimal surface in $B^3$ is the \emph{gluing method}. Here is a list of the existence results obtained with this approach.
\begin{itemize}
\item Folha--Pacard--Zolotareva constructed in \cite{FolhaPacardZolotareva2017} examples having genus zero or one and any \emph{sufficiently large} number of boundary components, by doubling the equatorial disc.\vspace{-1ex}

\item Kapouleas and Li developed in \cite{KapouleasLi2017} methods to desingularize the formal union of a disc and a critical catenoid to obtain free boundary minimal surfaces in $B^3$ with large genus and exactly three boundary components. To get a pictorial description, this family can be regarded as a free boundary version of the Costa--Hoffman--Meeks minimal surfaces in $\R^3$.\vspace{-1ex}

\item In \cite{KapouleasWiygul2017}, Kapouleas--Wiygul constructed free boundary minimal surfaces in $B^3$ having connected boundary and prescribed high genus, by tripling the equatorial disc. In the introduction of the same article, they also describe a desingularization scheme to construct free boundary minimal surfaces having connected boundary and sufficiently large genus by regularising the intersection of two orthogonal discs in $B^3$ (by means of a suitable Scherk surface).\vspace{-1ex}

\item Kapouleas--McGrath, in \cite{KapouleasMcGrath2020}, generalized the \emph{linearized doubling approach}, which, roughly speaking, produces a minimal surface resembling two copies of a minimal surface (with certain conditions) joined by many small catenoidal bridges. 
Then, they applied this result to construct a family of free boundary minimal surfaces in $B^3$ with four boundary components and sufficiently large genus, by doubling the critical catenoid.\vspace{-1ex}

\item Kapouleas and Zou in \cite{KapouleasZou2021} constructed another family of free boundary minimal surfaces with genus zero and sufficiently large number of boundary components, converging in the sense of varifolds to the unit sphere $S^2$.\vspace{-1ex}

\item Recently, Carlotto--Schulz--Wiygul in \cite{CarlottoSchulzWiygul2022} constructed another family of free boundary minimal surfaces in $B^3$ with three boundary components and sufficiently large genus, which have the same symmetry group as the ones constructed in \cite{KapouleasLi2017}. In particular, this proves that the topology and the symmetry group does not determine a free boundary minimal surface in $B^3$ uniquely.
\end{itemize}

\subsection{Equivariant min-max theory} \label{method:EquivMinMax}

The third and last method that has been employed to construct free boundary minimal surfaces in $B^3$, and in which we focus in this thesis, is an \emph{equivariant min-max theory} à la Simon--Smith and Colding--De Lellis \cite{ColdingDeLellis2003}.
This procedure was first proposed by Pitts--Rubinstein in \cite{PittsRubinstein1988} and then developed in \cite{Ketover2016Equivariant} for the closed case and in \cite{Ketover2016FBMS} for the free boundary one, with the goal of constructing new families of minimal surfaces in $S^3$ and of free boundary minimal surfaces in $B^3$. 
Indeed, encoding the right symmetry group in the min-max procedure allows us to produce surfaces with fully controlled topology.
In particular, in \cite{Ketover2016FBMS}, Ketover proved the existence of free boundary minimal surfaces in $B^3$ with three boundary components and arbitrary large genus. In the same article, the author describes also the construction of examples with symmetry groups of the Platonic solids and of examples with genus zero and any number of boundary components, resembling the doubling of the equatorial disc.

The min-max theory developed by Simon--Smith and Colding--De Lellis differs from the one by Almgren--Pitts \cite{Almgren1965,Pitts1981} and Marques--Neves \cite{MarquesNeves2014,MarquesNeves2017} since, in the former one, stronger regularity and convergence conditions are imposed on the sweepouts (cf. \cite{MarquesNeves2014}*{Section 2.11}) and the ambient dimension is assumed to be equal to three. 
This makes the Simon--Smith approach not suitable for certain applications; indeed, note for example that the sweepout constructed by Marques and Neves in \cite{MarquesNeves2014} to prove the Willmore conjecture does not satisfy the Simon--Smith assumptions.

On the other hand, the stronger regularity makes proofs easier and enables to obtain a partial control on the topology of the resulting surface. This has been an advantage of the Simon--Smith approach when one wants to construct a minimal surface with a certain topological type. Indeed, the only way we know so far to obtain some information on the topology of the resulting minimal hypersurface in the Almgren--Pitts min-max theory is to first control the index and then use that the index bounds the topology (see for example \cite{AmbrozioCarlottoSharp2018Index,Maximo2018,Song2020}, and \cite{AmbrozioCarlottoSharp2018IndexFBMS} for the free boundary case), although such bounds (even when effective) are far from sharp.

Here, we outline a proof of the equivariant min-max theorem (see \cref{thm:EquivMinMax} below), recovering and unifying the methods in \cite{ColdingDeLellis2003,DeLellisPellandini2010,Li2015,Ketover2019,ColdingGabaiKetover2018,Ketover2016Equivariant,Ketover2016FBMS}.
We decided to do so because many arguments are fragmented in the literature and because, sometimes, we need variations of the known results.

Furthermore, we enhance the equivariant min-max theorem by obtaining also an upper bound on the \emph{equivariant index} of the surface resulting from this procedure. 
The proof consists in adapting the arguments adopted in the Almgren--Pitts setting by \cite{MarquesNeves2016} (see also \cite{MarquesNeves2021}, and \cite{GuangLiWangZhou2021} for the free boundary setting) to our `smoother' setting (see \cite{MarquesNeves2016}*{Section 1.3}), but with the additional presence of a group of isometries imposed to the min-max procedure.

Finally, we employ this machinery to obtain a new family of free boundary minimal surfaces in $B^3$, fully solving the existence problem in $B^3$ for the class of free boundary minimal surfaces with connected boundary. 

Let us now describe precisely the statement of the equivariant min-max theorem in a three-dimensional manifold with boundary. To this purpose, we first give some preliminary definitions.
Afterwards, we present our application to the construction of a new family of free boundary minimal surfaces in $B^3$.

\section{Setting and definitions}

Our setting is a three-dimensional compact Riemannian manifold $(M^3,\smetric)$ with strictly mean convex boundary and with a \emph{finite} group $G$ of orientation-preserving isometries.

\begin{remark}
We assume the manifold $M$ to have strictly mean convex boundary because we want it to satisfy property \hypP.
Here we ask for more, namely, the strict mean convexity of the boundary of the ambient manifold (which implies \hyperref[HypP]{$(\mathfrak{P})$}), because we want it to be an open property with respect to the class of smooth metrics. 
\end{remark}

\begin{remark} \label{rem:NotProperInLi}
Li in \cite{Li2015} studies the free boundary min-max problem without curvature assumptions on the ambient manifold $M$. However, as remarked in \cref{sec:Properness} (see also the introduction), there are problems in defining the Morse index for free boundary minimal surfaces that are not properly embedded in $M$. Moreover, the compactness results in \cite{AmbrozioCarlottoSharp2018Compactness}, which we use for example in \cref{sec:BumpyIsGGeneric}, need property \hyperref[HypP]{$(\mathfrak{P})$}.
That is another good reason why we require \hyperref[HypP]{$(\mathfrak{P})$} to hold.
\end{remark}

\begin{remark}
Here we assume the group $G$ to be finite, but it is due mentioning that the equivariant min-max theory has been developed also in the case of a compact \emph{connected} Lie group with cohomogeneity not $0$ or $2$ in \cite{Liu2021} and in the Almgren--Pitts setting with a compact Lie group of cohomogeneity greater or equal than $3$ in \cite{Wang2020}.
\end{remark}

In what follows, we denote by $I^n=[0,1]^n$ the product of $n$-copies of the unit interval. Moreover, when we say that $\Sigma^2$ is a \emph{surface} in an ambient manifold $M^3$, we mean that $\Sigma$ is smooth, complete and properly embedded, i.e., $\partial\Sigma=\Sigma\cap\partial M$.

\begin{definition} \label{def:GSweepout}
A family $\{\Sigma_t\}_{t\in I^n}$ of subsets of a three-dimensional Riemannian manifold $M$ is said to be a \emph{generalized family} of surfaces if there are a finite subset $T$ of $I^n$ and a finite set of points $P$ of $M$ such that:
\begin{enumerate}[label={\normalfont(\roman*)}]
    \item $t\mapsto\Sigma_t$ is continuous in the sense of varifolds;
    \item $\Sigma_t$ is a surface for every $t\not\in T\cup \partial I^n$;
    \item for $t\in T\setminus\partial I^n$, $\Sigma_t$ is a surface in $M\setminus P$.
\end{enumerate}

Moreover, the generalized family $\{\Sigma_t\}_{t\in I^n}$ is said to be \emph{smooth} if it holds also that:
\begin{enumerate}[resume*]
    \item $t\mapsto\Sigma_t$ is smooth for $t\not\in T\cup\partial I^n$;
    \item for $t\in T\setminus \partial I^n$, $\Sigma_s\to\Sigma_t$ smoothly in $M\setminus P$ as $s\to t$.
\end{enumerate}

Finally, we say that $\{\Sigma_t\}_{t\in I^n}$ is a \emph{$G$-sweepout} if it is a smooth generalized family of surfaces and $\Sigma_t$ is $G$-equivariant for all $t\in I^n$, i.e., $h(\Sigma_t) = \Sigma_t$ for all $h\in G$ and $t\in I^n$, and \emph{orientable} for all $t\in I^n\setminus T$.
\end{definition}

\begin{remark}
Slightly different variations of this definition are given in several references (see for example \cite{ColdingDeLellis2003}*{Definition 1.2}, \cite{DeLellisPellandini2010}*{Definition 0.5}, \cite{ColdingGabaiKetover2018}*{pp. 2836}). 
In order to have the regularity and the index bound on the resulting surface, it is sufficient to consider generalized families of surfaces, however smoothness is needed for the genus bound (see \cite{DeLellisPellandini2010}*{Definition 0.5}).
\end{remark}

\begin{remark}
The assumption that $\Sigma_t$ is orientable for all $t\in I^n\setminus T$ is needed to prove the genus bound \eqref{eq:GenusBound} below. Probably one could deal with nonorientable surfaces, but the equation would be modified and the proof more complicated. 
Note that the orientability is required also in the other references stating the genus bound, see for example \cite{DeLellisPellandini2010}*{Definition 0.5}, \cite{Li2015}*{Section 9}, \cite{Ketover2019}*{Section 1}.
\end{remark}

\begin{definition} \label{def:GEquivIsotopy}
We say that a smooth map $\Phi\colon I^n\times M\to M$ is a \emph{$G$-equivariant isotopy} if, for all $t\in I^n$, $\Phi_t\eqdef \Phi(t,\cdot)\colon M\to M$ is a diffeomorphism such that $\Phi_t\circ h = \Phi_t$ for all $h\in G$. We denote by $\Is_G(I^n, M)$ the set of all such $G$-equivariant isotopies. See also \cref{def:EquivariantObjects}.
\end{definition}

\begin{definition} \label{def:SaturationWidth}
Given a $G$-sweepout $\{\Sigma_t\}_{t\in I^n}$, we define its \emph{$G$-saturation} as
\[
\Pi \eqdef \{ \{\Phi_t(\Sigma_t)\}_{t\in I^n} \st \Phi\in \Is_G(I^n,M), \ \Phi_t=\id \ \forall t\in \partial I^n \}.
\]
Then the \emph{min-max width} of $\Pi$ is defined as
\[
W_\Pi \eqdef \inf_{\{\Lambda_t\}\in \Pi} \sup_{t\in I^n} \area(\Lambda_t).
\]
\end{definition}
\begin{remark}
Note that here we consider isotopies $\Phi$ such that $\Phi_t\colon M\to M$ is a diffeomorphism of the ambient manifold. This choice does not coincide with the definition in \cite{Li2015}, where \emph{outer isotopies} are considered (see also \cref{rem:NotProperInLi}).
\end{remark}
\begin{remark}
Note that the uniform bound on the `bad points' (the points belonging to $P$ in \cref{def:GSweepout}), which is required in \cite{ColdingDeLellis2003}*{Remark 1.3} (see also \cite{DeLellisPellandini2010}*{Remark 0.2}), is trivially satisfied for the $G$-saturation of a $G$-sweepout.
\end{remark}

\section{Equivariant min-max theorem with index bound}

\vspace{-1ex}

We obtain the following equivariant min-max theorem, which is the combination of regularity, genus bound and equivariant index bound, the latter being original work of this thesis. Indeed, the regularity and the genus bound were addressed in \cite{Ketover2016FBMS}*{Theorem 3.2} (see also \cite{Ketover2016Equivariant}*{Theorem 1.3}, where Ketover dealt with the closed case, for which many arguments are similar), as the result of many previous contributions.

\begin{theorem} \label{thm:EquivMinMax}
Let $(M^3, \smetric)$ be a three-dimensional Riemannian manifold with strictly mean convex boundary and let $G$ be a finite group of orientation-preserving isometries of $M$.
Let $\{\Sigma_t\}_{t\in I^n}$ be a $G$-sweepout and let $\Pi$ be its $G$-saturation. Assume that\vspace{-1ex}
\[
W_\Pi > \sup_{t\in \partial I^n} \area(\Sigma_t).
\]

Then, there exist a minimizing sequence $\{\{\Sigma^j_t\}_{t\in I^n}\}_{j\in\N}\subset \Pi$ (which means that $\lim_{j\to\infty} \sup_{t\in I^n} \area(\Sigma_t^j) = W_\Pi$) and a sequence $\{t_j\}_{j\in\N}\subset I^n$ such that $\{\Sigma^j = \Sigma^j_{t_j}\}_{j\in\N}$ is a min-max sequence (i.e., $\lim_{j\to\infty} \area(\Sigma^j) = W_\Pi$) converging in the sense of varifolds to $\limitv\eqdef \sum_{i=1}^k m_i\limitv_i$, where $\limitv_i$ are disjoint free boundary minimal surfaces and $m_i$ are positive integers.
Moreover, the $G$-equivariant index of the support $\operatorname{spt}(\limitv) = \bigcup_{i=1}^k \limitv_i$ of $\limitv$ is less or equal than $n$, namely \vspace{-1ex}
\begin{equation} \label{eq:IndexBound}
\ind_G(\operatorname{spt}(\Xi)) = \sum_{i=1}^k \ind_G(\Xi_i) \le n,\vspace{-2ex}
\end{equation}
and the following genus bound holds
\begin{equation} \label{eq:GenusBound}
\sum_{i\in\mathcal O} \genus(\limitv_i) + \frac 12 \sum_{i\in \mathcal N} (\genus(\limitv_i)-1) \le \liminf_{j\to+\infty}\liminf_{\substack{\tau\to t_j\\\tau\not\in T}} \genus(\Sigma^j_{\tau}), 
\end{equation}\vspace{-1ex}
where $\mathcal O$ is the set of indices $i\in\{1,\ldots,k\}$ such that $\Xi_i$ is orientable and $\mathcal N$ is the set of indices $i\in\{1,\ldots,k\}$ such that $\Xi_i$ is nonorientable.
\end{theorem}

The precise definition of equivariant index is given in \cref{sec:EquivSpectrum}, but it is what one would expect, i.e., the maximal dimension of a linear subspace of the \emph{$G$-equivariant vector fields} where the second variation of the area functional is negative definite.

\begin{remark}
The assumptions that $M$ has dimension $3$ and that $G$ consists of \emph{orientation-preserving} isometries are required in \cite{Ketover2016Equivariant} and \cite{Ketover2016FBMS}, which we present in \cref{sec:ProofRegularityGenus}. 
For what regards the index bound, such assumptions are used only in \cref{sec:ConvBoundGIndex}, while elsewhere the arguments work for any finite group of isometries $G$ and any ambient dimension $3\le m+1\le 7$.
\end{remark}

\begin{remark}
The equivariant index estimate in \cref{thm:EquivMinMax} implies the analogous result in the closed case, since, if we repeat the proofs forgetting about the boundary, everything works in the same way (actually more easily). 
\end{remark}

\begin{remark}
Observe that we state the genus bound \emph{without} multiplicity as in \cite{DeLellisPellandini2010} and \cite{Li2015}*{Section 9}.
We expect the result to hold with multiplicity as obtained in the closed case in \cite{Ketover2019}, but it is not clear how to generalize the result to the free boundary setting. In particular, \cite{Ketover2019}*{Section 5} (namely, the proof of the genus bound given the `improved lifting lemma' \cite{Ketover2019}*{Proposition 2.2}) uses in an essential way that the surface in question is compact without boundary.
\end{remark}

\begin{remark}\label{rem:OtherFamilies}
Similarly to what we do below for the surfaces $M_g$ constructed in \cref{thm:main-fbms-b1}, we can obtain the equivariant index bound stated in \cref{thm:EquivMinMax} for any other surface obtained via an equivariant min-max procedure. To our knowledge, the known equivariant min-max constructions so far are:
\begin{itemize}
\item The minimal surfaces in $S^3$ of \cite{Ketover2016Equivariant}*{Sections 6.2, 6.3, 6.5, 6.6 and 6.7}.
\item The minimal surfaces in $S^3$ of \cite{KetoverMarquesNeves2020}*{Theorem 3.6}.
\item The free boundary minimal surfaces in $B^3$ of \cite{Ketover2016FBMS}*{Theorems 1.1, 1.2 and 1.3}.
\end{itemize}

For all these surfaces we get that the equivariant index (with respect to the corresponding symmetry group) is less or equal than $1$. Then, one has to find a suitable equivariant test function to get the equality, on a case-by-case basis. However, in most of the cases above the constant function $1$ is such a test function. Indeed,  the constant $1$ is a negative direction for the second variation of the area functional for every minimal surface in $S^3$ and every free boundary minimal surface in $B^3$. Moreover, if the unit normal is equivariant (i.e., $h_*\nu=\nu$ for all $h$ in the symmetry group, where $\nu$ is a choice of unit normal), the constant function is also equivariant.
In other cases some more work is needed, an example are the surfaces $M_g$ constructed in \cref{thm:main-fbms-b1} (see \cref{lem:EquivIndMg}).
\end{remark}

\section[FBMS with connected boundary]{Free boundary minimal surfaces with connected boundary}

Thanks of a suitable application of the equivariant min-max theorem \cref{thm:EquivMinMax}, we are able to construct a new family of free boundary minimal surfaces in $B^3$ with connected boundary and arbitrary genus (see also \cref{fig:fbms-b1} in the \hyperref[chpt:intro]{Introduction}).

\begin{theorem}\label{thm:main-fbms-b1}
For each $1\le g\in\N$ there exists an embedded free boundary minimal surface $M_g$ in $B^3$ with connected boundary, genus $g$, dihedral symmetry $\dih_{g+1}$ and $\dih_{g+1}$-equivariant index equal to $1$. 

Moreover, the surface $M_g$ contains the $g+1$ horizontal axes
\[
\xi_k\vcentcolon=\left\{\left(r\cos\left(\frac{k\pi}{g+1}\right),r\sin\left(\frac{k\pi}{g+1}\right),0\right)\st r\in[-1,1]\right\}, \quad \text{for $k=1,\ldots,g+1$},
\]
and converges in the sense of varifolds to the equatorial disc with multiplicity three, as $g\to\infty$.
\end{theorem}

Given $2\le n\in\N$, the \emph{dihedral group} $\dih_{n}$ is the symmetry group (or order $2n$) of a regular $n$-sided polygon. As a subgroup of Euclidean isometries (acting on $\overline{B^3}$) is defined as the group generated by the rotation of angle $\pi$ around the $n$ horizontal axes $\xi_k$ for $k\in\{1,\ldots,n\}$. We refer to \cref{sec:DihedralGroup} for further properties.

In applying \cref{thm:EquivMinMax} to prove \cref{thm:main-fbms-b1} we first need, for any positive integer $g$, to carefully design a suitable genus $g$ equivariant sweepout so to ensure that not only the natural mountain-pass condition holds, but also (and more importantly) that the limit surface we obtain is attained with multiplicity \emph{one}. This is a general issue that arises whenever one relies on min-max techniques, and it is in fact a rather delicate point in our construction. In turn, this aspect is crucial to make sure that both the number of boundary components and the genus are controlled throughout the process, i.e., as we take the limit of a min-max sequence. 

Thanks to the topological lower-semicontinuity in \cref{thm:EquivMinMax}, the genus of $M_g$ is \emph{less or equal than} $g$. Then, the equality follows by the general characterisation of equivariant surfaces given in \cref{sec:structure}.
On the other hand, the question of controlling the number of boundary components in min-max constructions is notoriously more delicate.  In that respect, Li writes \cite[pp. 324]{Li2015}: `we note that it is impossible to get a similar bound on the connectivity (i.e., number of free boundary components) of the minimal surface.' In the context of the present paper, the conclusion that the boundary of $M_g$ must be connected is achieved through a rather surprising application of Simon's Lifting Lemma (cf. \cite[Proposition 2.1]{DeLellisPellandini2010}), which we present in \cref{sec:control}.

Observe that, while interesting examples have been obtained in abundance, on the one hand the gluing methodologies do not (for their very nature) allow to obtain low-genus examples, and are only asymptotically effective, while on the other hand nontrivial technical obstacles arise if one aims at full topological control of min-max free boundary minimal surfaces. 

The questions whether some of the surfaces we construct are unique for their given topological type, or whether they coincide (\emph{for large genus}) with the family constructed in \cite{KapouleasWiygul2017} via stacking methods remain open, as stands the related question whether $M_g$ could be characterised in terms of a maximising property for its first Steklov eigenvalue under the natural normalisation constraint. Furthermore, it would certainly be interesting to know how the family of free boundary minimal surfaces obtained, via $k$-dimensional Almgren--Pitts min-max schemes, in Corollary 1.4 of \cite{LiZhou2021} (cf. Remark 1.8 in \cite{MarquesNeves2017}) compares to our examples, and to the other ones listed above. 

Finally, we remark that the result about the $\dih_{g+1}$-equivariant index of $M_{g}$ is meant to be a first step toward the computation of the (nonequivariant) Morse index of the family of surfaces $M_g$ (or possibly any another family arising from an equivariant min-max procedure, even with multiple parameters in play). 
This would be particularly relevant, as the only free boundary minimal surfaces in $B^3$ for which we know the index are the equatorial disc (with index $1$) and the critical catenoid (with index $4$, see \cite{Devyver2019,SmithZhou2019,Tran2020}).
The idea is that the equivariant index gives information on the index of any of the isometric parts (which are $2(g+1)$ in the case of $M_g$) of the equivariant surface. 
This problem will be addressed elsewhere.

\section{Structure of  \texorpdfstring{\cref{part:Existence}}{Part II}}

We provide an outline of the contents of this part by means of the following diagram.

\begin{center}
\usetikzlibrary{trees}
\tikzstyle{every node}=[anchor=west]
\tikzstyle{tit}=[shape=rectangle, rounded corners,
    draw, minimum width = 3cm]
\tikzstyle{sec}=[shape=rectangle, rounded corners]
\tikzstyle{optional}=[dashed,fill=gray!50]
\begin{tikzpicture}[scale=0.9]

\tikzstyle{mybox} = [draw, rectangle, rounded corners, inner sep=10pt, inner ysep=20pt]
\tikzstyle{fancytitle} =[fill=white, text=black, draw]

\begin{scope}[align=center, font=\small]
 \node [tit] (pre) {Preliminaries};
 \node [sec] [below of = pre, yshift=-0.5cm, xshift=-5cm] (iso) {Finite group of \\ isometries (\cref{chpt:GroupIsom})};
 \node [sec] [below of = pre, yshift=-0.5cm] (eqsp) {Equivariant spectrum\\ (\cref{chpt:EquivSpectrum})};
 \node [sec] [below of = pre, yshift=-0.5cm, xshift=5cm] (eqfb) {Equivariant FBMS\\ (\cref{chpt:EquivFBMS})};
\end{scope}

\draw[gray,dashed,rounded corners] ($(pre)+(-8,0.7)$)  rectangle ($(pre)+(8,-2.6)$);

\begin{scope}[align=center, yshift=-4cm]
\node [tit] (res) {Main results};
\node [sec] [below of = res, yshift=0.3cm, xshift=-5cm] (reg) {Regularity of \\ min-max \\ (\cref{sec:ProofRegularityGenus})};
\node [sec] [below of = res, yshift=-2.1cm, xshift=-5cm] (ind) {Index bound \\ (\cref{chpt:IndexBound})};

\node [sec] [below of = res, yshift=-2.1cm] (thm) {Equivariant \\ min-max theorem\\ (\cref{thm:EquivMinMax})};

\node [sec] [below of = res, yshift=-2.1cm, xshift=5cm] (fbb1) {FBMS with $b=1$\\ (\cref{chpt:FBMSb1})};

\draw[blue,rounded corners] ($(thm)+(-8,1.3)$)  rectangle ($(thm)+(8,-1.7)$);
\node[blue] at ($(thm)+(-7.8,-1.3)$) {\footnotesize Core of the part};
\end{scope}

\tikzset{bigarr/.style={
decoration={markings,mark=at position 1 with {\arrow[scale=1.5]{>}}},
postaction={decorate}}}
 
\begin{scope}[gray]
\draw[bigarr] (reg) to (thm);
\draw[bigarr] (ind) to (thm);
\draw[bigarr] (thm) to (fbb1);
\end{scope}
\end{tikzpicture}
\end{center}


\chapter{Finite group of isometries} \label{chpt:GroupIsom}

In this chapter, we present some definitions, notation and results about finite groups of isometries on a Riemannian manifold with boundary.

\section{Basic definitions and notation}

Let $(M^{m+1},\smetric)$ be a compact Riemannian manifold with boundary, with \emph{arbitrary} dimension $m+1\ge 2$, and let $G$ be a finite group of isometries of $M$. Let us give some basic definitions, starting from the notation of isotropy group at a point, the singular locus and then some basic $G$-equivariant objects.

\begin{definition} [{cf. \cite{Ketover2016FBMS}*{Section 3}}] \label{def:IsotropyGroup-SingularLocus}
For all $x\in M$, let us define the \emph{isotropy subgroup at $x$} as
\[
G_x\eqdef \{\selem \in G\st \selem(x)=x\}.
\]
Equivalently, we say that \emph{$x$ is of isotropic type $G_x$}.
Moreover, we denote by $\mathcal{S}\subset M$ the \emph{singular locus} of $G$, defined as
\[
\mathcal{S} \eqdef \{x\in M \st  G_x\not=\{\id\}\}.
\]
\end{definition}

\begin{definition} \label{def:EquivariantObjects}
\begin{enumerate} [label={\normalfont(\roman*)}]
\item We say that $X\in\vf(M)$ is a \emph{$G$-equivariant (tangent) vector field} if $\selem_*X = X$ for all $\selem\in G$. We denote by $\vf_G(M)$ the set of all $G$-equivariant vector fields.

\item We say that $\varphi\colon M\to M$ is a \emph{$G$-equivariant diffeomorphism} if $\varphi = \selem^{-1}\circ \varphi\circ \selem$ for all $\selem\in G$. We denote by $\operatorname{Diff}_G(M)$ the set of all $G$-equivariant diffeomorphisms of $M$.

\item We say that a smooth map $\Phi\colon I^n\times M\to M$ is a \emph{$G$-equivariant isotopy} if $\Phi_t\eqdef \Phi(t,\cdot) \in \operatorname{Diff}_G(M)$ for all $t\in I^n$. We denote by $\Is_G(I^n, M)$ the set of all such $G$-equivariant isotopies in $M$. In the case $n=1$, we write $\Is_G(M)=\Is_G(I,M)$.

\item We say that $U\subset M$ is a \emph{$G$-equivariant subset} if $\selem(U)=U$ for all $\selem\in G$.

\item We say that $U\subset M$ is a \emph{$G$-compatible subset} if $\selem(U)$ is either equal to $U$ or disjoint from $U$ for all $\selem\in G$, and we define $G_U\eqdef\{\selem\in G\st \selem(U)=U\}$.
\end{enumerate}
\end{definition}

\begin{remark} \label{rem:BallsCompatible}
Every geodesic ball with sufficiently small radius is $G$-compatible.
\end{remark}

We now introduce the basic language about varifolds and, in particular, $G$-equivariant varifolds.

\begin{definition}
Given a varifold $V$ in $M$, we denote by $\norm{V}$ the Radon measure in $M$ associated with $V$.
Moreover, we say that $V$ in $M$ is a \emph{$G$-equivariant varifold} if $\selem_*V=V$ for all $\selem\in G$.
We denote by $\mathcal{V}^m_G(M)$ the set of $m$-dimensional $G$-equivariant varifolds supported in $M$, endowed with the weak topology. Recall that $\mathcal{V}^m_G(M)$ is metrizable and we denote by $\vard$ a metric metrizing it (see \cite{Pitts1981}*{pp. 66} or \cite{MarquesNeves2014}*{pp. 703}).

Observe that all the metrics on a compact manifolds are equivalent, hence the space $\mathcal{V}^m_G(M)$ does not depend on $G$. Therefore, we can also assume to fix a metric $\vard$ independently of the metric on $M$.
\end{definition}

\begin{remark} \label{rem:GequivVecFieldToFlow}
Given a $G$-equivariant vector field $X\in\vf_G(M)$, let $\Phi\colon [0,+\infty)\times M\to M$ be the associated flow. Then, if $V$ is a $G$-equivariant varifold, $(\Phi_t)_*V$ is $G$-equivariant as well for all $t\in[0,+\infty)$. Indeed, $\Phi\colon[0,+\infty)\times M\to M$ is a $G$-equivariant isotopy (see \cref{def:GEquivIsotopy}).
\end{remark}

\section{Orientation-preserving isometries} \label{sec:OrientationPreserving}

In most of the this part of the thesis, we require $M$ to have dimension $m+1=3$ and $G$ to be a finite group of \emph{orientation-preserving} isometries of $M$. The reason is that, in this case, the intersection of a $G$-equivariant surface with the singular locus, defined below, is `well-controlled', in the sense described in \cref{cor:IntersectionSingularLocus}.

Hence, let us consider here a finite group $G$ of orientation-preserving isometries on a compact Riemannian manifold $(M^3,\smetric)$ with boundary.
Note that, for all $x\in M$, $G_x$ acts on $T_xM$ as a finite subgroup of $\operatorname{SO}(3)$ and the finite subgroups of $\operatorname{SO}(3)$ are characterized as follows.

\begin{lemma} [{\cite{CooperHodgsonKerckhoff2000}*{Theorem~2.4}, see also \cite{Ketover2016Equivariant}*{Lemma 3.3}}] 
A finite nontrivial subgroup of $\operatorname{SO}(3)$ is cyclic $\Z_n$ for some $n\ge 2$, dihedral $\dih_n$ for some $n\ge 2$, or the group of rotational symmetries of a regular solid. Moreover, we have that
\begin{enumerate} [label={\normalfont(\arabic*)}]
\item for $n\not=2$, the groups $\Z_n$ and $\dih_n$ leave invariant only one plane;
\item $\Z_2$ fixes two orthogonal planes;
\item $\dih_2$ fixes three orthogonal planes;
\item the group of rotational symmetries of a regular solid does not fix any plane.
\end{enumerate}
\end{lemma}

\begin{remark} \label{rem:StructureSingLocus}
The previous lemma implies that the singular locus $\mathcal{S}$ can be decomposed as $\mathcal{S}_0\cup\mathcal{S}_1$, where $\mathcal{S}_0$ is a finite set and $\mathcal{S}_1$ is the union of curves consisting of points of isotropic type $\Z_n$ for some $n\ge2$.
\end{remark}

\begin{corollary} \label{cor:IntersectionSingularLocus}
Let $(M^3,\smetric)$ be a compact Riemannian manifold with boundary and let $G$ be a finite group of orientation-preserving isometries.
Moreover, let $\Sigma^2\subset M$ be a smooth properly embedded $G$-equivariant surface. Then $\Sigma\cap\mathcal{S}$ is the union of
\begin{enumerate} [label={\normalfont(\arabic*)}]
\item isolated points of isotropic type $\Z_n$ or $\dih_n$ for some $n\ge 2$;
\item\label{isl:curves} curves, consisting of points of isotropic type $\Z_2$;
\item meeting points of curves from the previous item, consisting of points of isotropic type $\dih_2$.
\end{enumerate}
\end{corollary}
\begin{proof}
The proof follows from the previous lemma (see also \cref{rem:StructureSingLocus}), together with \cite{Ketover2016Equivariant}*{Lemmas 3.4 and 3.5}.
\end{proof}

\begin{remark}
Define $\tilde{\mathcal{S}}_1\subset \Sigma$ as the set of points in the curves of item \ref{isl:curves} in the corollary above and $\tilde{\mathcal S}_0\eqdef (\Sigma\cap\mathcal{S})\setminus \tilde{\mathcal{S}}_1$. Then observe that $\tilde{\mathcal{S}}_i$ does not coincide with $\Sigma\cap\mathcal{S}_i$ for $i=0,1$. Indeed, a curve in $\mathcal{S}_1$ can intersect $\Sigma$ transversally, in which case the intersection belongs to $\tilde{\mathcal{S}}_0$.
\end{remark}

\section{A special case: the dihedral group} \label{sec:DihedralGroup}

As anticipated in the \hyperref[chpt:intro]{Introduction} and in \cref{intro:existence}, we develop the equivariant min-max procedure in a general ambient manifold subject to the action of a finite group of orientation-preserving isometries, then we specialize to the case where the ambient manifold is the three-dimensional unit ball $B^3$ and the finite group of orientation-preserving isometries is the dihedral group $\dih_{n}$.

The dihedral group $\dih_{n}$ is the symmetry group of a regular $n$-sided polygon. However, in our specific context, it is more convenient to define it as follows.

\begin{definition} \label{def:DihedralGroup}
Given $2\leq n\in\N$, we define the \emph{dihedral group} $\dih_{n}$ of order $2n$ to be the subgroup of Euclidean isometries (acting on $\overline{B^3}$) generated by the rotations of angle $\pi$ around the $n$ horizontal axes $\xi_k\vcentcolon=\{(r\cos(k\pi/n),r\sin(k\pi/n),0)\st r\in[-1,1]\}$ for $k\in\{1,\ldots,n\}$.
\end{definition}

Composing the rotations of angle $\pi$ around $\xi_1$ and $\xi_2$, we obtain a rotation of angle $2\pi/n$ around the vertical axis $\xi_0\vcentcolon=\{(0,0,r)\st r\in[-1,1]\}$.
Moreover, observe that the singular locus of $\dih_n$ is given by
\[\mathcal{S}=\xi_0\cup\xi_1\cup\ldots\cup\xi_n.\]
In particular, the origin $0=(0,0,0)\in\overline{B^3}$ is of isotropic type $\dih_n$, the points in $\xi_0\setminus\{0\}$ are of isotropic type $\Z_n$ and the points in $(\xi_1\cup\ldots\cup \xi_n)\setminus\{0\}$ are of isotropic type $\Z_2$.

\section{Structure of \texorpdfstring{$\dih_{g+1}$-equivariant surfaces}{surfaces equivariant with respect to the dihedral group}}\label{sec:structure}

In this section, we prove a couple of lemmas that are useful to control the topology of the surfaces resulting from the $\dih_{g+1}$-equivariant min-max procedure, in the proof of \cref{thm:main-fbms-b1}.
We present these result here, and not in \cref{chpt:FBMSb1}, because they work in general for any proper $\dih_{g+1}$-equivariant surface in $B^3$, with $g\ge1$, irrespective of minimality.

\begin{lemma}\label{lem:endpoints}
Let $\Gamma\subset B^3$ be any smooth, properly embedded, $\dih_{g+1}$-equivariant surface that contains the horizontal axes $\xi_1,\ldots,\xi_{g+1}$. 
Then, the endpoints \[q_k \vcentcolon= \left(\cos\left(\frac{k\pi}{g+1}\right),\sin\left(\frac{k\pi}{g+1}\right),0\right)  , \quad \text{for $k=1,\ldots,g+1$},\] of these axes are all contained in the same connected component of $\partial\Gamma$.
\end{lemma}

\begin{proof}
Given $k\in\{1,\ldots,g+1\}$, there exists a connected component $\sigma$ of $\partial\Gamma$ containing the point $q_k$ because $\xi_k\subset\Gamma$ and $\Gamma$ is properly embedded. 
Let $\psi_k\in\dih_{g+1}$ be the rotation of angle $\pi$ around $\xi_k$. 
Since $\partial\Gamma$ is $\dih_{g+1}$-equivariant, we have in particular $\psi_k(\sigma)\subset\partial\Gamma$. 
In fact, $\psi_k(\sigma)=\sigma$ because $\sigma$ is a connected component of $\partial\Gamma$ intersecting $\psi_k(\sigma)$ at least in the point $q_k=\psi_k(q_k)$.  
Moreover, as $\Gamma$ is properly embedded and smooth, $\sigma\subset\partial B^3$ must be a smooth, simple closed curve. 
Any such curve divides $\partial B^3$ into two connected open domains $A'$ and $A''$. 
We then note that $\psi_k$ leaves the set $\partial A'=\sigma=\partial A''$ invariant, and that $A'=\psi_k(A'')$.
It follows that $A'$ and $A''$ have the same area because $\psi_k$ is an isometry. 
Moreover, $A'=\psi_k(A'')$ implies that the antipodal point $q_{k+g+1}\in\xi_k$, which is fixed under $\psi_k$, must also be contained in $\sigma$ because it cannot be contained in $A'$ nor $A''$. 

Now suppose, for the sake of a contradiction, that the point $q_\ell$ belongs to a different connected component $\varsigma$ of $\partial\Gamma$, for some $\ell\in\{1,\ldots,g+1\}$ with $\ell\neq k$. 
Then either $\varsigma\subset A'$ or $\varsigma\subset A''$ because $\sigma$ and $\varsigma$ are disjoint by definition. 
However, the whole argument in the previous paragraph applies to $\varsigma$ as well. Yet, in either case (i.e., both when $\varsigma\subset A'$ and $\varsigma\subset A''$), it is impossible that $\varsigma$ divides $\partial B^3$ into two domains of equal area, and this concludes the proof.
\end{proof}

\begin{lemma}\label{lem:structural}
Given $1\leq g\in\N$, let $\Gamma\subset B^3$ be any compact, connected, properly embedded, $\dih_{g+1}$-equivariant surface of genus $\genus(\Gamma)\in\{1,\ldots,g\}$ containing the horizontal axes $\xi_k$ for $k\in\{1,\ldots,g+1\}$. 
Then $\Gamma$ has genus $\genus(\Gamma)=g$.  
\end{lemma}

\begin{proof}
In the case $g=1$, there is nothing to prove. 
Therefore, let us assume $g\geq2$. 
Since $\Gamma\subset B^3$ is compact, properly embedded, $\dih_{g+1}$-equivariant and contains the origin by assumption, the number of intersections with the vertical axis $\xi_0=\{(0,0,r)\st r\in[-1,1]\}$ is equal to $2j+1$ for some nonnegative integer $j$.  
The boundary $\partial\Gamma\subset\partial B^3$ is a collection of pairwise disjoint, simple closed curves, which means that the number $b\in\N$ of boundary components winding around $\xi_0$ is well-defined. 
By \cref{lem:endpoints}, the endpoints of $\xi_k$ are contained in the same boundary component for all $k\in\{1,\ldots,g+1\}$. 
Hence, by $\dih_{g+1}$-equivariance, $b$ is odd. 
Moreover, the number of boundary components which does not wind around $\xi_0$ is divisible by $(g+1)$. 
To summarise, $\partial\Gamma$ has $(g+1)\alpha+2\beta+1$ connected components, where $\alpha,\beta$ are nonnegative integers. 
In particular, the Euler characteristic of $\Gamma$ is given by 
\begin{align}\label{eqn:20200513-1}
\chi(\Gamma)=1-(g+1)\alpha-2\beta-2\genus(\Gamma).
\end{align}
Let $C_{n}<\dih_{n}$ be the cyclic subgroup of order $2\leq n\in\N$ generated by the rotation of angle ${2\pi}/{n}$ around the vertical axis $\xi_0$. 
The quotient $\Gamma'=\Gamma/C_{g+1}$ is an orientable topological surface with $\alpha+2\beta+1$ many boundary components. Hence, its Euler characteristic is given by 
\begin{align}\label{eqn:20200513-2}
\chi(\Gamma')=1-\alpha-2\beta-2g'
\end{align}
for some nonnegative integer $g'$. 
A suitable version of the Riemann--Hurwitz formula (see \cref{rem:Riemann-Hurwitz} below) implies 
\begin{align}\label{eqn:20200513-3}
\chi(\Gamma)
&=(g+1)\chi(\Gamma')
-(2j+1)g. 
\end{align}
Combining \eqref{eqn:20200513-1}, \eqref{eqn:20200513-2} and \eqref{eqn:20200513-3} and simplifying, we obtain   
\begin{align}\label{eqn:20200513-4} 
\genus(\Gamma)&=(g+1)g' + (\beta+j)g. 
\end{align} 
Since all variables in \eqref{eqn:20200513-4} are nonnegative integers and since $\genus(\Gamma)\leq g$ by assumption, we obtain $g'=0$ and $\beta+j=1$, hence $\genus(\Gamma)=g$ as claimed. 
\end{proof}

\begin{remark}[Riemann--Hurwitz formula, see e.\,g. Chapter IV.3 in \cite{Freitag2011}] \label{rem:Riemann-Hurwitz}
Let $\Gamma$ and $\Gamma'$ be as in the proof of \cref{lem:structural}. 
Let $T'$ be a triangulation of $\Gamma'$ such that every branch point of $\Gamma'$ is a vertex. 
Away from the branch points, the canonical projection $\Gamma\to\Gamma'$ is a covering map. 
Therefore, after refining $T'$ if necessary, the preimage of every triangle is a disjoint union of triangles in $\Gamma$ which leads to a triangulation $T$ of $\Gamma$. 
Note that $T$ has $(g+1)$ times as many faces and edges as $T'$. 
However, the $(2j+1)$ many branch points in $\Gamma'$ (which correspond to the points in $\Gamma\cap\xi_0$) have only one preimage rather than $(g+1)$. 
By the very definition of Euler characteristic we then have 
\begin{align*}
\chi(\Gamma)=(g+1)\chi(\Gamma')-(2j+1)g.
\end{align*}
It is appropriate to note that, since $\Gamma'$ is actually an orbifold, one could get to the same conclusions by invoking, in lieu of the Riemann--Hurwitz formula, a suitable version of the Gauss--Bonnet theorem for surfaces with conical singularities.
\end{remark}

\chapter[Equivariant spectrum of elliptic operators]{Equivariant spectrum of elliptic\\ operators} \label{chpt:EquivSpectrum}

In this chapter, we study some properties of an elliptic operator with mixed boundary conditions (i.e., Dirichlet in a portion of the boundary and Robin in the rest of the boundary) in presence of a symmetry group. 
Here, we work with a compact Riemannian manifold $(\Sigma^m,\smetric)$ with \emph{arbitrary} dimension $m\ge 1$.

\section{Existence of a discrete spectrum}

First of all, we prove existence of a discrete equivariant spectrum for an elliptic operator with mixed boundary conditions. The proof is very similar to the one in the case without equivariance (see for example the notes \cite{ArendtEtc2015}), but we report it here for completeness.

\begin{lemma} \label{lem:ImprovedTraceIneq}
Let $(\Sigma^m,\smetric)$ be a compact Riemannian manifold with boundary $\partial \Sigma\not=\emptyset$. Then, there exists a constant $C>0$ such that
\[
\norm{u}_{L^2(\partial\Sigma)}^2 \le C \norm{u}_{L^2(\Sigma)} \norm{u}_{H^1(\Sigma)}
\]
for all $u\in H^1(\Sigma)$.
\end{lemma}
\begin{proof}
The result is a slight variation of the standard trace inequality $\norm{u}_{L^2(\partial\Sigma)} \le C\norm{u}_{H^1(\Sigma)}$. Indeed, one can look at the proof of \cite{Brezis2011}*{Lemma 9.9} where $\Sigma=\R^m_+=\{(x_1,\ldots,x_m)\in\R^m\st x_m\ge 0\}$ and note that, integrating the second last line and applying Cauchy--Schwarz inequality, we get that there exists a constant $C>0$ such that
\[
\norm{u}^2_{L^2(\partial\Sigma)}  \le C \norm{u}_{L^2(\Sigma)} \norm{u}_{H^1(\Sigma)}
\]
for all $u\in C^1_c(\R^m)$. Then the proof of the lemma follows from a standard partition argument.
\end{proof}

\begin{lemma}\label{lem:CompactResolvent}
Let $V, H$ be Hilbert spaces such that there exists a compact (continuous) embedding $j\colon V\xhookrightarrow{d} H$ with dense image. Let $a\colon V\times V\to\R$ be a bounded symmetric $H$-elliptic form, i.e., assume that there exist $\omega\in\R$ and $c>0$ such that
\[
a(u,u) + \omega \norm{j(u)}^2_H \ge c \norm{u}^2_{V}
\]
for all $u\in V$. 
Moreover, let $A\colon D(A)\subset V\to H$ be the operator associated with the symmetric form $a$, i.e., given $x\in V$ and $y\in H$ we have $x\in D(A)$ and $Ax=y$ if and only if $a(x,u) = (y, j(u))_H$ for all $u\in V$. 
Then, the operator $A+\omega \id \colon D(A)\subset V\to H$ is invertible with bounded compact inverse $(A+\omega\id)^{-1}\colon H\to D(A)\hookrightarrow H$.
\end{lemma}
\begin{proof}
We briefly sketch the proof.
Let us consider the operator $b\colon V\times V\to \R$ given by $b(u,v) \eqdef a(u,v) + \omega (j(u),j(v))_H$.
Since $b$ is bounded and coercive, we can apply Lax--Milgram theorem (cf. \cite{Brezis2011}*{Corollary 5.8}) and obtain that $\mathcal B\colon V\to V^*$, defined as $\mathcal Bu(v) \eqdef b(u,v)$, is an isomorphism. Finally one can prove that the operator $j\circ \mathcal B^{-1}\circ k$, where $k\colon H\to V^*$ is given by $k(y) = (y, j(\cdot))_H$, coincides with $(A+\omega\id)^{-1}$.
\end{proof}

\begin{definition} \label{def:EquivFuncSpaces}
Given a compact connected Riemannian manifold $(\Sigma^m,\smetric)$, a finite group $G$ of isometries of $\Sigma$ and a multiplicative function $\operatorname{sgn}_\Sigma\colon G\to \{-1,1\}$, we define the spaces
\begin{align*}
C^\infty_G(\Sigma) &\eqdef \{u\in C^\infty(\Sigma) \st u\circ \selem = \operatorname{sgn}_\Sigma(\selem) u\ \ \forall \selem\in G\}, \\
L^2_G(\Sigma) &\eqdef \overline{C^\infty_G(\Sigma)}^{\norm{\cdot}_{L^2}}, \\
H^1_G(\Sigma) &\eqdef \overline{C^\infty_G(\Sigma)}^{\norm{\cdot}_{H^1}}.
\end{align*}
\end{definition}

\begin{remark}
Observe that $L^2_G(\Sigma)$ is a Hilbert space endowed with the scalar product $(u,v)_{L^2}=\int_\Sigma uv \de \Haus^m$ and $C^\infty_G(\Sigma)$ is dense in $L^2_G(\Sigma)$.
\end{remark}

\begin{theorem} \label{thm:DiscreteSpectrum}
Let $(\Sigma^m,\smetric)$ be a compact connected Riemannian manifold with nonempty boundary, $G$ be a finite group of isometries of $\Sigma$ and $\operatorname{sgn}_\Sigma\colon G\to \{-1,1\}$ be a multiplicative function. Let $\alpha\colon\Sigma\to\R$ and $\beta\colon\partial \Sigma\to\R$ be smooth $G$-equivariant functions, i.e., $\alpha\circ \selem = \alpha$ and $\beta\circ \selem = \beta$ for all $\selem\in G$, and let $\partial_D\Sigma\subset\partial\Sigma$ be a $G$-equivariant subset of the boundary of $\Sigma$.
Then, there exists an orthonormal basis $(\varphi_k)_{k\ge 1}\subset C^\infty_G({\Sigma})$ of $L^2_G(\Sigma)$ and an increasing sequence $(\lambda_k)_{k\ge 1}\subset \R$ converging to $+\infty$ such that
\[
\begin{cases}
-\Delta \varphi_k - \alpha\varphi_k = \lambda_k \varphi_k & \text{in $\Sigma$}\\
\varphi_k = 0 & \text{in $\partial_D\Sigma$}\\
\partial_\eta \varphi_k + \beta\varphi_k = 0 & \text{in $\partial_R\Sigma\eqdef\partial\Sigma\setminus\partial_D\Sigma$}.
\end{cases}
\]
\end{theorem}

\begin{remark}
Note that the connectedness of $\Sigma$ is instrumental for well-defining the spaces in \cref{def:EquivFuncSpaces}. However, one can work separately in each connected component to prove existence and properties of the spectrum of an elliptic operator.
\end{remark}

\begin{remark}
Observe that we consider `mixed' boundary terms, in the sense that we ask for Dirichlet boundary conditions in the portion $\partial_D\Sigma$ of the boundary and for Robin boundary conditions in the portion $\partial_R\Sigma=\partial\Sigma\setminus\partial_D\Sigma$. Note that, formally, Dirichlet boundary conditions arise from Robin boundary conditions in the limit as $\beta\to\infty$.
\end{remark}

\begin{proof}
Let us define the Hilbert space
\[
V\eqdef \{v\in H^1_G(\Sigma)\st \tr v = 0 \text{ on $\partial_D\Sigma$}  \}\subset H^1_G(\Sigma)
\]
and consider the quadratic form $a\colon V\times V\to \R$ defined as 
\begin{align*}
a(u,v) &\eqdef \int_\Sigma (\grad u \cdot \grad v - \alpha u v ) \de \Haus^m + \int_{\partial_R \Sigma} \beta uv \de \Haus^{m-1} \\
&= \int_\Sigma (\grad u \cdot \grad v - \alpha u v ) \de \Haus^m + \int_{\partial \Sigma} \beta uv \de \Haus^{m-1},
\end{align*}
where in the last equality we used that $\tr v =0$ on $\partial_D\Sigma=\partial\Sigma\setminus\partial_R\Sigma$.
Note that $a$ is continuous, because 
\begin{align*}
\abs{a(u,v)} &\le \norm{\grad u}_{L^2(\Sigma)} \norm{\grad v}_{L^2(\Sigma)} + \norm{\alpha}_{L^\infty(\Sigma)}\norm{u}_{L^2(\Sigma)} \norm{v}_{L^2(\Sigma)} +{}\\
&\pheq {} + \norm{\beta}_{L^\infty(\partial \Sigma)}\norm{\tr u}_{L^2(\partial \Sigma)} \norm{\tr v}_{L^2(\partial \Sigma)},
\end{align*}
and the trace $\tr\colon V\to L^2_G(\partial\Sigma)$ is continuous.
Moreover, $a$ is $L^2_G(\Sigma)$-elliptic, i.e.,
\[
a(u,u) +\omega \norm{u}_{L^2(\Sigma)}^2 \ge c \norm{u}^2_{H^1(\Sigma)}
\]
for all $u\in V$, for some $c>0$ and $\omega\in\R$. Here we used that \cref{lem:ImprovedTraceIneq} implies
\[
\norm{\beta}_{L^\infty} \int_{\partial\Sigma} \abs{u}^2\de\Haus^{m-1} \le c_1 \norm{u}_{L^2(\Sigma)} \norm{u}_{H^1(\Sigma)} \le \frac 12\norm{u}_{H^1(\Sigma)}^2 + \frac {c_1}{2} \norm{u}_{L^2(\Sigma)}^2
\]
for some $c_1>0$. 

Since the embedding $j\colon V\to H\eqdef L^2_G(\Sigma)$ is compact, we can thus apply \cref{lem:CompactResolvent} and obtain that $A+\omega\id\colon D(A)\subset V\to H$ is invertible with bounded compact inverse $(A+\omega\id)^{-1}\colon H\to D(A)\subset H$, where $A\colon D(A)\subset H\to H$ is the operator associated with the symmetric form $a$, as in the statement of the lemma.

Now we want to prove that $A=-\Delta-\alpha$ and $D(A)=D_R$, where
\begin{align*}
D_R\eqdef  \Big\{u\in V\st &\Delta u \in H,\\ &\int_\Sigma (\Delta u) v + \grad u\cdot \grad v\de \Haus^m =- \int_{\partial \Sigma} \beta u v \de\Haus^{m-1}\ \ \forall v\in V \Big\}.
\end{align*}
Let $u\in D(A)$ and $f\in H$ be such that $Au = f$, then
\[
\int_{\Sigma} (\grad u\cdot \grad v -\alpha uv) \de \Haus^m + \int_{\partial \Sigma} \beta uv \de \Haus^{m-1} = \int_\Sigma fv\de \Haus^m
\]
for all $v\in V$. In particular, if we consider $v\in C^\infty_G(\Sigma)$ with compact support in $\operatorname{int}(\Sigma)$, we get that $\Delta u\in H$ and $-\Delta u - \alpha u = f$. Exploiting this equation we obtain
\[
\int_\Sigma (\Delta u) v + \grad u\cdot \grad v\de \Haus^m =- \int_{\partial \Sigma} \beta u v \de\Haus^{m-1}
\]
for all $v\in V$. This proves that $u\in D_R$ and $Au=-\Delta u-\alpha u$.
Conversely, given $u\in D_R$, we can easily show that $u\in D(A)$.

Now observe that the bounded compact operator $(A+\omega\id)^{-1}\colon H\to H$ is positive definite.
Therefore, by the spectral theorem, there exists an orthonormal basis $(\varphi_k)_{k\in\N}$ of $H$ of eigenvectors of $(A+\omega\id)^{-1}$ with corresponding eigenvalues $(\mu_k)_{k\in\N}$ such that $\mu_0\ge \mu_1\ge\mu_2\ge \ldots\to0$. Then note that $\varphi_k\in D_R$, in particular $\varphi_k=0$ on $\partial_D\Sigma$, and it is an eigenvector of $A$ with corresponding eigenvalue $\lambda_k\eqdef 1/\mu_k-\omega$ for all $k\in\N$.
Finally, the fact that $\varphi_k\in C_G^\infty(\Sigma)$ follows from standard regularity theory and, from the fact that $\varphi_k\in D_R$, we obtain also that $\partial_\eta\varphi_k +\beta\varphi_k=0$ in $\partial_R\Sigma$.
\end{proof}

\begin{remark}
Observe that any \emph{$G$-equivariant} eigenfunction as in the statement of \cref{thm:DiscreteSpectrum} is actually an eigenfunction. However, in general, we cannot say to which nonequivariant eigenvalue the $k$-th equivariant eigenfunction corresponds.
\end{remark}

\section{Equivariant Courant's nodal domain theorem}

In this section we want to prove an equivariant version of Courant's nodal domain theorem. This is not directly used in the thesis, but we think it gives a better insight on how the equivariant spectrum behaves and it could be useful in further applications.

\begin{theorem} \label{thm:EquivariantCourant}
Let $\Sigma^m$ be a compact Riemannian manifold with nonempty boundary and let $G$ be a finite group of isometries of $\Sigma$. Let $\alpha\colon\Sigma\to\R$ and $\beta\colon\partial\Sigma\to\R$ be $G$-equivariant smooth functions, i.e., $\alpha\circ h=\alpha$ and $\beta\circ h =\beta$ for all $h\in G$, and let $\partial_D\Sigma\subset\partial\Sigma$ be $G$-equivariant. Let $\varphi_k\in C^\infty_G(\Sigma)$ be an eigenfunction of $-\Delta -\alpha$ with Dirichlet boundary conditions on $\partial_D\Sigma$ and with Robin boundary conditions $\partial_\eta + \beta = 0$ on $\partial_R\Sigma\eqdef\partial\Sigma\setminus\partial_D\Sigma$, associated to the $k$-th eigenvalue $\lambda_k$ for some $k\ge1$ (see \cref{thm:DiscreteSpectrum}). Namely
\[
\begin{cases}
-\Delta \varphi_k - \alpha\varphi_k = \lambda_k \varphi_k & \text{in $\Sigma$}\\
\varphi_k=0 & \text{in $\partial_D\Sigma$}\\
\partial_\eta \varphi_k + \beta\varphi_k = 0 & \text{in $\partial_R\Sigma$}.
\end{cases}
\]
Then, the number of nodal domains of $\varphi_k$ is at most $k \abs G$.
\end{theorem}

In order to prove the theorem, we need two preliminary lemmas.

\begin{lemma} \label{lem:AttachTwoLipschitz}
Let $X$ be a geodesic metric space and $C_1$, $C_2$ be closed subsets of $X$ such that $C_1\cup C_2=X$. Consider another metric space $Y$ and two Lipschitz functions $u_1\colon C_1\to Y$ and $u_2\colon C_2\to Y$ such that $u_1 = u_2$ in $C_1\cap C_2$.
Then, the function $u\colon X\to Y$ defined as
\[
u(x) = \begin{cases}
           u_1(x) & \text{for $x\in C_1$}\\
           u_2(x) & \text{for $x\in C_2$}
       \end{cases}
\]
is Lipschitz.
\end{lemma}

\begin{proposition}
Let $\Sigma^m$ be a compact Riemannian manifold, possibly with smooth boundary, and let $u\colon \Sigma\to\R$ be a smooth function. Consider an open subset $\Omega\subset\Sigma$ such that $u=0$ on $\partial\Omega\setminus\partial\Sigma$. Then 
\[
\int_{\Omega} u\Delta u \de \Haus^m+ \int_\Omega \abs{\grad u}^2 \de \Haus^m= \int_{\partial\Sigma\cap\Omega} u\partial_\eta u \de \Haus^{m-1}.
\]
\end{proposition}
\begin{proof}
Let us consider the vector field $X$ on $\Sigma$ defined as
\[
X = \begin{cases}
        u\grad u & \text{in $\Omega$}\\
        0 & \text{in $\Sigma\setminus\Omega$}.
       \end{cases}
\]
By \cref{lem:AttachTwoLipschitz}, $X$ is Lipschitz and therefore the divergence theorem holds
\[
\int_\Sigma\div(X) \de \Haus^m = \int_{\partial\Sigma} X \cdot \eta\de\Haus^{m-1}.
\]
However note that $\div(X) = 0$ almost everywhere in $\Sigma\setminus\Omega$ and $\div(X) = \div(u\grad u)$ in $\Omega$, which leads to the desired result.
\end{proof}

\begin{proof} [Proof of \cref{thm:EquivariantCourant}]
We follow the lines of the proof of Courant's theorem in \cite{SchoenYau1994}*{Chapter III Section 6}. Assume by contradiction that $\varphi_k$ has more than $k\abs G$ nodal domains $\Omega_1,\Omega_2,\ldots, \Omega_h$ with $h>k\abs G$.
Now define the $G$-equivariant subset $\tilde\Omega_1 \eqdef \bigcup_{h\in G} h(\Omega_1)$. Then, without loss of generality, we can assume that $\Omega_2\not\subset \tilde\Omega_1$ and define $\tilde\Omega_2 \eqdef \bigcup_{h\in G} h(\Omega_2)$. We repeat and find $\tilde\Omega_1,\tilde\Omega_2,\ldots,\tilde\Omega_l$ until $\Omega_i\subset \bigcup_{j=1}^l\tilde\Omega_j$ for all $i=1,\ldots,h$.
Note that $l>k$ since $\tilde \Omega_j$ consists of the disjoint union of at most $\abs{G}$ nodal domains for all $j=1,\ldots,l$.

At this point define the functions
\[
\varphi_k^j = \begin{cases}
                  \varphi_k & \text{in $\tilde\Omega_j$}\\
                  0 &\text{otherwise}
              \end{cases}
\]
and note that $\varphi_k^j\in L^2_G(\Sigma)$ for all $j=1,\ldots,l$.
Since we have that $\dim \operatorname{span}\{\varphi_k^1,\ldots,\varphi_k^k\} > \dim\operatorname{span}\{\varphi_1,\ldots,\varphi_{k-1}\}$, we can find a $G$-equivariant nontrivial linear combination $u\eqdef\sum_{j=1}^k a_j\varphi_k^j$ of $\varphi_k^1,\ldots,\varphi_k^k$ orthogonal to $\varphi_1,\ldots,\varphi_{k-1}$. Then, by the variational characterization of the eigenvalues, we get
\begin{align*}
\lambda_k \int_{\Sigma}u^2\de\Haus^m&\le \int_\Sigma (\abs{\grad u}^2 -\alpha u^2) \de \Haus^m + \int_{\partial_R\Sigma}\beta u^2\de\Haus^{m-1} \\
&= \sum_{j=1}^ka_j^2\int_{\Sigma} (\abs{\grad \varphi_k^j}^2-\alpha(\varphi_k^j)^2) \de \Haus^m + a_j^2\int_{\partial_R\Sigma}\beta (\varphi_k^j)^2\de\Haus^{m-1}\\
&=\sum_{j=1}^k a_j^2\int_{\tilde\Omega_j} (\abs{\grad \varphi_k}^2-\alpha\varphi_k^2) \de \Haus^m + a_j^2\int_{\partial_R\Sigma\cap\tilde\Omega_j}\beta \varphi_k^2\de\Haus^{m-1} \\
&= \lambda_k\sum_{j=1}^k a_j^2\int_{\tilde\Omega_j} \varphi_k^2 \de \Haus^m =\lambda_k\sum_{j=1}^k a_j^2\int_{\tilde\Omega_j} (\varphi_k^j)^2 \de \Haus^m = \lambda_k\int_\Sigma u^2\de\Haus^m.
\end{align*}
Therefore the inequality above is an equality and thus $u$ is an eigenfunction relative to the eigenvalue $\lambda_k$. However, this contradicts the fact that $u=0$ in the open set $\tilde\Omega_{k+1}$, by unique continuation.
\end{proof}

\begin{remark}
In the proof of Courant's theorem presented in \cite{SchoenYau1994}*{Chapter~III Section~6}, there is a mistake in Cheng's argument at pp. 123. In fact it is not true for $m\ge 3$ that, if $f\colon\R^m\to\R$ is a smooth function such that $f(x) = P_N(x)+\bigO(\abs{x}^{N+\varepsilon})$ for some homogeneous harmonic polynomial $P_N$ of degree $N$ and $0<\varepsilon<1$, then there exists a $C^1$-diffeomorphism $\phi$ such that $f(x) = P_N(\phi(x))$. A counterexample is given by $f(x,y,z) = xz-z^4$ and $P_2(x,y,z) = xy$; this is presented in \cite{BerardMeyer1982}*{Appendice E}.
A discussion on the regularity of the nodal domains for Robin boundary conditions in dimension $m=2$ can be found in \cite{GittinsHelffer2019}*{Appendix B}.
For a discussion about this issue see also Remark 2 in \cite{Chavel1984}*{Chapter I, Section 5}.
\end{remark}

\chapter{Equivariant free boundary minimal surfaces} \label{chpt:EquivFBMS}

In this chapter, we present some preliminaries about free boundary minimal hypersurfaces that are invariant with respect to the action of a finite group of isometries of the ambient manifold $(M^{m+1},\smetric)$, which has \emph{arbitrary} dimension $m+1\ge 3$.

\section{Equivariant Morse index} \label{sec:EquivSpectrum}

In this section, first we observe that being free boundary stationary with respect to $G$-equivariant variations is equivalent to being free boundary stationary, then we introduce the notion of $G$-equivariant index.
Throughout the section, we assume that $(M^{m+1},\smetric)$ is a compact Riemannian manifold with boundary and $G$ is a finite group of isometries of $M$.

\begin{definition}
Given a $G$-equivariant varifold $V\in\mathcal{V}^m_G(M)$, we say that $V$ is \emph{free boundary $G$-stationary} if $\delta V(X)=0$ for all $G$-equivariant vector fields $X\in \vf_G(M)$.
\end{definition}

\begin{remark} \label{rem:stationaryGstationary}
By the principle of symmetric criticality by Palais (see \cite{Palais1979}, or \cite{Ketover2016Equivariant}*{Lemma 3.8} for the result in this setting), we have that $V\in\mathcal{V}^m_G(M)$ is free boundary $G$-stationary if and only if it is free boundary stationary.
\end{remark}

\begin{definition} \label{def:GequivIndex}
Let $\Sigma^m\subset M$ be a $G$-equivariant free boundary minimal hypersurface in $M$ and let $\Gamma_G(N\Sigma)$ denote the sections of the normal bundle of $\Sigma$ obtained as restriction to $\Sigma$ of $G$-equivariant vector fields in $M$.
Then, the \emph{$G$-equivariant (Morse) index} $\ind_G(\Sigma)$ of $\Sigma$ is defined as the maximal dimension of a linear subspace of $\Gamma_G(N\Sigma)$ where $Q^\Sigma$ (defined in \cref{sec:SecondVariation}) is negative definite.
\end{definition}

\begin{remark} \label{rem:GequivExtension}
Note that, given a $G$-equivariant submanifold $\Sigma$, any $G$-equivariant vector field defined along $\Sigma$ can be extended to a $G$-equivariant vector field on the ambient manifold $M$. Indeed, we can first extend it to a vector field $\tilde X\in\vf(M)$ (not necessarily $G$-equivariant) and then define \[X\eqdef\frac{1}{\abs{G}} \sum_{\selem\in G} \selem_*\tilde X.\]
\end{remark}

\subsection{Equivariant index for two-sided hypersurfaces} \label{sec:twoside}

Given a $G$-equivariant, \emph{two-sided}, free boundary minimal hypersurface $\Sigma^m\subset M^{m+1}$ and a $G$-equivariant section $Y\in\Gamma_G(N\Sigma)$ of the normal bundle, we can write $Y=u\nu$, where $u\in C^\infty(\Sigma)$ and $\nu$ is a choice of unit normal to $\Sigma$.
Then, by \cref{sec:SecondVariation}, we know that the second variation of the volume of $\Sigma$ along $Y$ is given by
\begin{align*}
Q^\Sigma(Y,Y)=Q_\Sigma(u,u) = -\int_\Sigma u \jac_\Sigma u \de \Haus^m + \int_{\partial \Sigma}(u\partial_\eta u + \II^{\partial M}(\nu,\nu) u^2 ) \de \Haus^{m-1},
\end{align*}
where $\jac_\Sigma \eqdef \lapl + \Ric_M(\nu,\nu) + \abs{A}^2$ is the Jacobi operator associated to $\Sigma$.

Now, let us further assume that $\Sigma$ is connected. Note that, if $Y=u\nu$ is $G$-equivariant, for all $h\in G$ the fact that $\selem_*(u\nu) = u\nu$ implies
\[
u(\selem(x))\nu(\selem(x)) =  \d \selem_x[u(x)\nu(x)] = u(x) \d \selem_x[\nu(x)] = \operatorname{sgn}_\Sigma(\selem) u(x) \nu(\selem(x)),
\]
where $\operatorname{sgn}_\Sigma(\selem) = 1$ if $\selem_*\nu = \nu$ and $\operatorname{sgn}_\Sigma(\selem)=-1$ if $\selem_*\nu = -\nu$. Indeed, $h_*\nu $ is equal to $\nu$ or $-\nu$, because $h$ is an isometry and $h(\Sigma)=\Sigma$ (hence $h_*(N\Sigma)=N\Sigma$), and $\Sigma$ is connected (thus the sign does not depend on the point on $\Sigma$). 
Note that $\operatorname{sgn}_\Sigma(\selem_1\circ\selem_2) = \operatorname{sgn}_\Sigma(\selem_1)\operatorname{sgn}_\Sigma(\selem_2)$.

In particular, we can define $C^\infty_G(\Sigma)$, $L^2_G(\Sigma)$ and $H^1_G(\Sigma)$ as in \cref{def:EquivFuncSpaces} and get that the $G$-equivariant (Morse) index of $\Sigma$ coincides with the maximal dimension of a subspace of $C^\infty_G(M)$ where $Q_\Sigma$ is negative definite.
Moreover, thanks to \cref{thm:DiscreteSpectrum}, the elliptic problem
\[
\begin{cases}
-\jac_\Sigma\varphi = \lambda \varphi & \text{in $\Sigma$}\\
\partial_\eta \varphi = -   \II^{\partial M}(\nu,\nu)\varphi & \text{in $\partial\Sigma$}
\end{cases}
\]
admits a discrete spectrum $\lambda_1\le \lambda_2\le \ldots\le \lambda_k\le \ldots\to+\infty$ with associated $L^2_G(\Sigma)$-orthonormal basis of eigenfunctions $(\varphi_k)_{k\ge 1}\subset C^\infty_G(\Sigma)$ of $L^2_G(\Sigma)$. The $G$-equivariant index equals the number of negative eigenvalues.

\begin{remark}
Note that the assumption of connectedness of $\Sigma$ is not restrictive. Indeed, if $\Sigma$ consists of several connected components $\Sigma_1,\ldots,\Sigma_k$ for some $k\ge 2$, then we can consider each connected component separately and $\ind_G(\Sigma) = \sum_{i=1}^k\ind_G(\Sigma_i)$.
\end{remark}

\subsection{Equivariant index for one-sided hypersurfaces}
Given a $G$-equivariant, \emph{one-sided}, free boundary minimal hypersurface $\Sigma^m\subset M^{m+1}$, the elliptic problem
\begin{equation}\label{eq:OneSidedProblem}
\begin{cases}
-\lapl_\Sigma^\perp Y -\Ric_M^\perp(Y,\cdot) - \abs{A}^2Y = \lambda Y & \text{in $\Sigma$}\\
\grad^\perp_\eta Y = -   (\II^{\partial M}(Y,\cdot))^\sharp & \text{in $\partial\Sigma$},
\end{cases}
\end{equation}
on the $G$-equivariant sections $\Gamma_G(N\Sigma)$ of the normal bundle, admits a discrete spectrum $\lambda_1\le \lambda_2\le \ldots\le \lambda_k\le \ldots\to+\infty$ and the $G$-equivariant index coincides with the number of negative eigenvalues. 

This follows from the two-sided case applied to the double cover of $\Sigma$.
Indeed, let us consider the double cover $\pi\colon\tilde\Sigma\to\Sigma$, then every isometry $\selem\in G$ lifts to two isometries $\tilde \selem_1, \tilde \selem_2$ of $\tilde\Sigma$, where $\tilde \selem_1\circ\tilde \selem_2^{-1}$ is given by the isometric involution $i\colon\tilde\Sigma\to\tilde\Sigma$ associated to the universal cover. We denote by $\tilde G$ the finite group generated by these isometries.
Every $Y\in\Gamma_G(N\Sigma)$ lifts to a vector field $\tilde Y\in\Gamma(N\tilde\Sigma)$ and we can write it as $\tilde Y= u\tilde\nu$, where $\tilde \nu$ is a global unit normal to $\tilde\Sigma$ and $u\in C^\infty_{\tilde G}(\tilde\Sigma)$.
Vice versa, for all $u\in C^\infty_{\tilde G}(\tilde\Sigma)$, the vector field $Y\eqdef \pi_*(u\tilde\nu)$ is well-defined and belongs to $\Gamma_G(N\Sigma)$. Hence we can just apply \cref{thm:DiscreteSpectrum} to $\tilde\Sigma$ and $\tilde G$ (as we did in the previous subsection) to obtain the desired properties on the spectrum of problem \eqref{eq:OneSidedProblem}.

\section[FBMS with bounded equivariant index]{Free boundary minimal surfaces with bounded equivariant index} \label{sec:ConvBoundGIndex}

In this section, we restrict again to the case of a three-dimensional compact Riemannian manifold $(M^3,\smetric)$ with boundary satisfying property \hyperref[HypP]{$(\mathfrak{P})$} and we consider a finite group of \emph{orientation-preserving} isometries of $M$, as in \cref{sec:OrientationPreserving}.
Roughly speaking, we prove that a limit of free boundary minimal surfaces with bounded equivariant index satisfies the same bound on the equivariant index.

Given a free boundary minimal surface $\Sigma^2\subset M$ and an open subset $U\subset M$, we let $\mu_1(\Sigma\cap U)$ denote the first eigenvalue of the problem
\begin{equation} \label{eq:LocalEigenvProb}
\begin{cases}
-\lapl_{\Sigma}^\perp Y -\Ric_M^\perp(Y,\cdot) - \abs{A}^2Y = \mu Y & \text{in $\Sigma\cap U$}\\
\grad^\perp_\eta Y = -   (\II^{\partial M}(Y,\cdot))^\sharp & \text{in $\partial\Sigma \cap U$}\\
Y=0 & \text{in $\partial(\Sigma\cap U)\setminus \partial \Sigma$}
\end{cases}
\end{equation}
on the sections $\Gamma(N(\Sigma\cap U))$ of the normal bundle. Furthermore, if $\Sigma$ and $\Sigma\cap U$ are $G$-equivariant, let $\lambda_1(\Sigma\cap U)$ be the first eigenvalue of the same problem \eqref{eq:LocalEigenvProb} on the \emph{$G$-equivariant} sections $\Gamma_G(N(\Sigma\cap U))$ of the normal bundle (see \cref{sec:EquivSpectrum} and \cref{thm:DiscreteSpectrum}).

\begin{remark}
We stress the fact that we use the letters $\lambda$ and $\mu$ to denote the eigenvalues respectively \emph{taking} and \emph{not taking} into account the symmetry group.
\end{remark}

\begin{lemma}[{cf. \cite{Ketover2016Equivariant}*{Proposition 4.6}}] \label{lem:FirstEigenvalues}
In the setting above, let $\Sigma^2\subset M$ be a $G$-equivariant free boundary minimal surface in $M$ and fix $x\in\Sigma\setminus\mathcal{S}$.  
Then, there exists $\varepsilon_0>0$ such that, for all $0<\varepsilon<\varepsilon_0$, the $G$-equivariant subset $U_\varepsilon(x)\eqdef \bigcup_{\selem\in G} \selem(B_\varepsilon(x))$ consists of exactly $\abs{G}$ disjoint balls $\{B_\varepsilon(\selem(x))\}_{\selem\in G}$, and it holds
\[
\mu_1(\Sigma\cap B_\varepsilon(x)) = \mu_1(\Sigma\cap U_\varepsilon(x)) = \lambda_1(\Sigma\cap U_\varepsilon(x)).
\]
\end{lemma}
\begin{remark}
Observe that the result is not true if $x\in\mathcal{S}$.
\end{remark}

\begin{proof}
First, note that $B_\varepsilon(\selem(x)) = \selem(B_\varepsilon(x))$ since every $\selem\in G$ is an isometry. Moreover, if $\varepsilon_0$ is sufficiently small (in particular smaller than half of the distance between $x$ and $\selem(x)$ for all $\selem\in G$),
then we have that $\selem(B_\varepsilon(x)) \cap B_\varepsilon(x) = B_\varepsilon(h(x)) \cap B_\varepsilon(x) \not= \emptyset$ for some $\selem\in G$ if and only if 
$h(x)=x$, which implies that $h=\id$ because $x\not\in\mathcal{S}$.
In particular, we get that the subset $U_\varepsilon(x)\eqdef \bigcup_{\selem\in G} \selem(B_\varepsilon(x))$ consists of exactly $\abs{G}$ disjoint balls.

Now, we let $\varepsilon>0$ be sufficiently small such that $\Sigma\cap B_\varepsilon(x)$ is two-sided. Then, we can write any section $X\in \Gamma(N(\Sigma\cap B_\varepsilon(x))$ of the normal bundle as $X = u\nu$, where $u\in C^\infty(\Sigma\cap B_\varepsilon(x))$ and $\nu$ is a choice of unit normal, as in \cref{sec:twoside}.
Consider the first nonnegative eigenfunction $\varphi\in C^\infty(\Sigma\cap B_\varepsilon(x))$ relative to the eigenvalue $\mu_1(\Sigma\cap B_\varepsilon(x))$, and define $\tilde\varphi\in C^\infty(\Sigma\cap U_\varepsilon(x))$ as
\[
\tilde\varphi(\selem(x)) \eqdef \operatorname{sgn}_\Sigma(\selem) \varphi(x) 
\]
for all $\selem\in G$. Note that $\tilde\varphi$ is well-defined and $G$-equivariant, i.e., $\tilde\varphi\in C^\infty_G(\Sigma\cap U_\varepsilon(x))$.
As a result, since $\tilde\varphi\in C^\infty_G(\Sigma\cap U_\varepsilon(x))$ is nonnegative and $G$-equivariant, we obtain that $\tilde\varphi$ is the first eigenfunction relative to the eigenvalue $\mu_1(\Sigma\cap U_\varepsilon(x))$, and also relative to $\lambda_1(\Sigma\cap U_\varepsilon(x))$.
This implies that the three numbers $\mu_1(\Sigma\cap B_\varepsilon(x))$, $\mu_1(\Sigma\cap U_\varepsilon(x))$ and $\lambda_1(\Sigma\cap U_\varepsilon(x))$ actually coincide.
\end{proof}

\begin{theorem} \label{thm:ConvBoundGIndex}
In the setting above, let $\{\Sigma_k\}_{k\in\N}$ be a sequence of $G$-equivariant free boundary minimal surfaces in $M$, with uniformly bounded area, such that $\ind_G(\Sigma_k)\le n$ for some fixed $n\in\N$. Then (up to subsequence) $\Sigma_k$ converges locally graphically and smoothly, possibly with multiplicity, to a free boundary minimal surface $\tilde\Sigma\subset M\setminus(\mathcal{S}\cup \mathcal{Y})$ in $M\setminus(\mathcal{S}\cup \mathcal{Y})$, where $\mathcal{S}\subset M$ is the singular locus of $G$ and $\mathcal{Y}$ is a finite subset of $M$ with $\abs{\mathcal{Y}}\le n\abs{G}$.
Furthermore, if there exists a $G$-equivariant free boundary minimal surface $\Sigma\subset M$ such that $\tilde\Sigma = \Sigma\setminus (\mathcal{S}\cup\mathcal{Y})$ (namely if $\tilde \Sigma$ extends smoothly to $M$), then $\ind_G(\Sigma)\le n$.
\end{theorem}
\vspace{-1.2ex}
\begin{remark} \label{rem:AlsoForConvMetrics}
The previous theorem holds even when $\Sigma_k$ is a free boundary minimal surface with respect to a $G$-equivariant Riemannian metric $\smetric_k$ on $M$, for all $k\in\N$, and the sequence $\{\smetric_k\}_{k\in\N}$ converges smoothly to $\smetric$.
\end{remark}
\begin{proof}
Let us consider the set
\[
\mathcal{Y}\eqdef  \left\{x\in M \setminus\mathcal{S}\st \forall \varepsilon>0,\ \limsup_{k\to\infty} \mu_1(\Sigma_k\cap B_\varepsilon(x)) < 0\right\}\subset M\setminus \mathcal{S}.
\]
By Theorems 18 and 19 in \cite{AmbrozioCarlottoSharp2018Compactness}, having uniformly bounded area, the surfaces $\Sigma_k$ converge (up to subsequence) locally graphically and smoothly in $M\setminus(\mathcal{S}\cup\mathcal{Y})$ with finite multiplicity to a free boundary minimal surface $\tilde\Sigma\subset M\setminus(\mathcal{S}\cup\mathcal{Y})$.

Assume by contradictions that $\mathcal{Y}$ contains $n'\eqdef n\abs{G}+1$ distinct points $x_1,\ldots,x_{n'}\in M\setminus\mathcal{S}$. Then, there exists $\varepsilon>0$ sufficiently small, such that the balls $B_\varepsilon(x_1),\ldots,B_\varepsilon(x_{n'})\subset M\setminus\mathcal{S}$ are disjoint and (up to subsequence) $\mu_1(\Sigma_k\cap B_\varepsilon(x_i))<0$ for all $k$ sufficiently large and $i=1,\ldots,n'$.
Now, defining the $G$-equivariant subsets $U_i =U_\varepsilon(x_i)\eqdef \bigcup_{\selem\in G} \selem(B_\varepsilon(x_i))$ for all $i=1,\ldots,n'$ as in \cref{lem:FirstEigenvalues} (taking $\varepsilon$ possibly smaller), we have that $\lambda_1(\Sigma_k\cap U_i) = \mu_1(\Sigma_k\cap B_\varepsilon(x_i))<0$.
Moreover, $U_i\cap U_j\not= \emptyset$ for some $i,j\in\{1,\ldots,n'\}$ if and only if $U_i=U_j$. Note that, fixed $i\in\{1,\ldots,n'\}$, there are at most $\abs{G}$ values of $j\in\{1,\ldots,n'\}$ such that $U_i=U_\varepsilon(x_i) = U_\varepsilon(x_j) = U_j$. Therefore, at least $n+1 = \lceil \frac{n'}{\abs{G}}\rceil$, say $U_1,\ldots,U_{n+1}$, of the sets $U_1,\ldots, U_{n'}$ are disjoint.
In particular, this contradicts the assumption $\ind_G(\Sigma_k) \le n$, since $\Sigma_k\cap U_1,\ldots,\Sigma_k\cap U_{n+1}$ are disjoint $G$-equivariant $G$-unstable subsets of $\Sigma_k$.  
Hence, we have that $\abs{\mathcal{Y}}\le n\abs{G}$.

Let us now suppose that $\tilde\Sigma$ extends smoothly to $\Sigma$ in $M$, and assume by contradiction that $\ind_G(\Sigma)>n$. Then, there exist $G$-equivariant vector fields $X_1,\ldots,X_{n+1}\in\Gamma_G(N\Sigma)$ on $\Sigma$ such that $\sum_{i=1}^{n+1}a_i X_i$ is a negative direction for the second variation of the area of $\Sigma$ for all $(a_1,\ldots,a_{n+1})\in \R^{n+1}\setminus\{0\}$. 
Since the isometries in $G$ are orientation-preserving, $\Sigma\cap\mathcal{S}$ consists of a finite union $\tilde{\mathcal{S}}_0$ of isolated points and of a finite union $\tilde{\mathcal{S}}_1=(\Sigma\cap\mathcal{S})\setminus\tilde{\mathcal{S}}_0$ of smooth curve segments (cf. \cref{cor:IntersectionSingularLocus}). 
In particular, the points on $\tilde{\mathcal{S}}_1$ are of type $\Z_2$, therefore for all $x\in\tilde{\mathcal{S}}_1$ there exists $\selem\in G$ such that $h(x)=x$ and $h_*\nu = -\nu$, where $\nu$ is a choice of unit normal to $\Sigma$ at the point $x$.
Therefore, by equivariance of $X_1,\ldots,X_{n+1}\in\Gamma_G(N\Sigma)$, we have that $X_i = 0$ on $\tilde{\mathcal{S}}_1$ for all $i=1,\ldots,n+1$.
As a result, thanks to a standard cutoff argument, we can assume without loss of generality that $X_1,\ldots,X_{n+1}$ are compactly supported in $M\setminus(\mathcal{S}\cup \mathcal{Y})$. Indeed, we can first suppose that $X_1,\ldots,X_{n+1}$ are compactly supported in $M\setminus\tilde{\mathcal{S}}_1$, because every continuous function on a surface that is zero along a smooth curve can be approximated in the $H^1$-norm with functions that are zero in a neighborhood of such curve (see e.g. the proof of Theorem 2 in \cite{Evans2010PDE}*{Section 5.5}).
Secondly, we can suppose that $X_1,\ldots,X_{n+1}$ are zero in a neighborhood of the finite set $\tilde{\mathcal{S}}_0\cup\mathcal{Y}$ using a standard log-cutoff argument. Observe that all these operations can be made in an equivariant way, because both $\tilde{\mathcal{S}}_1$ and $\tilde{\mathcal{S}}_0\cup\mathcal{Y}$ are $G$-equivariant.
Since $\Sigma_k$ smoothly converges (possibly with multiplicity) to $\Sigma$ in $M\setminus (\mathcal{S}\cup \mathcal{Y})$, we thus get that $\ind_G(\Sigma_k) > n$ for $k$ sufficiently large, which contradicts our assumption. This proves $\ind_G(\Sigma)\le n$, as desired.
\end{proof}

\begin{corollary} \label{cor:ConvGStableSurf}
In the setting above, let $\{\Sigma_k\}_{k\in\N}$ be a sequence of $G$-equivariant, $G$-stable, free boundary minimal surfaces in $M$, with uniformly bounded area and uniformly bounded genus. Then (up to subsequence) $\Sigma_k$ converges in the sense of varifolds to a $G$-equivariant, $G$-stable, free boundary minimal surface $\Sigma$. Moreover, the convergence is smooth away from finitely many points on the singular locus $\mathcal{S}$.
\end{corollary}
\begin{proof}
Thanks to Lemmas 1 and 2 in \cite{Ilmanen1998} (see also \cite{EjiriMicallef2008}, or \cite{Lima2017} for the free boundary case), $\Sigma_k$ converges in the sense of varifolds to a $G$-equivariant free boundary minimal surface $\Sigma$ and the convergence is smooth away from finitely many points. We can assume that these points lie on the singular locus $\mathcal{S}$, because $G$-stability of $\Sigma_k$ implies curvature estimate away from $\mathcal{S}$ and therefore (up to subsequence) we have smooth convergence away from the singular locus. Finally, $\Sigma$ is $G$-stable by \cref{thm:ConvBoundGIndex}.
\end{proof}

\begin{remark} \label{rem:GStableAsStable}
Thanks to the previous corollary, in most of the arguments we can treat $G$-stable free boundary minimal surfaces as stable free boundary minimal surfaces. Indeed, an important property of stable free boundary minimal surfaces is that they are compact with respect to the smooth converge. This property is replaced by \cref{cor:ConvGStableSurf} in the case of $G$-stable free boundary minimal surfaces. Observe that the convergence is not smooth everywhere, but this is sufficient in the applications.
\end{remark}

\section{Bumpyness is \texorpdfstring{$G$}{G}-generic} \label{sec:BumpyIsGGeneric}

Let $M^{m+1}$ be a smooth, compact, connected manifold with boundary and let $G$ be a finite group of diffeomorphisms of $M$.
The aim of this section is to prove that, for a generic choice of a $G$-equivariant metric, $M$ contains countably many $G$-equivariant free boundary minimal hypersurfaces.

\begin{definition} [{cf. \cite{AmbrozioCarlottoSharp2018Compactness}*{Theorem 9}}] \label{def:BumpyMetrics}
Let $q$ be a positive integer $\ge 3$ or $q=\infty$. Denote by $\Gamma^q_G$ be the set of $G$-equivariant $C^q$ metrics on $M$ endowed with the $C^q$ topology. Moreover, let $\mathcal B^q_G\subset \Gamma^q_G$ be the subset of metrics $\smetric\in \Gamma^q_G$ such that no compact, smooth, $G$-equivariant manifolds with boundary that are $C^q$ properly embedded as free boundary minimal hypersurfaces in $(M,\smetric)$, and no finite covers of any such hypersurface, admit a nontrivial Jacobi field.
\end{definition}
\begin{remark}
Note that we \emph{do not} require the Jacobi field in the previous definition to be $G$-equivariant.
\end{remark}

\begin{proposition} \label{prop:CountableGequivSurfaces}
Let $\gamma\in \mathcal{B}_G^\infty$ be such that $(M,\smetric)$ satisfies property \hyperref[HypP]{$(\mathfrak{P})$}. Then $(M,\smetric)$ contains countably many $G$-equivariant free boundary minimal hypersurfaces.
\end{proposition}
\begin{proof}
This follows from Theorem 5 in \cite{AmbrozioCarlottoSharp2018Compactness}, similarly to Corollary 8 therein.
\end{proof}

\begin{theorem} \label{thm:BumpyIsGGeneric}
Let $q$ be a positive integer $\ge 3$ or $q=\infty$.
The subset $\mathcal B^q_G$ defined in \cref{def:BumpyMetrics} is comeagre in $\Gamma^q_G$.
\end{theorem}

\begin{remark}
In \cite{White2017}*{Theorem 2.1}, White proved a stronger result in the closed case, namely that a generic, $G$-equivariant, $C^q$ Riemannian metric on a smooth closed manifold is bumpy in the following sense: no closed, minimal \emph{immersed} submanifold has a nontrivial Jacobi field. We decided to state \cref{thm:BumpyIsGGeneric} only for properly embedded $G$-equivariant submanifolds, since White's generalization requires more technical work and we do not need it here. However, a similar proof should works also in the case with boundary.
\end{remark}

Before proceeding to the proof of \cref{thm:BumpyIsGGeneric}, we need to introduce some notation, which are the adaptations to the $G$-equivariant setting to the ones given in \cite{AmbrozioCarlottoSharp2018Compactness}*{pp. 22}.
Consider a compact, connected, smooth manifold $\Sigma^m$ with boundary and fix any $\alpha\in (0,1)$. For $w\in C^{q-1,\alpha}(\Sigma,M)$, let
\[
[w]\eqdef \{w\circ\varphi \st \varphi\in \operatorname{Diff}(\Sigma)\}.
\]
Then define
\begin{align*}
\mathcal{PE}^q_G \eqdef \{[w] \st &w\in C^{q-1,\alpha}(\Sigma,M) \text{ is a proper embedding, $w(\Sigma)$ is $G$-equivariant}\}
\end{align*}
and
\begin{align*}
\mathcal{S}^q_G \eqdef \{(\gamma,[w]) \in \Gamma^q_G\times \mathcal{PE}^q_G \st &w \text{ is a free boundary minimal} \\ &\text{proper embedding w.r.t. $\gamma$} \}.
\end{align*}
Finally denote by $\pi^q_G\colon \mathcal{S}^q_G\to \Gamma^q_G$ the projector onto the first factor, namely $\pi^q(\gamma,[w]) = \gamma$.

Given these definitions, we can now state the main ingredient to prove \cref{thm:BumpyIsGGeneric}. This is a structure theorem, whose first version was discovered by White in \cite{White1987Space}, stating the relation between critical points of $\pi^q$ and degenerate minimal hypersurfaces.

\begin{theorem} [Structure Theorem, cf. \cite{White2017}*{Theorem 2.3} and \cite{AmbrozioCarlottoSharp2018Compactness}*{Theorem 35}] \label{thm:StructureThm}
In the setting described above, $\mathcal S^q_G$ is a separable $C^1$ Banach manifold and $\pi^q_G\colon \mathcal{S}^q_G\to \Gamma^q_G$ is a $C^1$ Fredholm map of Fredholm index $0$.
Furthermore, $(\gamma, [w])$ is a critical point of $\pi^q_G$ if and only if $w(\Sigma)$ admits a nontrivial $G$-equivariant Jacobi field.
\end{theorem}

\begin{remark}
Theorem 2.3 in \cite{White2017} is the version of the theorem in the case without boundary, while Theorem 35 in \cite{AmbrozioCarlottoSharp2018Compactness} is the case with boundary but where no group of symmetries is considered.
\end{remark}

\begin{proof}
As observed by White in the proof of \cite{White2017}*{Theorem 2.3}, the proof of the Structure Theorem works the same in the equivariant case replacing `metric' with `$G$-equivariant metric', `functions' with `$G$-equivariant functions' and working on a $G$-equivariant hypersurface instead of on any hypersurface. In fact, the reader can follow the proof of the Structure Theorem in the case with boundary contained in \cite{AmbrozioCarlottoSharp2018Compactness}*{Section 7.2} with these modifications. In particular note that:
\begin{itemize}
\item Here we used $M^{m+1}$ for the ambient manifold and $\Sigma^m$ for the embedded hypersurface, instead of $\mathcal N^{n+1}$ and $M^n$ as in \cite{AmbrozioCarlottoSharp2018Compactness}.
\item The background metric $\gamma_*$ shall be chosen $G$-equivariant. This way, the exponential map with respect to $\gamma_*$ induces a diffeomorphism between $V^r$ (as defined in \cite{AmbrozioCarlottoSharp2018Compactness}*{pp. 25}) and a $G$-equivariant open neighborhood of $w(\Sigma)$. Then a class $[w]$ such that $w\in C^{q-1,\alpha}(\Sigma,M)$ is a proper embedding with $G$-equivariant image corresponds to a function $u\in C_G^{q-1,\alpha}(\Sigma,V)$ with $G$-equivariant symmetry (cf. \cref{def:EquivFuncSpaces}). 
\item As observed in \cref{rem:stationaryGstationary}, being free boundary $G$-stationary is the same as being free boundary stationary. Hence \cite{AmbrozioCarlottoSharp2018Compactness}*{Proposition 41} does not require modifications.
\item The computations in \cite{AmbrozioCarlottoSharp2018Compactness}*{Proposition 45} are the same. The only change that we need is in point (2), where the operator $L(\gamma,u)$ should be consider between the spaces $C^{q-1,\alpha}_G(\Sigma,V)$ and $C^{q-3,\alpha}_G(\Sigma,V)\times C^{q-2,\alpha}_G(\partial\Sigma,V)$. Then, similarly as in \cite{AmbrozioCarlottoSharp2018Compactness}, one can prove that $L(\gamma,u)\colon C^{q-1,\alpha}_G(\Sigma,V)\to C^{q-3,\alpha}_G(\Sigma,V)\times C^{q-2,\alpha}_G(\partial\Sigma,V)$ is a Fredholm operator of Fredholm index $0$.
\item Finally, \cite{AmbrozioCarlottoSharp2018Compactness}*{Proposition 46} and the final proof of Structure Theorem itself (at \cite{AmbrozioCarlottoSharp2018Compactness}*{pp. 31}) can be modified accordingly to fit the equivariant setting. Just note that the functions in $\ker L(\gamma, u)$, where $L(\gamma,u)$ is seen as a map between the equivariant spaces, correspond exactly to the $G$-equivariant Jacobi fields.  \qedhere
\end{itemize}
\end{proof}

\begin{proof}[Proof of \cref{thm:BumpyIsGGeneric}]
Given \cref{thm:StructureThm}, the proof is exactly the same as the proof of Theorem 9 in \cite{AmbrozioCarlottoSharp2018Compactness}*{Section 7.1} substituting $\Gamma^q$, $\mathcal S^q$, $\mathcal B^q$ with $\Gamma^q_G$, $\mathcal S^q_G$, $\mathcal B^q_G$. 
In fact, the only difference in the proofs is hidden in the application of \cite{White2017}*{Lemma 2.6}, which is the point where the Structure Theorem (in the different variants: \cref{thm:StructureThm}, \cite{White2017}*{Theorem 2.3} or \cite{AmbrozioCarlottoSharp2018Compactness}*{Theorem 35}) is used. However, White in \cite{White2017} works in the equivariant setting as we do here, hence the argument there can be applied in our context without changes, since the presence of the boundary does not play a role in this lemma, as observed also in \cite{AmbrozioCarlottoSharp2018Compactness}*{pp. 24}.
As said above (in the proof of \cref{thm:StructureThm}), the reader should pay attention to the fact that here we used $M^{m+1}$ for the ambient manifold and $\Sigma^m$ for the embedded hypersurface, instead of $\mathcal N^{n+1}$ and $M^n$ as in \cite{AmbrozioCarlottoSharp2018Compactness}.
\end{proof}


\chapter{Regularity of equivariant min-max surfaces} \label{sec:ProofRegularityGenus}

In this chapter, we outline the results about the regularity and the genus bound for the surface obtained via a Simon--Smith equivariant min-max procedure, as described in \cref{thm:EquivMinMax}. 
Besides being interesting on its own, we also need to unfold some of those arguments for the proof of the equivariant index bound in \cref{thm:EquivMinMax}. Indeed, in \cref{chpt:IndexBound}, we do not prove that any surface obtained from the min-max procedure has controlled equivariant index, but we show that it is possible to properly modify the proof in order to obtain a surface with the desired bound on the equivariant index.

We also include some detailed proofs for completeness, even if they are very similar to existing ones. This is because sometimes we need slight variations of existing results (for example here we consider $n$-dimensional sweepouts) or the proofs are split in many different references.

\section[Min-max sequence converging to a free boundary stationary varifold]{Min-max sequence converging to a free boundary stationary\\ varifold} \label{sec:LimitStatVar}

The first step is to show that there exists a min-max sequence converging to a free boundary stationary varifold.

\begin{definition}
In the setting of \cref{thm:EquivMinMax}, given a minimizing sequence $\{\so^j\}_{j\in\N}$ of sweepouts $\so^j=\{\Sigma^j_t\}_{t\in I^n}\in\Pi$ (namely such that $\lim_{j\to\infty}\sup_{t\in I^n} \area(\Sigma_t^j) = W_\Pi$), we define its \emph{critical set} as
\begin{align*}
    C(\{\so^j\}_{j\in\N}) = \{ & V\in\mathcal{V}^2_G(M) \st \text{$\exists$ $j_k\to+\infty$, $t_{j_k}\in I^n$ with $\vard(\Sigma^{j_k}_{t_{j_k}},V)\to 0$},\\ &\norm{V}(M)=W_\Pi\}.
\end{align*}
\end{definition}

\begin{proposition} [{cf. \cite{Ketover2016Equivariant}*{Proposition 3.9}}] \label{prop:ConvergenceToStationary}
In the setting of \cref{thm:EquivMinMax}, given a minimizing sequence $\{\boldsymbol{\Lambda}^j\}_{j\in\N}=\{\{\Lambda_t^j\}_{t\in I^n}\}_{j\in\N}$, there exists another minimizing sequence $\{\so^j\}_{j\in\N}=\{\{\Sigma_t^j\}_{t\in I^n}\}_{j\in \N}$ such that, if $\{\Sigma_{t_j}^j\}_{j\in\N}$ is a min-max sequence, then (up to subsequence) $\Sigma^j_{t_j}$ converges in the sense of varifolds to a free boundary stationary varifold.
Moreover, it holds $C(\{\so^j\}_{j\in\N}) \subset C(\{\boldsymbol{\Lambda}^j\}_{j\in\N})$.
\end{proposition}
\begin{proof}
We include the proof for completeness. However, the argument is the same as in the proof of \cite{ColdingDeLellis2003Errata}*{Proposition 2.1}, given the suitable modifications to the equivariant setting.

We first restrict to the compact space
\[
\mathcal{X} \eqdef \{V\in \mathcal{V}^2_G(M) \st \norm{V}(M)\le 2W_\Pi\}
\]
and we denote by $\mathcal{V}_\infty$ the subspace of free boundary $G$-stationary varifolds of $\mathcal{X}$, which coincides with the subspace of $G$-equivariant free boundary stationary varifolds by \cref{rem:stationaryGstationary}, and we define $\mathcal{Y}\eqdef \mathcal{X}\setminus \mathcal{V}_\infty$.
Now, for every $V\in \mathcal{Y}$, consider a $G$-equivariant smooth vector field $\xi_V\in\vf_G(M)$ such that $\delta V(\xi_V) < 0$. Up to rescaling, we can assume without loss of generality that $\norm{\xi_V}_{C^k}\le 1/k$, whenever $\vard(V,\mathcal{V}_\infty)\le 2^{-k}$ and $k\ge 1$ is any positive integer.

Then, for every $V\in \mathcal{Y}$, let $0 < \rho(V) < \vard(V,\mathcal{V}_\infty)/2$ be such that $\delta W(\xi_V)<0$ for every $W\in B_{\rho(V)}^{\vard}(V)$. Note that $\{B_{\rho(V)}^{\vard}(V)\}_{V\in \mathcal{Y}}$ is a cover of $\mathcal{Y}$ and thus it admits a subordinate partition of unity $\{\varphi_V\}_{V\in \mathcal{Y}}$. Therefore, for every $V\in \mathcal{Y}$, we can define the smooth $G$-equivariant vector field
\[
H_V\eqdef \sum_{W\in \mathcal{Y}} \varphi_W(V) \xi_W.
\]
Note that the map $\mathcal{Y}\ni V \mapsto H_V \in \vf_G(M)$ is continuous. Moreover, observe that $\delta V(H_V) < 0$ for every $V\in \mathcal{Y}$ and that, if $\vard(V,\mathcal{V}_\infty) \le 2^{-k-1}$ for some positive integer $k\ge 1$, then $\norm{H_V}_{C^k}\le 1/k$. 
This last property follows from the fact that, if $\vard(V,\mathcal{V}_\infty) \le 2^{-k-1}$ for some positive integer $k\ge 1$ and $\varphi_W(V) > 0$ for some $W\in \mathcal{Y}$, then $\vard(W,\mathcal{V}_\infty)\le 2\, \vard(V,\mathcal{V}_\infty) \le 2^{-k}$ and therefore $\norm{\xi_W}_{C^k}\le 1/k$.

As a result, the map $V\mapsto H_V$ defined on $\mathcal{Y}$ can be extended to a continuous function $\mathcal{X}\to \vf_G(M)$ by setting it identically equal to $0$ on $\mathcal{V}_\infty$.
Then, for every $V\in \mathcal{X}$, the flow $\Psi_V\colon[0,+\infty)\to \operatorname{Diff}(M)$ generated by $H_V$ is a $G$-equivariant isotopy. Since the map $V\in \mathcal{X}\mapsto H_V\in \vf_G(M)$ is continuous, then also the map $V\mapsto \Psi_V$ is continuous. In particular the function $V\mapsto \delta(\Psi_V(s,\cdot)_\sharp V)(H_V)$ is continuous as well for all $s\ge 0$ and $\Psi_V(0,\cdot)_\sharp V = V$. 
Hence, for every $V\in \mathcal{Y}$, let $0<\tau(V)\le 1$ be the maximum such that
\[
\delta(\Psi_V(s,\cdot)_\sharp V)(H_V) \le \frac 12 \delta V(H_V) < 0 \ \text{for all $s\in [0,\tau(V)]$}.
\]
The map $\tau\colon \mathcal{Y}\to [0,1]$ is continuous, hence we can redefine $H_V$ multiplying it by $\tau(V)$ in $\mathcal{Y}$, i.e., we reset it to be equal to $\tau(V)H_V$ in $\mathcal{Y}$ and equal to $0$ elsewhere. Note that this gives a well-defined vector field $\mathcal{X}\ni V\mapsto H_V\in \vf_G(M)$, because $\tau$ is bounded in $\mathcal{Y}$ and therefore the newly-defined $H_V$ is continuous if we set it equal to $0$ in $\mathcal{V}_\infty$.
As a result, we also redefine the flow $\Psi_V$ and it holds
\[
\delta(\Psi_V(s,\cdot)_\sharp V) (H_V) < 0 \ \text{for all $V\in \mathcal{Y}$ and $s\in [0,1]$}.
\]

Now let $\{\boldsymbol{\Lambda}^j\}_{j\in\N}=\{\{\Lambda_t^j\}_{t\in I^n}\}_{j\in\N}\subset \Pi$ be a minimizing sequence and define the $G$-equivariant surface $\tilde\Sigma_t^j \eqdef \Psi_{\Lambda_t^j}(1,\Lambda_t^j)$ for all $j\in \N$ and $t\in I^n$. Observe that the map $(t,x)\mapsto \tilde \Phi^j(t,x)\eqdef \Psi_{\Lambda_t^j}(1,x)$ is $C^1$ in the parameter $t$;
however, it is not necessarily smooth, hence $\{\tilde\Sigma_t^j\}_{t\in I^n}$ is not necessarily contained in $\Pi$.
Therefore, for all $j\in\N$, let us define a smooth map $(t,x)\mapsto \Phi^j(t,x)$ by convolving $\tilde \Phi^j$ with a smooth kernel in the parameter $t\in I^n$ in such a way that
\begin{equation} \label{eq:hnearc1}
\norm{\Phi^j(t,\cdot)-\tilde \Phi^j(t,\cdot)}_{C^1} \le \frac{1}{j+1}.
\end{equation}
Note that $\Phi^j(t,\cdot)$ is $G$-equivariant, because we are convoluting in the parameter $t$ and $\tilde\Phi^j(t,\cdot)$ is $G$-equivariant.
Then, define the $G$-equivariant surface $\Sigma_t^j\eqdef \Phi^j(t,\Lambda_t^j)$ for all $j\in\N$ and $t\in I^n$. By smoothness of the map $t\mapsto \Phi^j(t,\cdot)$, we have that $\{\Sigma_t^j\}_{t\in I^n}\in\Pi$ for all $j\in\N$.
Moreover, by \eqref{eq:hnearc1}, it holds that
\begin{equation} \label{eq:TildeNearSigma}
\lim_{j\to+\infty} \sup_{t\in I^n} \vard(\Sigma_t^j,\tilde \Sigma_t^j) = 0.
\end{equation}
As a result, using that $\area(\tilde\Sigma_t^j) = \area(\Psi_{\Lambda_t^j}(1,\Lambda_t^j)) \le \area(\Lambda_t^j)$, we have that
\[
m_0\le \limsup_{j\in\N}\sup_{t\in I^n} \area(\Sigma_t^j) = \limsup_{j\in\N}\sup_{t\in I^n} \area(\tilde\Sigma_t^j) \le \limsup_{j\in\N}\sup_{t\in I^n} \area(\Lambda_t^j) = m_0,
\]
thus, a posteriori, all the inequalities are equalities; in particular we have that \[\limsup_{j\in\N}\sup_{t\in I^n} \area(\Sigma_t^j) = m_0.\]

Our aim now is to prove that the minimizing sequence $\{\boldsymbol{\Sigma}^j\}_{j\in\N}=\{\{\Sigma_t^j\}_{t\in I^n}\}_{j\in \N}$ has the properties claimed in the statement. Let $t_j\in I^n$ be such that $\{\Sigma_{t_j}^j\}_{j\in\N}$ is a min-max sequence. Up to subsequence, $\Sigma_{t_j}^j$ converges in the sense of varifolds to a varifold $V$, which coincides with the limit in the sense of varifolds of $\tilde\Sigma_{t_j}^j$ by \eqref{eq:TildeNearSigma}.
Possibly passing to a subsequence, we can also assume that $\Lambda_{t_j}^j$ converges in the sense of varifolds to some varifold $W$ and, by continuity of the map $\Psi$, we have that $\Psi_W(1,\cdot)_\sharp W = V$.

First consider the case when $W$ is not free boundary stationary.
Then $\norm{W}(M) = m_0$, but on the other hand it holds
\begin{align*}
m_0 &= \lim_{j\to+\infty} \area(\tilde\Sigma_{t_j}^j) = \norm{V}(M) = \norm{\Psi_W(1,\cdot)_\sharp W}(M) \\
&= \norm{W}(M) + \int_0^1 \delta(\Psi_W(s,\cdot)_\sharp W)(H_W) \de s < m_0,
\end{align*}
which is a contradiction. As a result $W$ is free boundary stationary and therefore $W=\Psi_W(1,\cdot)_\sharp W= V$, which implies that $V$ is free boundary stationary, as desired. Moreover, since $\Lambda_{t_j}^j\to W =V$, we also get that $C(\{\so^j\}_{j\in\N}) \subset C(\{\boldsymbol{\Lambda}^j\}_{j\in\N})$.
\end{proof}

\section{Min-max sequence \texorpdfstring{$G$}{G}-almost minimizing in annuli} \label{sec:SeqAlmMin}

The second step is to prove that we can further assume that the min-max sequence found in the previous section has a certain `regularizing property', namely that of being $G$-almost minimizing in suitable subsets. The precise definitions and statements are given below.

\begin{definition} \label{def:AlmostMinimizing}
Given $\delta,\varepsilon>0$, a $G$-equivariant open set $U\subset M$, and a $G$-equivariant surface $\Sigma\subset U$, we say that $\Sigma$ is \emph{$(G,\delta,\varepsilon)$-almost minimizing} in $U$ if there does not exist any $G$-equivariant isotopy $\Psi\in \Is_G(M)$ supported in $U$ such that
\begin{enumerate}[label={\normalfont(\roman*)}]
\item $\area(\Psi(s,\Sigma)) \le \area(\Sigma) + \delta$ for all $s\in[0,1]$;
\item $\area(\Psi(1,\Sigma)) \le \area(\Sigma) - \varepsilon$.
\end{enumerate}
A sequence $\{\Sigma^j\}_{j\in\N}$ is said to be \emph{$G$-almost minimizing} in $U$ if there exist sequences of positive numbers $\delta_j,\varepsilon_j>0$ with $\varepsilon_j\to 0$ such that $\Sigma^j$ is $(G,\delta_j,\varepsilon_j)$-almost minimizing in $U$.
\end{definition}

\begin{definition}
We say that a subset $\mathrm{An}\subset M$ is a \emph{$G$-equivariant annulus} if there are $x\in M$ and $0<r<s$ such that $\mathrm{An} = \bigcup_{g\in G} g(\mathrm{An}(x,r,s))$, where $\mathrm{An}(x,r,s) = B_s(x) \setminus \overline{B_r(x)}$ is the open annulus of center $x$, inner radius $r$ and outer radius $s$. 
Note that, for $r,s$ sufficiently small $\bigcup_{g\in G} g(\mathrm{An}(x,r,s))$ consists of a finite union of disjoint open annuli (note that some annuli in the union could possibly be the same). 
Given $x\in M$, we denote by $\mathcal{AN}^G_{r}(x)$ the union of all $G$-equivariant annuli $\bigcup_{g\in G} g(\mathrm{An}(x,r_1,r_2))$ as the radii $r_1,r_2$ vary in the range $0<r_1<r_2\le r$.

We say that a family of $G$-equivariant annuli is \emph{$L$-admissible} for some positive integer $L$, if it consists of $L$ $G$-equivariant annuli $\mathrm{An}_1,\ldots, \mathrm{An}_L$ such that $\mathrm{An}_i = \bigcup_{g\in G} g(\mathrm{An}(x,r_i,s_i))$ for some $x\in M$ and $r_i,s_i>0$ with $r_{i+1}>2s_i$ for all $i=1,\ldots,L-1$.
\end{definition}
\begin{definition}
We say that a surface $\Sigma\subset M$ is $(G,\delta,\varepsilon)$-almost minimizing in an $L$-admissible family $\mathcal{A}$ of $G$-equivariant annuli if it is $(G,\delta,\varepsilon)$-almost minimizing in at least one $G$-equivariant annulus of $\mathcal{A}$. Moreover, we say that a sequence $\{\Sigma^j\}_{j\in\N}$ is $G$-almost minimizing in an $L$-admissible family $\mathcal{A}$ of $G$-equivariant annuli if there exist sequences $\delta_j,\varepsilon_j>0$ with $\varepsilon_j\to 0 $ such that $\Sigma^j$ is $(G,\delta_j,\varepsilon_j)$-almost minimizing in $\mathcal{A}$; equivalently, $\{\Sigma^j\}_{j\in\N}$ is $G$-almost minimizing in $\mathcal{A}$ if it is $G$-almost minimizing in at least one $G$-equivariant annulus of $\mathcal{A}$.
\end{definition}

\begin{lemma} [{cf. \cite{ColdingGabaiKetover2018}*{Lemma A.1}}] \label{lem:AlmostMinSeqInAdmissible}
Let $\{\so^j\}_{j\in\N}=\{\{\Sigma_t^j\}_{t\in I^n}\}_{j\in \N}$ be a minimizing sequence obtained from \cref{prop:ConvergenceToStationary},
then there exists $t_j\in I^n$ for all $j\in\N$ such that $\{\Sigma^j_{t_j}\}_{j\in\N}$ is a $G$-equivariant min-max sequence that is $G$-almost minimizing in every $L$-admissible family of $G$-equivariant annuli, where $L=(3^n)^{3^n}$.
\end{lemma}
\begin{proof}
Let $\{\so^j\}_{j\in\N}$ be the minimizing sequence given by \cref{prop:ConvergenceToStationary}. Fixed any $J>0$, we prove that there exist $j>J$ and $t_j\in I^n$ such that $\Sigma^j_{t_j}$ is $(G,1/(2^{n+2}J),1/J)$-almost minimizing in every $L$-admissible family of $G$-equivariant annuli and $\area(\Sigma^j_{t_j}) \ge m_0 - 1/J$.
To this purpose, consider any $j>J$ and define $K\eqdef \{t\in I^n\st \area(\Sigma^j_t) \ge m_0 - 1/J\}$. Assume that for all $t\in K$ there exists an $L$-admissible family $\mathcal{A}_t=\{\mathrm{An}_t^1,\ldots,\mathrm{An}_t^L\}$ of $G$-equivariant annuli where $\Sigma_t^j$ is not $(G,1/(2^{n+2}J),1/J)$-almost minimizing, namely for all $t\in K$ and $i=1,\ldots,L$ there exists a $G$-equivariant isotopy $\Psi_t^i\in\Is_G(M)$ supported in $\mathrm{An}_t^i$ such that
\begin{itemize}
    \item $\area(\Psi_t^i(s,\Sigma_t^j)) \le \area(\Sigma_t^j) + 1/(2^{n+2}J)$ for all $s\in[0,1]$;
    \item $\area(\Psi_t^i(1,\Sigma_t^j)) \le \area(\Sigma_t^j) - 1/J$.
\end{itemize}
By continuity in the sense of varifolds of the map $t\mapsto \Sigma_t^j$, for every $t\in K$ we can find a neighborhood $U_t\subset I^n$ such that
\begin{itemize}
    \item $\area(\Psi_t^i(s,\Sigma_r^j)) \le \area(\Sigma_r^j) + 1/(2^{n+1}J)$ for all $s\in[0,1]$, $r\in U_t$;
    \item $\area(\Psi_t^i(1,\Sigma_r^j)) \le \area(\Sigma_r^j) - 1/(2J)$ for all $r\in U_t$.
\end{itemize}
Since $\{U_t\}_{t\in K}$ is an open cover of $K$ and $K$ is compact, we can extract a finite subcover $\mathcal{U}$. We now want to find another cover $\mathcal{O}$ of $K$ and for every $O\in\mathcal{O}$ we want to choose a $G$-equivariant annulus $\mathrm{An}_O$ such that
\begin{enumerate}[label={\normalfont(\roman*)}]
\item\label{opencoverK:i} for all $O\in\mathcal{O}$ there exists $U_t\in\mathcal{U}$ such that $O\subset U_t$ and $\mathrm{An}_O = \mathrm{An}_t^i$ for some $i=1,\ldots,L$;
\item\label{opencoverK:ii} for all $O_1,O_2\in\mathcal{O}$ with nonempty intersection we have that $\mathrm{An}_{O_1} \cap \mathrm{An}_{O_2} = \emptyset$.
\end{enumerate}

To this purpose, we follow \cite{Pitts1981}*{Chapter 4}). For all $m\in\N$, let $I(1,m)$ be the CW complex structure, supported on the $1$-dimensional interval $I^1=[0,1]$, whose $0$-cells are $[0], [3^{-m}], [2\cdot 3^{-m}], \ldots, [1-3^{-m}],[1]$ and whose $1$-cells are the intervals $[0,3^{-m}], [3^{-m}, 2\cdot 3^{-m}], \ldots, [1-3^{-m}, 1]$. Then consider the CW complex structure $I(n,m)$ on the $n$-dimensional interval $I^n$, whose $p$-cells for $0\le p\le n$ can be written as $\sigma_1\otimes\sigma_2\otimes\ldots\otimes \sigma_n$, where $\sigma_1,\ldots,\sigma_n$ are cells of the CW complex $I(1,m)$ such that $\sum_{i=1}^n \dim(\sigma_i) = p$.

For every cell $\sigma = \sigma_1\otimes\sigma_2\otimes\ldots\otimes \sigma_n\in I(n,m)$, let $T(\sigma)$ be the open subset of $I_n$ defined as $T(\sigma)\eqdef T(\sigma_1)\times T(\sigma_2)\times \ldots \times T(\sigma_n)$, where 
\[
T(\sigma_i) \eqdef \begin{cases}
                       (x-3^{-m-1},x+3^{-m-1}) & \text{if $\sigma_i = [x]$ is a $0$-cell}\\
                       (x,y) & \text{if $\sigma_i = [x,y]$ is a $1$-cell}. 
                   \end{cases}
\]
Note that $\{T(\sigma)\st \sigma\in I(n,m), T(\sigma)\cap K \not=\emptyset\}$ is an open cover of $K$. Then define $\mathcal{O}$ to be such an open cover for some $m$ sufficiently large such that $3^{-m}\sqrt{n}$ is less than the Lebesgue number of $\mathcal{U}$, thus $O$ has diameter less than the Lebesgue number of $\mathcal{U}$ for all $O\in\mathcal{O}$.
Then, for all $O\in\mathcal{O}$, let $\mathcal{A}_O$ be an $L$-admissible family of $G$-equivariant annuli in $M$ such that there exists $U_t\in\mathcal{U}$ with $O\subset U_t$ and $\mathcal{A}_O=\mathcal{A}_t$.

Now, note that $T(\sigma_1)\cap T(\sigma_2)\not=\emptyset$ for some $\sigma_1,\sigma_2\in I(n,m)$ if and only if $\sigma_1$ and $\sigma_2$ are faces (possibly of different dimension) of a cell in $I(n,m)$. 
We can thus exploit the following proposition to associate to every $O=T(\sigma)\in\mathcal{O}$ a $G$-equivariant annulus $\mathrm{An}_O\in\mathcal{A}_O$.

\begin{proposition} [\cite{Pitts1981}*{Proposition 4.9}] \label{prop:PittsChoice}
For every $\sigma\in I(n,m)$, let $\mathcal{A}(\sigma)$ be an $L$-admissible family of $G$-equivariant annuli in $M$ with $L=(3^n)^{3^n}$. Then for every $\sigma\in I(n,m)$ it is possible to choose a $G$-equivariant annulus $\mathrm{An}(\sigma)\in\mathcal{A}(\sigma)$ such that $\mathrm{An}(\sigma_1)\cap \mathrm{An}(\sigma_2)=\emptyset$ for every $\sigma_1\not=\sigma_2\in I(n,m)$ that are faces of a cell in $I(n,m)$.
\end{proposition}

Hence, for all $O = T(\sigma)\in \mathcal{O}$, let us define $\mathrm{An}_O \eqdef \mathrm{An}(\sigma)$, where $\mathrm{An}(\sigma)$ is given from \cref{prop:PittsChoice}. This proves the existence of the open cover $\mathcal{O}$ of $K$ with properties \ref{opencoverK:i} and \ref{opencoverK:ii} above. Moreover, by definition of $\mathcal{U}$ and $\mathcal{O}$, for all $O\in\mathcal{O}$ there exists a $G$-equivariant isotopy $\Psi_O\in\Is_G(M)$ supported in $\mathrm{An}_O$ such that
\begin{itemize}
\item $\area(\Psi_O(s,\Sigma_r^j))\le \area(\Sigma_r^j) + 1/(2^{n+1}J)$ for all $s\in[0,1]$, $r\in O$;
\item $\area(\Psi_O(1,\Sigma_r^j))\le \area(\Sigma_r^j) - 1/(2J)$ for all $r\in O$.
\end{itemize}

Now, for all $O\in \mathcal O$ choose a $C^\infty$ function $\varphi_O\colon I^n\to [0,1]$ supported in $O$, such that for all $s\in K$ there exists $O\in\mathcal{O}$ such that $\varphi_O(s) = 1$.
Then, given any $t\in I^n$, consider the function $\Phi_t\in \operatorname{Diff}_G(M)$ defined as
\[
\Phi_t(x) \eqdef \begin{cases}
                     \Psi_O(\varphi_O(t), x) & \text{if $x\in \mathrm{An}_O$ for some $O\in\mathcal{O}$ with $t\in O$,}\\
                     x &\text{if $t\not\in O$ for all $O\in\mathcal O$.}
                 \end{cases}
\]
Observe that $\Phi_t$ is well-defined since $\mathrm{An}_{O_1}\cap \mathrm{An}_{O_2}=\emptyset$ whenever $O_1,O_2\in\mathcal O$ are not disjoint.
Moreover, $\Phi=\{t\in I^n\mapsto\Phi_t\}\in\Is_G(I^n,M)$ is a $G$-equivariant isotopy with $\Phi_t=\id$ for $t\in\partial I^n$. Indeed, note that $d(K,\partial I^n) > 0$, hence we can choose $m$ possibly larger in such a way that $O\cap \partial I^n=\emptyset$ for all $O\in\mathcal O$.
Therefore, we can define $\tilde\Sigma_t^j = \Phi_t(\Sigma_t^j)$ and we have that $\{\tilde\Sigma_t^j\}_{t\in I^n}\in \Pi$. 
Now note that, for every $t\in K$, there exists $O\in \mathcal O$ such that $\varphi_O(t) = 1$, moreover $t\in O$ for at most other $2^n-1$ sets $O\in\mathcal O$. As a result, we can estimate
\[
\area(\tilde\Sigma_t^j) \le \area(\Sigma_t^j) - \frac{1}{2J} + \frac{2^n-1}{2^{n+1}J} =\area(\Sigma_t^j) - \frac{1}{2^{n+1}J}
\]
for all $t\in K$, and
\[
\area(\tilde\Sigma_t^j) \le \area(\Sigma_t^j) + \frac{2^n}{2^{n+1}J} < m_0-\frac 1{2J}
\]
for all $t\in I^n\setminus K$.
By arbitrariness of $j$, this would give \[\liminf_{j\to+\infty}\sup_{t\in I^n} \area(\tilde\Sigma_t^j) < m_0,\] which is a contradiction. This proves that, for every $J>0$, there exists $j>J$ and $t_j\in I^n$ such that $\Sigma_{t_j}^j$ is $(G,1/(2^{n+2}J),1/J)$-almost minimizing in every $L$-admissible family of $G$-equivariant annuli and $\area(\Sigma_{t_j}^j) \ge m_0-1/J$. In particular, we have that $\{\Sigma^j_{t_j}\}_{j\in\N}$ is a $G$-equivariant min-max sequence that is $G$-almost minimizing in every $L$-admissible family of $G$-equivariant annuli.
\end{proof}

\begin{remark}
Note that the min-max sequence that we obtain is $(G,1/(2^{n+2}J),1/J)$-almost minimizing. This coincides with the result obtained in \cite{ColdingDeLellis2003} (see Definition 3.2 therein) in the case of $1$-parameters sweepouts, i.e., $n=1$.
\end{remark}

\section{Equivariant Meeks--Simon--Yau reductions}

In this section, we present the adaptation to the equivariant setting of the reductions introduced by Meeks--Simon--Yau in \cite{MeeksSimonYau1982}. The idea is to perform a `topological simplification' on a surface, modifying it only slightly with respect to the varifold sense.
Let us start recalling a preliminary lemma about `fillings' of a surface in small balls.

\begin{lemma} [{cf. \cite{MeeksSimonYau1982}*{Lemma 1}, \cite{Li2015}*{Lemma 7.5}}] \label{lem:CompactFilling}
For every $r_0>0$ sufficiently small (depending on $M$), there exists $\delta\in(0,1)$ such that, given any surface $\Sigma^2\subset M$ with $\partial \Sigma=\Sigma\cap\partial M$ and $\area(\Sigma\cap B_{r_0}(x))<\delta^2r_0^2$ for all $x\in M$, there exists a unique compact set $K_\Sigma\subset M$ such that
\begin{enumerate} [label={\normalfont(\arabic*)}]
\item $\partial K_\Sigma \cap \operatorname{int}(M)=\Sigma$;
\item $\Haus^3(K_\Sigma\cap B_{r_0}(x))\le\delta^2r_0^3$ for all $x\in M$;
\item $\Haus^3(K_\Sigma)\le c_0\area(\Sigma)^{3/2}$, where $c_0$ is a constant depending only on $M$.
\end{enumerate}
Moreover, if $\Sigma$ is homeomorphic to a two-dimensional sphere, then $K_\Sigma$ is homeomorphic to a three-dimensional ball.
\end{lemma}

Observe that, without loss of generality, up to rescaling the manifold $M$ we can assume that $r_0=1$. From now on, we make this assumption and we consider the associated $\delta$ given by \cref{lem:CompactFilling}. Moreover, we choose $0<\gamma<\delta^2/9$.
We can now define the notion of equivariant $\gamma$-reduction in the interior of $M$ and away from points of type $\dih_n$ for some $n\ge 2$. The reason why it is sufficient to work under these assumptions will be clear in the next sections. The idea is that the points at the boundary away from the singular locus are treated in \cite{Li2015}, and the (finitely many) points of type $\dih_n$ with $n\ge 2$ and points at the intersection between the boundary and the singular locus can be avoided.

Therefore, for the rest of the section let us assume that $U\subset M$ is a $G$-compatible open subset such that $U\subset\operatorname{int}(M)$ and $G_U=\Z_n$ for some $n\ge 2$. Since here we only work in $U$, we drop the subscript $U$ when we talk about $G_U$-equivariant objects. Moreover, if not otherwise stated, when we consider a surface $\Sigma\subset U$, we assume that $\Sigma$ is smooth, compact, embedded, possibly disconnected, and $G$-equivariant (i.e., $G_U$-equivariant).

\begin{definition}[{cf. \cite{MeeksSimonYau1982}*{Section 3}, \cite{Li2015}*{Definition 7.6}}] \label{def:EquivariantGammaReduction}
We say that a $G$-equivariant surface $\Sigma_2\subset U$ is a \emph{$G$-equivariant $\gamma$-reduction} of a $G$-equivariant surface $\Sigma_1\subset U$ in $U$ and write
\[
\Sigma_2 \reduc[\gamma] \Sigma_1
\]
(omitting the dependence on $G$ and $U$), if $\Sigma_2$ is obtained from $\Sigma_1$ through a $G$-equivariant surgery in $U$ satisfying the following conditions.
\begin{enumerate} [label={\normalfont(\roman*)}]
\item $\overline{\Sigma_1\setminus \Sigma_2} = A\subset U$ is diffeomorphic to the disjoint union of $k$ annuli $A_1,\ldots,A_k$ and it is $G$-equivariant.
\item $\overline{\Sigma_2\setminus\Sigma_1} = D\subset U$ is diffeomorphic to the disjoint union of $2k$ discs $D_1^1$, $D_1^2$, \ldots, $D_k^1$, $D_k^2$ and $\cup_{i=1}^k D_i^1$ and $\cup_{i=1}^k D_i^2$ are $G$-equivariant.
\item $\partial A_i = \partial D_i^1\cup \partial D_i^2$ for all $i=1,\ldots,k$ and there exist $k$ compact sets $Y_1,\ldots,Y_k$ homeomorphic to the closed three-dimensional ball such that $\partial Y_i = A_i\cup D_i^1\cup D_i^2$ and $Y\eqdef Y_1\cup\ldots\cup Y_k$ satisfies $(Y\setminus\partial Y)\cap(\Sigma_1\cup\Sigma_2)=\emptyset$.
\item $\area(A) + \area(D) < 2\gamma$.
\item\label{egr:nontrivial} Let $\Sigma_1^*$ be a connected component of $\Sigma_1$ intersecting $A$. Then, if $\Sigma_1^*\setminus A$ is disconnected, each connected component of $\Sigma_1^*\setminus A$ is either not simply connected or has area $\ge \delta^2/(2k)$.
\end{enumerate}

We say that a surface $\Sigma\subset M$ is \emph{$G$-equivariant $\gamma$-irreducible} in $U$ if there does not exist $\tilde\Sigma$ such that $\tilde \Sigma \reduc[\gamma] \Sigma$.
\end{definition}

\begin{proposition} \label{rem:EquivalenceIrreducible}
A surface $\Sigma\subset U$ is $G$-equivariant $\gamma$-irreducible in $U$ if and only if whenever $\Delta\subset U$ is $G$-equivariant and it is a union of closed discs with $\partial \Delta= \Delta\cap \Sigma$ and $\area(\Delta) < \gamma$, then there is a disjoint union of discs $\tilde\Delta\subset\Sigma\cap U$ with $\partial\Delta  = \partial \tilde\Delta $ and $\area(\tilde \Delta)<\delta^2/2$.
\end{proposition}
\begin{proof}
If there exists $\Delta$ for which there is no such $\tilde\Delta$, then we can cut a neighborhood of $\partial\Delta$ and so $\Sigma$ is $G$-equivariant $\gamma$-reducible in $U$.
Vice versa, assume that $\Sigma$ is $G$-equivariant $\gamma$-reducible in $U$. 
Take a surgery as in \cref{def:EquivariantGammaReduction} and assume that $\area(\cup_{i=1}^k D_i^1) \le \area(\cup_{i=1}^k D_i^2)$.
Now take $\Delta \eqdef \cup_{i=1}^k D_i^1$, for which it holds $\area(\Delta)= \area(\cup_{i=1}^k D_i^1) \le \area(D)/2 < \gamma$. Note that it cannot exist a disjoint union of discs $\tilde\Delta\subset \Sigma\cap U$ with $\partial\Delta =\partial \tilde\Delta$ and $\area(\tilde\Delta)<\delta^2/2$, otherwise \ref{egr:nontrivial} would be contradicted.
\end{proof}

\begin{definition}[{cf. \cite{MeeksSimonYau1982}*{Section 3}, \cite{Li2015}*{Definition 7.8}}]
We say that a $G$-equivariant surface $\Sigma_2\subset U$ is a \emph{strong $G$-equivariant $\gamma$-reduction} of a $G$-equivariant surface $\Sigma_1\subset U$ in $U$ and write
\[
\Sigma_2 \sreduc[\gamma] \Sigma_1
\]
(again, omitting the dependence on $G$ and $U$), if there is $\tilde\Sigma_1 \in \Is_G(\Sigma_1,U)$ with $\area(\tilde\Sigma_1\triangle\Sigma_1) < \gamma$ and $\Sigma_2\reduc[\gamma] \tilde\Sigma_1$.
\end{definition}

\begin{remark} \label{rem:StronglyIrrWithDiffGamma}
If $\gamma_1\ge \gamma_2$ and $\Sigma$ is (strongly) $G$-equivariant $\gamma_1$-irreducible in $U$, then it is also (strongly) $G$-equivariant $\gamma_2$-irreducible in $U$.
\end{remark}

We are now ready to prove that a strongly $G$-equivariant $\gamma$-irreducible surface that is close to be a minimizer of the area is roughly the union of disjoint discs.
This is an adaptation to the equivariant setting of \cite{MeeksSimonYau1982}*{Theorem 2} (see also \cite{Li2015}*{Theorem 7.10}).

\begin{theorem} \label{thm:StructureIrreducible}
In the setting above, let $B\subset U$ be a $G$-equivariant subset of $U$, diffeomorphic to the closed three-dimensional ball. Suppose that $\Sigma\subset U$ is a $G$-equivariant surface such that:
\begin{enumerate} [label={\normalfont(\roman*)}]
\item $E_G(\Sigma, U)\eqdef \area(\Sigma) - \inf_{\tilde\Sigma\in\Is_G(\Sigma,U)}\area(\tilde\Sigma)\le\gamma/4$;
\item $\Sigma$ is strongly $G$-equivariant $\gamma$-irreducible in $U$;
\item $\Sigma$ intersects $\partial B$ transversally;
\item for each component $\omega$ of $\Sigma\cap \partial B$, let $F_\omega$ be a disc in $\partial B$ with $\partial F_\omega=\omega$ and least area; then $\sum_{j=1}^q \area(F_{j}) \le\gamma/8$, where $F_j \eqdef F_{\omega_j}$ and $\omega_1,\ldots,\omega_q$ are the connected components of $\Sigma\cap\partial B$.
\end{enumerate}
Moreover, let $\Sigma_0$ be the union of all components of $\Sigma$ contained in a compact subset of $U$ diffeomorphic to the closed three-dimensional ball.
Then, $\area(\Sigma_0)\le E_G(\Sigma,U)$ and there exist pairwise disjoint closed discs $D_1,\ldots,D_p \subset U$ for some $p\ge 0$ such that
\begin{enumerate} [label={\normalfont(\arabic*)}]
\item \label{strongdisc:i} $D=\cup_{i=1}^pD_i$ is $G$-equivariant;
\item \label{strongdisc:ii}$D_i\subset\Sigma\setminus\Sigma_0$ and $\partial D_i\subset\partial B$ for $i=1,\ldots,p$;
\item \label{strongdisc:iii} $\sum_{i=1}^p \area(D_i) \le \sum_{j=1}^q \area(F_{j}) + E_G(\Sigma,U)$;
\item \label{strongdisc:iv} $(\cup_{i=1}^pD_i)\cap B = (\Sigma\setminus\Sigma_0)\cap B$;
\item \label{strongdisc:v} $\cup_{i=1}^p(\varphi_1(D_i)\setminus\partial D_i) \subset B\setminus \partial B$ for some isotopy $\{\varphi_t\}_{t\in[0,1]}$ of $U$ such that $\varphi_t(x) = x$ for all $(t,x)\in [0,1]\times W$, where $W$ is a neighborhood of $(\Sigma\setminus\Sigma_0) \setminus \cup_{i=1}^p(D_i\setminus\partial D_i)$.
\end{enumerate}
\end{theorem}

\begin{proof}
The proof is very similar to the one of \cite{MeeksSimonYau1982}*{Theorem 2}. Hence, we just present it briefly for completeness. The idea is to induct on $q$. Observe that we can assume $\Sigma_0=\emptyset$, as remarked in \cite{MeeksSimonYau1982}*{pp. 631}.
Moreover, up to relabeling, we can assume that $F_q\cap \omega_j=\emptyset$ for all $j\not=q$ and $\area(F_q)\le \area(F_j)$ for all $j=1,\ldots,q$.

First, consider the case in which $\omega_q$ intersects the singular locus, then the only possibility is that $G=\Z_2$ and $\omega_q$ is a $\Z_2$-equivariant curve intersecting the singular locus (in the two points of $\mathcal{S}\cap\partial B$). Then, $\area(F_q) = \area(\partial B)/2$ and therefore we have that $q=1$ (by choice of the order in which we are processing the indices). Observe that, there exists a $\Z_2$-equivariant disc $\Delta\subset U$ such that $\partial \Delta = \omega_1$ and $\area(\Delta)\le \area(F_1)\le \gamma/8$. Hence, by \cref{rem:EquivalenceIrreducible}, there is a $G$-equivariant disc $D_1\subset\Sigma\cap U$ with $\partial D_1=\omega_1$ and $\area(D_1) < \delta^2/2$. Then, $D_1$ easily satisfies the requirements in the statement (we give below more details in the other case, which is more complicated).

Now, let us assume that $\omega_q$ does not intersect the singular locus.
Define $\tilde\omega \eqdef G(\omega_q) = \cup_{\selem\in G}\selem(\omega_q)$, which we may assume to be equal to the $k$ components $\omega_{q-k+1},\ldots,\omega_q$, for some $k\ge1$. 
Observe that, in this case, we can assume that $F_s = \selem(F_q)$ for all $s=q-k+1,\ldots,q$, where $\selem\in G$ is such that $\omega_s=\selem(\omega_q)$. Hence, in particular $F_s\cap \omega_j=\emptyset$ for all $s=q-k+1,\ldots,q$ and $j\not=s$, and $\tilde F\eqdef \cup_{s=q-k+1}^qF_s$ is $G$-equivariant. 

Observe that, by \cref{rem:StronglyIrrWithDiffGamma}, $\Sigma$ is strongly $G$-equivariant $\tilde\gamma$-irreducible in $U$ with \[\tilde\gamma\eqdef\frac \gamma4+E_G(\Sigma,U)+4\sum_{j=1}^q\area(F_j)\le \frac\gamma 4 + \frac\gamma4+\frac\gamma2 = \gamma\] and it holds $\area(\tilde F) \le \sum_{j=1}^q\area(F_j)< \tilde\gamma$. Hence, by \cref{rem:EquivalenceIrreducible}, there exists a disjoint union of discs $\tilde D = \tilde D_{q-k+1}\cup\ldots\cup\tilde D_q\subset\Sigma\cap U$ with $\partial \tilde D_s = \omega_s$ for $s=q-k+1,\ldots,q$ and $\area(\tilde D) < \delta^2/2$.

Thus $\tilde D\cup \tilde F$ is homeomorphic to the disjoint union of two-dimensional spheres and has area less than $\delta^2$. Hence, by \cref{lem:CompactFilling}, we know that $\tilde D\cup \tilde F$ is the boundary of an open set $Y$ such that $\overline{Y}$ is homeomorphic to the disjoint union of closed three-dimensional balls. 

Since we assumed that $\Sigma_0=\emptyset$, we must have $(\Sigma\setminus\tilde D)\cap Y = \emptyset$. Now, consider $\Sigma_* \eqdef (\Sigma\setminus\tilde D)\cup \tilde F$ and $\tilde F_\varepsilon \eqdef \{x\in U\st d(x, \tilde F)<\varepsilon\}$ for $\varepsilon>0$. Then we can select a continuous isotopy $\varphi \in \Is_G(U)$ supported in $\tilde F_\varepsilon$, i.e., $\varphi_t(\tilde F_\varepsilon)\subset\tilde F_\varepsilon$ and $\varphi_t=\id$ in $U\setminus \tilde F_\varepsilon$, such that $\area(\Sigma_*\cap \tilde F_\varepsilon) \le \area(\varphi_1(\Sigma_*\cap \tilde F_\varepsilon))\le \area(\Sigma_*\cap \tilde F_\varepsilon)+\varepsilon$ and $\varphi_1(\Sigma_*\cap \tilde F_\varepsilon)\cap \partial B = \emptyset$. Hence, define the $G$-equivariant surface $\hat\Sigma_*\eqdef \varphi_1(\Sigma_*)\in \Is_G(\Sigma,U)$, for which it holds
\begin{itemize}
    \item $\hat\Sigma_* \cap \partial B = \cup_{j=1}^{q-k} \omega_j$;
    \item $\area(\hat\Sigma_*\triangle \Sigma) < \area(\tilde D) + \area(\tilde F)+\varepsilon \le \area(\tilde D) + 2\area(\tilde F)$, by choosing $\varepsilon \le\area(\tilde F)$;
    \item $\area(\hat\Sigma_*) < \area(\Sigma) + \area(\tilde F) - \area(\tilde D) +\varepsilon$, or equivalently $E_G(\hat\Sigma_*) < E_G(\Sigma) + \area(\tilde F) - \area(\tilde D) +\varepsilon \le E_G(\Sigma) + 2\area(\tilde F) - \area(\tilde D)$, again using $\varepsilon \le\area(\tilde F)$.
\end{itemize}
In particular, we deduce that $\hat\Sigma_*$ is strongly $G$-equivariant $\gamma_*$-irreducible with 
\begin{align*}
\gamma_* &= \tilde\gamma - (2\area(\tilde F)-\area(\tilde D)) =  \frac \gamma4+E_G(\Sigma)+4\sum_{j=1}^{q-k}\area(F_j) + 2\area(\tilde F) - \area(\tilde D)\\
&\ge \frac\gamma 4 + E_G(\hat\Sigma_*) +4\sum_{j=1}^{q-k}\area(F_j).
\end{align*}
Moreover, it holds $E_G(\hat\Sigma_*)\le \gamma/2 - 2\sum_{j=1}^{q-k+1}\area(F_j)$, hence $\hat\Sigma_*$ satisfies the inductive hypothesis and thus there exist pairwise disjoint discs $\hat\Delta_1,\ldots,\hat\Delta_{\hat p}$ for some $\hat p\ge 0$ satisfying \ref{strongdisc:i}, \ref{strongdisc:ii}, \ref{strongdisc:iii}, \ref{strongdisc:iv}, \ref{strongdisc:v} with respect to $\hat\Sigma_*$. In particular, we have
\[
\sum_{i=1}^p\area(\tilde\Delta_i) \le \sum_{j=1}^{q-k}\area(F_j) + E_G(\hat\Sigma_*) .   
\]
As a result, by construction of $\hat\Sigma_*$, it is easy to obtain pairwise disjoint discs $\Delta_1,\ldots,\Delta_p$ for some $p\ge 0$ such that 
\begin{itemize}
    \item their union is $G$-equivariant;
    \item $ \Delta_i\subset \Sigma_* = (\Sigma\setminus\tilde D)\cup\tilde F$ and $\partial \Delta_i\subset\partial B$ for $i=1,\ldots,p$;
    \item $\sum_{i=1}^p\area(\Delta_i) \le \sum_{j=1}^{q-k}\area(F_j) + E_G(\hat\Sigma_*)< \sum_{j=1}^{q-k}\area(F_j) + E_G(\Sigma) + \area(\tilde F) - \area(\tilde D) +\varepsilon$, and choosing $\varepsilon$ sufficiently small this implies 
    \begin{align*}\sum_{i=1}^p\area(\Delta_i) &\le  \sum_{j=1}^{q-k}\area(F_j) + E_G(\Sigma) + \area(\tilde F) - \area(\tilde D) \\ &=\sum_{j=1}^{q}\area(F_j) + E_G(\Sigma) - \area(\tilde D);\end{align*}
    \item $(\cup_{i=1}^p\Delta_i)\cap B = \Sigma_*\cap B$;
    \item $\cup_{i=1}^p(\psi_1(\Delta_i)\setminus\partial \Delta_i) \subset B\setminus \partial B$ for some isotopy $\{\psi_t\}_{t\in[0,1]}$ of $U$ such that $\psi_t(x) = x$ for all $(t,x)\in [0,1]\times W$, where $W$ is a neighborhood of $\Sigma_* \setminus \cup_{i=1}^p(\Delta_i\setminus\partial \Delta_i)\setminus\tilde F_\varepsilon$.
\end{itemize}
Finally, observe that $\Sigma$ is obtained by $\Sigma_*$ by a continuous $G$-equivariant isotopy, morally sending $\tilde D$ into $\tilde F$. Therefore, it is now easy to define the discs relative to $\Sigma$ as required in the statement, distinguishing the cases $\tilde F\subset\cup_{i=1}^p\Delta_i$ and $\tilde F\not\subset \cup_{i=1}^p \Delta_i$ (in the second case we need $p+k$ discs instead of $p$). 
\end{proof}

\section{Almost minimizing sequences of isotopic surfaces}

\begin{definition} [{\cite{DeLellisPellandini2010}*{Definition 3.1}}]
Let $\mathscr{I}$ be a class of isotopies of $M$ and $\Sigma^2\subset M$ a smooth embedded surface. Given $\{\Phi^k\}_{k\in\N}\subset \mathscr{I}$ we say that $\Phi^k(1,\Sigma)$ is a \emph{minimizing sequence for $\Prob(\Sigma,\mathscr{I})$} if
\[
\lim_{k\to\infty}\area(\Phi^k(1,\Sigma)) = \inf_{\Psi\in\mathscr{I}} \area(\Psi(1,\Sigma)).
\]
\end{definition}

\begin{definition}
Given a $G$-compatible open set $U\subset M$, a smooth embedded $G$-equivariant surface $\Sigma\subset U$ with $\partial\Sigma \subset \partial U$ and a positive real number $\delta>0$, we define
\[
\Is_G^\delta(U,\Sigma) \eqdef \{\Psi\in\Is_G(U)\st \area(\Psi(t,\Sigma)) \le \area(\Sigma)+\delta \ \ \forall t\in [0,1]\}.
\]
\end{definition}

\begin{theorem} \label{thm:AlmostMinimizing}
Let $U\subset M$ be a $G$-compatible open set such that $\overline{U}$ does not contain points in $\mathcal{S}\cap \partial M$ and points of type $\dih_n$ for some $n\ge 2$.
Moreover, let $\Sigma\subset U$ be a smooth embedded $G$-equivariant surface with $\partial\Sigma\subset\partial U$ and let $\delta>0$ be a positive real number. If $\{\Delta^k\eqdef\Phi^k(1,\Sigma)\}_{k\in\N}$ is a minimizing sequence for $\Prob(\Sigma,\Is_{G}^\delta(U,\Sigma))$ converging in the sense of varifolds to $V$, then $V$ is (up to multiplicity) a smooth $G$-stable free boundary minimal surface, with genus bounded by the genus of $\Sigma$.

Moreover, if $U$ is a ball with sufficiently small radius, then multiplicity does not occur, namely $V$ is a smooth free boundary minimal surface $\Delta^\infty$, and $\partial\Delta^\infty\setminus\partial M=\partial\Sigma\setminus\partial M$.
\end{theorem}
\begin{remark}
We can assume that $U$ does not contain points in $\mathcal{S}\cap \partial M$ and points of type $\dih_n$ for some $n\ge 2$, because in the arguments where we apply \cref{thm:AlmostMinimizing} we can avoid a finite number of points. Probably one can remove this assumption using a removable singularity theorem as in \cite{ColdingDeLellis2003}*{Proposition 6.3}, but we preferred to avoid that to keep the proof cleaner.
\end{remark}

\begin{lemma} [Squeezing Lemma, {\cite{ColdingDeLellis2003}*{Lemma 7.6} and \cite{DeLellisPellandini2010}*{Lemma 6.1}}] \label{lem:Squeezing}
Let $\{\Delta^k\}_{k\in\N}$ as in \cref{thm:AlmostMinimizing}, $x\in \overline{U}$ and $\delta>0$. Then there exists $r_0>0$ and $K\in\N$ with the following property.
If $k\ge K$ and $\Psi\in\Is_G(B_{r_0}(x)\cap U)$ is such that $\area(\Psi(1,\Delta^k)) \le \area(\Delta^k)$, then there exists $\Psi'\in \Is_G(B_{r_0}(x)\cap U)\cap \Is_G^\delta(U,\Delta^k)$ such that $\Psi'(1,\cdot)=\Psi(1,\cdot)$.
\end{lemma}
\begin{proof}
The proof is exactly the same as in \cite{ColdingDeLellis2003}*{Lemma 7.6} and \cite{DeLellisPellandini2010}*{Lemma 6.1}, observing that the transformations considered in the arguments (in particular, radially symmetric transformations in sufficiently small balls) are $G$-equivariant (see \cite{Ketover2016Equivariant}*{Lemma 4.13}).
\end{proof}

\begin{proof}[Proof of \cref{thm:AlmostMinimizing}]
First observe that, since $\{\Delta^k\}_{k\in\N}$ is a minimizing sequence for $\Prob(\Sigma,\Is_G^\delta(U,\Sigma))$, $V$ is a free boundary stationary $G$-stable varifold.
Moreover, given the local versions of Meeks--Simon--Yau argument \cite{MeeksSimonYau1982} in the interior \cite{ColdingDeLellis2003}*{Theorem 7.3} and at the boundary \cite{Li2015}*{Theorem 7.3}, we conclude that $V$ is a smooth free boundary minimal surface up to multiplicity in $U\setminus\mathcal{S}$, with the same argument as in the proof of Lemma 7.4 in \cite{ColdingDeLellis2003}*{Section 7.3}. Hence, we are left to prove the regularity of $V$ at points in $\mathcal{S}$.

Given $x\in\mathcal{S}\cap U$, by \cref{lem:Squeezing}, there exists $r_0>0$ sufficiently small such that $\{\Delta^k\}_{k\in\N}$ is minimizing for $\Prob(\Sigma,\Is_G(B_{r_0}(x), \Sigma))$ and $B_{r_0}(x)$ is $G$-compatible with $G_{B_{r_0}(x)} = \Z_n$ for some $n\ge2$. As a result, one can peruse the proof in \cite{MeeksSimonYau1982}*{Section 5} using \cref{thm:StructureIrreducible} in place of \cite{MeeksSimonYau1982}*{Theorem 2} and obtain that $V$ is a smooth free boundary minimal surface in $B_{r_0}(x)$. Indeed, we remark that the convergence of a sequence of minimizing discs (object of \cite{AlmgrenSimon1979}*{Sections 5 and 6}) goes through also in the equivariant setting (cf. \cite{Ketover2016Equivariant}*{Section 4.3} and \cite{Ketover2016FBMS}*{Section 7.2}).

Finally, the fact that the multiplicity does not occur if $U$ is a ball with sufficiently small radius follows as in \cite{DeLellisPellandini2010}*{Proposition 3.2}. And, as a consequence, Simon's Lifting Lemma \cite{DeLellisPellandini2010}*{Proposition 2.1} holds in our equivariant setting and the genus bound follows.
\end{proof}

\section[Regularity and genus bound as a consequence of \texorpdfstring{$G$}{G}-almost minimality]{Regularity and genus bound as a consequence of \texorpdfstring{$G$}{G}-almost\\ minimality}\label{sec:ProofRegularityGenusFinal}

The third and final step is to show that the properties of the min-max sequence obtained in \cref{sec:LimitStatVar,sec:SeqAlmMin} are sufficient to prove that this sequence converges to a $G$-equivariant free boundary minimal surface (possibly with multiplicity) for which the genus bound in \cref{thm:EquivMinMax} holds.

The idea of the proof is the same as in \cite{ColdingDeLellis2003}*{Section 7}, namely constructing replacements for the varifold limit of the min-max sequence.
In our equivariant setting, \cite{ColdingDeLellis2003}*{Definition 6.1} is replaced by the following definition.

\begin{definition} \label{def:GReplacement}
Let $V\in\mathcal{V}_G(M)$ be $G$-equivariant free boundary stationary and let $U\subset M$ be an open subset. A $G$-equivariant free boundary stationary varifold $V'\in\mathcal{V}_G(M)$ is said to be a \emph{$G$-replacement} for $V$ in $U$ if the following properties hold:
\begin{enumerate} [label={\normalfont(\roman*)}]
\item $V'=V$ in $M\setminus\overline{U}$ and $\norm{V'}=\norm{V}$;
\item \label{replacement:stable} $V$ is a $G$-stable minimal surface $\Sigma$ (up to multiplicity) and $\overline{\Sigma}\setminus\Sigma\subset \partial U$.
\end{enumerate}
\end{definition}
\begin{remark} \label{rem:RegularityForGReplacements}
Observe that the difference between the definition of $G$-replacement and the one of replacement in \cite{ColdingDeLellis2003}*{Definition 6.1} is in condition \ref{replacement:stable}. However, thanks to \cref{cor:ConvGStableSurf} (see also \cref{rem:GStableAsStable}) in place of (8) in \cite{ColdingDeLellis2003}, the regularity theorem \cite{ColdingDeLellis2003}*{Proposition 6.3} for varifolds with good replacements works also with this equivariant definition of replacements.
\end{remark}

\begin{proposition} \label{prop:AlmostMinSeq}
In the setting of \cref{thm:EquivMinMax}, given a $G$-equivariant min-max sequence $\{\Sigma^j\}_{j\in\N}$ that is $G$-almost minimizing in every $L$-admissible family of $G$-equivariant annuli for $L=(3^n)^{3^n}$ (as the one obtained in \cref{lem:AlmostMinSeqInAdmissible}), there exists a $G$-equivariant function $r\colon M \to \R^+$ such that (up to subsequence)  $\{\Sigma^j\}_{j\in\N}$ is $G$-almost minimizing in $\mathrm{An}\in \mathcal{AN}_{r(x)}^G(x)$ for all $x\in M$.
Moreover, $\{\Sigma^j\}_{j\in\N}$ converges in the sense of varifolds as $j\to+\infty$ to $\limitv = \bigcup_{i=1}^k m_i\limitv_i$, where $\limitv_i$ is a $G$-equivariant free boundary minimal surface and $m_i$ is a positive integer for all $i=1,\ldots,k$, and the genus bound in \cref{thm:EquivMinMax} holds.
\end{proposition}
\begin{proof}
Thanks to \cite{ColdingGabaiKetover2018}*{Lemma A.3}, we have that there exists a $G$-equivariant function $r\colon M\to \R^+$ such that (up to subsequence) $\{\Sigma^j\}_{j\in\N}$ is $G$-almost minimizing in every $\mathrm{An}\in \mathcal{AN}_{r(x)}^G(x)$ for all $x\in M$. Indeed, the only difference in the proof is that, instead of considering annuli $\mathrm{An}(x,r,s)$ centered at some point $x\in M$, here we have to consider $G$-equivariant annuli $\bigcup_{\selem\in G}\selem(\mathrm{An}(x,r,s))$ `centered' at some point $x\in M$.
Choosing $r\colon M\to\R^+$ sufficiently small, for all $\mathrm{An}\in\mathcal{AN}_{r(x)}^G(x)$ we can assume that:
\begin{enumerate} [label={\normalfont(\roman*)}]
\item \label{an:almin} for all $j$ sufficiently large, $\Sigma^j$ is a smooth $G$-equivariant $(G,\delta_j,\varepsilon_j)$-almost minimizing surface in $\mathrm{An}$, for some $\delta_j,\varepsilon_j>0$, $\varepsilon_j\to0$ (by definition of $G$-almost minimality);
\item \label{an:nobdry} for all $x\in M\setminus\partial M$, $r(x)<d(x,\partial M)$;
\item \label{an:nosing} for all $x\in M\setminus\mathcal{S}$, $r(x)<d(x,\mathcal{S})$;
\item \label{an:nobadsing} for all $x\in \mathcal{S}\setminus\mathcal{S}_0$\footnote{Recall the definition of $\mathcal{S}_0$ in \cref{rem:StructureSingLocus}.}, $r(x)<d(x,\mathcal{S}_0)$;
\item \label{an:nosingbdry} for all $x\in M\setminus(\mathcal{S}\cap \partial M)$, $r(x)<d(x,\mathcal{S}\cap\partial M)$.
\end{enumerate}

Now, up to subsequence, we can assume that $\Sigma^j$ converges in the sense of varifolds to $\Xi$, as $j\to\infty$.
We follow the proof in \cite{ColdingDeLellis2003}*{Section 7.2}, using \cref{thm:AlmostMinimizing} in place of the nonequivariant analogue \cite{ColdingDeLellis2003}*{Lemma 7.4}, to prove that $\Xi$ admits a replacement in any $G$-equivariant annulus $\mathrm{An}\in \mathcal{AN}_{r(x)}^G(x)$.
Note that $\mathrm{An}$ satisfies all the assumptions of \cref{thm:AlmostMinimizing}; in particular it does not contain any point in $\mathcal{S}\cap \partial M$ and points of type $\dih_n$ for $n\ge2$, thanks to \ref{an:nobdry}, \ref{an:nosing}, \ref{an:nobadsing}, \ref{an:nosingbdry} above.

At this point, we apply \cref{thm:AlmostMinimizing} to a minimizing sequence $\{\Sigma^{j,k}\}_{k\in\N}$ for $\operatorname{Prob}(\Sigma^j,\Is^{\delta_j}_G(\mathrm{An},\Sigma^j))$, obtaining that the minimizing sequence converges in the sense of varifolds to $V^j$, which is a smooth $G$-stable free boundary minimal surface, up to multiplicity.
Moreover, the surfaces $V^j$ have uniformly bounded area and uniformly bounded genus in $\mathrm{An}$, because the genus is bounded by the genus of $\Sigma^j$ and $\genus(\Sigma^j)\le \sup_{t\in I^n\setminus \partial I^n} \genus(\Sigma_t) < +\infty$, because the $G$-sweepout $\{\Sigma_t\}_{t\in I^n}$ depends smoothly on $t\in I^n\setminus\partial I^n$ apart from finitely many $t$ (see \cref{def:GSweepout}). 
Therefore, by \cref{cor:ConvGStableSurf}, (up to subsequence) $V^j$ converges in the sense of varifolds to a varifold $V$, which is a $G$-equivariant, $G$-stable, free boundary minimal surface (up to multiplicity) in $\mathrm{An}$.

Observe that $\Haus^2(\Sigma^j)-\varepsilon_j \le \norm{V^j}\le \Haus^2(\Sigma^j)$, since $\Sigma^j$ is $(G,\delta_j,\varepsilon_j)$-almost minimizing in $\mathrm{An}$ and $V^j$ coincides with $\Sigma^j$ outside of $\mathrm{An}$. This implies that $\norm{V} = \lim_{j\to\infty}\norm{V^j} = \lim_{j\to\infty} \norm{\Sigma^j}= \norm{\Xi}$, because $\varepsilon_j\to0$, and that $V$ coincides with $\Xi$ outside of $\mathrm{An}$.
As a result, to conclude that $V$ is a $G$-replacement for $\Xi$, we just need to show that $V$ is free boundary stationary. However, exactly the same proof as in \cite{ColdingDeLellis2003}*{Section 7.2} works here as well.
Then, we conclude that $\Xi$ is a finite union of $G$-equivariant free boundary minimal surfaces (possibly with multiplicity) with the same argument as in \cite{ColdingDeLellis2003}*{Section 7.5}, since \cite{ColdingDeLellis2003}*{Proposition 6.3} works also using \cref{def:GReplacement} instead of \cite{ColdingDeLellis2003}*{Definition 6.1} (see \cref{rem:RegularityForGReplacements}).

Finally, the proof of the genus bound in \cref{thm:EquivMinMax} follows as in \cite{DeLellisPellandini2010} using \cref{thm:AlmostMinimizing} (in particular the second part of the statement, in small balls) in place of \cite{DeLellisPellandini2010}*{Proposition 3.2}. See \cite{Li2015}*{Section 9} for the adaptation to the free boundary case.
\end{proof}

\begin{remark}
Note that in the references above about the regularity and the genus bound (namely \cite{ColdingDeLellis2003,Li2015,DeLellisPellandini2010,Ketover2019,Ketover2016Equivariant,Ketover2016FBMS}) the definition of almost minimizing sequence (corresponding to \cref{def:AlmostMinimizing} in this paper) requires $\delta_j=\varepsilon_j/8$. However, the relation between $\delta_j$ and $\varepsilon_j$ in the definition of almost minimality is not important in the proof of \cref{prop:AlmostMinSeq}.
\end{remark}


\chapter[Index of equivariant min-max surfaces]{Index of equivariant min-max\\ surfaces} \label{chpt:IndexBound}

This chapter is devoted to the proof of the equivariant index bound in \cref{thm:EquivMinMax}.

\section{Deformation theorem}

In this section we prove that, in the setting of \cref{thm:EquivMinMax}, it is possible to modify a minimizing sequence in a way that, so to say, it avoids a given free boundary minimal surface with $G$-equivariant index greater or equal than $n+1$. The deformation theorem, \cref{thm:Deformation}, is the analogue of Deformation Theorem A in \cite{MarquesNeves2016}.
Before presenting it, we recall three lemmas contained in \cite{MarquesNeves2016} that are needed in the proof.
\begin{lemma} \label{lem:IndexDiffeo}
Let $(M^3,\smetric)$ be a three-dimensional Riemannian manifold with strictly mean convex boundary and let $G$ be a finite group of isometries of $M$.
Given a finite union $\avoids^2\subset M$ of $G$-equivariant free boundary minimal surfaces (possibly with multiplicity) with $\ind_G(\operatorname{spt}(\avoids))\ge n+1$, there exist $0<c_0<1$, $\delta>0$ and a smooth family $\{F_v\}_{v\in \overline{B}^{n+1}}\subset \operatorname{Diff}_G(M)$ of $G$-equivariant diffeomorphisms with:
\begin{enumerate}[label={\normalfont(\arabic*)}]
\item $F_0=\id$, $F_{-v} = F_v^{-1}$ for all $v\in\overline{B}^{n+1}$;
\item for any $V\in \overline{B}_{2\delta}^{\vard}(\avoids)$, the smooth function $A^V\colon\overline{B}^{n+1}\to[0,+\infty)$ given by \[ A^V(v) = \norm{(F_v)_\# V}(M)\] has a unique maximum at $m(V)\in B^{n+1}_{c_0/\sqrt{10}}(0)$ and it satisfies $-c_0^{-1}\id \le \Diff^2 A^V(v) \le -c_0 \id$ for all $v\in\overline{B}^{n+1}$.
\end{enumerate}
\end{lemma}
\begin{proof}
Using that $\ind_G(\operatorname{spt}(\avoids))\ge n+1$, we can find $G$-equivariant normal vector fields $X_1,\ldots,X_{n+1}$ on $\operatorname{spt}(\avoids)$ such that 
\[
Q^{\operatorname{spt}(\avoids)}\left(\sum_{i=1}^{n+1}a_iX_i,\sum_{i=1}^{n+1}a_iX_i\right) < 0
\]
for all $(a_1,\ldots,a_{n+1})\not=0\in\R^{n+1}$. These vector fields can be extended to $G$-equivariant vector fields defined in all $M$ (see \cref{rem:GequivExtension}).
Furthermore recall that, given the flow $\Phi\colon[0,+\infty)\times M\to M$ of a $G$-equivariant vector field $Y$, we have that $\Phi_t\in \operatorname{Diff}_G(M)$ for all $t\in[0,+\infty)$ (see \cref{rem:GequivVecFieldToFlow}). 
Given these observations, the proof follows exactly as in \cite{MarquesNeves2016}*{Proposition 4.3}, see also \cite{MarquesNeves2016}*{Definition 4.1}.
\end{proof}

The following two lemmas coincide exactly with \cite{MarquesNeves2016}*{Lemmas 4.5 and 4.4}, since the $G$-equivariance does not play any role in these two results. We report them here for the sake of expository convenience.
\begin{lemma}[{\cite{MarquesNeves2016}*{Lemma 4.5}}] \label{lem:PushAway}
In the setting of \cref{lem:IndexDiffeo}, for every $G$-equivariant varifold $V\in \overline{B}_{2\delta}^{\vard}(\avoids)$, let $\Phi^V\colon[0,+\infty)\times \overline{B}^{n+1}\to\overline{B}^{n+1}$ be the one-parameter flow generated by the vector field
\[
u\mapsto -(1-\abs{u}^2)\grad A^V(u), \quad u\in \overline{B}^{n+1},
\]
as defined also in \cite{MarquesNeves2016}*{pp. 476}.
Then, for all $0<\eta<1/4$, there is $T=T(\eta,\delta,\avoids,\{F_v\}_{v\in \overline{B}^{n+1}}, c_0) \ge 0$ such that, for all $V\in \overline{B}_{2\delta}^{\vard}(\avoids)$ and $v\in \overline{B}^{n+1}$ with $\abs{v-m(v)}\ge \eta$, we have 
\[
A^V(\Phi^V(T,v)) < A^V(0) - \frac{c_0}{10} \quad\text{and} \quad \abs{\Phi^V(T,v)} >\frac{c_0}{4}.
\]
\end{lemma}
\begin{remark}
Note that $\Phi^V$ is smooth since $A^V$ is. Moreover, for all $u\in \overline{B}^{n+1}$, the map $s\mapsto A^V(\Phi^V(s,u))$ is nonincreasing.
\end{remark}

\begin{lemma}[{\cite{MarquesNeves2016}*{Lemma 4.4}}] \label{lem:FarIsFar}
There exists $\overline{\eta}=\overline{\eta}(\delta,\avoids,\{F_v\}_{v\in \overline{B}^{n+1}}) >0$ such that, for any $G$-equivariant varifold $V\in (B_\delta^{\vard}(\avoids))^c$ with
\[
\norm{(F_v)_\# V}(M)\le \norm{V}(M) + \overline{\eta}
\]
for some $v\in \overline{B}^{n+1}$, we have $\vard((F_v)_\#V, \avoids)\ge 2\overline{\eta}$.
\end{lemma}

\begin{theorem}[Deformation theorem] \label{thm:Deformation}
Let $\{\so^j\}_{j\in\N}=\{\{\Sigma^j_t\}_{t\in I^n}\}_{j\in\N}$ be a minimizing sequence in the setting of \cref{thm:EquivMinMax}. Moreover, assume that
\begin{enumerate}[label={\normalfont(\roman*)}]
\item $\avoids^2$ is a finite union of $G$-equivariant free boundary minimal surfaces (possibly with multiplicity) with $\ind_G(\operatorname{spt}(\avoids))\ge n+1$;
\item $\area({\avoids}) = W_\Pi$;
\item $K$ is a compact set of varifolds such that $\avoids\not\in K$ and $\Sigma_t^j\not\in K$ for all $j\in\N$, $t\in I^n$.
\end{enumerate}
Then there exist $\varepsilon>0$ and another minimizing sequence $\{\boldsymbol{\Lambda}^j\}_{j\in\N}=\{\{\Lambda_t^j\}_{t\in I^n}\}_{j\in\N}\subset \Pi$ such that $\Lambda_t^j\cap(\overline{B}^{\vard}_\varepsilon(\avoids)\cup K) =\emptyset$ for all $j$ sufficiently large.
\end{theorem}

\begin{proof}
As aforementioned, the result is the analogue of Deformation Theorem A in \cite{MarquesNeves2016} and the idea of the proof is to exploit the fact that we have many negative directions for the second variation of the area along $\avoids$, hence it is possible to push the minimizing sequence $\{\so^j\}_{j\in\N}$ away from $\avoids$, keeping the fact that it is a minimizing sequence. Indeed, the sweepout is $n$-dimensional, while we have $n+1$ negative directions.

Since $\operatorname{spt}(\avoids)$ has $G$-equivariant index greater or equal than $n+1$, we can apply \cref{lem:IndexDiffeo} and obtain $0<c_0<1$, $\delta>0$ and the family $\{F_v\}_{v\in\overline{B}^{n+1}}\subset \operatorname{Diff}_G(M)$ given by the lemma. 
Moreover, up to modifying $\delta$ and $\{F_v\}_{v\in\overline{B}^{n+1}}$, we can assume that 
\begin{equation} \label{eq:AvoidK}
\vard(\avoids,F_v(\tilde\avoids)) \le \vard(\avoids, K)/2 \quad\text{for all $\tilde\avoids\in \overline{B}^{\vard}_{2\delta}(\avoids)$ and $v\in \overline{B}^{n+1}$}.
\end{equation}

Fixed $j\in\N$, define the open subset $U_j\subset I^n$ given by
\[
U_j\eqdef \{t\in I^n\st \vard(\avoids,\Sigma_t^j)<7\delta/4\}.
\]
Consider the continuous function $m_j\colon U_j\to B^{n+1}_{c_0/\sqrt{10}}(0)$ given by $m_j(t) = m(\Sigma_t^j)$, where the function $m$ is defined in \cref{lem:IndexDiffeo}. Since $\dim U_j = n < n+1= \dim (B^{n+1}_{c_0/\sqrt{10}}(0))$, by the transversality theorem given e.g. in \cite{Hirsch1994}*{Theorem 2.1}, there exists $\tilde m_j\colon U_j\to B^{n+1}_{c_0/\sqrt{10}}(0)$ such that $\tilde m_j(t)\not= 0$ and $\abs{\tilde m_j(t) - m_j(t)} < 2^{-j}$ for all $t\in U_j$. Hence, consider the function $a_j\colon U_j\to B^{n+1}_{2^{-j}}(0)$ given by $a_j(t) = m_j(t)-\tilde m_j(t)$ and note that $a_j(t)\not=m_j(t)$ for all $t\in U_j$. In particular, we can assume that there is $\eta_j>0$ such that $\abs{a_j(t) - m_j(t)}\ge \eta_j$ for all $t\in U_j$ (possibly taking $\delta$, and so $U_j$, smaller).

Now, for all $t\in U_j$, consider the one-parameter flow $\{\Phi^{t,j}(s,\cdot)\}_{s\ge 0} = \{\Phi^{\Sigma_t^j}(s,\cdot)\}_{s\ge 0}\subset \operatorname{Diff}(\overline{B}^{n+1})$ defined in \cref{lem:PushAway} and \[T_j = T(\eta_j,\delta, \avoids,\{F_v\}_{v\in\overline{B}^{n+1}}, c_0) \ge 0\] given by the lemma. Then, given a nonincreasing smooth function $\rho\colon [0,+\infty)\to [0,1]$ that is $1$ in $[0,3\delta/2]$ and $0$ in $[7\delta/4,+\infty)$, we define the continuous function
\[
v_j\colon I^n\to\overline{B}^{n+1}, \quad v_j(t) = \begin{cases}
\Phi^{t,j}(\rho(\vard(\avoids,\Sigma_t^j))T_j,\rho(\vard(\avoids,\Sigma_t^j))a_j(t)) & \text{for $t\in U_j$}\\
0 &\text{for $t\not\in U_j$},
\end{cases}
\]
and then set 
\[
\Lambda_t^j = \begin{cases}
                  F_{v_j(t)}(\Sigma_t^j)& \text{for $t\in U_j$}\\
\Sigma_t^j &\text{for $t\not\in U_j$}.
              \end{cases}
\]
Note that $\Lambda_t^j$ is $G$-equivariant since $\Sigma_t^j$ is $G$-equivariant and $F_{v_j(t)}\in \operatorname{Diff}_G(M)$.
However, a priori $\{\Lambda_t^j\}_{t\in I^n}$ is not contained in $\Pi$, since $(t,x)\mapsto F_{v_j(t)}(x)$ is not necessarily smooth but only continuous.
Anyway, let us first show that $\lim_{j\to+\infty}\sup_{t\in I^n} \area(\Lambda_t^j) \le W_\Pi$ and that $\Lambda_t^j \cap \overline{B}_\varepsilon^{\vard}(\avoids) =\emptyset$ for all $t\in I^n$, for $j$ sufficiently large, where $0<\varepsilon<\delta$ has to be chosen. Later, we will describe a regularization argument to get a sequence of sweepouts with the same properties of $\{\boldsymbol{\Lambda}^j\}_{j\in\N}$, but also contained in $\Pi$.

Observe that $\area(\Lambda_t^j) = \area(\Sigma_t^j)$ for $t\not\in U_j$ and, for $t\in U_j$, we have
\[
\area(\Lambda_t^j) = \area(F_{v_j(t)}(\Sigma_t^j)) \le \area(F_{\rho(\vard(\avoids,\Sigma_t^j)) a_j(t)}(\Sigma_t^j)).
\]
However $\abs{\rho(\vard(\avoids,\Sigma_t^j)) a_j(t)}\le 2^{-j}$, which implies that 
\[
\lim_{j\to+\infty} \max_{t\in I^n} \area(\Lambda_t^j) \le \lim_{j\to+\infty} \max_{t\in I^n} \area(\Sigma_t^j) = W_\Pi.
\]

Now let us prove that $\Lambda_t^j\cap\overline{B}_\varepsilon^{\vard}(\avoids) = \emptyset$ for all $t\in I^n$, for $j$ sufficiently large. 
Up to taking $\delta>0$ possibly smaller (note that \cref{lem:IndexDiffeo} still holds for $\delta$ smaller), we can assume that $\abs{\area(\tilde\avoids)-\area(\avoids)} \le c_0/20$ for all $\tilde\avoids\in \overline{B}_{2\delta}^{\vard}(\avoids)$.
Then, let us distinguish three cases:
\begin{itemize}
\item If $t\in I^n$ is such that $\vard(\avoids,\Sigma_t^j)\ge 7\delta/4$, then $\Lambda_t^j=\Sigma_t^j$ and therefore $\vard(\avoids,\Lambda_t^j) = \vard(\avoids,\Sigma_t^j) \ge 7\delta/4> \varepsilon$.

\item If $t\in I^n$ is such that $\vard(\avoids,\Sigma_t^j)\le 3\delta/2$, then we have $v_j(t) = \Phi^{t,j}(T_j,a_j(t))$ and therefore, by \cref{lem:PushAway}, it holds
\begin{align*}
\area(\Lambda_t^j) &= \area(F_{v_j(t)}(\Sigma_t^j)) = A^{\Sigma_t^j}(\Phi^{t,j}(T_j,a_j(t)))) < A^{\Sigma_t^j}(0) - \frac{c_0}{10} \\
&= \area(\Sigma_t^j) -\frac{c_0}{10} \le \area(\avoids) - \frac{c_0}{20},
\end{align*}
where the last inequality holds for $j$ sufficiently large.
Hence, it is possible to choose $\varepsilon>0$ possibly smaller (depending on $\avoids$ and $c_0$) such that this implies that $\vard(\avoids,\Lambda_t^j)>\varepsilon$ (indeed note that $c_0$ does not depend on $\varepsilon$).

\item If $t\in I^n$ is such that $3\delta/2\le \vard(\avoids,\Sigma_t^j)\le 7\delta/4$, then we apply \cref{lem:FarIsFar}.
Indeed, given $\overline{\eta} = \overline{\eta}(\delta, \avoids, \{F_v\}_{v\in \overline{B}^{n+1}})$ as in the lemma, for $j$ sufficiently large it holds that
\[
\area(\Lambda_t^j) = \area(F_{v_j(t)}(\Sigma_t^j)) \le \area(F_{\rho(\vard(\avoids,\Sigma_t^j) )a_j(t)}(\Sigma_t^j)) \le \area(\Sigma_t^j)+\overline{\eta},
\]
since $\abs{\rho(\vard(\avoids,\Sigma_t^j) )a_j(t)}\le 2^{-j}\to 0$. This implies that $\vard(\Lambda_t^j, \avoids)\ge 2\overline{\eta}$. Choosing $\varepsilon<2\overline{\eta}$, we then get that $\vard(\Lambda_t^j, \avoids)>\varepsilon$ for $j$ sufficiently large, as desired.
\end{itemize}

To conclude the proof, we need to address the regularity issue. For all $j\in\N$, let $\tilde v_j\colon I^n\to \overline{B}^{n+1}$ be a smooth function such that $\tilde v_j = 0$ on $I^n\setminus U_j$ and $\abs{\tilde v_j(t)-v_j(t)} \le 2^{-j}$ for all $t\in U_j$. Then, define
\[
\tilde\Lambda_t^j = \begin{cases}
                  F_{\tilde v_j(t)}(\Sigma_t^j)& \text{for $t\in U_j$}\\
\Sigma_t^j &\text{for $t\not\in U_j$}.
              \end{cases}
\]
Note that $\{\tilde\Lambda_t^j\}_{t\in I^n}\in \Pi$ for all $j\in\N$. Moreover, $\sup_{t\in I^n} \vard(\tilde\Lambda_t^j, \Lambda_t^j) \to 0$ as $j\to+\infty$, which implies that 
\[
\lim_{j\to+\infty} \sup_{t\in I^n}\area(\tilde\Lambda_t^j) = \lim_{j\to+\infty} \sup_{t\in I^n} \area(\Lambda_t^j) \le W_\Pi,
\]
and that $\tilde\Lambda_t^j\cap \overline{B}_\varepsilon^{\vard}(\avoids) = \emptyset$ for all $t\in I^n$, for $j$ sufficiently large (possibly taking $\varepsilon>0$ smaller).

Finally note that, thanks to \eqref{eq:AvoidK}, for all $t\in U_j$ it holds that $\tilde \Lambda_t^j = F_{\tilde v_j(t)}(\Sigma_t^j) \not \in K$, since $\Sigma_t^j\in \overline{B}^{\vard}_{2\delta}(\avoids)$ and $\tilde v_j(t)\in\overline{B}^{n+1}$. Moreover, for all $t\in I^n\setminus U_j$, we have $\tilde\Lambda_t^j = \Sigma_t^j\not\in K$. 
Hence $\{\tilde\Lambda_t^j\}_{t\in I^n}$ also avoids $K$ and thus satisfies the desired properties.
\end{proof}

\section{Proof of the equivariant index bound}

We now have all the tools to complete the proof of \cref{thm:EquivMinMax}. Inspired by the proofs of \cite{MarquesNeves2016}*{Theorems 6.1 and 1.2}, the idea consists in repeatedly applying \cref{thm:Deformation} in order to obtain a minimizing sequence in $\Pi$ that stays away from the $G$-equivariant free boundary minimal surfaces with $G$-equivariant index greater than $n$.

\begin{proof}[Proof of \cref{thm:EquivMinMax}]
First of all, let us assume that the metric $\smetric$ on the ambient manifold $M$ is contained in $\mathcal{B}_G^\infty$, defined in \cref{def:BumpyMetrics}, i.e., it is bumpy.
Let us consider the set $\mathcal V^{n+1}$ of finite unions (possibly with multiplicity) of $G$-equivariant free boundary minimal surfaces in $M$ with area $W_\Pi$ and whose supports have $G$-equivariant index greater or equal than $n+1$. We want to prove that there exists a minimizing sequence $\{\so^j\}_{j\in\N}\subset \Pi$ such that $C(\{\so^j\}_{j\in\N})\cap \mathcal V^{n+1} = \emptyset$.
First note that, since $\smetric\in\mathcal{B}_G^\infty$, the set $\mathcal V^{n+1}$ is at most countable thanks to \cref{prop:CountableGequivSurfaces} (because $\mathcal{V}^{n+1}$ consists of finite unions with integer multiplicities of $G$-equivariant free boundary minimal surfaces). Therefore, we can write $\mathcal V^{n+1} = \{\avoids_1,\avoids_2,\ldots\}$. Now, the idea is to repeatedly apply \cref{thm:Deformation} in order to avoid all the elements in $\mathcal V^{n+1}$.

Let us consider a minimizing sequence $\{\so^j\}_{j\in\N}$ and apply \cref{thm:Deformation} with $\avoids=\avoids_1$. Then we get that there exist $\varepsilon_1>0$, $j_1\in\N$ and another minimizing sequence $\{\so^{1,j}\}_{j\in\N}\subset\Pi$ such that $\Sigma^{1,j}_t\cap \overline{B}^{\vard}_{\varepsilon_1}(\avoids_1)=\emptyset$ for all $j\ge j_1$ and $t\in I^n$. Moreover, we can assume that no $\avoids_k$ belongs to $\partial {B}^{\vard}_{\varepsilon_1}(\avoids_1)$.
Let us now consider $\avoids_2$: if it belongs to $\overline{B}^{\vard}_{\varepsilon_1}(\avoids_1)$, we choose $\varepsilon_2=\varepsilon_1-\vard(\avoids_1,\avoids_2)>0$ (here we use that $\avoids_2\not\in \partial {B}^{\vard}_{\varepsilon_1}(\avoids_1)$); otherwise we apply \cref{thm:Deformation} with $\avoids=\avoids_2$ and $K=\overline{B}^{\vard}_{\varepsilon_1}(\avoids_1)$. In both cases, we get $\varepsilon_2>0$, $j_2\in\N$ and another minimizing sequence $\{\so^{2,j}\}_{j\in\N}\subset\Pi$ such that $\Sigma^{2,j}_t\cap (\overline{B}^{\vard}_{\varepsilon_1}(\avoids_1)\cup \overline{B}^{\vard}_{\varepsilon_2}(\avoids_2))=\emptyset$ for all $j\ge j_2$ and $t\in I^n$. Moreover, we can assume again that no $\avoids_k$ belongs to $\partial {B}^{\vard}_{\varepsilon_2}(\avoids_2)$.

Then we proceed inductively for all $\avoids_k$'s and we have two possibilities:
\begin{itemize}
\item The process ends in finitely many steps. In this case there exist $m>0$, a minimizing sequence $\{\so^{m,j}\}_{j\in\N}\subset\Pi$, $\varepsilon_1,\ldots,\varepsilon_m>0$ and $j_m\in\N$ such that 
\[\Sigma^{m,j}_t\cap (\overline{B}^{\vard}_{\varepsilon_1}(\avoids_{1})\cup\ldots \cup \overline{B}^{\vard}_{\varepsilon_m}(\avoids_{m}))=\emptyset \]
for all $j\ge j_m$ and $t\in I^n$ and $\mathcal V^{n+1}\subset {B}^{\vard}_{\varepsilon_1}(\avoids_{1})\cup\ldots \cup {B}^{\vard}_{\varepsilon_m}(\avoids_{m})$.
\item The process continues indefinitely. In this case for all $m>0$ there exist a minimizing sequence $\{\so^{m,j}\}_{j\in\N}\subset\Pi$, $\varepsilon_m>0$ and $j_m\in\N$ such that $\Sigma^{m,j}_t\cap (\overline{B}^{\vard}_{\varepsilon_1}(\avoids_{1})\cup\ldots \cup \overline{B}^{\vard}_{\varepsilon_m}(\avoids_{m})) =\emptyset$ for all $j\ge j_m$ and $t\in I^n$ and no $\avoids_k$ belongs to $\partial {B}^{\vard}_{\varepsilon_1}(\avoids_{1})\cup\ldots \cup \partial {B}^{\vard}_{\varepsilon_m}(\avoids_{m})$.
\end{itemize}
In the first case we define $\boldsymbol{\Lambda}^i = \so^{m,i}$, while in the second case we set $\boldsymbol{\Lambda}^i= \so^{i,l_i}$ for all $i\in\N$, for some $l_i\ge j_i$ such that $\{\boldsymbol{\Lambda}^i\}_{i\in\N}\subset\Pi$ is a minimizing sequence and $C(\{\boldsymbol{\Lambda}^i\}_{i\in\N})\cap \mathcal V^{n+1}=\emptyset$. 
Hence, we can apply \cref{prop:ConvergenceToStationary}, \cref{lem:AlmostMinSeqInAdmissible,prop:AlmostMinSeq} to conclude the proof in the case of $\smetric\in\mathcal{B}_G^\infty$.

Now, consider the case of an arbitrary metric $\smetric$ and let $\{\smetric_k\}_{k\in\N}$ be a sequence of metrics in $\mathcal{B}_G^\infty$ converging smoothly to $\smetric$, which exists because of \cref{thm:BumpyIsGGeneric}. Thanks to the first part of the proof (in particular applying \cref{prop:ConvergenceToStationary,lem:AlmostMinSeqInAdmissible} to the minimizing sequence found in the first part of the proof with respect to $\gamma_k$), for every $k\in\N$ there exists a $G$-equivariant min-max sequence $\{\Lambda^{(k),j}_{t_j}\}_{j\in\N}\subset\Pi$ (i.e., $\area_{\gamma_k}(\Lambda^{(k),j}_{t_j}) \to W_{\Pi,\smetric_k}$, the width of $\Pi$ with respect to the metric $\smetric_k$) that is $G$-almost minimizing (with respect to $\smetric_k$) in every $L$-admissible family of $G$-equivariant annuli with $L=(3^n)^{3^n}$ and that converges in the sense of varifolds to a finite union $\limitv_k$ of $G$-equivariant free boundary minimal (with respect to $\gamma_k$) surfaces (possibly with multiplicity). Moreover, it holds $\ind_G(\operatorname{spt} (\limitv_k))\le n$ and $\area_{\gamma_k}(\limitv_k) = W_{\Pi, \smetric_k}$.

Note that $W_{\Pi, \smetric_k}$ converges to the width $W_\Pi = W_{\Pi, \smetric}$ (the proof is the same as in the Almgren--Pitts setting, for which one can see \cite{IrieMarquesNeves2018}*{Lemma 2.1}). Hence, since the varifolds $\limitv_k$ have uniformly bounded mass, up to subsequence $\limitv_k$ converges in the sense of varifolds to a varifold $\limitv$ with mass equal to $W_\Pi$. Moreover, taking a suitable diagonal subsequence of $\{\Lambda^{(k),j}_{t_j}\}_{j,k\in\N}$ we can obtain a min-max sequence $\{\Lambda^j\}_{j\in\N}$ for $\Pi$ with respect to $\smetric$, converging in the sense of varifolds to $\limitv$ and which is $G$-almost minimizing (with respect to $\smetric$) in every $L$-admissible family of $G$-equivariant annuli.
More precisely, we can get the min-max sequence $\{\Lambda^j\}_{j\in\N}$ in such a way that $\Lambda_j$ is $(G,\varepsilon_j/2^{n+3},2\varepsilon_j)$-almost minimizing in every $L$-admissible family of $G$-equivariant annuli, for some sequence of positive numbers $\varepsilon_j\to0$. 
Indeed, given $\varepsilon_j>0$, if $\smetric_k$ is sufficiently close (depending on $\varepsilon_j$) to $\smetric$, and $\Lambda_j$ is $(G,\varepsilon_j/2^{n+2},\varepsilon_j)$-almost minimizing in every $L$-admissible family of $G$-equivariant annuli with respect to $\smetric_k$, then $\Lambda_j$ is  $(G,\varepsilon_j/2^{n+3},2\varepsilon_j)$-almost minimizing in every $L$-admissible family of $G$-equivariant annuli with respect to $\smetric$.

Now, given the min-max sequence $\{\Lambda^j\}_{j\in\N}$, we can apply \cref{prop:AlmostMinSeq}, and obtain that $\limitv$ is a disjoint union of $G$-equivariant free boundary minimal surfaces (possibly with multiplicity) and the genus bound in the statement holds.
One can look at \cref{fig:scheme} for a scheme of the argument.

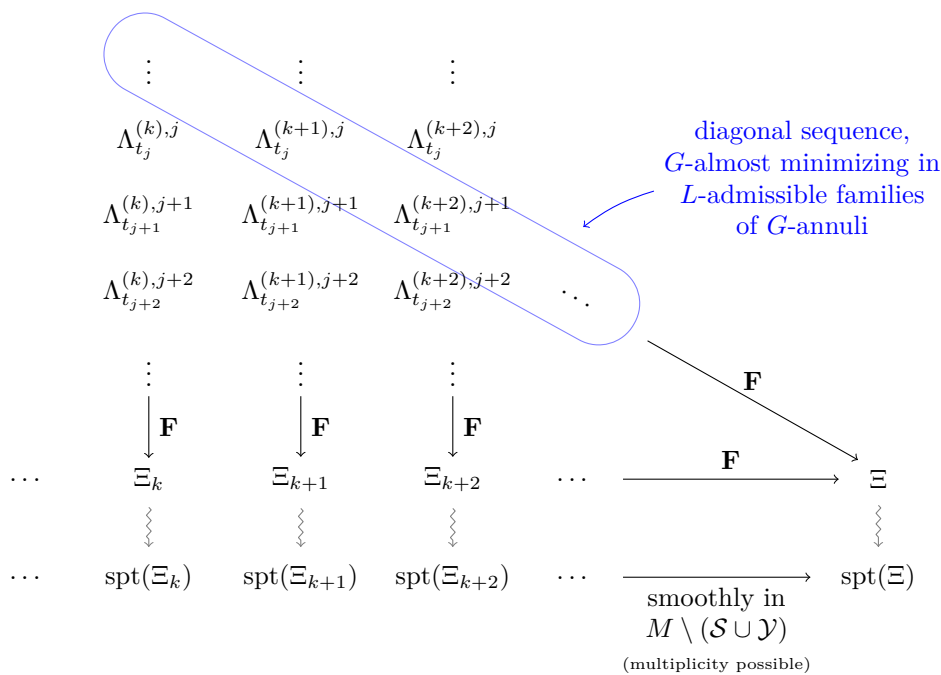
\begin{figure}[htpb]
\centering
\usetikzlibrary{trees}
\tikzstyle{every node}=[anchor=west]
\tikzstyle{tit}=[shape=rectangle, rounded corners,
    draw, minimum width = 3cm]
\tikzstyle{sec}=[shape=rectangle, rounded corners]
\tikzstyle{optional}=[dashed,fill=gray!50]
\begin{tikzpicture}[scale=0.9]

\tikzstyle{mybox} = [draw, rectangle, rounded corners, inner sep=10pt, inner ysep=20pt]
\tikzstyle{fancytitle} =[fill=white, text=black, draw]

\begin{scope}[align=center, font=\small]
\node[sec] (d1) {$\vdots$};
\node[sec] (lk0) [below of = d1] {$\Lambda_{t_j}^{(k),j}$};
\node[sec] [below of = lk0] (lk1) {$\Lambda_{t_{j+1}}^{(k),j+1}$};
\node[sec] [below of = lk1] (lk2) {$\Lambda_{t_{j+2}}^{(k),j+2}$};
\node[sec] [below of = lk2] (d12) {$\vdots$};
\node[sec] [below of = d12, yshift=-.5cm] (x0) {$\Xi_k$};
\node[sec] [below of = x0, yshift=-.3cm] (s0) {$\operatorname{spt}(\Xi_k)$};
\draw[->] (d12) -- node {$\vard$} (x0);

\node[sec] [right of = d1, xshift=1cm] (d2) {$\vdots$};
\node[sec] (l10) [below of = d2] {$\Lambda_{t_j}^{(k+1),j}$};
\node[sec] [below of = l10] (l11) {$\Lambda_{t_{j+1}}^{(k+1),j+1}$};
\node[sec] [below of = l11] (l12) {$\Lambda_{t_{j+2}}^{(k+1),j+2}$};
\node[sec] [below of = l12] (d22) {$\vdots$};
\node[sec] [below of = d22, yshift=-.5cm] (x1) {$\Xi_{k+1}$};
\node[sec] [below of = x1, yshift=-.3cm] (s1) {$\operatorname{spt}(\Xi_{k+1})$};
\draw[->] (d22) -- node {$\vard$} (x1);

\node[sec] [right of = d2, xshift=1cm] (d3) {$\vdots$};
\node[sec] (l20) [below of = d3] {$\Lambda_{t_j}^{(k+2),j}$};
\node[sec] [below of = l20] (l21) {$\Lambda_{t_{j+1}}^{(k+2),j+1}$};
\node[sec] [below of = l21] (l22) {$\Lambda_{t_{j+2}}^{(k+2),j+2}$};
\node[sec] [below of = l22] (d32) {$\vdots$};
\node[sec] [below of = d32, yshift=-.5cm] (x2) {$\Xi_{k+2}$};
\node[sec] [below of = x2, yshift=-.3cm] (s2) {$\operatorname{spt}(\Xi_{k+2})$};
\draw[->] (d32) -- node {$\vard$} (x2);

\node[sec] [left of = x0, xshift = -.6cm] (dd1) {\ldots};
\node[sec] [below of = dd1, yshift=-.3cm] {\ldots};
\node[sec] [right of = x2, xshift = .6cm] (do) {\ldots};
\node[sec] [below of = do, yshift=-.3cm] (dos) {\ldots};

\node[sec]  [right of = do, xshift = 3cm] (lim) {$\Xi$};
\node[sec] [below of = lim, yshift=-.3cm] (limspt) {$\operatorname{spt}(\Xi)$};
\draw[->] ([xshift=3mm]do.east) -- node[above] {$\vard$} ([xshift=-3mm]lim.west);
\draw[->] ([xshift=3mm]dos.east) -- node[below] {smoothly in \\ $M\setminus(\mathcal{S}\cup\mathcal{Y})$\\
\tiny(multiplicity possible)} ([xshift=-3mm]limspt.west);

\node[sec] [right of = l21, xshift=.6cm, yshift=-1cm, rotate=9] (diag) {$\ddots$};

\begin{scope}[on background layer]
\draw[blue!50!white,rounded corners=15pt,opacity=1]
    let \p1=($(l10)!-32mm!(diag)$),
        \p2=($(diag)!-10mm!(l10)$),
        \p3=($(\p1)!6mm!90:(\p2)$),
        \p4=($(\p1)!6mm!-90:(\p2)$),
        \p5=($(\p2)!6mm!90:(\p1)$),
        \p6=($(\p2)!6mm!-90:(\p1)$)
    in
    (\p3) -- (\p4)-- (\p5) -- (\p6) -- cycle;
\end{scope}

\draw[->] ([xshift=1cm,yshift=-.3cm]diag.south) -- node[above] {$\vard$} ([xshift=-.3cm]lim.north);
\node[sec] [blue,above of = diag,xshift=3cm,yshift=5mm] (nn) {diagonal sequence,\\$G$-almost minimizing in\\ $L$-admissible families\\ of $G$-annuli};
\draw[->,blue] ([yshift=-2mm]nn.west) to [bend right=10] ([yshift=5mm,xshift=2mm]diag.north);

\draw [->, gray, line join=round, decorate, decoration={
    zigzag, segment length=4, amplitude=.9,post=lineto, post length=2pt
}]  ([yshift=-1mm]x0.south) -- ([yshift=1mm]s0.north);
\draw [->, gray, line join=round, decorate, decoration={
    zigzag, segment length=4, amplitude=.9,post=lineto, post length=2pt
}]  ([yshift=-1mm]x1.south) -- ([yshift=1mm]s1.north);
\draw [->, gray, line join=round, decorate, decoration={
    zigzag, segment length=4, amplitude=.9,post=lineto, post length=2pt
}]  ([yshift=-1mm]x2.south) -- ([yshift=1mm]s2.north);
\draw [->, gray, line join=round, decorate, decoration={
    zigzag, segment length=4, amplitude=.9,post=lineto, post length=2pt
}]  ([yshift=-1mm]lim.south) -- ([yshift=1mm]limspt.north);
\end{scope}

\end{tikzpicture}
\caption{Scheme of the convergence argument in the proof of \cref{thm:EquivMinMax}.}
\label{fig:scheme}
\end{figure}

Note that, thanks to \cref{thm:ConvBoundGIndex} (see also \cref{rem:AlsoForConvMetrics}), up to extracting a further subsequence, we can assume that $\operatorname{spt}(\limitv_k)$ converges smoothly (possibly with multiplicity) to a free boundary minimal surface away from the singular locus $\mathcal{S}$ and, possibly, from finitely many additional points $\mathcal{Y}$. This free boundary minimal surface coincides with the limit of $\operatorname{spt}(\limitv_k)$ in the sense of varifolds, which is $\operatorname{spt}(\limitv)$, in $M\setminus(\mathcal{S}\cup\mathcal{Y})$. In particular, we get that $\ind_G(\operatorname{spt}(\limitv))\le n$, which concludes the proof.
\end{proof}


\chapter[FBMS with connected boundary]{Free boundary minimal surfaces\\ with connected boundary} \label{chpt:FBMSb1}

In this chapter, we employ \cref{thm:EquivMinMax} to construct free boundary minimal surfaces with connected boundary and arbitrary genus in the three-dimensional unit ball $B^3$.

\section{Effective sweepouts}
\label{sec:sweepout}

First of all, for all $g\ge 1$, we need to construct a $\dih_{g+1}$-sweepout, as defined in \cref{def:GSweepout}. This is achieved in the following lemma.

\begin{lemma}\label{lem:ConstructionSweepout}
Given $1\leq g\in\N$ there exists a $\dih_{g+1}$-sweepout $\{\Sigma_t\}_{t\in[0,1]}$ of $B^3$ such that $\Haus^2(\Sigma_0)=\Haus^2(\Sigma_1)=\pi$ and such that for every $0<t<1$ 
\begin{itemize} 
\item the surface $\Sigma_t$ has genus $g$, 
\item the boundary of $\Sigma_t$ is connected,
\item the area of $\Sigma_t$ is strictly less than $3\pi$. 
\end{itemize}
\end{lemma}

The idea behind our construction is to equivariantly glue \emph{three} parallel discs through suitably controlled ribbons. 

\begin{remark}\label{rem:KetSweep}
In a partly similar way Ketover \cite[Theorem 5.1]{Ketover2016FBMS} glued two discs in order to variationally construct free boundary minimal surfaces of genus zero and $b\geq 2$ boundary components (to be compared with the existence result by Fraser and Schoen \cite[Theorem 1.1]{FraserSchoen2016}).
\end{remark}

Let $1\leq g\in\N$ be fixed. 
Let $\Eq=\{x\in \overline{B^3}\st x_3=0\}$ be the equatorial disc in the closed unit ball and let $\overline{B_\varepsilon}(p)=\{x\in\R^3\st \abs{x-p}\leq \varepsilon\}$ denote the closed ball of radius $\varepsilon>0$ around any given $p\in\R^3$. 
For all $k\in\{0,\ldots,g\}$ we consider the points 
\begin{align*}
p_{k}^{\pm}&\vcentcolon=\left(
\cos\left(\tfrac{2k\pm\frac{1}{2}}{g+1}\pi\right),
\sin\left(\tfrac{2k\pm\frac{1}{2}}{g+1}\pi\right),
0\right)
\end{align*}
on the equator and the subsets
\begin{align}\label{eqn:D+-epsilon}
\Eq_{\varepsilon}^{\pm}&\vcentcolon=\Eq\setminus\bigcup^{g}_{k=0} B_\varepsilon(p_k^{\pm}), 
\end{align}
as shown in \cref{fig:20190919}, which we then scale and translate upwards (or downwards) to define
\begin{align}\label{eqn:20191106-D}
\Eq_{t,\varepsilon}^{\pm}&\vcentcolon=\bigl(\sqrt{1-t^2}\Eq_\varepsilon^{\pm}\bigr)\pm(0,0,t)
\end{align}
for all $t\in[0,1]$.  
\begin{figure}%
\centering
\pgfmathsetmacro{\eps}{0.25}
\pgfmathsetmacro{\genus}{2} 
\pgfmathsetmacro{\alph}{2*asin(\eps/2)}
\pgfmathsetmacro{\globalscale}{3}
\begin{tikzpicture}[scale=\globalscale,line cap=round,line join=round,baseline={(0,0)},rotate=90/(\genus+1)]
\draw[fill=lightgray](-\alph:1)
\foreach\k in {0,...,\genus}{
arc({\k*360/(\genus+1)+180+180/2-\alph/2}:{\k*360/(\genus+1)+180-180/2+\alph/2}:\eps)arc(\k*360/(\genus+1)+\alph:{\k*360/(\genus+1)+1*360/(\genus+1)-\alph}:1)
};
\foreach\k in {0,...,\genus}{
\draw({\k*360/(\genus+1)}:1)node(p\k){$\scriptstyle\bullet$}
node[anchor=180+(2*\k+1/2)*180/(\genus+1)]{$p_{\k}^{+}$};
}
\draw[latex-latex](p0.center)--++({180-90+180/(\genus+1)}:\eps)node[midway,above]{$\varepsilon$};
\draw[latex-latex](0,0)--({-90/(\genus+1)}:1)node[midway,below]{$1$};
\draw[very thin](0:0.4)arc(0:{180/(\genus+1)}:0.4);
\draw({90/(\genus+1)}:0.27)node{$\frac{\pi}{g+1}$};
\draw[dotted]
(0,0)--(p0.center)
(0,0)--({180/(\genus+1)}:1)
(p0.center)--(p1.center)
;
\draw(0,0)node[below=3ex]{$\Eq_{\varepsilon}^{+}$};
\end{tikzpicture}
\hfill
\begin{tikzpicture}[scale=\globalscale,line cap=round,line join=round,baseline={(0,0)},rotate=-90/(\genus+1)]
\draw[fill=lightgray](-\alph:1)
\foreach\k in {0,...,\genus}{
arc({\k*360/(\genus+1)+180+180/2-\alph/2}:{\k*360/(\genus+1)+180-180/2+\alph/2}:\eps)arc(\k*360/(\genus+1)+\alph:{\k*360/(\genus+1)+1*360/(\genus+1)-\alph}:1)
};
\foreach\k in {0,...,\genus}{
\draw({\k*360/(\genus+1)}:1)node(p\k){$\scriptstyle\bullet$}
node[anchor=180+(2*\k-1/2)*180/(\genus+1)](pn\k){$p_{\k}^{-}$};
}
\draw(0,0)node[below=3ex]{$\Eq_{\varepsilon}^{-}$};
\foreach\k in {-0.5,0,...,\genus}{
\draw[dashed](0,0)--({(\k+1/4)*360/(\genus+1)}:1);
};
\foreach[count=\i]\k in {0,...,\genus}{
\draw ({(\k+1.5)*180/(\genus+1)}:0.55)node[inner sep=2pt,circle,anchor={180+\k*180/(\genus+1)}]{$\xi_\i$};
};
\end{tikzpicture}
\caption{Top view of the sets $\Eq_{\varepsilon}^{\pm}$  for $\varepsilon=1/4$ and $g=\genus$.} 
\label{fig:20190919} 
\end{figure}
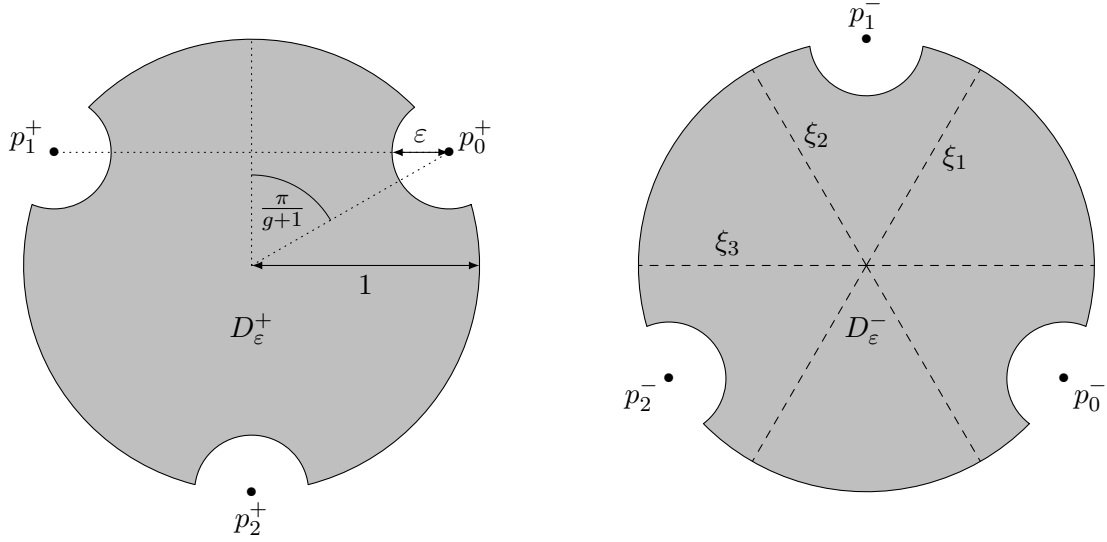
Now we connect the three sets $\Eq_{t,\varepsilon}^{+}$, $\Eq_{t,\varepsilon}^{-}$ and $(\Eq_{\varepsilon}^{+}\cap \Eq_{\varepsilon}^{-})$ in a $\dih_{g+1}$-equivariant way by means of $2(g+1)$ ribbons. 
Note that each of these sets is a nonempty, connected subset of $\overline{B^3}$ provided that $\varepsilon<\sin(\pi/(2g+2))$. 
Let $0<t_0<1$ be a fixed, small value which will be specified later in  \eqref{eqn:20191018-h0}. 
For each $t\in[t_0,1]$ we define 
\begin{align*} 
\Omega_{t,\varepsilon}^{\pm}&\vcentcolon=\bigcup_{s\in[0,t]}\Eq_{s,\varepsilon}^{\pm}, &
S_{t,\varepsilon}^{\pm}&\vcentcolon=\overline{\partial\Omega_{t,\varepsilon}^{\pm}\setminus(\partial B^3\cup \Eq)},
\end{align*}
where the symbol $\partial$ refers to the topological boundary in $\R^3$, and
\begin{align}\label{eqn:20191106-E}
\Sigma_{t}&\vcentcolon=S_{t,\varepsilon}^{+}\cup S_{t,\varepsilon}^{-}\cup(\Eq_{\varepsilon}^{+}\cap \Eq_{\varepsilon}^{-}).
\end{align}
Moreover, we allow $\varepsilon\colon\Interval{t_0,1}\to\intervaL{0,\varepsilon_0}$ to be a continuous function of $t$, bounded from above by some sufficiently small $\varepsilon_0>0$ which we choose later in \eqref{eqn:20191018-h0} depending on $t_0$ and $g$, such that $\varepsilon(t)\to0$ as $t\nearrow1$. 
Then we define $\Sigma_1$ to be the union of the equatorial disc $\Eq$ with the (shortest) geodesic arcs connecting $p_k^+$ with the north pole and $p_k^{-}$ with the south pole for each $k\in\{0,\ldots,g\}$. 
The construction is visualized in the first, second and third image of \cref{fig:sweepout1}. 

\begin{figure}
\centering
\includegraphics[page=1]{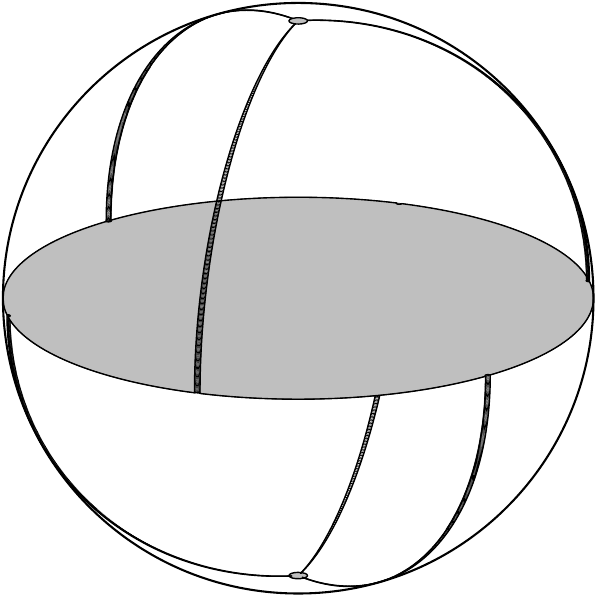}$\mathllap{\Sigma_{\mathrlap{0.9995}}}$%
\hfill
\includegraphics[page=2]{pic-sweepout1.pdf}$\mathllap{\Sigma_{\frac{2}{3}\hphantom{t_0}}}$%

\bigskip 

\includegraphics[page=3]{pic-sweepout1.pdf}$\mathllap{\Sigma_{\mathrlap{t_0}}}$%
\hfill
\includegraphics[page=4]{pic-sweepout1.pdf}$\mathllap{\Sigma_{\frac{1}{2}t_0}}$%

\bigskip 

\includegraphics[page=5]{pic-sweepout1.pdf}$\mathllap{\Sigma_{\mathrlap{\frac{1}{4}t_0}}}$%
\hfill
\includegraphics[page=2]{pic-sweepout4.pdf}$\mathllap{\Sigma_{\frac{1}{16}t_0}}$%

\caption{Construction of an effective sweepout in the case $g=2$. 
In the first three images, $\varepsilon$ has been increased and relation \eqref{eqn:20191018-h0} ignored for the sake of clarity.  }%
\label{fig:sweepout1}%
\end{figure}

Arriving at $\Sigma_{t_0}$, one would like to increase $\varepsilon$ (as we further \emph{decrease} $t$) in order to retract the three sets $\Eq_{t,\varepsilon}^{+}$, $\Eq_{t,\varepsilon}^{-}$ and $(\Eq_{\varepsilon}^{+}\cap \Eq_{\varepsilon}^{-})$ as illustrated in the fourth image of \cref{fig:sweepout1}. 
However, this requires refined control on the area of the widening ribbons as we point out in the following statement.

\begin{lemma}\label{lem:20191106-21:35}
Let $\Sigma_{t}$ be as given in \eqref{eqn:20191106-E}. 
Then its area satisfies
\begin{align}
\label{eqn:20191018}
\Haus^2({\Sigma_t})&\leq 3\pi-2\pi\bigl(t^2-(g+1)\varepsilon t\bigr).
\end{align} 
\end{lemma}
 
\begin{proof}
Obviously, the set $S_{t,\varepsilon}^{+}$ has the same area as $S_{t,\varepsilon}^{-}$, by symmetry. Furthermore $S_{t,\varepsilon}^{+}$ is the union of $\Eq_{t,\varepsilon}^{+}$ defined in \eqref{eqn:20191106-D} with $(g+1)$ ribbons. 
By construction, 
\begin{align}\label{eqn:20191108}
\Haus^2({\Eq_{t,\varepsilon}^{+}})&\leq \pi(1-t^2),  &
\Haus^2({\Eq_{\varepsilon}^{+}\cap \Eq_{\varepsilon}^{-}})&\leq\pi. 
\end{align}
The intersection of one ribbon with the horizontal plane at height $s\in[0,t]$ is an arc of length less than $\pi\sqrt{1-s^2}\varepsilon$.
Hence, using the coarea formula one gets at once that the area of one ribbon is bounded from above by
\begin{align*}
\int^{t}_{0}\sqrt{1+\frac{s^2}{1-s^2}}\sqrt{1-s^2}\,\pi\varepsilon\,ds=\pi\varepsilon t.
\end{align*} 
Therefore, 
\(
\Haus^2({\Sigma_t})=2\Haus^2({S_{t,\varepsilon}^{+}})+\Haus^2({\Eq_{\varepsilon}^{+}\cap \Eq_{\varepsilon}^{-}})
\leq 2\pi(1-t^2)+2(g+1)\pi\varepsilon t+\pi
\), which allows to conclude by simply rearranging the terms.
\end{proof}

\begin{remark}
\cref{lem:20191106-21:35} implies that if $\varepsilon>0$ is \emph{small} compared to $t_0$, more precisely, if $(g+1)\varepsilon<t_0$, then $\Haus^2({\Sigma_t})<3\pi$ holds for all $t\in[t_0,1]$. In estimate \eqref{eqn:20191108} we did not take into account that by definition~\eqref{eqn:D+-epsilon} of $D_{\varepsilon}^{\pm}$, small balls of radius $\varepsilon>0$ are removed around the points $p_k^{\pm}$ (and thus similarly for $\Eq_{t,\varepsilon}^{\pm}$). 
If we subtracted these contributions, one could then easily prove that the inequality $\Haus^2({\Sigma_{t_0}})<3\pi$ would also hold for $\varepsilon>0$ \emph{large} compared to $t_0$.
However, even the improved right-hand side of \eqref{eqn:20191018} would \emph{not} stay below $3\pi$ if we increased $\varepsilon$ continuously from small to large values. For instance, that bound would be violated if one took $\varepsilon=2t_0/(g+1)$.
This is the reason why we need to refine the construction for $t<t_0$ and appeal to the so-called \emph{catenoid estimate} instead. 
\end{remark}

At this stage, it would be possible to proceed by appealing to a suitable variant (for boundary points) of \cite[Theorem 2.4]{KetoverMarquesNeves2020}. However, for our specific scopes we work out the explicit construction in our Euclidean setting. 

Fix $0<r<\sin(\pi/(2g+2))$ and $0< h < \min\{\tanh(1)/2, 1/5 \}r = r/5$. Moreover we choose $h$ such that we also have that $-\log h > 8(g+1)$. For every $s\geq 0$ consider the surfaces
\begin{align}\label{eqn:20191115-18:51}
C_{s}^{r,h}\vcentcolon=\Bigl\{x\in\R^3\st \sqrt{x_1^2+x_2^2}=\frac{r\cosh(s x_3)}{\cosh(s h)},~ \abs{x_3}\leq h\Bigr\},
\end{align}
which all span two parallel circles of radius $r$ and distance $2h$. 
If $s$ is chosen such that 
\begin{align}\label{eqn:catenoid}
r s=\cosh(s h),
\end{align}
then $C_{s}^{r,h}$ is a subset of a (rescaled) catenoid and hence a minimal surface. 
For our choices of $r$ and $h$, equation \eqref{eqn:catenoid} has two positive solutions $s_1(r,h)<s_2(r,h)$. 
The smaller one corresponds to the stable catenoid and the larger one to the unstable catenoid.

The family of surfaces in question interpolates between the cylinder at $s=0$ and the union of a line segment with two discs of radius $r$ which we denote by $C_{\infty}^{r,h}$. 
The unstable catenoid can be regarded as the slice of largest area in this family, as we prove in \cref{sec:maximal} below.
The catenoid estimate given in \cite[Proposition\;2.1]{KetoverMarquesNeves2020}, combined with \cref{lem:maximal} (postponed to the next section), implies that we can choose $h$ possibly smaller (only depending on $r$) such that for all $s\geq0$
\begin{align}\label{eqn:20191115-18:52}
\Haus^2({C_s^{r,h}})&\leq \Haus^2({C_\infty^{r,h}})+\frac{4\pi h^2}{(-\log h)}
=2\pi r^2+\frac{4\pi h^2}{(-\log h)}. 
\end{align} 
Let $E_R\vcentcolon=\{x\in\R^3\st x_1^2+(x_2-R)^2+x_3^2< R^2\}$ be the ball of radius $R>1$ around the point $(0,R,0)$. 
By symmetry, the catenoid estimate holds under restriction to the half-space $E_\infty\vcentcolon=\{x\in\R^3\st x_2>0\}$, i.e,
\begin{align*}
\sup_{s\ge 0}\ \Haus^2({C_{s}^{r,h}\cap E_\infty})&\leq \Haus^2({C_{\infty}^{r,h}\cap E_\infty})+\frac{2\pi h^2}{(-\log h)}. 
\end{align*} 
By a simple continuity argument, there exists $R_0 = R_0(r,h)>1$ such that
\begin{align*}
\sup_{s\ge 0}\ \Haus^2({C_{s}^{r,h}\cap E_{R_0}})&\leq \Haus^2({C_{\infty}^{r,h}\cap E_{R_0}})+\frac{4\pi h^2}{(-\log h)}.
\end{align*}
Hence, renaming $r\mapsto {r}/{R_0}$ and $h\mapsto {h}/{R_0}$ (which corresponds to rescaling the whole picture by a factor of $1/R_0$), we obtain
\begin{equation}
\label{eqn:20191021-1051}
\begin{split}
\sup_{s\ge 0}\ \Haus^2({C_{s}^{r,h}\cap E_1})&\leq \Haus^2({C_{\infty}^{r,h}\cap E_1})+\frac{4\pi h^2}{(-\log h + \log R_0)}\\
&\le  \Haus^2({C_{\infty}^{r,h}\cap E_1})+\frac{4\pi h^2}{(-\log h)}.
\end{split}
\end{equation} 
Observe that the conditions imposed on the smallness of $r$, $h$ and $h/r$ are still fulfilled.

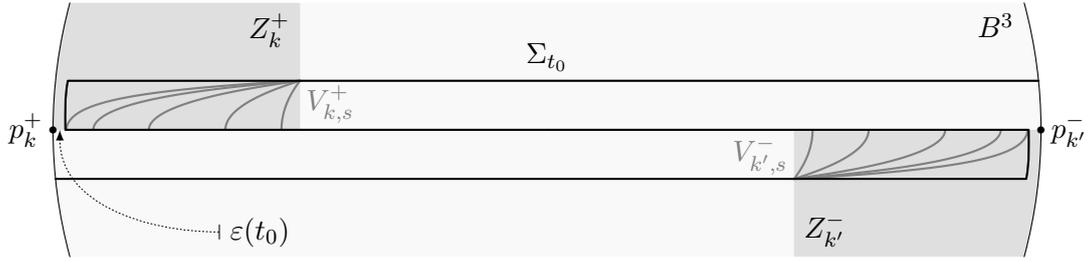
\begin{figure} 
\centering
\begin{tikzpicture}[line cap=round,line join=round,scale=6.5]
\pgfmathsetmacro{\h}{0.1}
\pgfmathsetmacro{\r}{0.5}
\pgfmathsetmacro{\eps}{0.025}
\pgfmathsetmacro{\alph}{15}
\pgfmathsetmacro{\Rh}{(-1+\eps)*sqrt(1-\h*\h)} 
\fill[lightgray!10]
(-\alph:1)arc(-\alph:\alph:1)--(180-\alph:1)arc(180-\alph:180+\alph:1);
\fill[lightgray!50]
(1-\r,0)--(1,0)arc(0:-\alph:1)-|node[midway,above right,black]{$Z_{k'}^-$}cycle
(-1+\r,0)--(-1,0)arc(180:180-\alph:1)-|node[midway,below left,black]{$Z_{k}^+$}cycle;
\path(-1+\r,0)|-(180+\alph:1)node[midway,above left](epsilon){$\varepsilon(t_0)$};
\begin{scope}[
thick,black!50
]
\foreach\s in {2,...,6}{
\draw[shift={(-1,0)},domain=0:\h,smooth,variable=\z]
plot({ \r*cosh(\s*\s*\z)/cosh(\s*\s*\h)},{\z});
\draw[shift={( 1,0)},domain=-\h:0,smooth,variable=\z]
plot({-\r*cosh(\s*\s*\z)/cosh(\s*\s*\h)},{\z});
} 
\draw
(-1+\r,0)node[above right,inner sep=2pt]{$V_{k,s}^+$}
( 1-\r,0)node[below left,inner sep=2pt]{$V_{k',s}^-$};
;
\end{scope}
\draw[thick]
(0,0)--(-1+\eps,0)to[out=90,in=-100](\Rh,\h)--({sqrt(1-\h*\h)},\h)
(0,0)--(1-\eps,0)to[out=-90,in= 80](-\Rh,-\h)--({-sqrt(1-\h*\h)},-\h);
\draw(180-\alph:1)arc(180-\alph:180+\alph:1);
\draw(-\alph:1)arc(-\alph:\alph:1)node[below left]{$B^3$};
\draw[](0,\h)node[above]{$\Sigma_{t_0}$};
\draw(1,0)node{$\scriptscriptstyle\bullet$}node[right]{$p_{k'}^-$}
(-1,0)node{$\scriptscriptstyle\bullet$}node[left]{$p_{k}^+$};
\draw[latex-|,densely dotted,thin](-1+\eps/1.8,0)to[out=-90,in=180](epsilon);
\end{tikzpicture}
\caption{Replacement of $\Sigma_{t_0}\cap Z_{k}^{\pm}$ with $V_{k,s}^{\pm}$. }
\label{fig:replacement}
\end{figure}
 
 We recall that $\Sigma_t$ was defined for all $t\in[t_0,1]$ in \eqref{eqn:20191106-E}, where we are free to choose first $t_0>0$ and then $\varepsilon_0>0$ such that 
\begin{align}\label{eqn:20191018-h0}
t_0&= h, &  \varepsilon_0&=\frac{t_0}{2(g+1)}.
\end{align}
By \cref{lem:20191106-21:35}, this choice for $\varepsilon_0$ ensures that for all $t\in[t_0,1]$
\begin{align}\label{eqn:20191107}
\Haus^2({\Sigma_t})\leq(3-t_0^2)\pi.
\end{align} 
For each $k\in\{0,\ldots,g\}$, let 
\begin{align*}
Z_k^+&\vcentcolon=\{x\in \overline{B^3}\st \operatorname{dist}((x_1,x_2,0),p_k^+)<r,~x_3>0\},  \\
Z_k^-&\vcentcolon=\{x\in \overline{B^3}\st \operatorname{dist}((x_1,x_2,0),p_k^-)<r,~x_3<0\}
\end{align*}
as shown in \cref{fig:replacement} and in \cref{fig:20191018} on the left. Note that, by the very way we have defined our parameters it follows that $r>5t_0 \ge 10(g+1)\varepsilon_0$.
We shall now replace $\Sigma_{t_0}\cap Z_k^+$ with a copy of the upper half of the surface $C_{s}^{r,t_0}\cap E_1$ after a suitable horizontal translation and rotation mapping $0\mapsto p_k^+$ and $E_1\mapsto B^3$. 
Similarly, $\Sigma_{t_0}\cap Z_k^-$ is replaced by a copy of the lower half of $C_{s}^{r,t_0}\cap E_1$. 
We denote those copies by $V_{k,s}^+$ and $V_{k,s}^-$, respectively.  
Initially, we choose $s=s_0$ such that  
\begin{align*}
\frac{r}{\cosh(s_0 t_0)}=\varepsilon(t_0)
\end{align*}
which ensures a continuous gluing of $V_{k,s_0}^{\pm}$ and $\Eq_{\varepsilon}^{+}\cap \Eq_{\varepsilon}^{-}\subset\Sigma_{t_0}$ at height $x_3=0$. 
Moreover, assuming that $\varepsilon(t_0)\in\interval{0,\varepsilon_0}$ is sufficiently small (so that $s_0$ will be very large), the surfaces $\Sigma_{t_0}\cap Z_k^\pm$ and $V_{k,s_0}^\pm$ are arbitrarily close such that we can continuously deform $\Sigma_{t_0}\cap Z_k^+$ into $V_{k,s_0}^+$ without significantly increasing the area. 
Then, as $t$ decreases further from $t_0$ to $t_0/2$, we decrease $s$ from $s_0$ to $0$ 
and define $\Sigma_t$ through similar gluings of $V_{k,s}^{\pm}$ and   $\Sigma_{t_0}\setminus Z_k^{\pm}$ as shown in \cref{fig:replacement} and \cref{fig:sweepout1}, third and fourth image. 
By \eqref{eqn:20191021-1051} and by \eqref{eqn:20191107}, we have 
for all $t\in\Interval{t_0/2,t_0}$ 
\begin{align}\notag
\Haus^2({\Sigma_t}) 
&\leq
\Haus^2({\Sigma_{t_0}})-(g+1)\Haus^2({C_{\infty}^{r,t_0}\cap E_1})+(g+1)\Haus^2({C_{s}^{r,t_0}\cap E_1})
\\
&\leq(3-t_0^2)\pi+(g+1)\frac{4\pi t_0^2}{(-\log t_0)}  
< 3\pi, \label{eqn:<3pi}
\end{align}
the last inequality relying on the fact that $-\log t_0 > 8(g+1)$.
Now, observe further that 
\begin{equation}\label{eq:lastregime}
\begin{split}
\Haus^2({\Sigma_{t_0/2}}) &\le 3\pi - 2(g+1)\left( \frac\pi 2 r^2 -2\pi r t_0 \right) \\
&= 3\pi - (g+1)\pi(r^2 - 4r t_0) < 3\pi,
\end{split}
\end{equation}
where we have used that $r> 5h = 5 t_0$. 
It is now easy to see that it is possible to define $\Sigma_t$ for $t\in [0,t_0/2]$ in such a way $\Haus^2({\Sigma_t})$ is decreasing as $t$ decreases and $\Sigma_0$ is the equatorial disc. Indeed, thanks to \eqref{eq:lastregime} we see at once that by increasing $r$ till the threshold value $\overline{r}=\sin(\pi/(2g+2))$, so by removing larger discs as we vary  $t\in [t_0/4,t_0/2]$ gives an area-decreasing deformation; then, for $t\in [0,t_0/4]$ we can just perform a simple retraction  (see \cref{fig:20191018} and \cref{fig:sweepout1}, fifth and sixth image).

\begin{figure}%
\centering
\pgfmathsetmacro{\eps}{0.075}
\pgfmathsetmacro{\genus}{2} 
\pgfmathsetmacro{\alph}{2*asin(\eps/2)}
\pgfmathsetmacro{\globalscale}{3}
\pgfmathsetmacro{\t}{0.3}
\pgfmathsetmacro{\radius}{sin(90/(\genus+1))/1.2}
\begin{tikzpicture}[scale=\globalscale,line cap=round,line join=round,baseline={(0,0)},rotate=90/(\genus+1)]
\draw[fill=lightgray]({-\alph-180/(\genus+1)}:1) 
\foreach\k in {-0.5,0,...,\genus}{
arc({\k*360/(\genus+1)+180+180/2-\alph/2}:{\k*360/(\genus+1)+180-180/2+\alph/2}:\eps)arc(\k*360/(\genus+1)+\alph:{\k*360/(\genus+1)+180/(\genus+1)-\alph}:1)
};
\draw[fill=lightgray] ({-\alph}:{sqrt(1-\t*\t)}) 
\foreach\k in {0,...,\genus}{
arc({\k*360/(\genus+1)+180+180/2-\alph/2}:{\k*360/(\genus+1)+180-180/2+\alph/2}:{\eps*sqrt(1-\t*\t)})arc(\k*360/(\genus+1)+\alph:{\k*360/(\genus+1)+360/(\genus+1)-\alph}:{sqrt(1-\t*\t)})
};
\foreach\k in {-0.5,0,...,\genus}{
\draw[dashed](0,0)--({(\k+1/4)*360/(\genus+1)}:1);
};
\foreach\k in {0,...,\genus}{
\draw({\k*360/(\genus+1)}:1)node[inner sep=0](p\k){$\scriptstyle\bullet$}
node[anchor=180+(\k+1/4)*360/(\genus+1)]{$p_{\k}^{+}$};
}
\begin{scope}
\clip(0,0)circle(1.15);
\foreach\k in {0,...,\genus}{
\draw[densely dotted]({\k*360/(\genus+1)}:1)circle(\radius);
\draw[latex-latex](p\k.center)--({\k*360/(\genus+1)}:{1-\radius})node[midway,anchor=90+\k*360/(\genus+1)]{$r$};
};
\end{scope}
\pgfresetboundingbox
\useasboundingbox[rotate=-90/(\genus+1)](-1.2,-1.2)rectangle(1.2,1);
\end{tikzpicture}
\hfill
\begin{tikzpicture}[scale=\globalscale,line cap=round,line join=round,baseline={(0,0)},rotate=90/(\genus+1)]
\draw[fill=lightgray!20]({-\alph-180/(\genus+1)}:1) 
\foreach\k in {-0.5,0,...,\genus}{
arc({\k*360/(\genus+1)+180+180/2-\alph/2}:{\k*360/(\genus+1)+180-180/2+\alph/2}:\eps)arc(\k*360/(\genus+1)+\alph:{\k*360/(\genus+1)+180/(\genus+1)-\alph}:1)
}; 
\draw[thick,fill=lightgray]({(0.5+1/4)*360/(\genus+1)}:1)
\foreach\k in {-0.5,0,...,\genus}{
--({(-\k+1/4)*360/(\genus+1)}:1)--({(-\k+1/4)*360/(\genus+1)}:0.9)..controls(0,0)and(0,0)..
({(-\k-1/4)*360/(\genus+1)}:0.9)--({(-\k-1/4)*360/(\genus+1)}:1)
}--cycle;
\foreach\k in {-0.5,0,...,\genus}{
\draw[dashed](0,0)--({(\k+1/4)*360/(\genus+1)}:1);
};
\begin{scope}
\clip(0,0)circle(1);
\foreach\k in {-0.5,0,...,\genus}{
\draw[densely dotted]({\k*360/(\genus+1)}:1)circle(\radius);
\draw[dotted]({\k*360/(\genus+1)}:1)circle({sin(90/(\genus+1))});
}
\end{scope}
\pgfresetboundingbox
\useasboundingbox[rotate=-90/(\genus+1)](-1.2,-1.2)rectangle(1.2,1);
\end{tikzpicture}
\caption{Implementing the catenoid estimate and making a further retraction.} 
\label{fig:20191018} 
\end{figure}
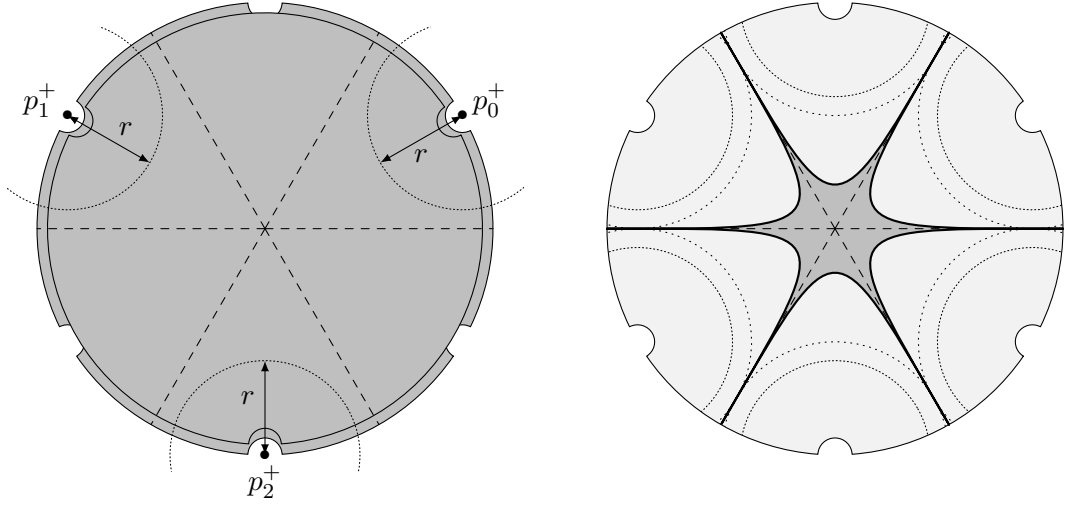

\begin{proof}[Proof of \cref{lem:ConstructionSweepout}]
Let $\{\Sigma_t\}_{t\in[0,1]}$ be as constructed above. 
We define the desired smooth surfaces by regularising $\Sigma_t$ for all $0<t<1$ equivariantly (without renaming), a process which can be performed without violating the strict $3\pi$ upper bound on the area.
We note that at $t=0$, the origin is a singular point, where the genus of $\Sigma_t$ collapses as $t\searrow0$. 
However, for $0<t<1$, we obtain a smooth family of genus $g$ surfaces, as claimed. 
\end{proof}


\section{Maximality of the unstable catenoid}\label{sec:maximal}

In \cref{sec:sweepout}, we introduced the surfaces $C_s^{r,h}$, parametrized by $s\geq0$, interpolating between the cylinder $C_0^{r,h}$ of radius $r$ and height $2h$ (for $s=0$) and the union of two discs of radius $r$ with a line segment (as one lets $s\to\infty$).
As we remarked, $C_s^{r,h}$ is minimal if $s$ is a solution to equation \eqref{eqn:catenoid}. Said $t_0>0$ the only positive solution of $\cosh(t)=t\sinh(t)$, equation \eqref{eqn:catenoid} has two positive solutions $s_1$ and $s_2$ provided that $0<h<r/\sinh(t_0)$ (note that the number $1/\sinh(t_0)$ is bounded from below by 0.6627).  
In our notation, $s_2$ is the larger solution and corresponds to the so-called \emph{unstable catenoid}. 
In this section we show that, as soon as we are willing to impose a slightly more restrictive condition on the ratio $h/r$, such a surface has largest area among all elements of the family in question. As a result, estimate \eqref{eqn:20191115-18:52}, which plays an essential role in \cref{sec:sweepout}, follows at once from \cite[Proposition 2.1]{KetoverMarquesNeves2020}. 

\begin{lemma}\label{lem:maximal}
The unstable catenoid has largest area among all surfaces in the family ${\{C_s^{r,h}\}}_{s\geq0}$ provided that $0<2h<r\tanh(1)$. 
\end{lemma}

\begin{proof}
As defined in \eqref{eqn:20191115-18:51}, the surface $C_s^{r,h}$ is obtained by rotating the graph of $\rho\colon[-h,h]\to\R$ given by 
\begin{align}\label{eqn:20191115-20:40}
\rho(z)=\frac{r\cosh(s z)}{\cosh(s h)}
\end{align}
around the vertical axis. Being a surface of revolution, the mean curvature of $C_s^{r,h}$ is easily computed to be 
\begin{align}\label{eqn:20191116-16:38}
H&=\frac{\rho\rho''-(\rho')^2-1}{\bigl(1+(\rho')^2\bigr)^{3/2}\rho}.
\end{align}

We notice that the denominator of \eqref{eqn:20191116-16:38} is strictly positive and that the numerator 
\begin{align*}
\rho\rho''-(\rho')^2-1&=\frac{r^2 s^2}{\cosh^2(s h)}-1
\end{align*}
is independent of $z$. 
In particular, it follows that $H$ has the same sign as the function $r s-\cosh(s h)$. 
Recalling equation \eqref{eqn:catenoid}, the inequality $r s>\cosh(s h)$ is equivalent to $s_1<s<s_2$ by strict convexity of $s\mapsto\cosh(s h)$. 
In this case $H>0$, which implies for $s\in\interval{s_1,s_2}$ the area of $C_s^{r,h}$ is an increasing function of $s$. 
Conversely, if $s>s_2$ or if $s<s_1$, then $H<0$ and the area of $C_s^{r,h}$ is decreasing in $s$. 
This shows that the area of $C_s^{r,h}$ has a local minimum at $s_1$ and a local maximum at $s_2$. 
In order to prove the claim that $s_2$ is in fact a global maximum, it remains to check $\Haus^2({C_{s_2}^{r,h}})>\Haus^2({C_{0}^{r,h}})$, i.e, that the unstable catenoid has larger area than the cylinder provided that ${h}/{r}$ is sufficiently small. 

The area $A(s)=\Haus^2({C_{s}^{r,h}})$ can be computed using the formula 
\begin{align*}
A(s)&=4\pi\int^{h}_{0}\rho\sqrt{(\rho')^2+1}\,dz. 
\end{align*}
With $\rho$ as defined in \eqref{eqn:20191115-20:40}, which in particular satisfies $\rho''=s^2\rho$, the function $f=4\pi\rho\sqrt{(\rho')^2+1}$ has a primitive given by 
\begin{align*}
F&=\frac{2\pi}{s^2}\left(\asinh\bigl(\rho'\bigr)+\rho'\sqrt{(\rho')^2+1}\right).
\end{align*}

Since $\rho'(0)=0$ and $\rho'(h)=r s\tanh(s h)$, we obtain 
\begin{align*}\label{eqn:20191117-17:01}
A(s)&=F(h)-F(0)\\
&=\frac{2\pi}{s^2}\left(\asinh\bigl(r s\tanh(s h)\bigr) +r s\tanh(s h)\sqrt{1 + r^2 s^2 \tanh^2(s h)}\right).
\end{align*}
Comparing the derivatives of $s\mapsto r s$ and $s\mapsto\cosh(sh)$ at the intersection $s=s_2$, we obtain  
\begin{align*}
r\leq h\sinh(s_2h)<h\cosh(s_2h)=r s_2 h
\end{align*}
which implies $s_2h>1$. 
In turn, this yields  
\begin{align*}
A(s_2)&=\frac{2\pi}{s_2^2}\Bigl(s_2 h+r s_2\tanh(s_2 h)\cosh(s_2 h)\Bigr)\\
&=\frac{2\pi h}{s_2}+2\pi r^2\tanh(s_2 h)
>2\pi r^2\tanh(1).
\end{align*}
The assumption $2h<r\tanh(1)$ then implies $A(s_2)>4\pi rh=\Haus^2({C_0^{r,h}})$ which completes the proof. 
\end{proof}


\section{Saturation of the sweepout and its width}\label{sec:swap}

In order to apply a min-max procedure, let us consider the $\dih_{g+1}$-saturation $\Pi$ of the sweepout $\{\Sigma_t\}_{t\in[0,1]}$ defined in \cref{lem:ConstructionSweepout}, and let $W_\Pi$ be its min-max width (see \cref{def:SaturationWidth}).
In this section, we prove that the min-max width $W_\Pi$ is strictly larger than $\pi=\area(\Sigma_0)=\area(\Sigma_1)$, namely that the mountain pass condition in \cref{thm:EquivMinMax} is satisfied. 
For this purpose, it is helpful to introduce some terminology about finite perimeter sets (the reader is referred to e.\,g. \cite[Chapter 12]{Maggi2012}).

Hereafter, any subset of the form $\{x\in {B^3}\st x\cdot v\geq0\}$ for some $v \in\R^3$ is called a \emph{half-ball}. 

\begin{definition}
We say that a finite perimeter set $E\subset {B^3}$ is $\dih_n$-equivariant if, for all $\varphi\in \dih_n$, the set $\varphi(E)$ coincides either with $E$ or ${B^3}\setminus E$ up to a negligible set. 
\end{definition}

\begin{lemma}\label{lem:SweepoutsCorrespondence}
Let $\{\Sigma_t\}_{t\in[0,1]}$ be the $\dih_{g+1}$-sweepout in \cref{lem:ConstructionSweepout} and let $\Pi$ be its $\dih_{g+1}$-saturation. Then for every $\{\Lambda_t\}_{t\in[0,1]}\in\Pi$ there exists a family $\{F_t\}_{t\in[0,1]}$ of $\dih_{g+1}$-equivariant finite perimeter sets such that the following properties hold.
\begin{enumerate} [label={\normalfont(\arabic*)}]
    \item \label{fpss:extremaltimes} $F_0$ is the upper half-ball and $F_1$ is the lower half-ball. 
    \item \label{fpss:continuity} The family $\{ F_t \}_{t\in[0,1]}$ is continuous in the sense of finite perimeter sets, i.e, $\Haus^3({F_t\triangle F_{t_0}}) \to 0$ whenever $t\to t_0$, where $F_t\triangle F_{t_0} \vcentcolon= (F_t\setminus F_{t_0}) \cup (F_{t_0}\setminus F_t)$. 
    \item \label{fpss:equivariance} The finite perimeter sets $F_t$ are $\dih_{g+1}$-equivariant for all $0\le t\le 1$. 
    \item \label{fpss:relbdry} For every $0\le t\le 1$, $\Lambda_t$ is the relative boundary of $F_t$ in ${B^3}$; namely $\Lambda_t \setminus \partial B^3 = \partial F_t\setminus\partial B^3$. 
    \item \label{fpss:volume} The volume of $F_t$ is half the volume of ${B^3}$ for all $0\le t\le 1$, i.e, $\Haus^3({F_t}) = \Haus^3({B^3})/2$ for all $0\le t\le 1$. 
\end{enumerate}
\end{lemma}

\begin{proof}
By the construction of the sweepout $\{\Sigma_t\}_{t\in[0,1]}$ in \cref{lem:ConstructionSweepout}, we easily obtain that there exists a family $\{ F_t^\Sigma\}_{t\in[0,1]}$ of $\dih_{g+1}$-equivariant finite perimeter sets with properties \ref{fpss:extremaltimes}--\ref{fpss:volume}.
In particular, we can choose $\{F_t^\Sigma\}_{t\in[0,1]}$ such that $F_0^\Sigma$ is the upper half-ball and $F_1^\Sigma$ is the lower half-ball.

Now let us consider any other sweepout $\{ \Lambda_t \}_{t\in[0,1]} \in \Pi$. By definition of saturation there exists a $\dih_{g+1}$-isotopy $\Phi\colon[0,1]\times \overline{B^3}\to\overline{B^3}$ such that $\Lambda_t = \Phi_t(\Sigma_t)$ for all $0\le t\le 1$.
We want to prove that $\{F_t\}_{t\in[0,1]}$ defined by $F_t \vcentcolon= \Phi_t(F_t^\Sigma)$ is a family of $\dih_{g+1}$-equivariant finite perimeter sets as required in the statement.
Observe that properties \ref{fpss:extremaltimes}, \ref{fpss:continuity}, \ref{fpss:equivariance} and \ref{fpss:relbdry} are straightforward, hence we just need to check \ref{fpss:volume}.

Let us consider $\psi\in\dih_{g+1}$ given by the rotation (of angle $\pi$) around the isotropy axis $\xi_1$ and observe that $\psi(F_t^\Sigma) ={B^3}\setminus F_t^\Sigma$ for every $t\in[0,1]$. 
Hence we have that
\begin{align*}
\psi(F_t) = \psi(\Phi_t(F_t^\Sigma)) = \Phi_t(\psi(F_t^\Sigma)) = \Phi_t ({B^3}\setminus F_t^\Sigma) = {B^3}\setminus \Phi_t(F_t^\Sigma) = {B^3}\setminus F_t,
\end{align*}
which proves $\Haus^3({F_t}) = \Haus^3({B^3})/2$ and concludes the proof.
\end{proof}

We now denote by $\Theta$ the set of all families of $\dih_{g+1}$-equivariant finite perimeter sets for which properties \ref{fpss:extremaltimes}--\ref{fpss:volume} in \cref{lem:SweepoutsCorrespondence} hold for some $\dih_{g+1}$-sweepout $\{\Lambda_t\}_{t\in[0,1]}\in\Pi$.
Since there is a one-to-one correspondence between $\Pi$ and $\Theta$, we derive the following conclusion.

\begin{corollary} \label{cor:CompareWidths}
Under the hypotheses of the previous lemma, we have that
\begin{align*}
    W_\Pi =\adjustlimits\inf_{\{F_t\}\in \Theta~}\sup_{t\in[0,1]} P(F_t;{B^3}),
\end{align*}
where $P(F_t;B^3)$ denotes the relative perimeter of the finite perimeter set $F_t$ in $B^3$.
\end{corollary}

In order to prove that $W_\Pi$ is indeed strictly larger than $\pi$, we first need the following stability lemma for the isoperimetric inequality.

\begin{lemma}[Stability of the isoperimetric inequality] \label{lem:StabIsop}
Fix $2\le n\in\N$. Then, for every $\varepsilon>0$ there exists $\delta>0$ such that, given a $\dih_n$-equivariant finite perimeter set $F\subset {B^3}$ with Lebesgue measure $\Haus^3(F) = \Haus^3({B^3})/2$ and relative perimeter $P(F;{B^3})\le \pi+\delta$, there exists a $\dih_n$-equivariant half-ball $\tilde F$ with $\Haus^3({F \triangle \tilde F})\le\varepsilon$. 
\end{lemma}
\begin{proof}
Towards a contradiction, assume that there exist $\varepsilon>0$ and a sequence $\{F_k\}_{k\in\N}$ of $\dih_n$-equivariant finite perimeter sets satisfying $\Haus^3({F_k}) = \Haus^3({B^3})/2$ and $P(F_k;{B^3})\le\pi+\delta_k$ for $\delta_k\to 0$ as well as $\Haus^3({F_k\triangle \tilde F})\ge \varepsilon$ for every $\dih_n$-equivariant half-ball $\tilde F$.

By the compactness theorem for finite perimeter sets (see \cite[Theorem 12.26]{Maggi2012}), there exists a $\dih_n$-equivariant finite perimeter set $F_\infty\subset {B^3}$ such that a subsequence of $\{F_k\}_{k\in\N}$, which we do not rename, satisfies $\Haus^3({F_\infty\triangle F_k})\to 0$ as $k\to\infty$. 
In particular we have that $\Haus^3({F_\infty}) = \Haus^3({B^3})/2$. 
Moreover, by lower semicontinuity of the perimeter (cf. \cite[Proposition 12.15]{Maggi2012}), it holds
\[P(F_\infty;{B^3}) \le \liminf_{k\to\infty} P(F_k;{B^3}) = \pi.\]
Hence, $F_\infty$ is a finite perimeter set in ${B^3}$ with $\Haus^3({F_\infty}) = \Haus^3({B^3})/2$ and $P(F_\infty;{B^3})\le \pi$, which implies that $F_\infty$ is a half-ball. Indeed, the reduced boundary $\partial^* F_\infty \cap B^3$ of $F_\infty$ in $B^3$ is smooth analytic with constant mean curvature by Theorem 27.4 in \cite{Maggi2012} (see pp. 386--389 therein for historical notes) and thus it is an equatorial disc by \cite[Satz 1]{BokowskiSperner1979} (see also \cite[Theorem 5]{Ros2005}). 
However, this contradicts the choice of the sequence $\{F_k\}_{k\in\N}$ and concludes the proof.
\end{proof}

\begin{proposition} \label{prop:NonTrivialWidthFinitePerSets}
Fix $2\le n\in \N$. Then, there exists $\delta_0>0$ with the following property. Let $\{F_t\}_{t\in[0,1]}$ be a family of $\dih_n$-equivariant finite perimeter sets in the unit ball ${B^3}$ such that
\begin{enumerate}[label={\normalfont(\roman*)}]
    \item $\{F_t\}_{t\in[0,1]}$ is continuous in the sense of finite perimeter sets, i.e., $\Haus^3({F_t\triangle F_{t_0}})\to 0$ whenever $t\to t_0$;
    \item $\Haus^3({F_{t}}) = \Haus^3({B^3})/2$ for all $0\le t\le 1$;
    \item $F_0 ={B^3}\setminus F_1$ up to a negligible set.
\end{enumerate}
Then, $\sup_{t\in[0,1]} P(F_t;{B^3})\ge \pi + \delta_0$.
\end{proposition}
\begin{proof}
Pick $\varepsilon=\Haus^3(B^3)/12= \pi/9$ and consider $\delta_0>0$ to be the associated $\delta$ given by \cref{lem:StabIsop}. 
If \[\sup_{t\in[0,1]}P(F_t; {B^3})< \pi+\delta_0,\] then for every $t\in[0,1]$ there exists a $\dih_n$-equivariant half-ball $\tilde F_t$ such that $\Haus^3({F_t\triangle \tilde F_t}) \le {\pi}/{9}$.
Note that the $\dih_n$-equivariant half-balls are the upper and the lower half-balls and, for $n=2$, also the two half-balls bounded by the plane containing $\xi_0,\xi_1$ and the two half-balls bounded by the plane containing $\xi_0,\xi_2$.
In any case we deduce that, for every $t\in[0,1]$, the $\dih_n$-equivariant half-ball $\tilde F_t$ is uniquely determined.
Therefore, by continuity of the family $\{F_t\}_{t\in[0,1]}$, $\tilde F_t$ must be constant, but this contradicts the assumption that $F_{0}$ is the complement of $F_1$ in $B^3$.
\end{proof}

\begin{corollary} \label{cor:WidthBiggerPi}
Let $\{\Sigma_t\}_{t\in[0,1]}$ be the $\dih_{g+1}$-sweepout given by \cref{lem:ConstructionSweepout} and let $\Pi$ be its $\dih_{g+1}$-saturation given in \cref{def:SaturationWidth}. Then the min-max width of $\Pi$ is larger than $\pi$ and smaller than $3\pi$, namely $\pi<W_\Pi <3\pi$.
\end{corollary}


\section{Controlling the topology}\label{sec:control}

In this section we prove the existence of the surface $M_g$ for $g\ge1$, stated in \cref{thm:main-fbms-b1}. Then, the last section will be devoted to the study of the asymptotic behavior of $M_g$ as $g\to\infty$, completing the proof of \cref{thm:main-fbms-b1}.
By \cref{cor:WidthBiggerPi}, all conditions for applying \cref{thm:EquivMinMax}, in particular the mountain pass condition \[W_\Pi>\max\{\Haus^2({\Sigma_0}),\Haus^2({\Sigma_1})\},\] are satisfied.  
We thus obtain a min-max sequence $\{\Sigma^j\}_{j\in\N}$ consisting of $\dih_{g+1}$-equivariant surfaces and converging in the sense of varifolds to $m\Gamma$, where $\Gamma$ is a smooth, properly embedded, compact connected free boundary minimal surface in ${B^3}$, and the multiplicity $m$ is a positive integer. 
Moreover, the following statements hold. 
\begin{enumerate}[label={\normalfont(\arabic*)}]
\item\label{min-max-theorem-axes} The surface $\Gamma$ contains the horizontal axes $\xi_1,\ldots,\xi_{g+1}$ and intersects $\xi_0$ orthogonally.
\item\label{min-max-theorem-odd} The integer $m$ is odd.
\item\label{min-max-theorem-genus} $\genus(\Gamma) \le g$.
\item\label{min-max-theorem-index} $\ind_{\dih_{g+1}}(\Gamma)\le1$.
\end{enumerate}

\begin{remark}\label{rem:DLP}
Observe that statement \ref{min-max-theorem-axes} is a consequence of the $\dih_{g+1}$-equivariance (cf. also Lemma 3.4 and Lemma 3.5 in \cite{Ketover2016Equivariant}). Point \ref{min-max-theorem-odd} follows from the invariance with respect to the rotation of angle $\pi$ around the axes $\xi_1,\ldots,\xi_{g+1}$ and its (self-contained) proof can be found at the end of Section 7.3 in \cite{Ketover2016FBMS}.
\end{remark}

\begin{lemma}\label{lem:multiplicity}
The multiplicity $m$ is equal to $1$ and $\Gamma$ is not a (topological) disc. 
\end{lemma}

\begin{proof}
Fraser and Schoen (see \cite[Theorem 5.4]{FraserSchoen2011}) proved that any free boundary minimal surface in ${B^3}$ has area at least $\pi$.
By varifold convergence of the min-max sequence, it holds $m\Haus^2(\Gamma) = W_\Pi$. By \cref{cor:WidthBiggerPi}, $\pi<m\Haus^2(\Gamma)<3\pi$
whence we conclude $m<3$. 
In fact, $m=1$ since $m$ must be odd by~\ref{min-max-theorem-odd}.
As a result, $\Haus^2({\Gamma})>\pi$, which implies that $\Gamma$ is not isometric to the equatorial disc. 
However, according to \cite{Nitsche1985}, the equatorial disc is the only possible free boundary minimal disc in $B^3$ up to ambient isometries.  
\end{proof}

\begin{lemma}\label{lem:number_of_boundary_components}
The number of boundary components of $\Gamma$ is 1.
\end{lemma}
\begin{proof}
Suppose, towards a contradiction, that $\partial\Gamma$ has more than one connected component. 
Then \cref{lem:endpoints} implies that one connected component, say $\alpha$, of $\partial\Gamma$ is disjoint from $\mathcal{S}=\xi_0\cup\xi_1\cup\ldots\cup\xi_{g+1}$. 
Recall that the vertical axis $\xi_0$ is always disjoint from $\partial\Gamma$ because $\Gamma\subset B^3$ is properly embedded and intersects $\xi_0$ orthogonally by item \ref{min-max-theorem-axes} above.
Moreover, let $\tilde\alpha$ be a simple closed curve in the interior of $\Gamma\setminus\mathcal{S}$ that is homotopic to $\alpha$ in $\Gamma\setminus\mathcal{S}$ (it is sufficient to slightly push $\alpha$ towards the interior of $\Gamma$). 

Let then $\delta>0$ be so small that 
\begin{align*}
U_{\delta}\Gamma\vcentcolon=\{x\in B^3\st \dist_{\R^3}(x,\Gamma)<\delta\}
\end{align*}
is a tubular neighbourhood of $\Gamma$ in $B^3$.
Since $\Gamma$ is connected and thus it is not a disc (since it has at least two boundary components), $\alpha$ and $\tilde\alpha$ are not contractible in $U_\delta\Gamma$.

Thanks to Proposition 4.10 in \cite{Li2015}, without loss of generality we can assume that the min-max sequence $\{\Sigma^j\}_{j\in\N}$ is outer almost minimising (see \cite[Definition 3.6]{Li2015}) in sufficiently small annuli.
Therefore, by Simon's Lifting Lemma (see \cite[Proposition 2.1]{DeLellisPellandini2010} and \cite[Section 9]{Li2015}), for every $j$ sufficiently large there exists a closed curve $\alpha^j\subset \Sigma^j\cap U_\delta\Gamma$ that is homotopic to $\tilde\alpha$ in $U_\delta\Gamma$ (note that we can apply the lemma since $\tilde\alpha$ is contained in the interior of $\Gamma$). In fact, given any $\rho>0$, it follows from the proof in Section 4.3 of \cite{DeLellisPellandini2010} that $\alpha^j$ can be taken in $U_\rho\tilde\alpha\vcentcolon=\{x\in B^3\st \dist_{\R^3}(x,\tilde\alpha)<\rho\}$.
Hence, choosing $\rho$ such that $U_{\rho}\tilde\alpha\subset U_{\delta/2}\Gamma\setminus\mathcal{S}$, we can guarantee that $\alpha^j\subset (\Sigma^j\cap U_{\delta/2}\Gamma)\setminus\mathcal{S}$ for every $j$ sufficiently large.

Now observe that $\Sigma^j$ is diffeomorphic to an element $\Sigma_{t_j}$ of the sweepout given by \cref{lem:ConstructionSweepout} for some $0<t_j<1$, through a $\dih_{g+1}$-equivariant diffeomorphism.
Thus, both connected components of $\Sigma^j\setminus(\xi_1\cup\ldots\cup\xi_{g+1})$ are topological discs, hence $\alpha^j \subset \Sigma^j\setminus\mathcal{S}$ is contractible.
Let us denote by $D^j$ a disc in the interior of $\Sigma^j$ such that $\partial D^j = \alpha^j$.
We claim that $\alpha^j$ bounds a disc in $U_{\delta}\Gamma$ as well, which would contradict the existence of $\alpha$, since $\alpha^j$ is homotopic to $\tilde{\alpha}$ in $U_{\delta}\Gamma$ and $\tilde{\alpha}$ is not contractible there.

We now exploit an argument similar to the one in Section 2.4 of \cite{DeLellisPellandini2010}.
Since the min-max sequence $\{\Sigma^j\}_{j\in\N}$ converges in the sense of varifolds to $\Gamma$, it follows that given any $\eta>0$ there exists $J=J({\delta,\eta})\in\N$ such that, for every $j\geq J$, 
\begin{align*}
\Haus^2(\Sigma^j\setminus U_{\delta/2}\Gamma)<\eta.
\end{align*}
Defining $V_s\Gamma\vcentcolon= \partial(U_s\Gamma)\cap B^3$ for $s\in\interval{0,\delta}$,
we observe that $\{V_s\Gamma\}_{s\in\interval{0,\delta}}$ is a smooth foliation of $U_{\delta}\Gamma$ and we can apply the coarea formula to conclude 
\begin{align*}
\int^{\delta}_{\delta/2}\Haus^1(\Sigma^j\cap V_s\Gamma)
&\leq \Haus^2(\Sigma^j\setminus U_{\delta/2}\Gamma)<\eta,
\end{align*}
for every $j$ sufficiently large.
Thus, there exists a subset $I\subset\interval{\delta/2,\delta}$ of measure at least ${\delta}/{4}$ such that for all $s\in I$
\begin{align*}
\Haus^1(\Sigma^j\cap V_s\Gamma)<\frac{4\eta}{\delta}. 
\end{align*}
By Sard's theorem there exists $s\in I$ such that the intersection 
\(
\Theta^{j}_{s}\vcentcolon=\Sigma^j\cap V_s\Gamma
\)
is transverse. 
This implies that any connected component of $\Theta^{j}_{s}$ is smooth and either a simple closed curve or an arc connecting two points of $\partial\Sigma^j$ in $V_s\Gamma$.

There exists $\lambda>0$ (depending on $\Gamma$ and $\delta$) such that for any $s\in\interval{\delta/2,\delta}$ any simple closed curve in $V_s\Gamma$ with length less than $\lambda$ bounds an embedded disc in $V_s\Gamma$. 
At this stage, we may choose $\eta>0$ such that $4\eta<\lambda\delta$ and then $j\geq J({\delta,\eta})$ to ensure that the length of each connected component of $\Theta_{s}^{j}$ is less than~$\lambda$.
Now observe that $D^j\cap V_s\Gamma\subset \Sigma^j\cap V_s\Gamma = \Theta_s^j$. In particular, $D^j\cap V_s\Gamma$ consists of a finite number of simple closed curves (since $D^j$ is contained in the interior of $\Sigma^j$) of length less than $\lambda$ and thus each connected component of $D^j\cap V_s\Gamma$ bounds a disc in $V_s\Gamma$.

Hence, defining $G^j\subset U_\delta\Gamma$ as the connected component of $D^j\cap U_s\Gamma$ containing $\partial D^j=\alpha^j$, it is possible to cap the boundary components of $G^j$ lying in $V_s\Gamma$ with discs (contained in $U_\delta\Gamma$) in such a way that the resulting surface, which we denote by $\tilde D^j$, satisfies $\partial \tilde D^j = \partial D^j = \alpha^j$ (for recall that $\alpha^j\subset U_{\delta/2}\Gamma$). Note that $\tilde D^j$ is a topological disc since it is obtained from the topological disc $D^j$ by removing interior discs and then gluing discs with those same boundaries.
Therefore it follows that $\tilde\alpha^j=\partial \tilde D^j$ is contractible in $U_\delta\Gamma$, which contradicts the initial choice of $\alpha$.
\end{proof}

\begin{lemma}\label{lem:genus}
The genus of $\Gamma$ is $g$.
\end{lemma}

\begin{proof}
By \ref{min-max-theorem-axes} and \ref{min-max-theorem-genus}, $\Gamma$ contains the horizontal axes $\xi_k$ for $k\in\{1,\ldots,g+1\}$ and has genus $\genus(\Gamma)\le g$.
Moreover, by \cref{lem:multiplicity,lem:number_of_boundary_components}, $\Gamma$ has one boundary component and it is not a topological disc, thus $\genus(\Gamma)\ge 1$.
As a result, \cref{lem:structural} applies, proving the claim.
\end{proof}

According to \cref{lem:number_of_boundary_components,lem:genus}, the free boundary minimal surface $M_g=\Gamma$ has genus $g$ and connected boundary. 
As said before, $\Gamma$ inherits the dihedral symmetry $\dih_{g+1}$ from the min-max sequence $\{\Sigma^j\}_{j\in\N}$.
Let us conclude this section by proving that the $\dih_{g+1}$-equivariant index of $M_g$ is $1$.

\begin{lemma} \label{lem:EquivIndMg}
The $\dih_{g+1}$-equivariant index of $M_g$ is $1$.
\end{lemma}
\begin{proof}
As stated in \ref{min-max-theorem-index} above (from the application of \cref{thm:EquivMinMax}), we have that the surface $M_g=\Gamma$ has $\dih_{g+1}$-equivariant index less or equal than $1$. We want to prove that the $\dih_{g+1}$-equivariant index is exactly $1$. To do so, it is sufficient to construct a test function on which the associated quadratic form attains a negative value.

Recall that $M_g$ contains the horizontal axes of symmetry $\xi_1,\ldots,\xi_{g+1}$. Hence, $\operatorname{sgn}_{M_g}(\selem) = -1$ for all $\selem\in\dih_{g+1}$ given from a rotation of angle $\pi$ around any horizontal axis. Observe that these isometries generate all $\dih_{g+1}$, because the composition of the rotations of angle $\pi$ around $\xi_1$ and $\xi_2$ is equal to the rotation of angle $2\pi/(g+1)$ around $\xi_0$ (see \cref{sec:DihedralGroup}).
Hence, we can infer the sign of all the elements of the group.
In particular, we obtain that the function $u(x_1,x_2,x_3) = x_3$ on $M_g$ belongs to $C^\infty_G(M_g)$ for all $g\in \dih_{g+1}$. Moreover, one can compute that
\[
Q_{M_g}(u,u) = -\int_{M_g}\abs{A}^2u^2\de\Haus^2 <0
\]
(see \cite{Devyver2019}*{Lemma 6.1}). This proves that the $\dih_{g+1}$-equivariant index of $M_g$ is exactly equal to $1$.
\end{proof}

\begin{remark}
Note that the function $x^\perp = \scal{x}{\nu}$ is in $C^\infty_G(M_g)$ and $Q(x^\perp, x^\perp) = 0$.
\end{remark}

\section{Asymptotic behavior}

In this section, we conclude the proof of \cref{thm:main-fbms-b1} by proving that the free boundary minimal surfaces $M_g\subset B^3$ converge in the sense of varifolds to the equatorial disc with multiplicity three, for $g\to\infty$.

\begin{proof}
First observe that, by \cref{cor:WidthBiggerPi}, we have that $\Haus^2(M_g) < 3\pi$ for all $g\ge 1$. Hence, by Theorem 2 in \cite{White2016}, we get that $\Haus^1(\partial M_g) < 6\pi$.
Moreover, note that $\partial M_g$ consists of $2(g+1)$ isometric parts, each of which connecting two distinct endpoints of the symmetry axes $\xi_1,\ldots,\xi_{g+1}$. Hence, by the estimate above on the length of $\partial M_g$, the length of each of these parts is less than $3\pi/(g+1)$, which converges to $0$ as $g\to\infty$. 

Now, let $h(g)$ be the maximal distance of a point of $M_g$ from the equatorial disc. By Theorem 15 in \cite{White2016}, $h(g)$ is realized by a point on the boundary of $M_g$. Thanks to the observations above, this point is contained in a curve of length less than $3\pi/(g+1)$ with endpoints on the equatorial disc. As a result, we obtain that $h(g) < 3\pi/(2g+2) \to0$.

Hence, we get that any subsequence (not relabelled) of $M_g$ converges in the sense of varifolds to an integer multiple $m$ of the equatorial disc $D$. Thus it is sufficient to prove that $m=3$, in order to conclude that the whole sequence $M_g$ converges to the equatorial disc with multiplicity three.
Since $\Haus^2(M_g) < 3\pi$, we have that $m\in\{1,2,3\}$. By Proposition 2.1 in \cite{Ketover2016FBMS}, we exclude $m=1$.

Now observe that, without loss of generality, we can assume that there is one horizontal symmetry axis $\xi$ in common among all the surfaces $M_g$. In fact, if we prove that a rotation of $M_g$ around the vertical axis converges to the triple equatorial disc, then also $M_g$ converges to the triple equatorial disc (being the equatorial disc invariant with respect to rotations around the vertical axis).

Consider a radius $0<R<1$ such that $M_g$ intersects transversally $B_R(0)$ for all the elements of the subsequence $M_g$ under consideration. Then we can apply \cref{lem:structural} and obtain that the genus of $M_g\cap B_R(0)$ is $0$ or $g$. In either cases, extracting a further subsequence of $M_g$, we find a point $x\in \xi$ such that $M_g$ has genus $0$ in $B_r(x)$ for some $r>0$ sufficiently small. Indeed we can take $x\in \xi \cap B_R(0)$ if $\genus(M_g\cap B_R(0)) = 0$ and $x\in \xi \cap \overline{B_R(0)}^c$ if $\genus(M_g\cap B_R(0)) = g$.

Therefore, by Lemmas 1 and 2 in \cite{Ilmanen1998}, it is possible to take $r>0$ possibly smaller in such a way that $M_g$ converges with multiplicity $m$ to $D$ smoothly away from finitely many points in $B_r(x)$. In particular there exist a point $y\in \xi$ and a ball $B_s(y)$ such that $M_g$ converges to $D$ smoothly with multiplicity $m$ in $B_s(y)$. 
Hence $M_g\cap B_s(y)$ consists of $m$ graphs converging smoothly to $D\cap B_s(y)$. Observe that exactly one of these graphs contains the axis $\xi\cap B_s(y)$. As a result, it easily follows from equivariance that the multiplicity $m$ must be odd (as in Section 7.3 of \cite{Ketover2016FBMS}).
Hence, we conclude that the multiplicity of the convergence is $3$, as desired.
\end{proof}

\backmatter

\phantomsection

\addcontentsline{toc}{chapter}{Bibliography}

\bibliographystyle{alpha}
\bibliography{biblio}

\begin{bibdiv}
\begin{biblist}

\bib{AmbrozioBuzanoCarlottoSharp2018}{article}{
      author={Ambrozio, Lucas},
      author={Buzano, Reto},
      author={Carlotto, Alessandro},
      author={Sharp, Ben},
       title={Bubbling analysis and geometric convergence results for free
  boundary minimal surfaces},
        date={2019},
     journal={J. \'{E}c. polytech. Math.},
      volume={6},
       pages={621\ndash 664},
}

\bib{AmbrosioCarlottoMassaccesi2018}{book}{
      author={Ambrosio, Luigi},
      author={Carlotto, Alessandro},
      author={Massaccesi, Annalisa},
       title={Lectures on elliptic partial differential equations},
      series={Lecture Notes. Scuola Normale Superiore di Pisa (New Series)},
   publisher={Edizioni della Normale, Pisa},
        date={2018},
      volume={18},
}

\bib{AmbrozioCarlottoSharp2018Compactness}{article}{
      author={Ambrozio, Lucas},
      author={Carlotto, Alessandro},
      author={Sharp, Ben},
       title={Compactness analysis for free boundary minimal hypersurfaces},
        date={2018},
     journal={Calc. Var. Partial Differential Equations},
      volume={57},
      number={1},
       pages={Paper No. 22, 39},
}

\bib{AmbrozioCarlottoSharp2018Index}{article}{
      author={Ambrozio, Lucas},
      author={Carlotto, Alessandro},
      author={Sharp, Ben},
       title={Comparing the {M}orse index and the first {B}etti number of
  minimal hypersurfaces},
        date={2018},
     journal={J. Differential Geom.},
      volume={108},
      number={3},
       pages={379\ndash 410},
}

\bib{AmbrozioCarlottoSharp2018IndexFBMS}{article}{
      author={Ambrozio, Lucas},
      author={Carlotto, Alessandro},
      author={Sharp, Ben},
       title={Index estimates for free boundary minimal hypersurfaces},
        date={2018},
     journal={Math. Ann.},
      volume={370},
      number={3-4},
       pages={1063\ndash 1078},
}

\bib{ArendtEtc2015}{book}{
      author={Arendt, Wolfgang},
      author={Chill, Ralph},
      author={Seifert, Christian},
      author={Vogt, Hendrik},
      author={Voigt, J\"urgen},
       title={Form methods for evolution equations, and applications},
   publisher={18th Internet Seminar on Evolution Equations},
        date={2015},
        note={URL:
  \url{https://www.mat.tuhh.de/veranstaltungen/isem18/pdf/LectureNotes.pdf}.
  Last visited on 2022/04/20},
}

\bib{Allard1972}{article}{
      author={Allard, William~K.},
       title={On the first variation of a varifold},
        date={1972},
     journal={Ann. of Math. (2)},
      volume={95},
       pages={417\ndash 491},
}

\bib{Almgren1965}{unpublished}{
      author={Almgren, Frederick},
       title={The theory of varifolds},
        date={1965},
        note={Mimeographed notes, Princeton},
}

\bib{AmbrosettiMalchiodi2007}{book}{
      author={Ambrosetti, Antonio},
      author={Malchiodi, Andrea},
       title={Nonlinear analysis and semilinear elliptic problems},
      series={Cambridge Studies in Advanced Mathematics},
   publisher={Cambridge University Press, Cambridge},
        date={2007},
      volume={104},
}

\bib{AmbrosettiRabinowitz1973}{article}{
      author={Ambrosetti, Antonio},
      author={Rabinowitz, Paul~H.},
       title={Dual variational methods in critical point theory and
  applications},
        date={1973},
     journal={J. Functional Analysis},
      volume={14},
       pages={349\ndash 381},
}

\bib{AlmgrenSimon1979}{article}{
      author={Almgren, Frederick~J., Jr.},
      author={Simon, Leon},
       title={Existence of embedded solutions of {P}lateau's problem},
        date={1979},
     journal={Ann. Scuola Norm. Sup. Pisa Cl. Sci. (4)},
      volume={6},
      number={3},
       pages={447\ndash 495},
}

\bib{BerardMeyer1982}{article}{
      author={B\'{e}rard, Pierre},
      author={Meyer, Daniel},
       title={In\'{e}galit\'{e}s isop\'{e}rim\'{e}triques et applications},
        date={1982},
     journal={Ann. Sci. \'{E}cole Norm. Sup. (4)},
      volume={15},
      number={3},
       pages={513\ndash 541},
}

\bib{BuzanoNguyenSchulz2021}{article}{
      author={Buzano, Reto},
      author={Nguyen, Huy~The},
      author={Schulz, Mario~B.},
       title={Noncompact self-shrinkers for mean curvature flow with arbitrary
  genus},
        date={2021},
     journal={preprint arXiv:2110.06027},
}

\bib{Brezis2011}{book}{
      author={Brezis, Haim},
       title={Functional analysis, {S}obolev spaces and partial differential
  equations},
      series={Universitext},
   publisher={Springer, New York},
        date={2011},
}

\bib{Brothers1986}{inproceedings}{
      author={Brothers, J.~E.},
       title={Some open problems in geometric measure theory and its
  applications suggested by participants of the 1984 {AMS} summer institute},
        date={1986},
   booktitle={Proc. {S}ympos. {P}ure {M}ath.},
   publisher={Amer. Math. Soc., Providence, RI},
       pages={441\ndash 464},
}

\bib{BokowskiSperner1979}{article}{
      author={Bokowski, J\"{u}rgen},
      author={Sperner, Emanuel, Jr.},
       title={Zerlegung konvexer {K}\"{o}rper durch minimale
  {T}rennfl\"{a}chen},
        date={1979},
     journal={J. Reine Angew. Math.},
      volume={311(312)},
       pages={80\ndash 100},
}

\bib{Calabi1967}{article}{
      author={Calabi, Eugenio},
       title={Minimal immersions of surfaces in {E}uclidean spheres},
        date={1967},
     journal={J. Differential Geometry},
      volume={1},
       pages={111\ndash 125},
}

\bib{Carlotto2015}{article}{
      author={Carlotto, Alessandro},
       title={Generic finiteness of minimal surfaces with bounded {M}orse
  index},
        date={2017},
     journal={Ann. Sc. Norm. Super. Pisa Cl. Sci. (5)},
      volume={17},
      number={3},
       pages={1153\ndash 1171},
}

\bib{CarlottoChodoshEichmair2016}{article}{
      author={Carlotto, Alessandro},
      author={Chodosh, Otis},
      author={Eichmair, Michael},
       title={Effective versions of the positive mass theorem},
        date={2016},
     journal={Invent. Math.},
      volume={206},
      number={3},
       pages={975\ndash 1016},
}

\bib{ColdingDeLellis2003}{incollection}{
      author={Colding, Tobias~H.},
      author={De~Lellis, Camillo},
       title={The min-max construction of minimal surfaces},
        date={2003},
   booktitle={Surveys in differential geometry, {V}ol. {VIII} ({B}oston, {MA},
  2002)},
      series={Surv. Differ. Geom.},
      volume={8},
   publisher={Int. Press, Somerville, MA},
       pages={75\ndash 107},
}

\bib{ColdingDeLellis2003Errata}{incollection}{
      author={Colding, Tobias~H.},
      author={De~Lellis, Camillo},
       title={{T}he min-max construction of minimal surfaces -- {E}rrata},
        date={2003},
   booktitle={Surveys in differential geometry},
      series={Surv. Differ. Geom.},
   publisher={Int. Press, Somerville, MA},
        note={URL:
  \url{https://www.math.ias.edu/delellis/sites/math.ias.edu.delellis/files/Errata_MinMax_survey_0.pdf}.
  Last visited on 2022/04/20},
}

\bib{ColdingDeLellis2005}{article}{
      author={Colding, Tobias~H.},
      author={De~Lellis, Camillo},
       title={Singular limit laminations, {M}orse index, and positive scalar
  curvature},
        date={2005},
     journal={Topology},
      volume={44},
      number={1},
       pages={25\ndash 45},
}

\bib{CarlottoFranz2020}{article}{
      author={Carlotto, Alessandro},
      author={Franz, Giada},
       title={Inequivalent complexity criteria for free boundary minimal
  surfaces},
        date={2020},
     journal={Adv. Math.},
      volume={373},
       pages={107322, 68 pp},
}

\bib{CarlottoFranzSchulz2020}{article}{
      author={Carlotto, Alessandro},
      author={Franz, Giada},
      author={Schulz, Mario},
       title={Free boundary minimal surfaces with connected boundary and
  arbitrary genus},
        date={2020},
     journal={preprint arXiv:2001.04920},
}

\bib{ColdingGabaiKetover2018}{article}{
      author={Colding, Tobias~Holck},
      author={Gabai, David},
      author={Ketover, Daniel},
       title={On the classification of {H}eegaard splittings},
        date={2018},
     journal={Duke Math. J.},
      volume={167},
      number={15},
       pages={2833\ndash 2856},
}

\bib{Chavel1984}{book}{
      author={Chavel, Isaac},
       title={Eigenvalues in {R}iemannian geometry},
      series={Pure and Applied Mathematics},
   publisher={Academic Press, Inc., Orlando, FL},
        date={1984},
      volume={115},
}

\bib{CooperHodgsonKerckhoff2000}{book}{
      author={Cooper, Daryl},
      author={Hodgson, Craig~D.},
      author={Kerckhoff, Steven~P.},
       title={Three-dimensional orbifolds and cone-manifolds},
      series={MSJ Memoirs},
   publisher={Mathematical Society of Japan, Tokyo},
        date={2000},
      volume={5},
}

\bib{ChodoshKetoverMaximo2017}{article}{
      author={Chodosh, Otis},
      author={Ketover, Daniel},
      author={Maximo, Davi},
       title={Minimal hypersurfaces with bounded index},
        date={2017},
     journal={Invent. Math.},
      volume={209},
      number={3},
       pages={617\ndash 664},
}

\bib{CarlottoLi2019}{article}{
      author={Carlotto, Alessandro},
      author={Li, Chao},
       title={Constrained deformations of positive scalar curvature metrics},
        date={2019},
     journal={J. Differential Geom. (\emph{to appear}), preprint
  arXiv:1903.11772},
}

\bib{ColdingMinicozzi2004IV}{article}{
      author={Colding, Tobias~H.},
      author={Minicozzi, William~P., II},
       title={The space of embedded minimal surfaces of fixed genus in a
  3-manifold. {IV}. {L}ocally simply connected},
        date={2004},
     journal={Ann. of Math. (2)},
      volume={160},
      number={2},
       pages={573\ndash 615},
}

\bib{ColdingMinicozzi2011}{book}{
      author={Colding, Tobias~Holck},
      author={Minicozzi, William~P., II},
       title={A course in minimal surfaces},
      series={Graduate Studies in Mathematics},
   publisher={American Mathematical Society, Providence, RI},
        date={2011},
      volume={121},
}

\bib{ChodoshMaximo2016}{article}{
      author={Chodosh, Otis},
      author={Maximo, Davi},
       title={On the topology and index of minimal surfaces},
        date={2016},
     journal={J. Differential Geom.},
      volume={104},
      number={3},
       pages={399\ndash 418},
}

\bib{ChodoshMaximo2018}{article}{
      author={Chodosh, Otis},
      author={Maximo, Davi},
       title={On the topology and index of minimal surfaces {II}},
        date={2018},
     journal={J. Differential Geom. (\emph{to appear}), preprint
  arXiv:1808.06572},
}

\bib{Courant1940}{article}{
      author={Courant, R.},
       title={The existence of minimal surfaces of given topological structure
  under prescribed boundary conditions},
        date={1940},
     journal={Acta Math.},
      volume={72},
       pages={51\ndash 98},
}

\bib{Courant1977}{book}{
      author={Courant, Richard},
       title={Dirichlet's principle, conformal mapping, and minimal surfaces},
   publisher={Springer-Verlag, New York-Heidelberg},
        date={1977},
}

\bib{ChoiSchoen1985}{article}{
      author={Choi, Hyeong~In},
      author={Schoen, Richard},
       title={The space of minimal embeddings of a surface into a
  three-dimensional manifold of positive {R}icci curvature},
        date={1985},
     journal={Invent. Math.},
      volume={81},
      number={3},
       pages={387\ndash 394},
}

\bib{CarlottoSchulzWiygul2022}{article}{
      author={Carlotto, Alessandro},
      author={Schulz, Mario},
      author={Wiygul, David},
       title={Infinitely many pairs of free boundary minimal surfaces with the
  same topology and symmetry group},
        date={2022},
     journal={preprint arXiv:2205.04861},
}

\bib{Devyver2019}{article}{
      author={Devyver, Baptiste},
       title={Index of the critical catenoid},
        date={2019},
     journal={Geom. Dedicata},
      volume={199},
       pages={355\ndash 371},
}

\bib{DeLellisPellandini2010}{article}{
      author={De~Lellis, Camillo},
      author={Pellandini, Filippo},
       title={Genus bounds for minimal surfaces arising from min-max
  constructions},
        date={2010},
     journal={J. Reine Angew. Math.},
      volume={644},
       pages={47\ndash 99},
}

\bib{DeLellisTasnady2013}{article}{
      author={De~Lellis, Camillo},
      author={Tasnady, Dominik},
       title={The existence of embedded minimal hypersurfaces},
        date={2013},
     journal={J. Differential Geom.},
      volume={95},
      number={3},
       pages={355\ndash 388},
}

\bib{EjiriMicallef2008}{article}{
      author={Ejiri, Norio},
      author={Micallef, Mario},
       title={Comparison between second variation of area and second variation
  of energy of a minimal surface},
        date={2008},
     journal={Adv. Calc. Var.},
      volume={1},
      number={3},
       pages={223\ndash 239},
}

\bib{ElSoufiIlias2000}{article}{
      author={El~Soufi, Ahmad},
      author={Ilias, Sa\"{\i}d},
       title={Riemannian manifolds admitting isometric immersions by their
  first eigenfunctions},
        date={2000},
     journal={Pacific J. Math.},
      volume={195},
      number={1},
       pages={91\ndash 99},
}

\bib{Evans2010PDE}{book}{
      author={Evans, Lawrence~C.},
       title={Partial differential equations},
     edition={Second},
      series={Graduate Studies in Mathematics},
   publisher={American Mathematical Society, Providence, RI},
        date={2010},
      volume={19},
}

\bib{FischerColbrie1985}{article}{
      author={Fischer-Colbrie, D.},
       title={On complete minimal surfaces with finite {M}orse index in
  three-manifolds},
        date={1985},
     journal={Invent. Math.},
      volume={82},
      number={1},
       pages={121\ndash 132},
}

\bib{FischerColbrieSchoen1980}{article}{
      author={Fischer-Colbrie, Doris},
      author={Schoen, Richard},
       title={The structure of complete stable minimal surfaces in
  {$3$}-manifolds of nonnegative scalar curvature},
        date={1980},
     journal={Comm. Pure Appl. Math.},
      volume={33},
      number={2},
       pages={199\ndash 211},
}

\bib{Finn1986}{book}{
      author={Finn, Robert},
       title={Equilibrium capillary surfaces},
      series={Fundamental Principles of Mathematical Sciences},
   publisher={Springer-Verlag, New York},
        date={1986},
      volume={284},
}

\bib{FraserLi2014}{article}{
      author={Fraser, Ailana},
      author={Li, Martin Man-chun},
       title={Compactness of the space of embedded minimal surfaces with free
  boundary in three-manifolds with nonnegative {R}icci curvature and convex
  boundary},
        date={2014},
     journal={J. Differential Geom.},
      volume={96},
      number={2},
       pages={183\ndash 200},
}

\bib{FolhaPacardZolotareva2017}{article}{
      author={Folha, Abigail},
      author={Pacard, Frank},
      author={Zolotareva, Tatiana},
       title={Free boundary minimal surfaces in the unit 3-ball},
        date={2017},
     journal={Manuscripta Math.},
      volume={154},
      number={3-4},
       pages={359\ndash 409},
}

\bib{Franz2021}{article}{
      author={Franz, Giada},
       title={Equivariant index bound for min-max free boundary minimal
  surfaces},
        date={2021},
     journal={preprint arXiv:2110.01020},
}

\bib{Freitag2011}{book}{
      author={Freitag, Eberhard},
       title={Complex analysis. 2},
      series={Universitext},
   publisher={Springer, Heidelberg},
        date={2011},
}

\bib{FraserSchoen2011}{article}{
      author={Fraser, Ailana},
      author={Schoen, Richard},
       title={The first {S}teklov eigenvalue, conformal geometry, and minimal
  surfaces},
        date={2011},
     journal={Adv. Math.},
      volume={226},
      number={5},
       pages={4011\ndash 4030},
}

\bib{FraserSchoen2013}{article}{
      author={Fraser, Ailana},
      author={Schoen, Richard},
       title={Minimal surfaces and eigenvalue problems, \emph{Geometric
  analysis, mathematical relativity, and nonlinear partial differential
  equations, 105--121}},
        date={2013},
     journal={Contemp. Math.},
      volume={599},
}

\bib{FraserSchoen2016}{article}{
      author={Fraser, Ailana},
      author={Schoen, Richard},
       title={Sharp eigenvalue bounds and minimal surfaces in the ball},
        date={2016},
     journal={Invent. Math.},
      volume={203},
      number={3},
       pages={823\ndash 890},
}

\bib{GittinsHelffer2019}{article}{
      author={Gittins, Katie},
      author={Helffer, Bernard},
       title={Courant-sharp {R}obin eigenvalues for the square and other planar
  domains},
        date={2019},
     journal={Port. Math.},
      volume={76},
      number={1},
       pages={57\ndash 100},
}

\bib{GirouardLagace2021}{article}{
      author={Girouard, Alexandre},
      author={Lagac\'{e}, Jean},
       title={Large {S}teklov eigenvalues via homogenisation on manifolds},
        date={2021},
     journal={Invent. Math.},
      volume={226},
      number={3},
       pages={1011\ndash 1056},
}

\bib{GromovLawson1980Spin}{article}{
      author={Gromov, Mikhael},
      author={Lawson, H.~Blaine, Jr.},
       title={Spin and scalar curvature in the presence of a fundamental group.
  {I}},
        date={1980},
     journal={Ann. of Math. (2)},
      volume={111},
      number={2},
       pages={209\ndash 230},
}

\bib{GulliverLawson1986}{incollection}{
      author={Gulliver, Robert},
      author={Lawson, H.~Blaine, Jr.},
       title={The structure of stable minimal hypersurfaces near a
  singularity},
        date={1986},
   booktitle={Geometric measure theory and the calculus of variations
  ({A}rcata, {C}alif., 1984)},
      series={Proc. Sympos. Pure Math.},
      volume={44},
   publisher={Amer. Math. Soc., Providence, RI},
       pages={213\ndash 237},
}

\bib{GuangLiWangZhou2021}{article}{
      author={Guang, Qiang},
      author={Li, Martin Man-chun},
      author={Wang, Zhichao},
      author={Zhou, Xin},
       title={Min-max theory for free boundary minimal hypersurfaces {II}:
  general {M}orse index bounds and applications},
        date={2021},
     journal={Math. Ann.},
      volume={379},
      number={3-4},
       pages={1395\ndash 1424},
}

\bib{GuangLiZhou2020}{article}{
      author={Guang, Qiang},
      author={Li, Martin Man-chun},
      author={Zhou, Xin},
       title={Curvature estimates for stable free boundary minimal
  hypersurfaces},
        date={2020},
     journal={J. Reine Angew. Math.},
      volume={759},
       pages={245\ndash 264},
}

\bib{GuangWangZhou2021}{article}{
      author={Guang, Qiang},
      author={Wang, Zhichao},
      author={Zhou, Xin},
       title={Compactness and generic finiteness for free boundary minimal
  hypersurfaces, {I}},
        date={2021},
     journal={Pacific J. Math.},
      volume={310},
      number={1},
       pages={85\ndash 114},
}

\bib{Hartman1964}{article}{
      author={Hartman, Philip},
       title={Geodesic parallel coordinates in the large},
        date={1964},
     journal={Amer. J. Math.},
      volume={86},
       pages={705\ndash 727},
}

\bib{Hirsch1994}{book}{
      author={Hirsch, Morris~W.},
       title={Differential topology},
      series={Graduate Texts in Mathematics},
   publisher={Springer-Verlag, New York},
        date={1994},
      volume={33},
}

\bib{HassNorburyRubinstein2003}{article}{
      author={Hass, Joel},
      author={Norbury, Paul},
      author={Rubinstein, J.~Hyam},
       title={Minimal spheres of arbitrarily high {M}orse index},
        date={2003},
     journal={Comm. Anal. Geom.},
      volume={11},
      number={3},
       pages={425\ndash 439},
}

\bib{Hsiang1983}{article}{
      author={Hsiang, Wu-Yi},
       title={Minimal cones and the spherical {B}ernstein problem. {I}},
        date={1983},
     journal={Ann. of Math. (2)},
      volume={118},
      number={1},
       pages={61\ndash 73},
}

\bib{Ilmanen1998}{book}{
      author={Ilmanen, Tom},
       title={Lectures on mean curvature flow and related equations},
        date={1998},
}

\bib{IrieMarquesNeves2018}{article}{
      author={Irie, Kei},
      author={Marques, Fernando~C.},
      author={Neves, Andr\'{e}},
       title={Density of minimal hypersurfaces for generic metrics},
        date={2018},
     journal={Ann. of Math. (2)},
      volume={187},
      number={3},
       pages={963\ndash 972},
}

\bib{Ketover2016Equivariant}{article}{
      author={Ketover, Daniel},
       title={Equivariant min-max theory},
        date={2016},
     journal={preprint arXiv:1612.08692},
}

\bib{Ketover2016FBMS}{article}{
      author={Ketover, Daniel},
       title={Free boundary minimal surfaces of unbounded genus},
        date={2016},
     journal={preprint arXiv:1612.08691},
}

\bib{Ketover2019}{article}{
      author={Ketover, Daniel},
       title={Genus bounds for min-max minimal surfaces},
        date={2019},
     journal={J. Differential Geom.},
      volume={112},
      number={3},
       pages={555\ndash 590},
}

\bib{KarpukhinKokarevPolterovich2014}{article}{
      author={Karpukhin, Mikhail},
      author={Kokarev, Gerasim},
      author={Polterovich, Iosif},
       title={Multiplicity bounds for {S}teklov eigenvalues on {R}iemannian
  surfaces},
        date={2014},
     journal={Ann. Inst. Fourier (Grenoble)},
      volume={64},
      number={6},
       pages={2481\ndash 2502},
}

\bib{KapouleasLi2017}{article}{
      author={Kapouleas, Nikolaos},
      author={Li, Martin Man-chun},
       title={Free boundary minimal surfaces in the unit three-ball via
  desingularization of the critical catenoid and the equatorial disc},
        date={2021},
     journal={J. Reine Angew. Math.},
      volume={776},
       pages={201\ndash 254},
}

\bib{KapouleasMcGrath2020}{article}{
      author={Kapouleas, Nikolaos},
      author={McGrath, Peter},
       title={Generalizing the linearized doubling approach, {I}: {G}eneral
  theory and new minimal surfaces and self-shrinkers},
        date={2020},
     journal={preprint arXiv:2001.04240},
}

\bib{KetoverMarquesNeves2020}{article}{
      author={Ketover, Daniel},
      author={Marques, Fernando~C.},
      author={Neves, Andr\'{e}},
       title={The catenoid estimate and its geometric applications},
        date={2020},
     journal={J. Differential Geom.},
      volume={115},
      number={1},
       pages={1\ndash 26},
}

\bib{KapouleasWiygul2017}{article}{
      author={Kapouleas, Nikolaos},
      author={Wiygul, David},
       title={Free-boundary minimal surfaces with connected boundary in the
  $3$-ball by tripling the equatorial disc},
        date={2017},
     journal={preprint arXiv:1711.00818},
}

\bib{KapouleasWiygul2020}{article}{
      author={Kapouleas, Nikolaos},
      author={Wiygul, David},
       title={The index and nullity of the {L}awson surfaces {$\xi_{g,1}$}},
        date={2020},
     journal={Camb. J. Math.},
      volume={8},
      number={2},
       pages={363\ndash 405},
}

\bib{KapouleasZou2021}{article}{
      author={Kapouleas, Nikolaos},
      author={Zou, Jiahua},
       title={Free boundary minimal surfaces in the {E}uclidean three-ball
  close to the boundary},
        date={2021},
     journal={preprint arXiv:2111.11308},
}

\bib{Lawson1970}{article}{
      author={Lawson, H.~Blaine, Jr.},
       title={Complete minimal surfaces in {$S^{3}$}},
        date={1970},
     journal={Ann. of Math. (2)},
      volume={92},
       pages={335\ndash 374},
}

\bib{Li2015}{article}{
      author={Li, Martin Man-chun},
       title={A general existence theorem for embedded minimal surfaces with
  free boundary},
        date={2015},
     journal={Comm. Pure Appl. Math.},
      volume={68},
      number={2},
       pages={286\ndash 331},
}

\bib{Li2020}{inproceedings}{
      author={Li, Martin Man-chun},
       title={Free boundary minimal surfaces in the unit ball: recent advances
  and open questions},
        date={2020},
   booktitle={Proceedings of the {I}nternational {C}onsortium of {C}hinese
  {M}athematicians 2017},
   publisher={Int. Press, Boston, MA},
       pages={401\ndash 435},
}

\bib{Lima2017}{article}{
      author={Lima, Vanderson},
       title={Bounds for the {M}orse index of free boundary minimal surfaces},
        date={2017},
     journal={Asian J. Math. (\emph{to appear}), preprint arXiv:1710.10971},
}

\bib{Liu2021}{article}{
      author={Liu, Zhenhua},
       title={The existence of embedded {$G$}-invariant minimal hypersurface},
        date={2021},
     journal={Calc. Var. Partial Differential Equations},
      volume={60},
      number={1},
       pages={Paper No. 36, 21},
}

\bib{LiokumovichMarquesNeves2018}{article}{
      author={Liokumovich, Yevgeny},
      author={Marques, Fernando~C.},
      author={Neves, Andr\'{e}},
       title={Weyl law for the volume spectrum},
        date={2018},
     journal={Ann. of Math. (2)},
      volume={187},
      number={3},
       pages={933\ndash 961},
}

\bib{LiZhou2021}{article}{
      author={Li, Martin Man-chun},
      author={Zhou, Xin},
       title={Min-max theory for free boundary minimal hypersurfaces
  {I}---{R}egularity theory},
        date={2021},
     journal={J. Differential Geom.},
      volume={118},
      number={3},
       pages={487\ndash 553},
}

\bib{Maggi2012}{book}{
      author={Maggi, Francesco},
       title={Sets of finite perimeter and geometric variational problems},
      series={Cambridge Studies in Advanced Mathematics},
   publisher={Cambridge University Press, Cambridge},
        date={2012},
      volume={135},
}

\bib{Maximo2018}{article}{
      author={Maximo, Davi},
       title={A note on minimal surfaces with bounded index},
        date={2018},
     journal={preprint arXiv:1812.10728},
}

\bib{Miao2002}{article}{
      author={Miao, Pengzi},
       title={Positive mass theorem on manifolds admitting corners along a
  hypersurface},
        date={2002},
     journal={Adv. Theor. Math. Phys.},
      volume={6},
      number={6},
       pages={1163\ndash 1182 (2003)},
}

\bib{MarquesNeves2014}{article}{
      author={Marques, Fernando~C.},
      author={Neves, Andr\'{e}},
       title={Min-max theory and the {W}illmore conjecture},
        date={2014},
     journal={Ann. of Math. (2)},
      volume={179},
      number={2},
       pages={683\ndash 782},
}

\bib{MarquesNeves2016}{article}{
      author={Marques, Fernando~C.},
      author={Neves, Andr\'{e}},
       title={Morse index and multiplicity of min-max minimal hypersurfaces},
        date={2016},
     journal={Camb. J. Math.},
      volume={4},
      number={4},
       pages={463\ndash 511},
}

\bib{MarquesNeves2017}{article}{
      author={Marques, Fernando~C.},
      author={Neves, Andr\'{e}},
       title={Existence of infinitely many minimal hypersurfaces in positive
  {R}icci curvature},
        date={2017},
     journal={Invent. Math.},
      volume={209},
      number={2},
       pages={577\ndash 616},
}

\bib{MarquesNeves2021}{article}{
      author={Marques, Fernando~C.},
      author={Neves, Andr\'{e}},
       title={Morse index of multiplicity one min-max minimal hypersurfaces},
        date={2021},
     journal={Adv. Math.},
      volume={378},
       pages={107527, 58},
}

\bib{MatthiesenPetrides2020}{article}{
      author={Matthiesen, Henrik},
      author={Petrides, Romain},
       title={Free boundary minimal surfaces of any topological type in
  {E}uclidean balls via shape optimization},
        date={2020},
     journal={preprint arXiv:2004.06051},
}

\bib{MeeksPerezRos2016}{article}{
      author={Meeks, William~H., III},
      author={P\'{e}rez, Joaqu\'{\i}n},
      author={Ros, Antonio},
       title={Local removable singularity theorems for minimal laminations},
        date={2016},
     journal={J. Differential Geom.},
      volume={103},
      number={2},
       pages={319\ndash 362},
}

\bib{MeeksSimonYau1982}{article}{
      author={Meeks, William, III},
      author={Simon, Leon},
      author={Yau, Shing~Tung},
       title={Embedded minimal surfaces, exotic spheres, and manifolds with
  positive {R}icci curvature},
        date={1982},
     journal={Ann. of Math. (2)},
      volume={116},
      number={3},
       pages={621\ndash 659},
}

\bib{Nadirashvili1996}{article}{
      author={Nadirashvili, N.},
       title={Berger's isoperimetric problem and minimal immersions of
  surfaces},
        date={1996},
     journal={Geom. Funct. Anal.},
      volume={6},
      number={5},
       pages={877\ndash 897},
}

\bib{Nitsche1985}{article}{
      author={Nitsche, Johannes C.~C.},
       title={Stationary partitioning of convex bodies},
        date={1985},
     journal={Arch. Rational Mech. Anal.},
      volume={89},
      number={1},
       pages={1\ndash 19},
}

\bib{Osserman1964}{article}{
      author={Osserman, Robert},
       title={Global properties of minimal surfaces in {$E^{3}$} and
  {$E^{n}$}},
        date={1964},
     journal={Ann. of Math. (2)},
      volume={80},
       pages={340\ndash 364},
}

\bib{Palais1979}{article}{
      author={Palais, Richard~S.},
       title={The principle of symmetric criticality},
        date={1979},
     journal={Comm. Math. Phys.},
      volume={69},
      number={1},
       pages={19\ndash 30},
}

\bib{Pitts1981}{book}{
      author={Pitts, Jon~T.},
       title={Existence and regularity of minimal surfaces on {R}iemannian
  manifolds},
      series={Mathematical Notes},
   publisher={Princeton University Press, Princeton, N.J.; University of Tokyo
  Press, Tokyo},
        date={1981},
      volume={27},
}

\bib{PittsRubinstein1988}{article}{
      author={Pitts, Jon~T.},
      author={Rubinstein, J.~H.},
       title={Equivariant minimax and minimal surfaces in geometric
  three-manifolds},
        date={1988},
     journal={Bull. Amer. Math. Soc. (N.S.)},
      volume={19},
      number={1},
       pages={303\ndash 309},
}

\bib{Ros2005}{incollection}{
      author={Ros, Antonio},
       title={The isoperimetric problem},
        date={2005},
   booktitle={Global theory of minimal surfaces},
      series={Clay Math. Proc.},
      volume={2},
   publisher={Amer. Math. Soc., Providence, RI},
       pages={175\ndash 209},
}

\bib{Ros2006}{article}{
      author={Ros, Antonio},
       title={One-sided complete stable minimal surfaces},
        date={2006},
     journal={J. Differential Geom.},
      volume={74},
      number={1},
       pages={69\ndash 92},
}

\bib{Sargent2017}{article}{
      author={Sargent, Pam},
       title={Index bounds for free boundary minimal surfaces of convex
  bodies},
        date={2017},
     journal={Proc. Amer. Math. Soc.},
      volume={145},
      number={6},
       pages={2467\ndash 2480},
}

\bib{Savo2010}{article}{
      author={Savo, Alessandro},
       title={Index bounds for minimal hypersurfaces of the sphere},
        date={2010},
     journal={Indiana Univ. Math. J.},
      volume={59},
      number={3},
       pages={823\ndash 837},
}

\bib{SchulzGallery}{misc}{
      author={Schulz, Mario},
       title={Geometric analysis gallery},
        date={2022},
        note={URL: \url{https://mbschulz.github.io/}. Last visited on
  2022/04/20},
}

\bib{Schoen1983Estimates}{incollection}{
      author={Schoen, Richard},
       title={Estimates for stable minimal surfaces in three-dimensional
  manifolds},
        date={1983},
   booktitle={Seminar on minimal submanifolds},
      series={Ann. of Math. Stud.},
      volume={103},
   publisher={Princeton Univ. Press, Princeton, NJ},
       pages={111\ndash 126},
}

\bib{Schoen1983Uniqueness}{article}{
      author={Schoen, Richard~M.},
       title={Uniqueness, symmetry, and embeddedness of minimal surfaces},
        date={1983},
     journal={J. Differential Geom.},
      volume={18},
      number={4},
       pages={791\ndash 809 (1984)},
}

\bib{Smith1982}{thesis}{
      author={Smith, Francis},
       title={On the existence of embedded minimal $2$-spheres in the
  $3$–sphere, endowed with an arbitrary riemannian metric},
        type={Ph.D. Thesis},
        date={1982},
}

\bib{Song2018}{article}{
      author={Song, Antoine},
       title={Existence of infinitely many minimal hypersurfaces in closed
  manifolds},
        date={2018},
     journal={preprint arXiv:1806.08816},
}

\bib{Song2020}{article}{
      author={Song, Antoine},
       title={Morse index, {B}etti numbers and singular set of bounded area
  minimal hypersurfaces},
        date={2019},
     journal={preprint arXiv:1911.09166},
}

\bib{SchoenSimon1981}{article}{
      author={Schoen, Richard},
      author={Simon, Leon},
       title={Regularity of stable minimal hypersurfaces},
        date={1981},
     journal={Comm. Pure Appl. Math.},
      volume={34},
      number={6},
       pages={741\ndash 797},
}

\bib{ShiohamaShioyaTanaka2003}{book}{
      author={Shiohama, Katsuhiro},
      author={Shioya, Takashi},
      author={Tanaka, Minoru},
       title={The geometry of total curvature on complete open surfaces},
      series={Cambridge Tracts in Mathematics},
   publisher={Cambridge University Press, Cambridge},
        date={2003},
      volume={159},
}

\bib{Struwe2008}{book}{
      author={Struwe, Michael},
       title={Variational methods},
     edition={Fourth Edition},
      series={Results in Mathematics and Related Areas. 3rd Series. A Series of
  Modern Surveys in Mathematics},
   publisher={Springer-Verlag, Berlin},
        date={2008},
      volume={34},
}

\bib{SchoenYau2017}{article}{
      author={Schoen, R.},
      author={Yau, S.-T.},
       title={Positive scalar curvature and minimal hypersurface
  singularities},
        date={2017},
     journal={preprint arXiv:1704.05490},
}

\bib{SchoenYau1979PMT}{article}{
      author={Schoen, Richard},
      author={Yau, Shing~Tung},
       title={On the proof of the positive mass conjecture in general
  relativity},
        date={1979},
     journal={Comm. Math. Phys.},
      volume={65},
      number={1},
       pages={45\ndash 76},
}

\bib{SchoenYau1983}{article}{
      author={Schoen, Richard},
      author={Yau, S.~T.},
       title={The existence of a black hole due to condensation of matter},
        date={1983},
     journal={Comm. Math. Phys.},
      volume={90},
      number={4},
       pages={575\ndash 579},
}

\bib{SchoenYau1994}{book}{
      author={Schoen, R.},
      author={Yau, S.-T.},
       title={Lectures on differential geometry},
      series={Conference Proceedings and Lecture Notes in Geometry and
  Topology, I},
   publisher={International Press, Cambridge, MA},
        date={1994},
}

\bib{SmithZhou2019}{article}{
      author={Smith, Graham},
      author={Zhou, Detang},
       title={The {M}orse index of the critical catenoid},
        date={2019},
     journal={Geom. Dedicata},
      volume={201},
       pages={13\ndash 19},
}

\bib{Tran2020}{article}{
      author={Tran, Hung},
       title={Index characterization for free boundary minimal surfaces},
        date={2020},
     journal={Comm. Anal. Geom.},
      volume={28},
      number={1},
       pages={189\ndash 222},
}

\bib{Volkmann2015}{thesis}{
      author={Volkmann, Alexander},
       title={Free boundary problems governed by mean curvature},
        type={Ph.D. Thesis},
        date={2015},
}

\bib{Wang2019}{article}{
      author={Wang, Zhichao},
       title={Existence of minimal hypersurfaces with non-empty free boundary
  for generic metrics},
        date={2019},
     journal={preprint arXiv:1909.01787},
}

\bib{Wang2020}{article}{
      author={Wang, Tongrui},
       title={Min-max theory for {$G$}-invariant minimal hypersurfaces},
        date={2020},
     journal={preprint arXiv:2009.10995},
}

\bib{White2009}{article}{
      author={White, Brian},
       title={Which ambient spaces admit isoperimetric inequalities for
  submanifolds?},
        date={2009},
     journal={J. Differential Geom.},
      volume={83},
      number={1},
       pages={213\ndash 228},
}

\bib{White2016}{article}{
      author={White, Brian},
       title={Introduction to minimal surface theory},
        date={2016},
      volume={22},
       pages={387\ndash 438},
}

\bib{White2017}{article}{
      author={White, Brian},
       title={On the bumpy metrics theorem for minimal submanifolds},
        date={2017},
     journal={Amer. J. Math.},
      volume={139},
      number={4},
       pages={1149\ndash 1155},
}

\bib{White1986}{article}{
      author={White, Brian},
       title={Complete surfaces of finite total curvature},
        date={1987},
     journal={J. Differential Geom.},
      volume={26},
      number={2},
       pages={315\ndash 326},
}

\bib{White1987Curvature}{article}{
      author={White, Brian},
       title={Curvature estimates and compactness theorems in {$3$}-manifolds
  for surfaces that are stationary for parametric elliptic functionals},
        date={1987},
     journal={Invent. Math.},
      volume={88},
      number={2},
       pages={243\ndash 256},
}

\bib{White1987Space}{article}{
      author={White, Brian},
       title={The space of {$m$}-dimensional surfaces that are stationary for a
  parametric elliptic functional},
        date={1987},
     journal={Indiana Univ. Math. J.},
      volume={36},
      number={3},
       pages={567\ndash 602},
}

\end{biblist}
\end{bibdiv}

%
%

\end{document}